\documentclass[onefignum,onetabnum]{siamart251104}

\usepackage{amsfonts}
\usepackage{thmtools}
\usepackage{amsopn}
\usepackage{amsmath}
\usepackage{amssymb}
\usepackage{multirow}
\usepackage{tabularx}
\usepackage{hhline}
\usepackage{scalerel,stackengine}
\usepackage{arydshln}
\usepackage[normalem]{ulem}
\usepackage{soul}
\usepackage{tikz}
\usepackage{pgf}
\usepackage{collcell}
\usepackage[utf8]{inputenc}
\usepackage{blkarray}
\usepackage{graphbox}
\usepackage[many]{tcolorbox}
\usepackage{euscript}
\usepackage{graphicx}
\usepackage{subcaption}
\usepackage{dsfont}

  \definecolor{grassgreen}{RGB}{92,135,39}
  \definecolor{lightgreen}{HTML}{07a56d}
  \definecolor{coolgray}{HTML}{829798}
  \definecolor{solidblue}{HTML}{454f78}
  \definecolor{lightblue}{cmyk}{0.88,0.44,0,0.07}

\newcolumntype{E}{>{\collectcell\usermacro}c<{\endcollectcell}}
\newcommand\usermacro[1]{\pgfmathparse{10000*#1}\pgfmathprintnumber\pgfmathresult}

\usepackage{algorithmicx}
\usepackage{algorithm}
\usepackage{algpseudocode}

\makeatletter
\patchcmd{\ALG@step}{\addtocounter{ALG@line}{1}}{\refstepcounter{ALG@line}}{}{}
\newcommand{\ALG@lineautorefname}{Step}
\makeatother

\usepackage{cite}
\usepackage[shortlabels]{enumitem}

\DeclareMathAlphabet{\mathup}{OT1}{\familydefault}{m}{n}

\DeclareSymbolFont{yhlargesymbols}{OMX}{yhex}{m}{n}
\DeclareMathAccent{\wideparen}{\mathord}{yhlargesymbols}{"F3}
\xdefinecolor{green}{rgb}{0.04, 0.85, 0.32}
\xdefinecolor{darkgreen}{rgb}{0.24, 0.7, 0.44}
\xdefinecolor{mint}{rgb}{0.24, 0.71, 0.54}
\xdefinecolor{officegreen}{rgb}{0.0, 0.5, 0.0}
\xdefinecolor{napiergreen}{rgb}{0.16, 0.5, 0.0}
\xdefinecolor{maroon}{rgb}{0.65,0.06,0.17}
\xdefinecolor{blue}{rgb}{0,0.2,0.6}
\xdefinecolor{phthalogreen}{rgb}{0.07,0.21,0.14}
\definecolor{grassgreen}{RGB}{92,135,39}

\crefname{algocf}{alg.}{algs.}
\Crefname{algocf}{Algorithm}{Algorithms}

\newcommand*{\Perm}[2]{{}^{#1}\!P_{#2}}  
\newcommand{\Oh}[1]{\mathcal{O}{\left(#1\right)}}
\newcommand{\logdet}[1]{\log\det\left(#1\right)}                    
\newcommand{\brlog}[1]{\log\left(#1\right)}                 
\newcommand{\del}[2]{\frac{\partial{#1}}{\partial{#2}}}             
\newcommand{\mat}[1]{\mathbf{{#1}}}                                 
\renewcommand{\vec}[1]{\mathbf{{#1}}}                                 

\DeclareMathOperator*{\argmin}{arg\,min}  
\DeclareMathOperator*{\argmax}{arg\,max}  

\newcommand{\proj}{\mathit{P}}       
\newcommand{\Proj}[2]{\proj_{#1}{\left(#2\right)}}       

\newcommand*{\opt}{^{\mkern-1.5mu\mathrm{opt}}}               
\newcommand*{\tran}{^{\mkern-1.5mu\mathsf{T}}}                
\newcommand*{\adj}{^{\mkern-1.5mu\mathsf{*}}}                 
\newcommand*{\inv}{^{\mkern-1.5mu\mathsf{-1}}}                
\newcommand{\domain}{\mathcal{D}}                             
\newcommand{\trace}{\mathrm{Tr}}                              
\newcommand{\Trace}[1]{\trace \left(#1\right)}                
\newcommand{\abs}[1]{\left| {#1} \right|}                  
\newcommand{\wnorm}[2]{\left\| {#1} \right\|_{#2}}            

\newcommand\restr[2]{{ \left.\kern-\nulldelimiterspace        
                     {#1}\vphantom{\big|} \right|_{#2}}}

\newcommand{\Rnum}{\mathbb{R}}  
\newcommand{\Znum}{\mathbb{Z}}  
\newcommand{\xcont}{u}                             

\renewcommand{\O}{\mathcal{O}}                                   
\newcommand{\ObsOperCont}{\O} 

\newcommand{\ObsOperind}[1]{\ObsOperCont_{#1}}

\newcommand{\y}{\mathbf{y}}                        
\newcommand{\obs}{\y}                              

\newcommand{\param}{\vec{\theta}}                  
\newcommand{\iparam}{\param}                       
\newcommand{\iparprior}{\iparam_{\rm pr}}          
\newcommand{\iparb}{\iparprior}                    
\newcommand{\iparpost}{\iparam_{\rm post}^\obs}    
\newcommand{\ipara}{\iparpost}                     

\newcommand{\Nparam}{{N_{\rm p}}}
\newcommand{\Nobs}{\textsc{N}_{\rm obs}}                        
\newcommand{\Nens}{\textsc{N}_{\rm ens}}                        
\newcommand{\nobs}{n_{t}}                                  
\newcommand{\nobstimes}{\nobs}                             
\newcommand{\Nsens}{N}                         
\newcommand{\Cparamprior}{\mat\Gamma_{{\rm pr}}}                
\newcommand{\Cparampost}{\mat\Gamma_{{\rm post}}}               
\newcommand{\Cobsnoise}{\mat{\Gamma}_{ {\rm noise}}}            
\newcommand{\Cparampriormat}{\Cparamprior}                      
\newcommand{\Cparampostmat}{\Cparampost}                        

\newcommand{\F}{\mathbf{F}}                                      
\newcommand{\Find}[2]{\F_{#1,#2}}

\newcommand{\Sol}{\mathcal{S}}  
\newcommand{\Solind}[2]{\Sol_{#1,#2}}

\newcommand{\utilityfunc}{\mathcal{U}}

\newcommand{\prob}{\mathbb{P}}                                   
\newcommand{\Prob}[1]{\prob{\left(#1\right)}}                    
\newcommand{\CondProb}[2]{\prob\left(#1\,|\,#2 \right)}  
\newcommand{\nCondProb}[3]{\prob^{(#1)}\left(#2\,|\,#3 \right)}  
\newcommand{\Var}{\mathsf{var}}                                  
\newcommand{\brVar}[1]{\Var\Bigl(#1\Bigr)}                         
\newcommand{\GM}[2]{\mathcal{N}\!\left( {#1}, {#2}\right)}       

\newcommand{\Expect}[2]{\mathbb{E}_{#1}{\left[ #2 \right]} }     
\newcommand{\baseline}{b}
\newcommand{\hyperparam}{p}     
\newcommand{\hyperparamvec}{\vec{\hyperparam}}     
\newcommand{\design}{\vec{\zeta}}                                       
\newcommand{\designvec}{\boldsymbol{\zeta}}                                       

\newcommand{\designdomain}{\Omega_{\design}}         

\newcommand{\directsum}[3]{\bigoplus\limits_{#1}^{#2}{\left(#3\right)}}

\stackMath
\newcommand\reallywidehat[1]{%
\savestack{\tmpbox}{\stretchto{%
  \scaleto{%
    \scalerel*[\widthof{\ensuremath{#1}}]{\kern-.6pt\bigwedge\kern-.6pt}%
    {\rule[-\textheight/2]{1ex}{\textheight}}
  }{\textheight}%
}{0.5ex}}%
\stackon[1pt]{#1}{\tmpbox}%
}
\newcommand{\commentout}[1]{\iffalse {#1} \fi}

\ifpdf
  \DeclareGraphicsExtensions{.eps,.pdf,.png,.jpg}
\else
  \DeclareGraphicsExtensions{.eps}
\fi

\newcommand{\fulltitle}{A Probabilistic Approach to Trajectory-Based Optimal Experimental Design}
\newcommand{\shorttitle}{Trajectory-Based OED}

\newcommand{\fundingtext}{%
  This material is based upon 
  work supported by the U.S. Department of Energy, 
  under contract number DE-AC02-06CH11357, with funding through, 
  (1) Office of Science, Office of Advanced Scientific Computing  Research: 
  Scientific Discovery through Advanced Computing (SciDAC) Program through the FASTMath Institute; and
  (2) Competitive Portfolios Project on Energy Efficient Computing: A Holistic Methodology.
}

\newsiamremark{remark}{Remark}
\newsiamremark{hypothesis}{Hypothesis}
\crefname{hypothesis}{Hypothesis}{Hypotheses}
\newsiamthm{claim}{Claim}
\crefname{lemma}{Lemma}{Lemmas}

\tcbuselibrary{theorems}

\newtcbtheorem[
  auto counter,
  number within=section,
  crefname={problem}{Problem},
  Crefname={Problem}{Problem},
]{problem}{Problem}{
  enhanced,
  sharp corners,
  attach boxed title to top left={
    xshift=5pt,
    yshift*=-\tcboxedtitleheight-0.5mm,
  },
  colback=white,
  colframe=coolgray,
  fonttitle=\bfseries,
  coltitle=white,
  boxed title style={
     frame code={\fill[tcbcolback]([xshift=-5pt]frame.north west) [rounded corners=2pt]--(frame.south west)--(frame.south east)[sharp corners]--([xshift=5pt]frame.north east)--cycle;},
    size=small,
    colback=coolgray,
    colframe=coolgray,
    drop shadow=coolgray!50!white,
  },
  leftrule=0pt,
  rightrule=0pt,
  breakable,
}{prbl}

\newtcbtheorem[
  auto counter,
  number within=section,
  crefname={definition}{Definition},
  Crefname={Definition}{Definition},
]{defn}{Definition}{
  enhanced,
  sharp corners,
  attach boxed title to top left={
    xshift=5pt,
    yshift*=-\tcboxedtitleheight-0.5mm,
  },
  colback=white,
  colframe=lightblue,
  fonttitle=\bfseries,
  coltitle=white,
  boxed title style={
     frame code={\fill[tcbcolback]([xshift=-5pt]frame.north west) [rounded corners=2pt]--(frame.south west)--(frame.south east)[sharp corners]--([xshift=5pt]frame.north east)--cycle;},
    size=small,
    colback=lightblue,
    colframe=lightblue,
    drop shadow=lightblue!50!white,
  },
  leftrule=0pt,
  rightrule=0pt,
  breakable,
}{defn}

\patchcmd{\SetTagPlusEndMark}{\$}{}{}{}
\patchcmd{\SetTagPlusEndMark}{\$}{}{}{}

\headers{\shorttitle 
  }{Ahmed Attia}

\title{\fulltitle\thanks{Submitted to the editors \today.
\funding{\fundingtext}
}}

\author{Ahmed Attia\thanks{Mathematics and Computer Science Division,
                   Argonne National Lab, USA
                  ( \email{aattia@anl.gov} ).}
}

\AtBeginDocument{
  \let\namerefOld\nameref
  \renewcommand{\nameref}[1]{\textit{\namerefOld{#1}}}
}

\ifpdf
    \hypersetup{
      pdftitle={\fulltitle},
      pdfauthor={Ahmed Attia}
    }
\fi

\begin{document}
  \maketitle

  \begin{abstract}
  We present a novel probabilistic approach for optimal path experimental design. In this approach a discrete path optimization problem is defined on a static navigation mesh, and trajectories are modeled as random variables governed by a parametric Markov policy. The discrete path optimization problem is then replaced with an equivalent stochastic optimization problem over the policy parameters, resulting in an optimal probability model that samples estimates of the optimal discrete path. This approach enables exploration of the utility function’s distribution tail and treats the utility function of the design as a black box, making it applicable to linear and nonlinear inverse problems and beyond experimental design. Numerical verification and analysis are carried out by using a parameter identification problem widely used in model-based optimal experimental design, namely a two-dimensional time-dependent advection diffusion problem in which the initial condition is the inference target. Experiments use both coarse and fine navigation meshes, with either a single moving sensor or a group of seven coordinated sensors, and the proposed approach is evaluated under D-, A-, and E-optimality criteria.
\end{abstract}




\section{Introduction} 
  \label{sec:introduction}
  Simulation models are critical for predicting the behavioral patterns of physical phenomena given 
  the initial condition or the model state. 
  These models, however, require calibration based on snapshots (observations) of 
  the physical phenomena of concern.
  Inverse problems enable estimating the unknown model parameters from partial 
  and noisy observations \cite{law2015data,asch2016data,attia2017hybrid,attia2017reduced}.
  Such observational data can be obtained from static observational sensors 
  or moving sensors such as relocatable sensors and land and aerial vehicles such as drones and robots. 
  Inverse problems underpin many applications, and their solution 
  quality depends on data quality, 
  making optimal data acquisition essential for accurate inference and prediction.

  Model-based optimal experimental design (OED) \cite{FedorovLee00,Pazman86,Pukelsheim93}
  is the general mathematical formalism 
  for optimizing the configuration of an experimental 
  or observational setup in simulation systems.
  An experimental design $\design\in\designdomain$ 
  encodes 
  the decision variables such as the spatial distribution of observational instruments, 
  the temporal frequency of data acquisition, 
  or the trajectory of moving sensors.
  OED problems generally seek a design that ``optimizes'' some 
  ``utility function'' $\utilityfunc(\design)$ that quantifies 
  the quality of the design $\design$. 
  The process includes \emph{minimizing} the variance of an inference parameter estimator 
  or \emph{maximizing} the expected information gain  
  \cite{Alexanderian21,attia2018goal,huan2014gradient}. 

  Unlike OED for static sensor placement 
  \cite{huan2024optimal,Alexanderian21}, 
  trajectory OED for mobile sensors remains underexplored.
  Most existing trajectory OED formulations \cite{tricaud2008d,Ucinski05,rafajlowicz1986optimum} 
  employ a parametrized form of the trajectory, such as ordinary differential equations,  
  and follow a gradient-based optimization approach to
  tune the parameters of the trajectory.
  A recent development pursues this gradient-based, parametrized route for
  infinite-dimensional Bayesian inverse problems governed by PDEs \cite{neuberger2026path}.
  Although valuable, parametrized trajectory development confines solutions to a specific trajectory family, limiting applicability.
  Moreover, incorporating hard constraints such as those imposed 
  by the domain is a significant challenge.
  Additionally, this approach mandates development of utility- and observation-specific gradients.

  Path planning for mobile sensors has been extensively investigated 
  in applications such as robotics and autonomous vehicles; 
  see, for example, \cite{lin2022review,latombe2012robot,lavalle2006planning,karur2021survey,gasparetto2015path,choset2005principles}.
  These methods are application-specific, however, and are not directly 
  applicable to model-based OED.
  For example, graph-based path optimization \cite{hart1968formal} methods 
  focus primarily on finding a minimum-cost path between two specific 
  locations (nodes on a graph). 
  While such methods can in principle be adapted to model-based OED, they 
  require the initial and final points to be specified, and they require 
  an OED utility function that balances information gain and path length. 
  Without such balance, classical shortest-path formulations would not 
  penalize longer paths and could favor arbitrarily long paths to maximize 
  the information gain.

  We propose a probabilistic path optimization framework 
  with black-box utility functions, 
  enabling out-of-the-box use with arbitrary criteria and observation patterns.
  The feasible domain of the mobile sensor is encoded as a directed graph 
  (a navigation mesh), and the trajectory of the sensor is modeled as a 
  random variable governed by a parametric Markov policy. 
  The initial position of the sensor is drawn from a categorical distribution 
  over graph vertices, and subsequent transitions are drawn from per-vertex 
  categorical distributions parameterized by the policy.
  The original discrete path OED problem is then reformulated as a stochastic 
  optimization problem over the policy parameters, which we solve using a 
  REINFORCE-style stochastic gradient with an optimal baseline for variance 
  reduction.
  Rather than providing only a single optimal path, the resulting 
  probabilistic policy samples near-optimal paths, thereby enabling 
  exploration of the tail (for example, the upper tail for maximization) 
  of the OED utility function distribution.

  \subsection*{Contributions}
    The main contributions of this work are summarized as follows.
    \begin{enumerate}[leftmargin=*]
      \item We define the path OED problem as a discrete optimization problem 
        over a graph. 
        The discrete path problem is then reformulated as a stochastic optimization problem 
        over probabilistic policy parameters, yielding an optimal probabilistic model.
      \item We propose three probabilistic policy models 
        based on first- and higher-order Markov chains, 
        analyze their complexity–cost trade-offs, 
        and derive stochastic optimization gradients with variance reduction.
      \item We provide an efficient algorithmic approach for solving the probabilistic 
        path optimization problem with emphasis on scalability to large-scale 
        inverse problems. 
      \item All developments in this work are made publicly available
        through PyOED \cite{chowdhary2025pyoed}.
    \end{enumerate}

  The rest of the paper is structured as follows.
  \Cref{sec:background} presents an overview of the mathematical background of standard OED, 
  and the generic probabilistic optimization framework.
  \Cref{sec:problem_formulation} describes the proposed path OED approach.
  Numerical results are 
  in \Cref{sec:numerical_experiments}, and 
  concluding remarks are given in \Cref{sec:conclusions}.

\section{Background: Model-Based OED} 
  \label{sec:background}
    An OED problem takes the  
    form
    \begin{equation}\label{eqn:OED_optimization}
        \designvec\opt \in \argmax_{\designvec \in \designdomain}\, \utilityfunc(\designvec) \,,
    \end{equation}
    where $\designvec\in \designdomain$ is an experimental design variable that encodes tunable 
    configurations of the system. 
    Here, $\utilityfunc$ is a utility function---possibly application-dependent---that 
    quantifies the quality of the design and generally requires multiple evaluations 
    of the dynamical system in model-based OED problems.
    Typical choices of  $\utilityfunc$ in Bayesian inference applications
    include the alphabetic criteria, such as A- and D-optimality \cite{Ucinski05,attia2018goal}, 
    which target minimizing posterior uncertainty of the inference parameter.

    The  design $\designvec$ in many 
    OED problems is binary 
    or discrete because it is typically associated with observational/experimental data that 
    is generally defined on a finite domain. 
    Solving the  OED optimization problem \eqref{eqn:OED_optimization} exactly is thus 
    a major challenge.
    Greedy and exchange-type approaches \cite{fedorov2013theory,wynn1972results} 
    have been traditionally followed for approximating the solution of 
    \eqref{eqn:OED_optimization} because of their simplicity and practical behavior. 
    Convergence guarantees and scalability of such algorithms, however, are limited.
    Thus the optimization problem \eqref{eqn:OED_optimization} is often
    replaced with an alternative formulation based on continuous relaxation of the design 
    $\designvec$. The relaxed formulation is then solved by following a gradient-based
    optimization approach, and the resulting design is rounded to approximate the solution 
    of the original problem \eqref{eqn:OED_optimization}, 
    as highlighted in \cite{attia2022optimal}.
    This approach requires developing the gradient of the utility function 
    with respect to the relaxed design $\nabla_{\designvec}\utilityfunc$ and thus requires the 
    utility function and any additional constraints such as sparsity-enforcing constraints, 
    to be differentiable with respect to the relaxed design $\designvec$.
    %

    \subsection{Probabilistic Optimization Approach}
    \label{subsec:probabilistic_optimization_approach}
      We adopt the 
      probabilistic optimization 
      approach \cite{attia2022stochastic,attia2025robust}, which replaces \eqref{eqn:OED_optimization}
      with the policy optimization problem: 
      \begin{equation}\label{eqn:probabilisitic_optimization}
          \hyperparamvec\opt \in \argmax_{\hyperparamvec} 
            \Expect{\designvec\sim\CondProb{\designvec}{\hyperparamvec}}{\utilityfunc(\designvec)} \,,
      \end{equation}
      where $\designvec$ is modeled as a random variable with 
      a parametric policy $\CondProb{\designvec}{\hyperparamvec}$. 
      A gradient ascent approach solves \eqref{eqn:probabilisitic_optimization}, 
      yielding parameters $\hyperparamvec\opt$ of an optimal policy 
      from which 
      $\designvec\sim\CondProb{\designvec}{\hyperparamvec\opt}$ 
      are sampled to approximate the original OED solution  \eqref{eqn:OED_optimization}.
      A stochastic approximation  of the gradient 
      $\nabla_{\hyperparamvec}\Expect{}{\utilityfunc(\design)}$ 
      is employed to solve \eqref{eqn:probabilisitic_optimization}:
      \begin{subequations}\label{eqn:generic_gradient}
        \begin{align}
          \vec{g}
            &=\nabla_{\hyperparamvec}
              \Expect{\designvec\sim\CondProb{\designvec}{\hyperparamvec}}{\utilityfunc(\designvec)} 
              =  
                \Expect{\designvec\sim\CondProb{\designvec}{\hyperparamvec}}{
                  \utilityfunc(\designvec) \nabla_{\hyperparamvec} \log{\CondProb{\designvec}{\hyperparamvec}}
                }  \label{eqn:full_gradient} \\
          &\approx 
            \frac{1}{\Nens} \sum_{k=1}^{\Nens} 
              \utilityfunc(\designvec[k]) \, \nabla_{\hyperparamvec} \log{\CondProb{\designvec[k]}{\hyperparamvec}} 
            := 
            \widehat{\vec{g}}
              \,, 
        \label{eqn:generic_stochastic_gradient}
        \end{align}
        with 
        $
          \designvec[k] \sim \CondProb{\designvec[k]}{\hyperparamvec}\,, \, k=1,\ldots, \Nens \,, 
        $ sampled from the parametric policy.
        Because only evaluations of $\utilityfunc$ at sampled designs are required, the probabilistic
        approach \eqref{eqn:probabilisitic_optimization} is applicable to black-box 
        objectives $\utilityfunc$.
      \end{subequations}

      The identity \eqref{eqn:full_gradient} follows by interchanging the 
      gradient with the (finite) sum defining the expectation and applying 
      the log-derivative identity. Specifically, since $\designvec$ ranges 
      over the finite set of feasible designs,
      \begin{equation*}
        \nabla_{\hyperparamvec} 
          \sum_{\designvec} \utilityfunc(\designvec)\, \CondProb{\designvec}{\hyperparamvec}
        = \sum_{\designvec} \utilityfunc(\designvec)\, 
             \nabla_{\hyperparamvec} \CondProb{\designvec}{\hyperparamvec}
        = \sum_{\designvec} \utilityfunc(\designvec)\, 
             \CondProb{\designvec}{\hyperparamvec}\, 
             \nabla_{\hyperparamvec} \log \CondProb{\designvec}{\hyperparamvec} \,,
      \end{equation*}
      which recovers \eqref{eqn:full_gradient} on rewriting the last sum as 
      an expectation. The same identity holds when $\designvec$ is a 
      continuous random variable, with the sums replaced by integrals.
      This score-function form of the gradient underlies the probabilistic OED 
      approach pursued in \cite{attia2022stochastic,attia2024probabilistic,attia2025robust}.

      This probabilistic approach requires a policy $\CondProb{\designvec}{\hyperparamvec}$ 
      that models the design $\designvec$ and design constraints \cite{attia2024probabilistic}, 
      along with a sampling procedure to generate $\designvec[k] \sim \CondProb{\designvec[k]}{\hyperparamvec}$, 
      and the gradient of the 
      log-probability $\nabla_{\hyperparamvec} \log{\CondProb{\designvec[k]}{\hyperparamvec}}$ 
      for stochastic optimization.

\section{Probabilistic Approach to Optimal Path OED}
  \label{sec:problem_formulation}
      We formulate the trajectory OED problem as a discrete graph 
      optimization (\Cref{prbl:path_oed_problem}), 
      where nodes represent feasible sensor locations and edges encode allowable 
      motion between them.
      \begin{problem}{Discrete Path OED Problem on a Graph}{path_oed_problem}
        Consider a navigation mesh represented by a directed graph $G=(V, E)$,  
        with vertices $V=\{v_1, \ldots, v_{N}\}$, 
        and arcs 
        $E\subseteq V \times V$.
        The path OED problem seeks a trajectory 
        $\designvec=\left(\design_1, \ldots, \design_{n}\right)  \in V^{n}$ 
        that solves
        \begin{equation}
          \label{eqn:path_OED}
          \designvec\opt\in 
            \argmax_{\designvec=\left(\design_1, \ldots, \design_{n}\right)  \in V^{n}}
              {\utilityfunc(\designvec)} \,,
        \end{equation}
        where $V^{n}=V\times\cdots\times V$ 
        contains  all possible trajectories of length $n$, and 
        $\utilityfunc(\designvec)$ is a utility function that quantifies the quality of a path $\designvec$.
      \end{problem}

      Note the following about \Cref{prbl:path_oed_problem}. 
      (1) Although related to navigation mesh path planning \cite{brandao2021explaining}, 
      this framework optimizes path OED objectives rather than shortest 
      paths between prescribed nodes.
      (2) The directed graph encodes admissible single-step transitions 
      between candidate sensor locations. Preservation of this single-step 
      constraint by the trajectory policies considered in this work is 
      addressed in \Cref{subsubsec:lag_weights}.
      (3) \Cref{prbl:path_oed_problem} is abstract and does not 
      prescribe trajectory–model synchronization; 
      when the design specifies an observational path, 
      graph transitions correspond to motion between candidate locations 
      over successive observation times. 
      Path synchronization is application-specific and is addressed in \Cref{sec:numerical_experiments}.
      (4) The utility function $\utilityfunc$ is assumed to be 
      a black box function that can be defined, for example,
      based on the information gain aggregated over the path.

      Standard discrete solution methods for \Cref{prbl:path_oed_problem}
      are infeasible due 
      to exponential design-space growth,  
      motivating the probabilistic formulation \eqref{eqn:probabilisitic_optimization}.

    Leveraging \eqref{eqn:probabilisitic_optimization} for path OED, requires 
    a parametric policy $\CondProb{\designvec}{\hyperparamvec}$ that defines the 
    path probability. 
    We propose three  policies in \Cref{subsec:the_probabilistic_models}. 
    A probabilistic optimization approach that utilizes these policies 
    is then described in \Cref{subsec:the_probabilistic_optimization}.

    \subsection{Probabilistic Policies for the Experimental Path}
    \label{subsec:the_probabilistic_models}
      The experimental path $\designvec$ 
      defines the sensor locations 
      at successive time instances. 
      At any time instance $t_{k}$, a moving sensor can exist at 
      one and only one location (vertex $v_i$).
      For the first-order policy of \Cref{subsubsec:first_order_MC_policy}, 
      the sensor can then move at the next time instance $t_{k+1}$ to a 
      vertex $v_j$ only if there is an edge connecting $v_i$ to $v_j$, that 
      is, $(v_i, v_j)\in E$. This single-step adjacency constraint is relaxed 
      by the higher-order policies of \Cref{subsubsec:higher_order_MC_policy}; 
      see \Cref{subsubsec:lag_weights}.
      A random trajectory on a graph can be modeled by using a discrete Markov chain (MC) with 
      the general probability model
      \begin{equation}\label{eqn:trajectory_probability}
        \Prob{\designvec} 
          = \Prob{\design_1, \design_2, \ldots, \design_{n-1}}
          = \Prob{\design_1} 
            \CondProb{\design_2}{\design_1} 
            \CondProb{\design_3}{\design_1, \design_2} \ldots 
            \CondProb{\design_{n}}{\design_{1}\ldots \design_{n-1}} \,,
      \end{equation}
      where
      the initial distribution $\Prob{\design_1}$ describes the probability of starting the 
      trajectory at any vertex $v_i$. 
      The conditional transition probability 
      $\CondProb{\design_{n}}{\design_{1}\ldots \design_{n-1}}$  
      describes the 
      probability of moving from $\design_{n-1}$ to $\design_{n}$ 
      given the full history $(\design_1, \ldots, \design_{n-1})$. 

      While a first-order chain assumes that the next location depends only on the current one, 
      high-order chains capture dependencies on multiple previous steps.
      We first develop a first-order policy based on memoryless MCs 
      in \Cref{subsubsec:first_order_MC_policy}, 
      then extend the discussion to higher-order MC models in \Cref{subsubsec:higher_order_MC_policy}.
      An  example of the proposed policies
      is given in the supplementary material (see \Cref{app:sec:Example}). 

      \subsubsection{First-Order Policy}
      \label{subsubsec:first_order_MC_policy}
        A memoryless
        MC assumes that 
        the conditional transition probability is given by
        $ 
          \CondProb{\design_{n}}{\design_{n-1}\ldots \design_1} 
            = \CondProb{\design_{n}}{\design_{n-1}}  
        $,  
        reducing \eqref{eqn:trajectory_probability} to
        \begin{equation}\label{eqn:first_order_trajectory_distribution}
          \Prob{\designvec} 
            = \Prob{\design_1} 
              \CondProb{\design_2}{\design_1} 
              \CondProb{\design_3}{\design_2} \ldots 
              \CondProb{\design_{n}}{\design_{n-1}}
            = \Prob{\design_1} \prod_{t=1}^{n-1} {\CondProb{\designvec_{t+1}}{\design_{i}}} \,. 
        \end{equation}

        Here ${\CondProb{\design_{i+1}}{\design_{i}}}$ describes 
        the probability of moving along the arc 
        $(v_{i}, v_{i+1})$ given that the sensor is at $v_i$.
        The matrix 
        $\mat{P}$ with $[\mat{P}]_{ij}
          =\CondProb{\designvec_{t+1}=v_j}{\designvec_{t}=v_i}
        $,  for any time $t$ and for all $v_i, v_j \in V$, 
        is called the transition probability matrix.
        
        The first-order trajectory distribution \eqref{eqn:first_order_trajectory_distribution} 
        requires modeling both the initial and the transition distribution.
        These models are developed next.
        
        \paragraph{Modeling the initial distribution}
          The initial distribution is given by
          \begin{equation}\label{eqn:initial_distribution}
            \boldsymbol{\pi} = (\pi_1, \pi_2, \ldots, \pi_{N}) \,;  
              \quad \text{such that}\quad  0\leq \pi_i \leq 1 , \,\, \text{ and } \,\,
              \sum_{i=1}^{N} \pi_i = 1 \,,
          \end{equation}
          where $\pi_i=\Prob{\design_1=v_i}$ is the probability that the 
          trajectory starts at $v_i\in V$.
          Starting at any node can be modeled as a Bernoulli experiment (start at the node or not). 
          Since only one location is allowed, 
          the initial distribution can be defined by employing first-order inclusion probabilities 
          \cite{sunter1986solutions}. 
          Specifically, assuming that each node $v_i$ is associated with a 
          Bernoulli probability 
          $\hyperparam_i\in(0, 1)$,  
          the initial distribution 
          is defined as
          \begin{equation}\label{eqn:initial_distribution_inclusion}
            \pi_i \equiv \Prob{\design_1=v_i} 
              = \frac{w_i}{\sum_{j=1}^{N} {w_j}} 
                \,; \quad 
                w_i = \frac{\hyperparam_i}{1-\hyperparam_i}   
                \,; \quad 
                \hyperparam_i\in (0, 1) 
                \,,\quad  
                i=1, 2, \ldots, \Nsens 
              \,,
          \end{equation}
          where, by construction, the inclusion probabilities satisfy
          $\sum_{i=1}^{N} \pi_i = 1$.
          Because the inclusion probabilities depend on the weights only
          through the normalization in \eqref{eqn:initial_distribution_inclusion},
          the map $\hyperparamvec\mapsto\pi$ is invariant under a common
          rescaling $w_i\mapsto k\,w_i$ of all weights---equivalently, along the
          corresponding curve in $\hyperparamvec$---so this over-parameterization
          leaves the policy, and hence the expected utility, unchanged, and thus
          does not affect the optimization.

          Initial parameters of the optimization procedure can be set, e.g., to
          $\hyperparam_i=0.5$, with preferences for certain nodes adjusted as needed.
          Note that in \eqref{eqn:initial_distribution_inclusion} and 
          throughout \Cref{subsec:the_probabilistic_models} 
          we assume 
          non-degenerate parameters $\hyperparam_i\in(0, 1)$.
          The case of boundary (degenerate) configurations
          $\hyperparam_i\in\{0,1\}$ are addressed uniformly in
          \Cref{subsubsec:degenerate_case}.

        \paragraph{Modeling the transition distribution}
          %
          The conditional transition distribution
          \begin{equation}\label{eqn:transition_distribution}
            \boldsymbol{\pi}^{(i)} = (\pi_1^{i}, \ldots, \pi_{N}^{i}) \,;  
              \quad 
              \sum_{j=1}^{N} \pi_j^{i} = 1  \,,\quad
              0\leq \pi^{i}_j \leq 1   \,,\quad  
              \pi^{i}_j =0 \text{ for } (v_i, v_j)\notin E  
              \,,
          \end{equation}
          defines $\pi_{j}^{i}=\CondProb{v_j}{v_i}\equiv\CondProb{\design_{t+1}=v_j}{\design_{t}=v_i}$,  
          the probability of moving in one step from 
          $v_i$ to $v_j$. 
          Assuming a stationary chain, the transition probability $\pi_{j}^{i}$ 
          is independent of time or the position of $(v_i,\, v_j)$ in the trajectory.

          The transition distribution \eqref{eqn:transition_distribution} 
          can be modeled by following the same approach
          employed for modeling the initial distribution. 
          Specifically, we associate with each vertex $v_i$ a vector of Bernoulli 
          probabilities---modeling the possibility of either to transition or not--- 
          of moving to another node $v_j$. 
          Since moving to a node $v_j$ such that $(v_i, v_j)\notin E$ 
          has to be zero, we need to model only the probability of moving to a node $v_j$ in the 
          out-neighborhood $\EuScript{N}(v_i) := \{v_j \in V : (v_i, v_j) \in E\}$ of $v_i$. 
          Specifically, by letting 
          \begin{equation}
            \hyperparamvec^{(i)}=(\hyperparam_1^{i}, \ldots, \hyperparam_{N_i}^{i} ) \,; \qquad
              0\leq \hyperparam_{j}^{i} \leq 1 \,\forall\, v_j \in \EuScript{N}(v_i)\,,\quad 
              \hyperparam_{j}^{i} = 0  \,\forall\, v_j \notin \EuScript{N}(v_i)\,, 
          \end{equation}
          the transition probabilities associated with node $v_i$ are defined as
          \begin{equation}\label{eqn:transition_probabilities}
            \pi^{i}_j 
            = \CondProb{\designvec_{t+1}\!=\!v_j}{\designvec_{t}\!=\!v_i} 
              = \frac{w_j^{i}}{\sum_{j=1}^{N_i} w_j^{i}} \,; \quad 
                w_j^{i} = \frac{\hyperparam_j^{i}}{1-\hyperparam_j^{i}} 
                \,; \quad 
                \hyperparam_{j}^{i}\in (0, 1) 
                \,. 
          \end{equation}
          As with the initial distribution, we assume non-degenerate
          transition parameters $\hyperparam^{i}_{j}\in(0, 1)$ here;
          boundary cases $\hyperparam^{i}_{j}\in\{0,1\}$ are addressed in
          \Cref{subsubsec:degenerate_case}.

        Note that $\hyperparam_i$ is the probability of starting from 
        $v_i$ irrespective of the other nodes.  
        Conversely, $\pi_i$ is the probability that a trajectory exclusively starts at $v_i$.
        Similarly, given that the sensor is at $v_{i}$, $\pi^{i}_{j}$ is 
        the probability of moving exclusively to
        $v_{j}$.
        These distinctions reflect that a sensor can occupy or move to only one node at a time.

        \Cref{defn:first_order_path_model} formalizes the proposed first-order path policy. 

        \begin{defn}{
            First-Order 
            Path Policy
          }{first_order_path_model}
          The probability distribution of a path 
          $\designvec=(\design_1, \design_2, \ldots, \design_{n})$ of fixed size $n$, 
          parametrized by the parameter $\hyperparamvec \in (0, 1)^{N(N+1)}$,
          is given  by
          \begin{subequations}\label{eqn:first_order_trajectory_distribution_full}
           {\allowdisplaybreaks
            \begin{align}
              \label{eqn:first_order_trajectory_distribution_pmf}
              &\Prob{\designvec}
                = \Prob{\design_1} 
                  \prod_{t=1}^{n-1} {\CondProb{\designvec_{t+1}}{\designvec_{t}}} \,, \\
              \label{eqn:first_order_trajectory_distribution_param}
              &\hyperparamvec
                = 
                    \Bigl(
                    \underbrace{
                      \hyperparam_1, \ldots, \hyperparam_{N}
                    }_{\text{initial parameters}}; \, \,\, 
                    \underbrace{
                      \hyperparam^{1}_1, \ldots, \hyperparam^{1}_{N};\, 
                      \, \ldots\, ;\,\, 
                    \hyperparam^{N}_1, \ldots, \hyperparam^{N}_{N} 
                    }_{\text{transition parameters}}
                  \,\, \Bigr)  \,, \\
                &\Prob{\design_1\!=\!v_i} = \pi_i 
                  = \frac{w_i}{\sum\limits_{j=1}^{N}{w_j}} \,; \quad
                    w_i = \frac{\hyperparam_i}{1-\hyperparam_i} \,, \, \quad i=1, \ldots, \Nsens  \,,
                      \label{eqn:first_order_trajectory_distribution_init_pmf}
                      \\
                  &\CondProb{\designvec_{t+1}\!=\!v_j}{\design_t\!=\!v_i}
                  = \pi^{i}_{j} 
                  = \frac{w^{i}_{j}}{\sum\limits_{k=1}^{N}{w^{i}_{k}}} \,; \,\,
                    w^{i}_{j} =\frac{\hyperparam^{i}_{j}}{1-\hyperparam^{i}_{j}} 
                      \,, 
                      \begin{matrix}
                        & i = 1,\ldots,N \,,  \\[-1.0pt] 
                        & j = 1,\ldots,N  \,.
                      \end{matrix}
                      \label{eqn:first_order_trajectory_distribution_transition_pmf}
            \end{align}
            }
          \end{subequations}
        \end{defn}

        Although \Cref{defn:first_order_path_model} defines 
        $\hyperparamvec \in(0, 1)^{N^2+N}$, in practice 
        the effective parameter space is much smaller since each node connects only to a few neighbors.
        Thus, one  needs to store only the nonzero parameters, 
        reducing the parameter space to $N(C+1)$, where $C$ is the maximum neighborhood cardinality.

        The gradient of the log-probability 
        $\nabla_{\hyperparamvec}{\brlog{\Prob{\designvec}}}$ 
        is required for the probabilistic optimization approach and is summarized by  
        \Cref{proposition:first_order_gradient}.

        \begin{proposition}\label{proposition:first_order_gradient}
          The policy 
          \eqref{eqn:first_order_trajectory_distribution_full}
          has the following gradient of log-probability:
          \begin{subequations}\label{eqn:first_order_trajectory_distribution_grad_log_prob_full}
            \begin{equation}\label{eqn:first_order_trajectory_distribution_grad_log_prob_init}
              \nabla_{\hyperparamvec}{\brlog{\Prob{\designvec}}}
                = \nabla_{\hyperparamvec}{\brlog{\Prob{\design_1}}}
                  + \sum_{i=1}^{n-1}  \nabla_{\hyperparamvec}{\brlog{\CondProb{\design_{i+1}}{\design_{i}}}}
                  \,,
            \end{equation}
            with elementwise components of the gradient vector given by
            {\allowdisplaybreaks
            \begin{align}
              \del{\log \pi_i} {\hyperparam_j}
                 &= 
                  \del{\log \Prob{\design_1=v_{i}}} {\hyperparam_j}
                  = 
                 \frac{1}{\left(1\!-\!\hyperparam_j\right)^2} \left(
                      \frac{\delta_{ij}}{w_j} 
                     - \frac{ 1 }{\sum\limits_{k=1}^{N}{w_k}} 
                   \right) \,, 
                   \label{eqn:first_order_trajectory_distribution_grad_log_prob_initial} 
                   \\
                \del{ \log \pi^{i}_{j} }{\hyperparam^{l}_{m}}
                  &= \del{\log \CondProb{\designvec_{t+1}=v_j}{\designvec_{t}=v_i}}{\hyperparam^{l}_{m}}
                  = \frac{\delta_{il}}{\left(1\!-\!\hyperparam^{l}_{m}\right)^2} \left(
                      \frac{\delta_{jm}}{w^{l}_{m}} 
                      - \frac{ 1 }{\sum\limits_{k=1}^{N}{w^{l}_{k}}} 
                     \right) \,,
                   \label{eqn:first_order_trajectory_distribution_grad_log_prob_transition} 
            \end{align}
            }
            where $\delta_{ij}$ is the Kronecker delta function with $\delta_{ij}=1$ 
            when $i=j$ and  $0$ otherwise.
          \end{subequations}
        \end{proposition}
        \begin{proof}
          See \Cref{app:sec:gradients_proofs}.
        \end{proof}

        The first-order policy (\Cref{defn:first_order_path_model}) is memoryless, assuming the next node depends only on the current one. 
        Trajectories may benefit from memory, which can be captured by using higher-order Markov 
        chains \cite[Chapter 6]{ching2006markov} as discussed in \Cref{subsubsec:higher_order_MC_policy}.

      \subsubsection{Higher-Order Path Policies}
      \label{subsubsec:higher_order_MC_policy}
        The $k$th-order MC satisfies
        \begin{equation}\label{eqn:kth_order_MC}
          \CondProb{\design_{n}}{
            \design_{1},
            \design_{2},
            \ldots, 
            \design_{n-2}, 
            \design_{n-1} 
          }
          = \CondProb{\design_{n}}{
            \design_{n-k},
            \design_{n-k+1}
            \ldots, 
            \design_{n-2}, 
            \design_{n-1} 
          } \,, \,\, k\in\Znum^{+} \,,
        \end{equation}
        where the case of $k=1$ reduces to the first-order (memoryless) 
        MC \eqref{eqn:first_order_trajectory_distribution}.
       
        Standard higher-order MC is rarely used in practice because 
        the size of the transition tensor \eqref{eqn:kth_order_MC} 
        is $\Oh{N^{k}}$, which grows exponentially 
        with the order $k$.
        Practical approaches, such as the Raftery model \cite{raftery1985model}, extend a first-order 
        model by introducing a single additional parameter for each lag: 
        \begin{equation}\label{eqn:Raftery_model}
          \CondProb{\design_{n}}{
            \design_{n-k},
            \design_{n-k+1}
            \ldots, 
            \design_{n-2}, 
            \design_{n-1} 
          } 
            = \sum_{i=1}^{k}{\lambda_{i}\, \CondProb{\design_{n}}{\design_{n-i}}} \,; \quad 
              \sum_{i=1}^{k} \lambda_{i} = 1 \,,
        \end{equation}
        where $\CondProb{\design_{n}}{\design_{n-i}}$ is the one-step transition 
        probability \eqref{eqn:transition_probabilities}.
        This is similar to the first-order autoregressive model, where a linear combination 
        is taken over subsequent lags.
        A generalization of \eqref{eqn:Raftery_model} was developed in 
        \cite{ching2004higher} to allow the transition probability 
        $\CondProb{\design_{n}}{\design_{n-i}}$ to be different for each lag,
        affording more flexibility.
        
        In the remainder of \Cref{subsubsec:higher_order_MC_policy} we develop two 
        probabilistic policies for the trajectory based on the 
        Raftery model and its generalization.

        \paragraph{Path policy based on the Raftery model}
          Here the path probability \eqref{eqn:trajectory_probability} 
          is 
          \begin{equation}\label{eqn:kth-order-trajectory}\begin{aligned}
            \Prob{\designvec} 
              &= \Prob{\design_{1}, \ldots, \design_{n} } 
               = \CondProb{\design_{n}}{\design_{1}, \ldots, \design_{n-1}}
                  \Prob{\design_{1},  \ldots, \design_{n-1}}
                \\
                &= \Prob{\design_{1}, \design_{2}, \ldots, \design_{k}} 
                \prod_{t=k}^{n-1} 
                  \CondProb{
                    \designvec_{t+1}
                  }{
                      \designvec_{t-k+1},
                      \ldots, 
                      \designvec_{t-1}, 
                      \designvec_{t}
                  }
                \,,
          \end{aligned}
          \end{equation}
          where the conditional probability rule is applied repeatedly, and 
          \eqref{eqn:kth_order_MC} is used to obtain \eqref{eqn:kth-order-trajectory}.
          We employ the first-order policy to model the initial distribution 
          (over $k$ steps) 
          $\Prob{\design_{1}, \ldots, \design_{k}}$.
          The parameters $\hyperparamvec$ of this model would be the same as 
          \eqref{eqn:first_order_trajectory_distribution_param} 
          in addition to the lag parameters $\lambda_{1}, \ldots, \lambda_{k}$ as formalized by 
          \Cref{defn:higher_order_path_model}.

          \begin{defn}{Higher-Order Path Policy}{higher_order_path_model}
            For a positive integer 
            $k\in\Znum^{+}$, the $k$th-order probability 
            distribution of a path 
            $\designvec=(\design_1, \ldots, \design_{n})$ of fixed size $n>k$,
            parametrized by 
            $\hyperparamvec \in (0, 1)^{N(N+1) + k}$, is
            \begin{subequations}\label{eqn:higher_order_trajectory_distribution}
              {\allowdisplaybreaks
              \begin{align}
                &\Prob{\designvec}
                  = \Prob{\design_1} 
                      \prod_{t=1}^{k-1}{
                      \left(\CondProb{\designvec_{t+1}}{\designvec_{t}}
                      \right)}
                    \prod_{t=k}^{n-1} \left( 
                      \sum_{i=1}^{k} {\lambda_{i}\, \CondProb{\designvec_{t+1}}{\designvec_{t+1-i}} }
                      \right)
                   \,, 
                      \label{eqn:higher_order_trajectory_distribution_pmf}
                  \\ 
                  &\hyperparamvec
                  = 
                    \left(
                      \hyperparam_1, \ldots, \hyperparam_{N}; \, 
                      \hyperparam^{1}_1, \ldots, \hyperparam^{1}_{N},
                      \, \ldots\, 
                      \hyperparam^{N}_1, \ldots, \hyperparam^{N}_{N}; \, 
                      \lambda_1, \ldots, \lambda_{k} 
                    \right) \,; \, \sum_{i=1}^{k} \lambda_{i} = 1 \,, 
                   \label{eqn:higher_order_trajectory_distribution_param}
                   \\ 
                  &\Prob{\design_1\!=\!v_i}
                   =\pi_i 
                   = \frac{w_i}{\sum_{j=1}^{N}{w_j}} \,; \quad 
                      w_i=\frac{\hyperparam_i}{1-\hyperparam_i} \,, i=1, \ldots, \Nsens  \,,
                        \label{eqn:higher_order_trajectory_initial_state}
                        \\
                  &\CondProb{v_j}{v_i}
                  = \pi^{i}_{j} 
                  = \frac{w^{i}_{j}}{\sum_{k=1}^{N}{w^{i}_{k}}} \,; \quad 
                      w^{i}_{j}=\frac{\hyperparam^{i}_{j}}{1-\hyperparam^{i}_{j}} 
                        \,, 
                        \begin{matrix}
                          & i = 1,\ldots,N \,,  \\[-1.0pt] 
                          & j = 1,\ldots,N  \,,
                        \end{matrix}
                        \label{eqn:higher_order_trajectory_distribution_transition}
              \end{align}
              }
            \end{subequations}
            where 
            $
              \CondProb{\design_{t+z}=v_j}{\design_{t}=v_i}
              = \pi^{i}_{j} 
            $ for any positive integer $z$.
          \end{defn}

          \begin{proposition}\label{proposition:higher_order_gradient}
            The gradient of the log-policy given by 
            \eqref{eqn:higher_order_trajectory_distribution}
            is given by
            \begin{subequations}\label{eqn:higher_order_trajectory_distribution_grad_log_prob}
            \begin{equation}\label{eqn:higher_order_path_distribution_grad_log_prob_operator}\small
              \left(
                \del{}{\hyperparam_{1}}, \ldots, \del{}{\hyperparam_{\Nsens}};\,\, 
                \del{}{\hyperparam^{1}_{1}}, \ldots, \del{}{\hyperparam^{1}_{\Nsens}};\,\ldots;\, \,
                \del{}{\hyperparam^{\Nsens}_{1}}, \ldots, \del{}{\hyperparam^{\Nsens}_{\Nsens}};\, \,
                \del{}{\lambda_{1}}, \ldots, \del{}{\lambda_{k}}
              \right)\,  \log\Prob{\designvec} \,,
            \end{equation}
            with elementwise components given as follows:

              {\allowdisplaybreaks
              \begin{align}
                \del{\log \Prob{\designvec} } {\hyperparam_j}
                  &= \del{\log \Prob{\design_{1} }} {\hyperparam_j}  
                   \,, \label{eqn:higher_order_path_distribution_grad_log_prob_operator_initial}
                   \\
                  \del{\log \Prob{\designvec}}{\hyperparam^{l}_{m}}
                  &= \sum_{t=1}^{k-1} 
                    \del{\log \CondProb{\designvec_{t+1}}{\designvec_{t}}}{\hyperparam^{l}_{m}}
                    + 
                    \sum_{t=k}^{n-1} 
                      \frac{ 
                        \sum_{i=1}^{k} {
                            \lambda_{i}\, 
                              \del{\CondProb{\designvec_{t+1}}{\designvec_{t+1-i}}}{\hyperparam^{l}_{m}}  
                          }
                      }{
                          \sum_{i=1}^{k} {\lambda_{i}\, \CondProb{\designvec_{t+1}}{\designvec_{t+1-i}} }
                      }\label{eqn:higher_order_path_distribution_grad_log_prob_operator_transitions}
                      \\
                  \del{\log \Prob{\designvec}}{\lambda_{l}}
                  &= \sum_{t=k}^{n-1} 
                  \frac{
                        \CondProb{\designvec_{t+1}}{\designvec_{t+1-l}} 
                      }{
                        \sum_{i=1}^{k} {\lambda_{i}\, \CondProb{\designvec_{t+1}}{\designvec_{t+1-i}} }
                      }\,, \label{eqn:higher_order_path_distribution_grad_log_prob_operator_lag}
              \end{align}
              }
            \end{subequations}
            where
            $
              \del{\log \Prob{\design_1=v_i}}{\hyperparam_j} =
              \del{\log \pi_i} {\hyperparam_j}
            $
            is given by 
            \eqref{eqn:first_order_trajectory_distribution_grad_log_prob_initial};  
            $
              \del{\log \CondProb{\designvec_{t+z}=v_j}{\designvec_{t}=v_i}}{\hyperparam^{l}_{m}}
              = \del{\log \pi^{i}_{j}}{\hyperparam^{l}_{m}}
            $
            is given by 
            \eqref{eqn:first_order_trajectory_distribution_grad_log_prob_transition}  for any positive integer $z$;  
            and $\del{\pi^{i}_j}{\hyperparam^{l}_{m}} = \pi^{i}_j  \del{\log\pi^{i}_j}{\hyperparam^{l}_{m}} $. 
          \end{proposition}
          \begin{proof}
            See \Cref{app:sec:gradients_proofs}.
          \end{proof}

        \paragraph{Path policy based on a generalized Raftery model}
          The original Raftery model \eqref{eqn:Raftery_model} freezes the transition 
          probabilities for all lags. 
          Specifically, as described by \eqref{eqn:higher_order_trajectory_distribution_transition},
          the transition probabilities are given by
          $
            \CondProb{\design_{*}\!=\!v_j}{\design_{*}\!=\!v_i}
          $
          irrespective of the lag, that is, irrespective of the order of the states 
          $\design_{*}$ in the trajectory or with respect to each other.
          This approach can be restrictive, however, especially 
          since in our application connectivity of the navigation mesh is sparse.
          Thus, for a given vertex the memory is mostly restricted to neighboring
          cells.
          Moreover, the time-dependent nature of predictive applications might not be sufficiently 
          captured by this model. 
          An alternative approach is to employ the lag-dependent conditional probability: 
          \begin{equation}\label{eqn:generalized_Raftery_model}
            \CondProb{\design_{n}}{
              \design_{n-k}
              \ldots, 
              \design_{n-2}, 
              \design_{n-1} 
            } 
              = \sum_{i=1}^{k}{\lambda_{i}\, 
                \nCondProb{i}{\design_{n}}{\design_{n-i}}} \,;\quad 
                \sum_{i=1}^{k} \lambda_{i} = 1 \,,
          \end{equation}
          where for each lag $l=1, \ldots, k$, the conditional transition probabilities form 
          a matrix $\mat{P}^{(l)}$ with
          $[\mat{P}^{(l)}]_{ij}=\{\nCondProb{l}{\designvec_{t+1}=v_j}{\designvec_{t}=v_i}\}$ 
          that satisfies the 
          conditions of a probability transition matrix; that is, the entries fall within $[0, 1]$ ,
          and each row adds 
          to $1$. 
          As suggested in 
          \cite{ching2004higher,ching2008higher}, we define 
          $\mat{P}^{(l)} = \prod_{i=1}^{l} \mat{P}$ for $l=1, \ldots, k$, where 
          $\mat{P}=\mat{P}^{(1)}$ is the first-order transition 
          probability matrix used in the original Raftery model \eqref{eqn:Raftery_model}.

          Note that this generalization does not change the number of parameters in the model since
          all transition probabilities depend on the first-order transition probability.
          The proposed path distribution is summarized by 
          \Cref{defn:generalized_higher_order_path_model}.

          \begin{defn}{
              Generalized Higher-Order Path Policy
          }{generalized_higher_order_path_model}
            For 
            $k\in\Znum^{+}$, the generalized $k$th-order probability 
            distribution of a path 
            $\designvec=(\design_1, \ldots, \design_{n})$ of fixed size $n>k$,
            parametrized by 
            $\hyperparamvec \in (0, 1)^{N(N+1) + k} $,
            is 
            \begin{subequations}\label{eqn:generalized_higher_order_trajectory_distribution}
              {\allowdisplaybreaks
              \begin{align}
                &\Prob{\designvec}
                  = \Prob{\design_1} 
                      \prod_{t=1}^{k-1}{
                        \left(
                            \CondProb{\designvec_{t+1}}{\designvec_{t}}
                        \right)
                      }
                    \, 
                    \prod_{t=k}^{n-1} \left( 
                      \sum_{i=1}^{k} {\lambda_{i}\, \nCondProb{i}{\designvec_{t+1}}{\designvec_{t+1-i}} }
                    \right)
                   \,, 
                      \label{eqn:generalized_higher_order_trajectory_distribution_pmf}
                  \\ 
                  &\hyperparamvec
                  = 
                    \left(
                      \hyperparam_1, \ldots, \hyperparam_{\Nsens}; \, 
                      \hyperparam^{1}_1, \ldots, \hyperparam^{1}_{N},
                      \, \ldots\, 
                      \hyperparam^{N}_1, \ldots, \hyperparam^{N}_{N}; \, 
                      \lambda_1, \ldots, \lambda_{k} 
                    \right) \,; \, \sum_{i=1}^{k} \lambda_{i} = 1  \,, 
                   \label{eqn:generalized_higher_order_trajectory_distribution_param}
                   \\ 
                  &\Prob{\design_1\!=\!v_i}
                   \equiv \pi_i 
                   = \frac{w_i}{\sum_{j=1}^{N}{w_j}} \,; \qquad 
                      w_i=\frac{\hyperparam_i}{1-\hyperparam_i} \,, 
                        \label{eqn:generalized_higher_order_trajectory_initial_state}
                        \\
                  &\CondProb{\designvec_{t+1}\!=\!v_j}{\designvec_{t}\!=\!v_i}
                  \equiv \pi^{i}_{j} 
                  = \frac{w^{i}_{j}}{\sum_{k=1}^{N}{w^{i}_{k}}} \,; \qquad \,\, 
                      w^{i}_{j}=\frac{\hyperparam^{i}_{j}}{1-\hyperparam^{i}_{j}} \,, 
                        \label{eqn:generalized_higher_order_trajectory_distribution_transition}
                        \\
                  &\nCondProb{r}{\designvec_{t}=v_j}{\designvec_{t-r}=v_i} 
                    = 
                    \sum_{k_{r-1}=1}^{N}
                    \cdots
                    \sum_{k_{2}=1}^{N}
                    \sum_{k_{1}=1}^{N}
                    \pi^{i}_{k_{r-1}} \pi^{k_{r-1}}_{k_{r-2}} \cdots  
                    \pi^{k_{2}}_{k_{1}} \pi^{k_{1}}_{k_{j}} 
                    \,.
                    \label{eqn:generalized_higher_order_transition}
              \end{align}
              }
            \end{subequations}
          \end{defn}
          
          While the $r$ summations in 
          \eqref{eqn:generalized_higher_order_transition} 
          are written over $N$ possible states, 
          in our application most of the terms vanish because of lack of connectivity
          in the navigation mesh.
          Each summation reduces to 
          the number of connected 
          nodes (via a path of length up to the predefined policy order $k$) to a given 
          node. 

          \begin{proposition}\label{proposition:generalized_higher_order_gradient}
            The gradient of the log-policy in 
            \eqref{eqn:generalized_higher_order_trajectory_distribution}
            is given by
            \begin{subequations}\label{eqn:generalized_higher_order_trajectory_distribution_grad_log_prob}
            \begin{equation}\label{eqn:generalized_higher_order_path_distribution_grad_log_prob_operator}\small
              \left(
                \del{}{\hyperparam_{1}}, \ldots, \del{}{\hyperparam_{\Nsens}};\,\, 
                \del{}{\hyperparam^{1}_{1}}, \ldots, \del{}{\hyperparam^{1}_{\Nsens}};\,\ldots;\, \,
                \del{}{\hyperparam^{\Nsens}_{1}}, \ldots, \del{}{\hyperparam^{\Nsens}_{\Nsens}};\, \,
                \del{}{\lambda_{1}}, \ldots, \del{}{\lambda_{k}}
              \right)\,  \log\Prob{\designvec} \,,
            \end{equation}
            with elementwise components given as follows:

            {\allowdisplaybreaks
            \begin{align}
              &\del{\log \Prob{\designvec} } {\hyperparam_j}
                = \del{\log \Prob{\design_{1}}} {\hyperparam_j}  
                   \,, \label{eqn:generalized_higher_order_path_distribution_grad_log_prob_operator_initial}
                 \\
                &\del{\log \Prob{\designvec}}{\hyperparam^{l}_{m}}
                = \sum_{t=1}^{k-1} 
                  \del{\log \CondProb{\designvec_{t+1}}{\designvec_{t}}}{\hyperparam^{l}_{m}}
                  + 
                  \sum_{t=k}^{n-1} 
                    \frac{ 
                      \sum_{i=1}^{k} {
                          \lambda_{i}\, 
                            \del{\nCondProb{i}{\designvec_{t+1}}{\designvec_{t+1-i}}}{\hyperparam^{l}_{m}}  
                        }
                    }{
                      \sum_{i=1}^{k} {\lambda_{i}\, \nCondProb{i}{\designvec_{t+1}}{\designvec_{t+1-i}} }
                    } 
                     \,, \label{eqn:generalized_higher_order_path_distribution_grad_log_prob_operator_transitions}
                    \\
                &\del{\log \Prob{\designvec}}{\lambda_{l}}
                = \sum_{t=k}^{n-1} 
                \frac{
                  \nCondProb{l}{\designvec_{t+1}}{\designvec_{t+1-l}} 
                    }{
                      \sum_{i=1}^{k} {\lambda_{i}\, \nCondProb{i}{\designvec_{t+1}}{\designvec_{t+1-i}} }
                    } \,,
                     \label{eqn:generalized_higher_order_path_distribution_grad_log_prob_operator_lag}
                  \\
                &\del{
                  \nCondProb{r}{\designvec_{t}\!=\!v_j}{\designvec_{t-r}\!=\!v_i} 
                }{\hyperparam^{l}_{m}} 
                  = 
                    \sum_{k_{r-1}=1}^{N}
                    \cdots
                    \sum_{k_{2}=1}^{N}
                    \sum_{k_{1}=1}^{N}
                  \del{\left( 
                    \pi^{i}_{k_{r-1}} \pi^{k_{r-1}}_{k_{r-2}} \cdots  
                      \pi^{k_{2}}_{k_{1}} \pi^{k_{1}}_{k_{j}}
                    \right)
}{\hyperparam^{l}_{m}} 
                 \,,
                     \label{eqn:generalized_higher_order_path_distribution_grad_log_prob_operator_higher_order_gradient}
            \end{align}
            }
            where 
            $
              \del{\log \pi_i} {\hyperparam_j}
            $
            is given by 
            \eqref{eqn:first_order_trajectory_distribution_grad_log_prob_initial} 
            and 
            $
              \del{\log \pi^{i}_{j}}{\hyperparam^{l}_{m}}
            $
            is given by 
            \eqref{eqn:first_order_trajectory_distribution_grad_log_prob_transition}.
            \end{subequations}
          \end{proposition}
          \begin{proof}
            See \Cref{app:sec:gradients_proofs}.
          \end{proof}

      \subsubsection{Lag Weights and the Single-Step Constraint}
      \label{subsubsec:lag_weights}
        By construction, the higher-order Markov policies condition the next 
        location on a history of past nodes rather than on the immediately 
        preceding node alone. A consequence is that the support of these 
        policies is not restricted to single-step graph adjacency: 
        transitions $(\design_{n-1}, \design_n) \notin E$ can receive nonzero 
        probability even though they are infeasible under the first-order 
        policy.
        This is a modeling trade-off rather than a feature of the underlying 
        setup. In our advection diffusion experiments 
        (\Cref{sec:numerical_experiments}), the graph encodes direct one-step 
        motions between candidate sensor locations, so transitions across 
        non-adjacent nodes do not correspond to physically realizable sensor 
        moves.         In settings where the graph edges encode something other than 
        direct spatial adjacency, for instance admissible velocities or 
        maneuver primitives, the same relaxation can carry a physical 
        interpretation, but this is outside the scope of the present work.

        At the parameter level, the three policies are defined over
        the same nominal transition parameter space
        $\{\hyperparam^{i}_{j} : (v_i, v_j) \in V \times V\}$, but they
        differ in which coordinates are exposed to optimization. For the
        first-order policy (\Cref{defn:first_order_path_model}) and the
        generalized higher-order policy
        (\Cref{defn:generalized_higher_order_path_model}), the active
        transition coordinates in $\hyperparamvec$ are restricted to
        directly-adjacent arcs $(v_i, v_j) \in E$. Non-existent arcs
        are encoded as $\hyperparam^{i}_{j} = 0$ and excluded from
        $\hyperparamvec$, so they remain zero throughout the
        optimization. The Raftery policy
        (\Cref{defn:higher_order_path_model}) is parameterized over an
        enriched active set that includes, in addition to the
        directly-adjacent arcs, the transition coordinates
        $\hyperparam^{i}_{j}$ for pairs $(v_i, v_j) \notin E$ that lie
        on a path of length at most $k$ in $G$. These extended
        coordinates start at zero and are updated by the optimizer.
        They can therefore become positive during optimization, which is
        what allows the Raftery policy to assign positive probability to
        transitions $(\design_{n-1}, \design_n) \notin E$ at the
        distributional level. The parameter-level treatment of all
        these coordinates is given in
        \Cref{subsubsec:degenerate_case}.

        The lag weights control how past steps influence the next move, and 
        they provide the main mechanism by which we mitigate the relaxation 
        above.
        They can be fixed a priori rather than optimized; in that case they 
        are not part of the policy parameters $\hyperparamvec$, and the 
        corresponding gradient 
        \eqref{eqn:generalized_higher_order_path_distribution_grad_log_prob_operator_lag} 
        is not computed during optimization. 
        For example, 
        $
          \lambda_{i}=\frac{1}{k} 
        $
        enforces uniform influence. 
        Stronger influence from more recent steps can be achieved, for example, 
        by setting $\lambda_{i}=\gamma^{i-1}$ with a discount factor $0<\gamma\leq 1$, 
        as followed in policy optimization methods \cite{bertsekas1995neuro}.

        In 
        \Cref{defn:higher_order_path_model}
        the conditional probability 
        $\CondProb{\design_{n}}{
            \design_{n-k},
            \design_{n-k+1}
            \ldots, 
            \design_{n-2}, 
            \design_{n-1} 
          }$
        is nonzero when the transition probability 
        $\CondProb{\design_{n}}{\design_{n-i}}$ from any node in the past, 
        up to the defined order $\design_{n-i}, i=1, \ldots, k$, is nonzero.
        Similarly, 
        \Cref{defn:generalized_higher_order_path_model}
        assigns nonzero values for the conditional transition probability 
        when there is a path of length $i$ between 
        $\design_{n-i}$ and 
        $\design_{n}$ for any $i\leq k$.
        This means that these models can allow trajectories involving
        jumps between nodes without direct arcs connecting them. 
        This effect is expected to magnify with longer trajectories and/or higher orders,
        as shown in the numerical results in \Cref{sec:numerical_experiments}.
        Thus, we recommend modeling the lag parameters with 
        a decaying sequence of fixed values. 
        For example, to avoid tuning a discount factor, in 
        \Cref{sec:numerical_experiments} we consider the following normalized geometric model:
        \begin{equation}\label{eqn:decreasing_lag_weights}
          \lambda_{i} 
            = 
            \frac{\frac{1}{i}}{\sum_{i=1}^{k}{\frac{1}{i}}} \,, 
            \qquad i=1, 2, \ldots, k \,.
        \end{equation}
        %

      \subsubsection{Degenerate Parameter Values and Boundary Behavior}
      \label{subsubsec:degenerate_case}
        The probability models in
        \eqref{eqn:initial_distribution_inclusion} and
        \eqref{eqn:transition_probabilities}, and their counterparts in
        the path policies defined above, were stated under the
        non-degenerate assumption
        $\hyperparam_i,\hyperparam^{i}_{j}\in(0, 1)$. The corresponding
        score components 
        \eqref{eqn:first_order_trajectory_distribution_grad_log_prob_initial}
        and
        \eqref{eqn:first_order_trajectory_distribution_grad_log_prob_transition}
        are similarly undefined when a parameter reaches the boundary.
        Boundary configurations are however both meaningful at modeling
        time (fixed starting positions, disallowed nodes, absent arcs) and
        possible at runtime (the optimization may drive a parameter to
        the boundary). The discussion here consolidates the treatment of
        both cases.

        Following the conditional Bernoulli framework of
        \cite{attia2024probabilistic}, specialized to budget $1$ since
        the trajectory selects a single initial node out of $N$
        candidates, the initial probability model
        \eqref{eqn:initial_distribution_inclusion} extends to the closed
        domain $\hyperparamvec\in[0, 1]^{N}$ as
        \begin{equation}\label{eqn:initial_distribution_inclusion_general}
          \pi_i \equiv \Prob{\design_1=v_i}
            = \begin{cases}
                1                                 & , \quad I=\{i\} \,, \\
                0                                 & , \quad I\neq\emptyset \text{ and } i\notin I \,, \\
                \dfrac{w_i}{\sum_{j=1}^{N} {w_j}} & , \quad I=\emptyset \,,
              \end{cases}
        \end{equation}
        where the weights $w_i = \hyperparam_i / (1-\hyperparam_i)$ are
        as in \eqref{eqn:initial_distribution_inclusion} with the
        convention that $w_i = 0$ when $\hyperparam_i = 0$, and where
        the index sets
        $I=\{i \,|\, \hyperparam_i=1\}$,
        $O=\{i \,|\, \hyperparam_i=0\}$, and
        $V=\{1,\ldots,N\}\setminus(I\cup O)$ collect the degenerate-on,
        degenerate-off, and non-degenerate coordinates of the current
        iterate, respectively. The case $|I|\geq 2$ is excluded since
        it would violate the budget constraint $\sum_i \pi_i = 1$.
        See \cite[Theorem 3.2]{attia2024probabilistic} for the general
        construction.
        The transition parameters $\hyperparam^{i}_{j}$ are treated by
        the same construction applied to each row of the transition
        matrix. The score components
        \eqref{eqn:first_order_trajectory_distribution_grad_log_prob_initial}
        and
        \eqref{eqn:first_order_trajectory_distribution_grad_log_prob_transition}
        admit an analogous closed-domain extension that remains finite at
        the boundary, see \cite[Proposition 3.2]{attia2024probabilistic}
        for the construction.

        The apparent boundary singularities of the score components
        \eqref{eqn:first_order_trajectory_distribution_grad_log_prob_initial}
        and
        \eqref{eqn:first_order_trajectory_distribution_grad_log_prob_transition}
        are harmless under the closed-domain extension, by two
        complementary observations.
        A coordinate with $\hyperparam_i = 0$ assigns zero inclusion
        probability to the corresponding node $v_i$ through
        \eqref{eqn:initial_distribution_inclusion_general}. The trajectory
        therefore never starts at $v_i$, and the score component
        associated with this event is never evaluated against a sampled
        trajectory.
        A coordinate with $\hyperparam_i = 1$ forces $\hyperparam_j = 0$
        for all $j \neq i$ by the budget constraint. The closed-domain
        extension of the inclusion probability is then $\pi_i = 1$, its
        derivative vanishes, and the score is well-defined and equal to
        zero. This is a direct consequence of
        \cite[Proposition 3.2]{attia2024probabilistic} specialized to
        budget $1$.
        The same arguments apply row-by-row to the transition parameters
        $\hyperparam^{i}_{j}$.

        One may fix an initial parameter at $\hyperparam_i = 0$ or
        $\hyperparam_i = 1$ at the modeling stage, before optimization
        begins. For example, setting $\hyperparam_i = \delta_{i j}$
        for a chosen $j \in \{1, \ldots, N\}$ enforces all sampled
        trajectories to start at the node $v_{j}$. A non-allowable
        initial node and an absent transition arc are handled in the
        same way, by excluding the corresponding coordinate at the value
        zero. 
        Excluded coordinates are not part of $\hyperparamvec$ and
        are not updated by \Cref{alg:probabilistic_path_optimization},
        so they remain pinned at their assigned values throughout the
        optimization. An exception applies to the Raftery policy
        (\Cref{defn:higher_order_path_model}): the transition
        coordinates $\hyperparam^{i}_{j}$ for pairs
        $(v_i, v_j) \notin E$ that lie on a path of length at most $k$
        in $G$ are included in $\hyperparamvec$ at the value zero and
        may become nonzero during optimization, as described in
        \Cref{subsubsec:lag_weights}.

        A coordinate that is included in $\hyperparamvec$ may also be
        driven to the boundary by the optimization procedure itself.
        The closed-domain extension of the score keeps the gradient
        finite there, and the expected score component vanishes at the
        boundary. Such a coordinate therefore stays pinned at the
        boundary under the mean dynamics of
        \Cref{alg:probabilistic_path_optimization}.

      \subsubsection{Sampling Trajectories}
      \label{subsubsec:sampling}
        Sampling trajectories from the parametric policy is critical for 
        the probabilistic optimization approach. 
        %
        Sampling trajectories from first-order policy 
        (\Cref{defn:first_order_path_model}) is described by \Cref{alg:sampling_first_order}, 
        and sampling the higher-order policies 
        (\Cref{defn:higher_order_path_model}, \Cref{defn:generalized_higher_order_path_model}) 
        is described by \Cref{alg:sampling_higher_order}.

          \begin{algorithm}[!htbp]
            \caption{
              Sample trajectories from first-order policy model 
                (\Cref{defn:first_order_path_model}).
            }\label{alg:sampling_first_order}
            
            \begin{algorithmic}[1] 
           
              \Require{
                Distribution parameters 
                $\hyperparamvec$ 
                \eqref{eqn:first_order_trajectory_distribution_param}; 
                trajectory length $n$; and
                sample size $\Nens$.
              }
              \Ensure{
                A sample 
                $ \{
                    \designvec[i]
                    \in V^{n}
                      \sim \CondProb{\designvec}{\hyperparamvec}
                    | i=1, \ldots, \Nens
                \}
                $ drawn from \eqref{eqn:first_order_trajectory_distribution_pmf}
              }


              \State Initialize a sample $S=\{\}$.

              \For {$i$ $\gets 1$ to $\Nens$}
                
                \State Extract initial parameters 
                  $\left(\hyperparam_1, \ldots, \hyperparam_{\Nsens}\right)$ from $\hyperparamvec$  
                
                \State Compute initial inclusion probabilities 
                  $\left(\pi_1, \ldots, \pi_{\Nsens}\right)$ 
                    \Comment{Use \eqref{eqn:first_order_trajectory_distribution_init_pmf}}

                \State Sample initial vertex $\design_{1}[i]\in V$ with 
                  probabilities $\Prob{v_i}=\pi_i,\, i=1, \ldots, \Nsens$

                \For {$j$ $\gets 2$ to $n$} \label{algstep:sampling_transitions}
                  
                  \State Extract transition parameters 
                    $\left(\hyperparam^{z}_{1}, \ldots\hyperparam^{z}_{\Nsens}\right)$ 
                    from $\hyperparamvec$ with 
                    $\design_{j-1}[i]=v_{z}\in V$ 

                  \State Compute transition probabilities 
                    $\pi^{z}_{1}, \ldots\pi^{z}_{\Nsens}$ 
                    \Comment{Use \eqref{eqn:first_order_trajectory_distribution_transition_pmf}}
                  
                  \State Sample next vertex 
                    $\design_{j}[i]\in V$ 
                    with 
                    probabilities $\Prob{v_i}=\pi^{z}_{i},\, i=1, \ldots, \Nsens$ 

                \EndFor

                \State Update $S \gets S \cup \{\designvec[i]\}$
              \EndFor
         
              \State \Return {$S$ }
              
            \end{algorithmic}
          \end{algorithm}

          \begin{algorithm}[!htbp]
            \caption{
              Sample 
              higher-order policy models
                \Cref{defn:higher_order_path_model} 
                  (or \Cref{defn:generalized_higher_order_path_model}).
            }\label{alg:sampling_higher_order}
            
            \begin{algorithmic}[1] 
           
              \Require{
                Policy order $k>1$,
                parameters 
                $\hyperparamvec$ 
                \eqref{eqn:higher_order_trajectory_distribution_param}
                (or \eqref{eqn:generalized_higher_order_trajectory_distribution_param}); 
                trajectory length $n>k$ ; and
                sample size $\Nens$.
              }
              \Ensure{
                Sample 
                $ \{
                    \designvec[i] 
                    \in \!V^{n}\!
                      \sim\! \CondProb{\designvec}{\hyperparamvec}
                    | i=1, \ldots, \Nens
                \}
                $ 
                from \eqref{eqn:higher_order_trajectory_distribution_pmf} (or 
                      \eqref{eqn:generalized_higher_order_trajectory_distribution_pmf}).
              }


              \State Initialize a sample $S=\{\}$.

              \For {$i$ $\gets 1$ to $\Nens$}
                
                \State Sample initial $k$ nodes $\left(\design_{1}[i],\ldots, \design_{k}[i] \right)$ of the path 
                    \Comment{Use \Cref{alg:sampling_first_order}}

                \For {$j$ $\gets k+1$ to $n$} \label{algstep:higher_order_sampling_transitions}
                  
                  \State Extract transition parameters 
                    $\left(\hyperparam^{z}_{1}, \ldots\hyperparam^{z}_{\Nsens}\right)$ 
                    from $\hyperparamvec$ with 
                    $\design_{j-1}[i]=v_{z}\in V$ 

                  \State Compute transition probabilities 
                    $\pi^{z}_{1}, \ldots\pi^{z}_{\Nsens}$ 
                    \Comment{Use 
                    \eqref{eqn:higher_order_trajectory_distribution_transition}
                    (or \eqref{eqn:generalized_higher_order_trajectory_distribution_transition})
                    }
                  
                  \State Sample next vertex 
                    $\design_{j}[i]\in V$ 
                    with 
                    probabilities $\Prob{v_i}=\pi^{z}_{i},\, i=1, \ldots, \Nsens$ 

                \EndFor

                \State Update $S \gets S \cup \{\designvec[i]\}$
              \EndFor
         
              \State \Return {$S$ }
              
            \end{algorithmic}
          \end{algorithm}

        In Step \ref{algstep:sampling_transitions} of 
        \Cref{alg:sampling_first_order} transitions need to be evaluated only over the local 
        neighborhood of the current node, reducing the sampling
        computational cost.
        In contrast, sampling the higher-order policies 
        (Step \ref{algstep:higher_order_sampling_transitions} of \Cref{alg:sampling_higher_order}) 
        involves linear combinations of transitions up to order $k$, 
        requiring an extended neighborhood that can be precomputed offline based on graph connectivity 
        and the model order.

      \subsection{Probabilistic Path OED Problem}
      \label{subsec:the_probabilistic_optimization}
        We replace the discrete 
        path OED \Cref{prbl:path_oed_problem}
        with the probabilistic formulation given by \Cref{prbl:probabilistic_path_oed_problem}. 

        \begin{problem}{Probabilistic Path OED Problem on a Graph}{probabilistic_path_oed_problem}
          Given the setup in \Cref{prbl:path_oed_problem}, 
          the probabilistic path OED problem is
          \begin{equation}\label{eqn:trajectory_probabilistic_optimization}
            \hyperparamvec\opt \in \argmax_{\designvec\sim\CondProb{\designvec}{\hyperparamvec}} 
              \Expect{\designvec\sim\CondProb{\designvec}{\hyperparamvec}}{\utilityfunc(\designvec)} \,, 
          \end{equation}
          where $\CondProb{\designvec}{\hyperparamvec}$ is 
          given by \Cref{defn:first_order_path_model}, 
          \Cref{defn:higher_order_path_model}, or \Cref{defn:generalized_higher_order_path_model}.
        \end{problem}

        We numerically solve \eqref{eqn:trajectory_probabilistic_optimization} using 
        the stochastic gradient $\widehat{\vec{g}}$ given by 
        \eqref{eqn:generic_gradient}. 
        At each iteration, the parameters are updated by projected stochastic 
        gradient ascent (or descent for minimization), with a prescribed step size, 
        ensuring feasibility by projecting onto $\left[0, 1\right]^{\Nparam}$.
        The update step of the policy parameter is thus given by
        \begin{equation}\label{eqn:update_step}
          \hyperparamvec^{(\ell)} \gets \proj(\hyperparamvec^{(\ell-1)} + \eta^{(\ell)} \vec{\widehat{\vec{g}}}) \,,
        \end{equation}
        where $0 < \eta^{(\ell)}\leq 1$ is a positive step size (learning rate) at the 
        $\ell$th iteration of the optimization procedure and  
        the plus sign in \eqref{eqn:update_step} is replaced with a minus sign for 
        minimization (steepest descent). 
        As proposed in \cite{attia2024probabilistic}, we employ the scaling projector:
        \begin{equation}\label{eqn:scaling_projector}
          \Proj{}{\vec{g}} = s \, \vec{g}\,; \quad
          s = \min\{1, \min\limits_{i=1, \ldots, \Nparam} \{s_i\}\}\,; \quad
          s_i =
            \begin{cases}
              \frac{1-\hyperparam_i}{\abs{g_i}} & \text{ if } \hyperparam_i\pm g_i > 1 \\
              \frac{\hyperparam_i}{\abs{g_i}} & \text{ if } \hyperparam_i \pm g_i < 0 \\
              1 & \text{ otherwise} \,,
            \end{cases}
             \,,
        \end{equation}
        which restricts (by scaling) the updated parameter $\vec{g}\pm \hyperparamvec$ to the domain 
        $[0, 1]^{\Nsens}$ 
        and $\pm$ defines the projected gradient
        for both maximization and minimization, respectively.
        When the higher-order models are employed, the lag weights 
        are normalized $\lambda_i \gets \lambda_i / \sum_{i=1}^{k}{\lambda_i}$ 
        after applying the projection operator. 

        With limited sample sizes $\Nens$ due to expensive utility function evaluations, 
        stochastic gradient variance can be high, leading to unstable learning. 
        This is mitigated by introducing a constant baseline function $\baseline$ to 
        \eqref{eqn:trajectory_probabilistic_optimization} by replacing 
        $\utilityfunc$ with $\utilityfunc-\baseline$. 
        This does  not alter the optimal solution since 
        $
            \argmax_{\designvec\sim\CondProb{\designvec}{\hyperparamvec}} 
              \Expect{\designvec\sim\CondProb{\designvec}{\hyperparamvec}}{\utilityfunc(\designvec)} 
            = \argmax_{\designvec\sim\CondProb{\designvec}{\hyperparamvec}} 
              \Expect{\designvec\sim\CondProb{\designvec}{\hyperparamvec}}{\utilityfunc(\designvec)-\baseline} 
            = \argmax_{\designvec\sim\CondProb{\designvec}{\hyperparamvec}} 
              \Expect{\designvec\sim\CondProb{\designvec}{\hyperparamvec}}{\utilityfunc(\designvec)} -\baseline
        $. Moreover, the gradient satisfies 
        $
          \vec{g}^{b} =
              \nabla_{\hyperparamvec}
                \Expect{\designvec\sim\CondProb{\designvec}{\hyperparamvec}}{\utilityfunc{(\designvec)}-\baseline}
          =
              \nabla_{\hyperparamvec}
                \Expect{\designvec\sim\CondProb{\designvec}{\hyperparamvec}}{\utilityfunc{(\designvec)}} 
          = \vec{g}
        $.
        The stochastic estimate of the gradient, however, becomes
        \begin{subequations}\label{eqn:baseline_gradient_group}
          \begin{equation}\label{eqn:generic_stochastic_gradient_baseline}
            \widehat{\vec{g}}^{\rm b} = 
              \frac{1}{\Nens} \sum_{k=1}^{\Nens} 
                \left( \utilityfunc(\designvec[k]) - \baseline \right) 
                  \, \nabla_{\hyperparam} \log{\CondProb{\designvec[k]}{\hyperparam}} \,; 
                  \quad \designvec[k] \sim \CondProb{\designvec}{\hyperparamvec}\,.
          \end{equation}

          Both \eqref{eqn:generic_stochastic_gradient} 
          and \eqref{eqn:generic_stochastic_gradient_baseline} 
          are unbiased estimators of 
          $\vec{g}$ \eqref{eqn:full_gradient}. 
          By letting 
          \begin{equation}\label{eqn:d_vec}
            \vec{d} = \frac{1}{\Nens} \sum_{k=1}^{\Nens}
              \nabla_{\hyperparamvec}\log\CondProb{\designvec[k]}{\hyperparamvec} \,; \qquad 
              \designvec[k] \sim \CondProb{\designvec}{\hyperparamvec}  \,,
          \end{equation}
          we note that 
          $
            \widehat{\vec{g}}^{\rm b} 
            = \widehat{\vec{g}} - \baseline \, \vec{d}
          $,  and thus 
          $
            \Expect{}{\widehat{\vec{g}}} 
            = \Expect{}{\widehat{\vec{g}}^{\rm b}} 
            = \vec{g},
          $
          where we used the fact that $\Expect{}{\vec{d}}=\vec{0}$.
          Moreover, the total variance of the two stochastic estimators is such that
          $
            \brVar{\widehat{\vec{g}}^{\rm b}} = \brVar{\widehat{\vec{g}}} 
            - 2 \baseline \Expect{}{\widehat{\vec{g}}\tran \vec{d}} 
              + \baseline^2 \Expect{}{\widehat{\vec{d}}\tran \vec{d}}  
          $, which is a quadratic in the baseline $\baseline$. 
          The value of the baseline that minimizes the variance $\brVar{\widehat{\vec{g}}^{\rm b}}$ 
          is thus 
          \begin{equation}\label{eqn:general_optimal_baseline}
            \baseline\opt = 
            \left( 
              \Expect{}{\widehat{\vec{g}}\tran \vec{d}} 
            \right)/\left( 
              \Expect{}{\widehat{\vec{d}}\tran \vec{d}} 
            \right)
            \,. 
          \end{equation}

          The optimal baseline \eqref{eqn:general_optimal_baseline} applies to any parametric 
          policy  $\CondProb{\designvec}{\hyperparam}$ 
          and has a denominator that depends only on the policy, allowing closed-form evaluation, 
          while the numerator depends on the utility and is generally intractable analytically. 
          Because of correlations and combinatorial complexity, \eqref{eqn:general_optimal_baseline} is 
          efficiently estimated by 
          \begin{equation}\label{eqn:general_optimal_baseline_estimate}
            \begin{aligned}
              \baseline\opt 
              &\approx \frac{
                \sum_{i=1}^{N_{\rm b}} \left(\widehat{\vec{g}}[i]\right)\tran \vec{d}[i]
              }{
                \sum_{i=1}^{N_{\rm b}} \left(\vec{d}[i]\right)\tran \vec{d}[i]
              } \,, \\
              \vec{d}[i] &= \frac{
                  \sum_{j=1}^{\Nens} \nabla_{\hyperparamvec} \log\Prob{\designvec[i, j]} 
              }{\Nens} 
              \,; \quad 
              \widehat{\vec{g}}[i] = \frac{
                  \sum_{j=1}^{\Nens} \utilityfunc(\designvec[i, j]) 
                    \nabla_{\hyperparamvec} \log\Prob{\designvec[i, j]} 
              }{\Nens} 
              \,, 
            \end{aligned}
          \end{equation}
          where
          $\{\designvec[i, j] \sim \CondProb{\designvec}{\hyperparamvec}\,|\, i=1, 
            \ldots, N_{\rm b},\, j=1, \ldots, \Nens\}$. 
        \end{subequations}  
        
        \subsubsection{Complete Algorithmic Statement}
        \label{subsubsec:complete_algorithm}
          \Cref{alg:probabilistic_path_optimization}  provides 
          a complete algorithmic statement of the 
          proposed probabilistic path optimization approach.

          \begin{algorithm}[!htbp]
            \caption{Probabilistic Path OED Optimization Algorithm for Solving 
              \eqref{prbl:probabilistic_path_oed_problem}. 
            }
            \label{alg:probabilistic_path_optimization}
            \begin{algorithmic}[1]

              \Require{
                Utility function $\utilityfunc$; 
                probabilistic policy (%
                  either \eqref{eqn:first_order_trajectory_distribution_full},
                  \eqref{eqn:higher_order_trajectory_distribution}, or 
                  \eqref{eqn:generalized_higher_order_trajectory_distribution}
                );
                initial distribution parameter $\hyperparamvec^{(0)}$;
                stepsize schedule $\eta^{(\ell)}$;
                sample sizes $\Nens,\, N_{\rm opt}$.
              }
              
              \Ensure{$\designvec\opt$}

              \State{initialize $\ell = 0$}

              \While{Not Converged}
                \State{\label{algstep:random_sample}
                  Sample 
                  $\{\designvec[k]; k=1,\ldots,\Nens \} $
                  \Comment{
                    $
                      \begin{cases} 
                        \text{\Cref{alg:sampling_first_order} } & \text{ for \eqref{eqn:first_order_trajectory_distribution_full}}  \\ 
                        \text{\Cref{alg:sampling_higher_order} } & \text{ for \eqref{eqn:higher_order_trajectory_distribution} or \eqref{eqn:generalized_higher_order_trajectory_distribution} }  \\ 
                      \end{cases}
                    $ 
                  }
                }

                \State {\label{algstep:evaluate_utility}
                  Evaluate $\{\utilityfunc(\designvec[k]); k=1,\ldots,\Nens \} $
                }
                
                \State {\label{algstep:evaluate_grad_log_pmf}
                  Evaluate $\{\nabla_{\hyperparamvec} \log\CondProb{\designvec[k]}{\hyperparamvec}; k=1,\ldots,\Nens \} $
                \Comment{Based on the chosen policy}
                }

                \State {
                  Calculate $ \widehat{\vec{g}}^{(\ell)} \gets
                    \frac{1}{\Nens} \sum_{k=1}^{\Nens} 
              \utilityfunc(\designvec[k]) \, \nabla_{\hyperparam} \log{\CondProb{\designvec[k]}{\hyperparamvec}}$
                }
                
                \State{
                  Calculate $
                    \vec{d} \gets \frac{1}{\Nens} \sum_{k=1}^{\Nens}
                      \nabla_{\hyperparamvec}\log\CondProb{\designvec[k]}{\hyperparamvec} 
                  $ 
                  \Comment{
                    Use \eqref{eqn:d_vec}
                  }
                }

                \State{\label{algstep:optimal_baseline_estimate}
                  Calculate  
                  $
                    \baseline \gets \frac{
                      \left(\widehat{\vec{g}}^{(\ell)}\right)\tran \, \vec{d}
                    }{
                      \vec{d}\tran \, \vec{d} 
                    } 
                  $ 
                  \Comment{
                    Use \eqref{eqn:general_optimal_baseline_estimate} with $N_{\rm b}>1$ if needed
                  }
                }

                \State{\label{algstep:parameter_update}
                  Update
                  $ \hyperparamvec^{(\ell+1)}
                      = \hyperparamvec^{(\ell)} + \eta^{(\ell)} \Proj{}{ \widehat{\vec{g}}^{(\ell)} - \baseline \vec{d} }
                  $
                  \Comment{Use $\proj$ given by \eqref{eqn:scaling_projector}}
                }

                \State {
                  Update $\ell \leftarrow \ell+1$
                }

              \EndWhile

              \State{
                Set $\hyperparamvec\opt = \hyperparamvec^{(\ell)}$
              }

              \State{
                  Sample $\{\designvec[k]; k=1,\ldots,N_{\rm opt} \} $
                  with the optimal parameter $\hyperparamvec\opt$.
              }

              \State \Return{\label{algstep:optimum_design_return}
                $\designvec\opt $: the design with largest
                $\utilityfunc$ 
                value in the sample.
              }

            \end{algorithmic}
          \end{algorithm}

          \Cref{alg:probabilistic_path_optimization} is presented for maximization; 
          minimization is handled by flipping the parameter update sign in Step 
          \ref{algstep:parameter_update}, and the optimum 
          design $\design\opt$ in Step \ref{algstep:optimum_design_return} 
          is the one with smallest $\utilityfunc$ value. 
          For simplicity, the baseline is estimated 
          (Step \ref{algstep:optimal_baseline_estimate}) 
          by using the same samples as the stochastic gradient, 
          which is validated in \Cref{sec:numerical_experiments}, 
          although larger batch sizes may be used.
          Termination is ensured by standard stopping criteria, 
          such as iteration limits or projected-gradient tolerances, 
          as discussed in \Cref{sec:numerical_experiments}.

      \paragraph{Convergence of the algorithm}
        By design, the proposed approach enjoys all properties of the probabilistic 
        optimization approach for black-box binary optimization presented, for example,  
        in \cite{attia2022stochastic,attia2024probabilistic}. 
        Convergence is thus guaranteed in expectation, and the set of optimal solutions 
        of the original optimization problem
        is guaranteed to be a subset of the optimal set of parameters 
        of the probabilistic formulation.
        The output of \Cref{alg:probabilistic_path_optimization} is thus expected to explore the 
        upper tail of the distribution of the utility function $\utilityfunc$; however, 
        global optimality is not generally guaranteed.

    \subsubsection{Computational Considerations and Scalability}
    \label{subsubsec:computational_conciderations}
      The probabilistic approach requires only utility evaluations, 
      making it independent of the OED criterion and well 
      suited for graph-based path optimization with black-box objectives. 
      Its only requirement is a discrete navigation mesh of the domain.
      Here we analyze computational bottlenecks, 
      scalability, parallelization opportunities, 
      and limitations of \Cref{alg:probabilistic_path_optimization}. 
      Discussion of policy selection is deferred to 
      \Cref{sec:conclusions}.

      Each iteration of \Cref{alg:probabilistic_path_optimization} 
      is dominated by policy sampling (Step \ref{algstep:random_sample}), 
      utility $\utilityfunc$ evaluation (Step \ref{algstep:evaluate_utility}),
      and log-probability gradients (Step \ref{algstep:evaluate_grad_log_pmf}).
     Since $\utilityfunc$ is treated as a black box, acceleration techniques 
      such as randomization or surrogates apply but are beyond this work; 
      we therefore focus on the costs of sampling and gradient evaluation.

      \begin{figure}[!htbp]
        \centering
        \includegraphics[width=1.0\linewidth]{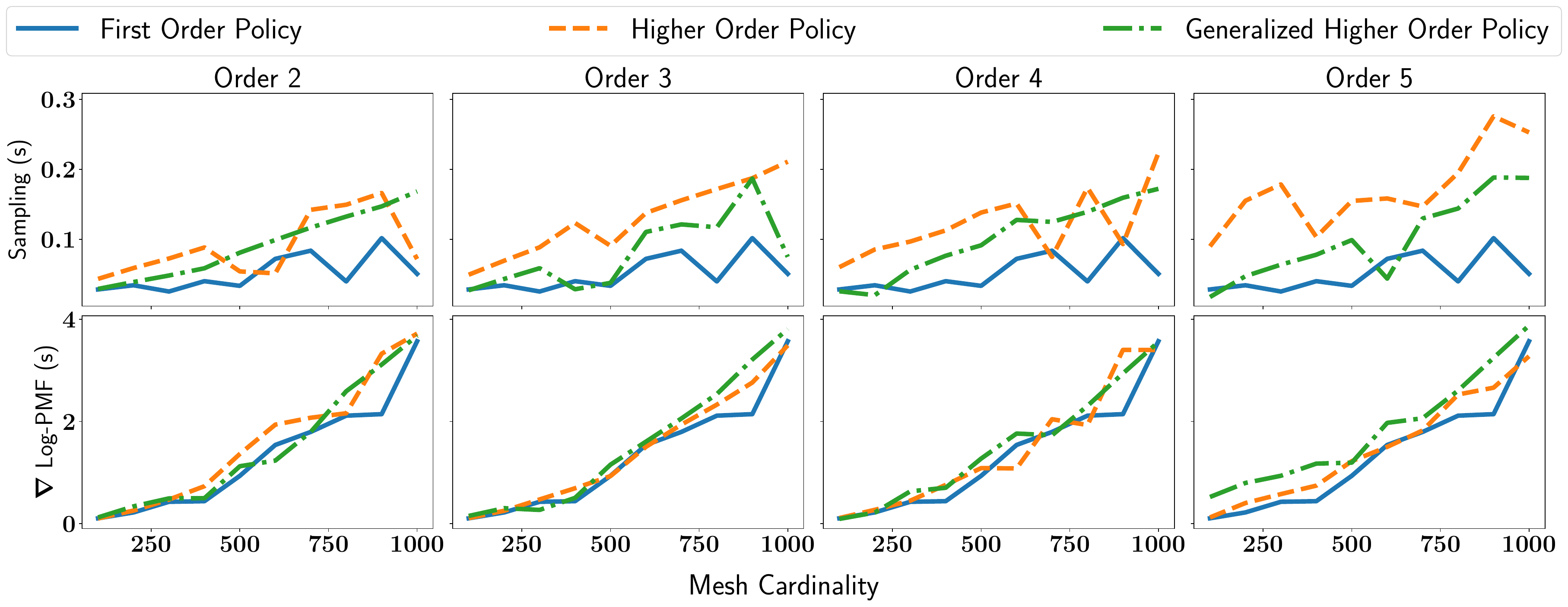}
        \caption{
          Wall times against mesh cardinality (number of nodes) 
          for sampling (top) and log-policy gradient (bottom) 
          across the three policy models, averaged over $32$ runs on an Apple M1 Mac 
          laptop without parallelization. 
          First-order, higher-order, 
          and generalized higher-order models correspond 
          to \Cref{defn:first_order_path_model}, 
          \Cref{defn:higher_order_path_model}, and 
          \Cref{defn:generalized_higher_order_path_model}, respectively.
          Results of the first-order policy are repeated for comparison.
        }
        \label{fig:policy_wall_times} 
      \end{figure}

      \Cref{fig:policy_wall_times} reports average wall times 
      for sampling and log-policy gradient across the three policy models. 
      Gradient evaluation (bottom row) is more expensive than sampling (top row), 
      and the higher-order policy incurs the highest sampling cost 
      due to transitions to all reachable nodes, 
      whereas the first-order and the generalized higher-order policies 
      restrict transitions to local or precomputed neighborhoods, 
      reducing cost.
      Additionally, log-policy gradient 
      scales almost linearly for the three proposed models.
      This applies to the first two policies. 
      For the third, gradient evaluation requires derivatives of higher-order 
      transitions \eqref{eqn:generalized_higher_order_transition}, 
      involving all paths up to order $k$ between nodes. 
      Since the mesh is static and $k$ is fixed, these paths can be precomputed, 
      and automatic differentiation can further aid scalability.

      Beyond utility evaluation, gradient computation can benefit the most from parallelization. 
      Stochastic gradients can be distributed across samples, and for each sample log-probability 
      gradients can be computed in parallel over path elements. 
      This approach is straightforward for the first two policies but requires extra care for
      the higher-order model due to derivative evaluation of higher-order transitions.

      A further consideration is the number of utility evaluations, which is the
      dominant cost when the utility is expensive. In the proposed approach this
      number equals the sample size $\Nens$ per iteration times the number of
      iterations. The optimal baseline allows a small batch size, here
      $N_{\rm b}=1$. In all experiments the budget is fixed at $\Nens=32$ samples
      over at most $300$ iterations, the same for every mesh resolution and sensor
      count. The design space itself is far larger. The coarse experiment alone
      admits $307{,}200$ feasible paths (\Cref{subsubsec:benchmark}), and the fine
      mesh is larger still. The number of utility evaluations is therefore set by
      the optimization budget and not by the size of the design space. When the
      utility is expensive this count is the limiting factor, and the black-box
      acceleration techniques noted above apply directly. A comparison with
      gradient-based, parametrized-path alternatives and the associated trade-offs
      is given in \Cref{sec:conclusions}.

\section{Numerical Experiments}
  \label{sec:numerical_experiments}
  We numerically test the proposed approach using 
  an advection-diffusion simulation common 
  in OED \cite{PetraStadler11,attia2018goal}, 
  presenting results from multiple experiments, 
  all conducted within the PyOED framework \cite{chowdhary2025pyoed}.

  \subsection{Experimental Setup}
  \label{subsec:advection_diffusion_results}
    The advection-diffusion model simulates the spatiotemporal evolution of a
    contaminant field $u=\xcont(\vec{x}, t)$ in a closed domain $\domain$.
    Given a navigation mesh, we seek 
    an optimal path for moving sensors to measure $u$.   
    The inference parameter is the initial contaminant distribution
    $\xcont(\vec{x}, 0)$, denoted $\iparam$ in the Bayesian inverse problem
    formulated below.
    %

    \paragraph{Model setup: advection-diffusion}
      The contaminant field $u$ 
      is governed by 
        \begin{equation}\label{eqn:advection_diffusion}
          \begin{aligned}
            \xcont_t - \kappa \Delta \xcont + \vec{v} \cdot \nabla \xcont &= 0     
              \quad \text{in } \domain \times [0,T],   \\
            \xcont(x,\,0) &= \iparam \quad \text{in } \domain,  \\
            \kappa \nabla \xcont \cdot \vec{n} &= 0  
              \quad \text{on } \partial \domain \times [0,T],
          \end{aligned}
      \end{equation}
      where $\kappa\!>\!0$ is the diffusivity and $T$ is the simulation final time. 
      The spatial domain is $\domain = [0,1]^2$ with 
      two rectangular interior regions representing buildings 
      where flow cannot enter. 
      The boundary $\partial \domain$ includes both the outer boundary 
      and the building walls. 
      Here $\vec{v}$ is the velocity (\Cref{fig:AD_Setup}, left) 
      that is obtained by solving a steady Navier--Stokes equation with 
      sidewall–driven flow. 
      The initial contaminant distribution shown in \Cref{fig:AD_Setup} 
      is used as the ground truth to create synthetic simulations.
      \begin{figure}[!htbp]
        \centering
        \includegraphics[width=0.225\linewidth]{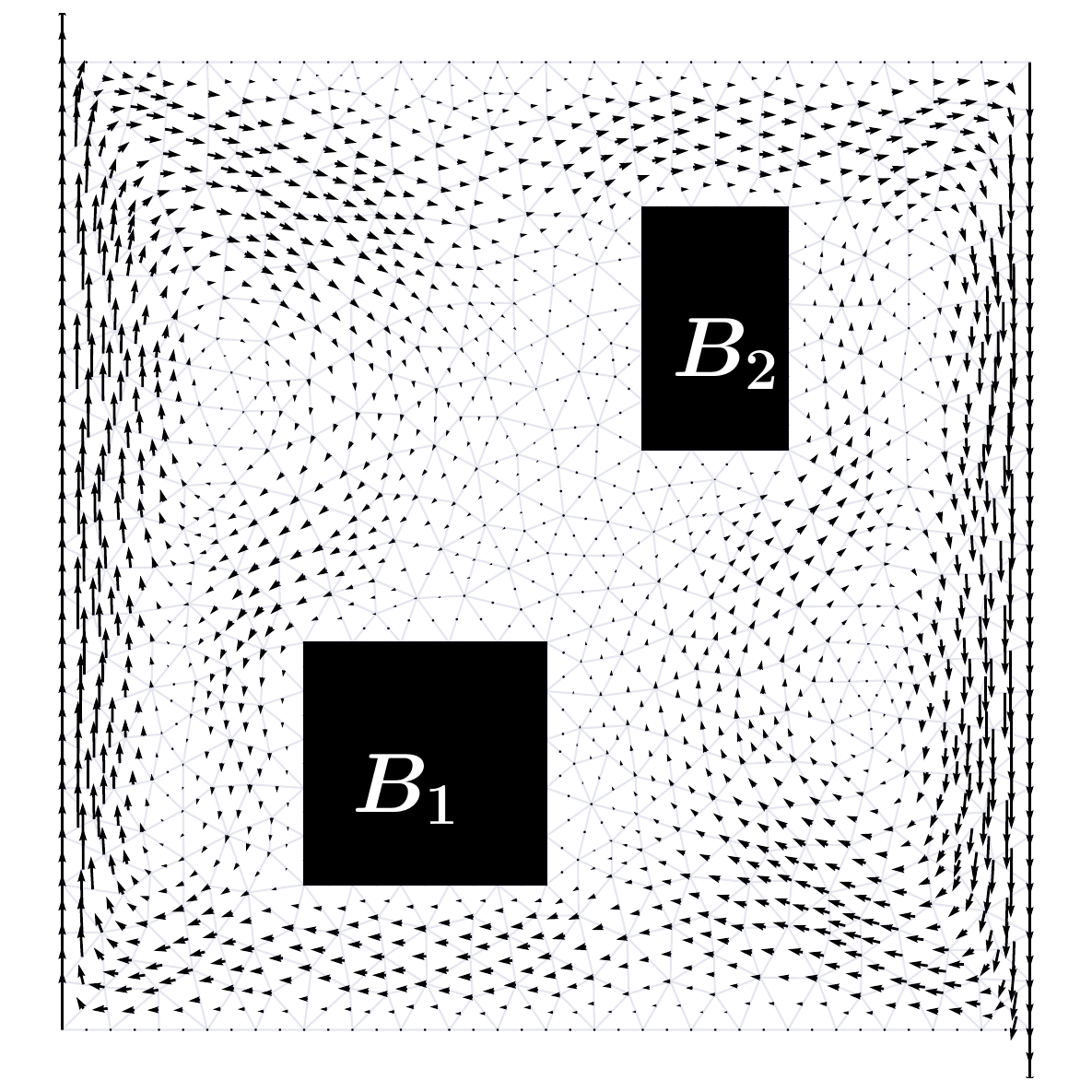}
        \includegraphics[width=0.245\linewidth]{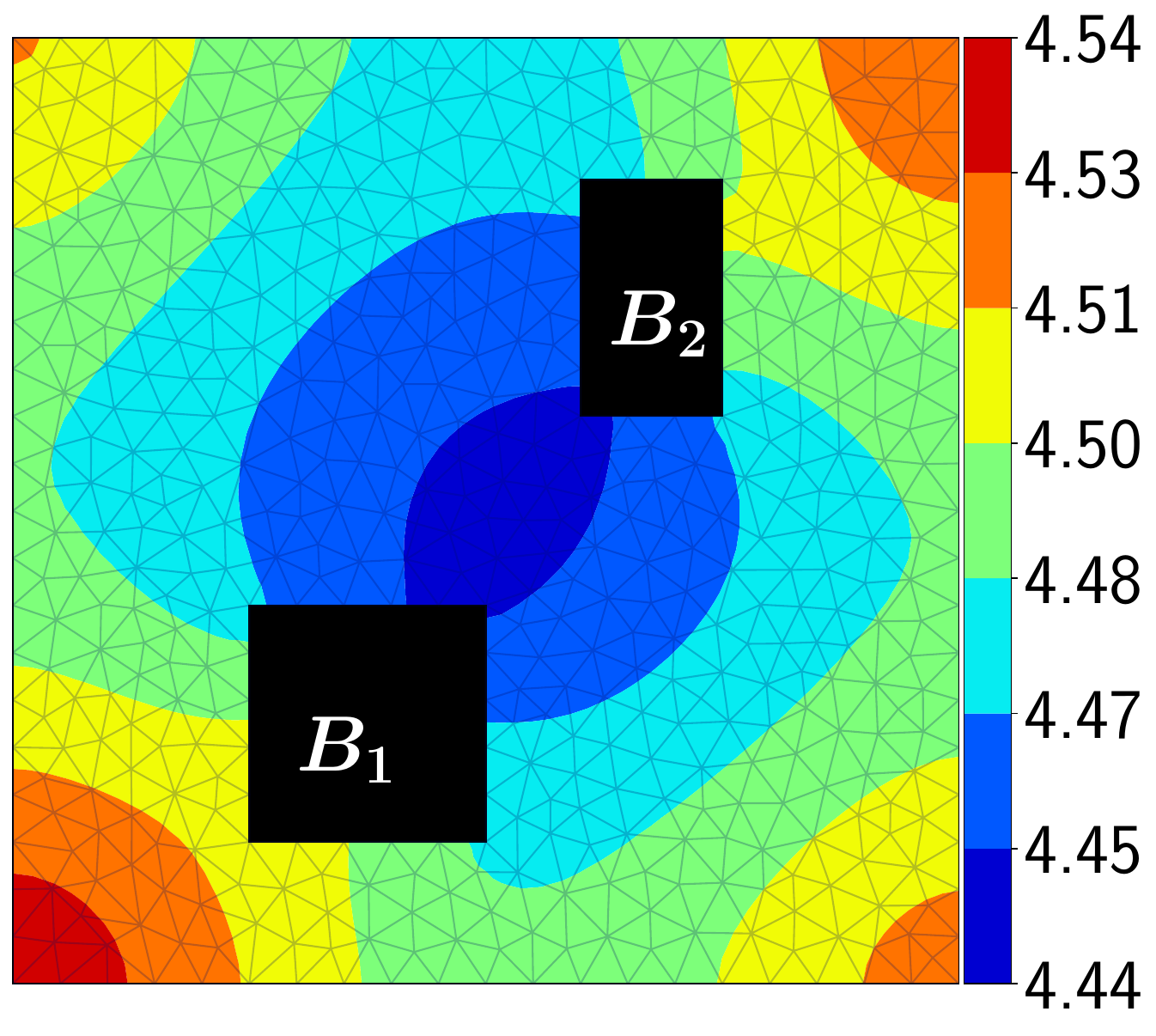}
        \includegraphics[width=0.245\linewidth]{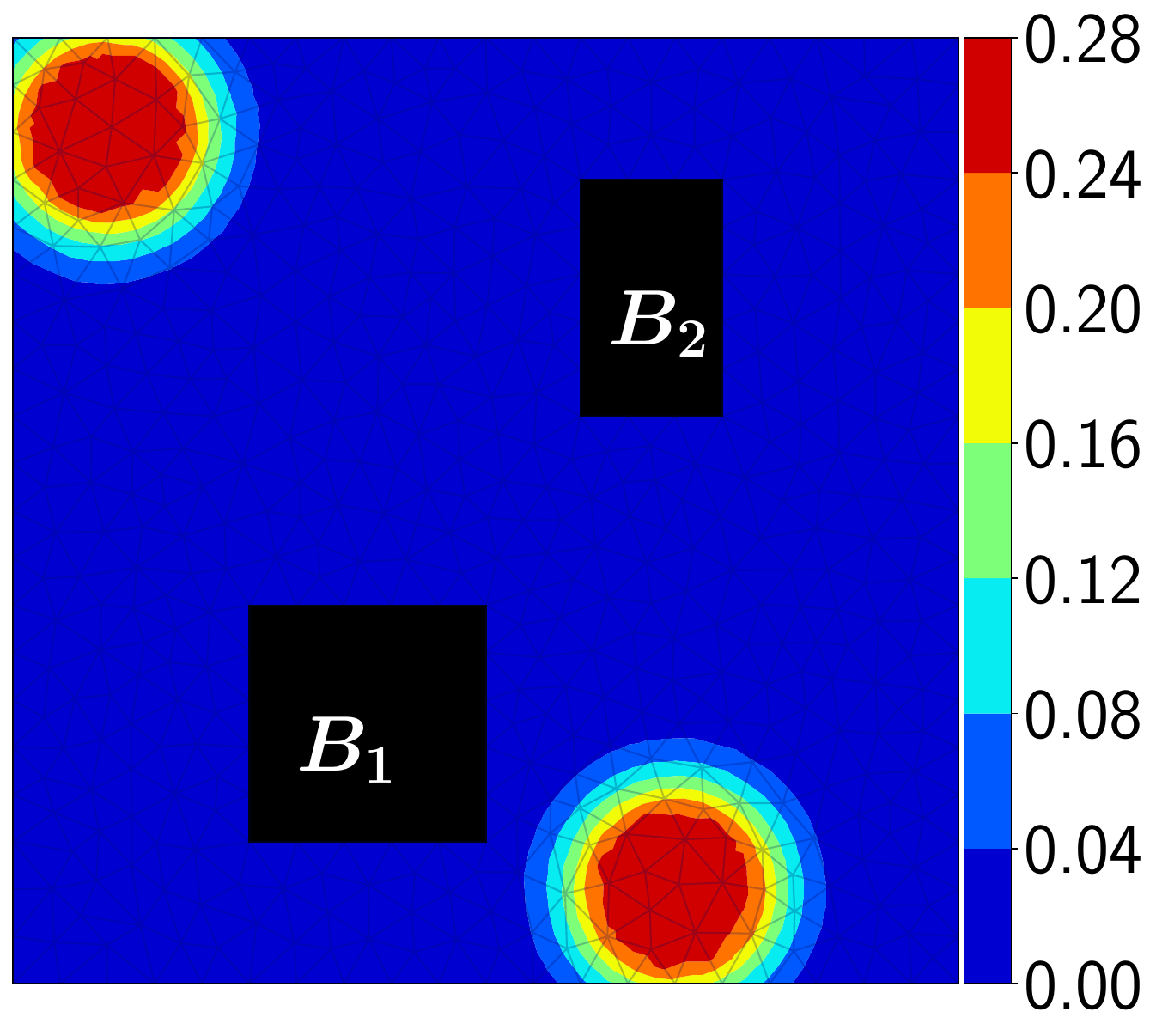}
        \includegraphics[width=0.245\linewidth]{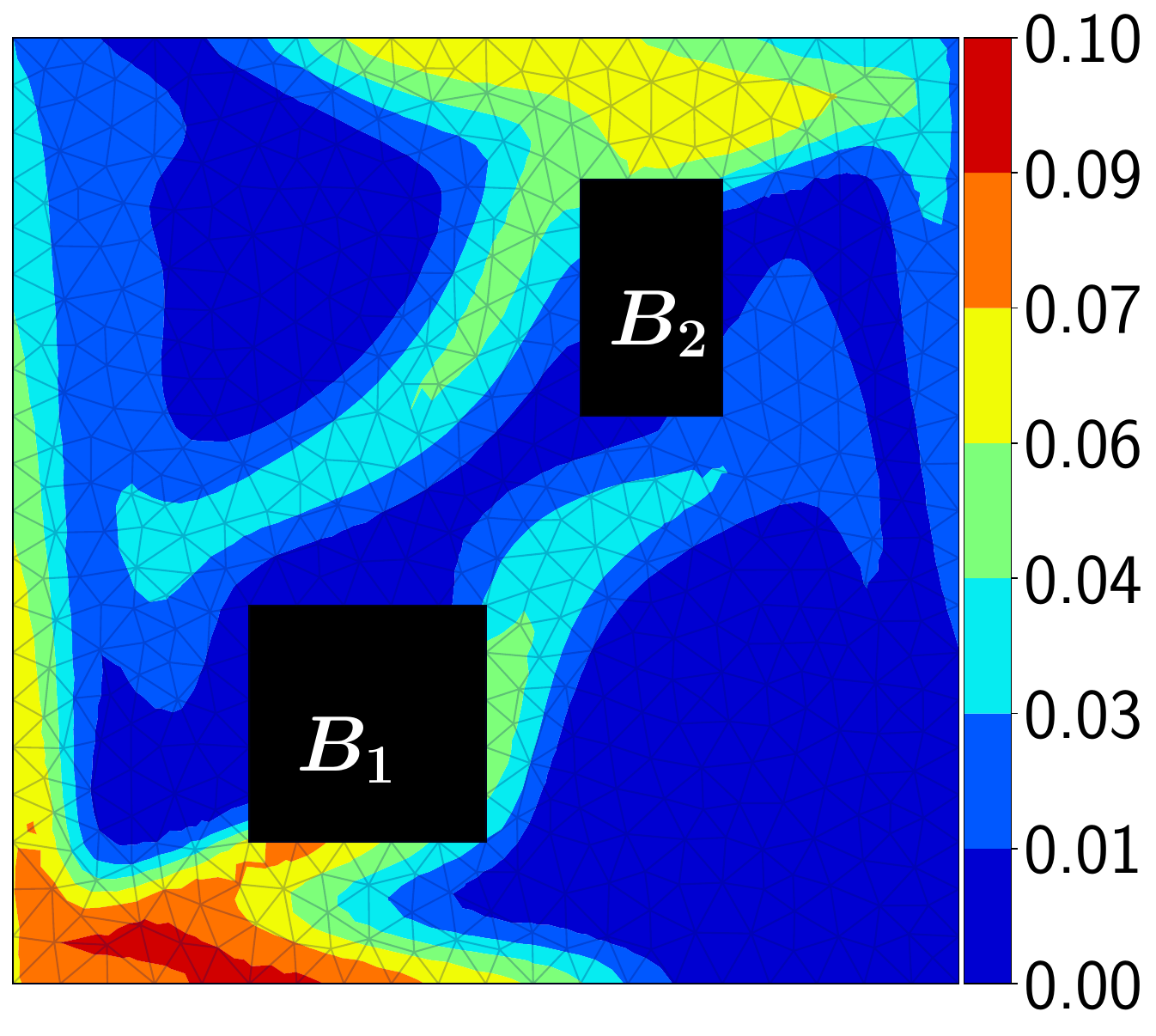}
        \caption{
          Advection-diffusion model \eqref{eqn:advection_diffusion}. 
          Panels from left to right are 
          (1) the constant velocity field;
          (2) the prior variance field, that is, the diagonal of the prior 
              covariance matrix;
          (3) ground truth of the inference parameter; 
            that is, the true initial condition;  and 
          (4) the model state at the final simulation time instance $T=3.6$.
          }
        \label{fig:AD_Setup} 
      \end{figure}
      %

    \paragraph{Forward operator, adjoint operator, and the prior}
      The forward operator $\F$ maps the model parameter $\iparam$, here the
      model initial condition, to the observation space.
      Specifically, $\F$ represents a forward simulation over the interval 
      $[0, T]$ followed by applying a restriction/observation operator 
      to extract concentrations $u$ at sensor locations at the observation
      times. 
      The observations are collected along the observation trajectory, and thus
      the observation operator is time-dependent and is tied to the path as described below.
      Following the standard decomposition (see, e.g.,~\cite{attia2022optimal}), the
      forward operator at time instance $t_k$ factors as
      $\Find{0}{k} = \ObsOperind{k}\,\Solind{0}{k}$.
      Here $\Solind{0}{k}$ is the solution operator that propagates the inference
      parameter $\iparam$ (the initial condition) over the time interval
      $[t_0, t_k]$ by solving~\eqref{eqn:advection_diffusion}, producing the model
      state at time $t_k$; this propagation is independent of the design.
      The observation operator $\ObsOperind{k}$ then evaluates the simulated state
      at the sensor locations at time $t_k$, and the dependence on the design, for
      example, follows from the fact that $\ObsOperind{k}\equiv\ObsOperind{k}(\design)$
      is determined by the spatial locations visited at time $t_k$.
      Here 
      $\F$ is linear, and the adjoint is defined as 
      $\F\adj \!= \!\mat{M}\inv\mat{F}\tran$, where
      $\mat{M}$ is the finite-element mass matrix.
      %
      The prior distribution of the parameter $\iparam$ is
      modeled by a Gaussian $\GM{\iparb}{\Cparampriormat}$, 
      where $\Cparampriormat$ is a discretization of $\mathcal{A}^{-2}$,
      with $\mathcal{A}$ being a Laplacian.
      The prior variance field 
      is plotted 
      in the second panel of \Cref{fig:AD_Setup}.

    \paragraph{Navigation mesh}
      We use two navigation meshes (\Cref{fig:navigation_meshes}): 
      a coarse mesh (left) for exact trajectory analysis 
      and a fine mesh (right). 
      Each node connects to its four nearest neighbors, approximating movement 
      in cardinal directions, 
      with no self-loops. 
      Navigation meshes can be designed to include additional connectivity constraints.

      \begin{figure}[!htbp]
        \centering
        \includegraphics[width=0.30\linewidth]{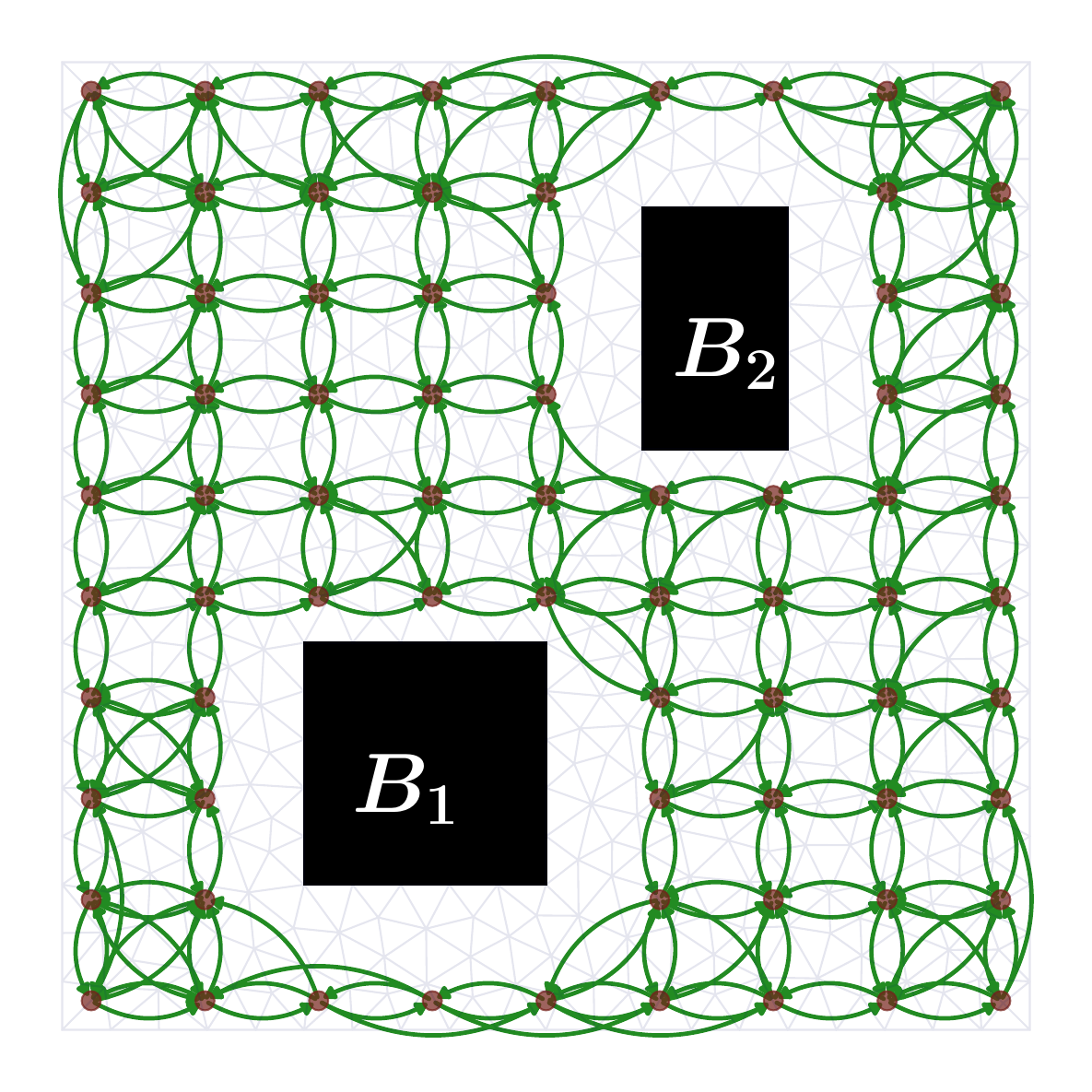}
        \quad 
        \includegraphics[width=0.30\linewidth]{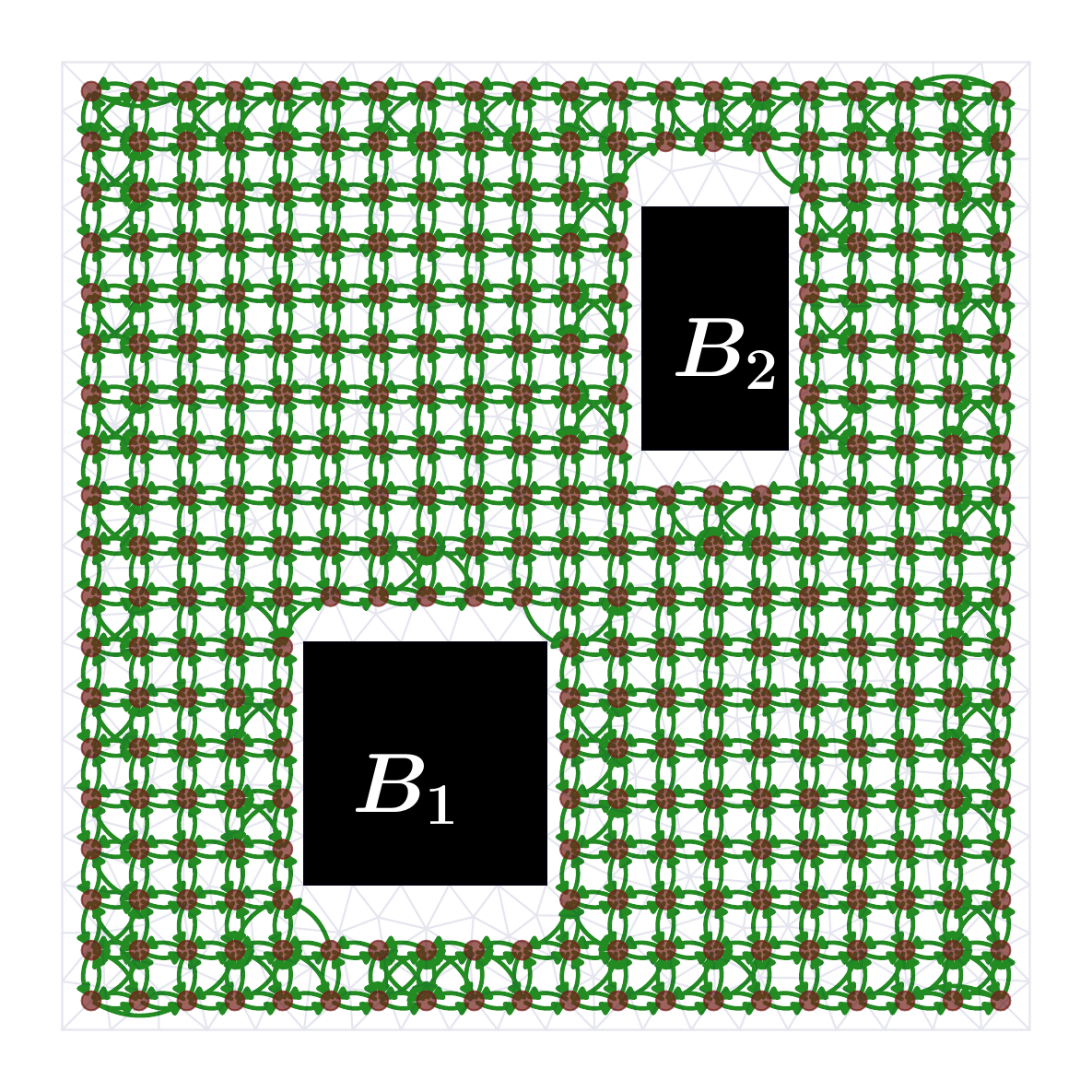}
        \caption{Navigation meshes used. 
          Left: a coarse mesh with $75$ candidate locations (graph vertices/nodes).
          Right: a fine navigation mesh with $332$ nodes.
          }
        \label{fig:navigation_meshes} 
      \end{figure}
      %

    \paragraph{Trajectory and observational setup}
      %
      An observation vector $\obs$ represents the concentration of the
      contaminant  observed along the trajectory.
      The observation times are $ i f \Delta t$, where  
      $\Delta t=0.2$ is the model simulation timestep;
      $f>0$ is the observation temporal frequency; 
      and $i=1, \ldots, n$.
      Thus, irrespective of $f$, the trajectory length (number of nodes) is $n$.
      In this setup, where each visited node corresponds to one observation time,
      the trajectory length $n$ coincides with the number of observation time
      instances, i.e., $n = \nobstimes$; in general, however, $n$ and $\nobstimes$
      could be viewed as independent quantities, allowing for trajectory nodes
      between observation times.
      For a group of $s$ sensors, the node on the path is the spatial gridpoint at the 
      center of the group, resulting in an observation vector of size 
      $\Nobs \!=\! s \!\times\! \nobstimes$.
      Here we show results obtained by using only one moving sensor ($s=1$) yielding scalar 
      observations along the trajectory.
      In the supplementary material (see \Cref{sup:sec:numerical_results})   
      we discuss numerical experiments carried out with $s=7$ moving sensors.
      %

    \paragraph{Observation frequency and path synchronization}
      Because the design $\design$ defines spatiotemporal observation locations, 
      the observational path must be synchronized with the model simulation times.
      The transition from one node to the next occurs at 
      $f \Delta t$. 
      In our setup we use observation frequency $f=7$ with a trajectory length $n=7$ nodes when 
      the coarse navigation mesh (\Cref{fig:navigation_meshes}, left) is used.
      When the fine navigation mesh (\Cref{fig:navigation_meshes}, right) is used, 
      the observation frequency is set to $f=1$, and the trajectory length is $n=19$ nodes.
      Thus, both trajectories end up at the same observation time matching the simulation 
      final time of $T=3.6$. 
      This enables matching the length of traveled paths in both cases to each other.

    %
    \paragraph{Bayesian inverse problem}
      The forward problem $\obs = \F(\iparam)+\vec{\epsilon}$ relates the inference 
      parameter $\iparam$ to the observations $\obs$. 
      The observation error $\vec{\epsilon}$ follows a Gaussian 
      distribution $\GM{\vec{0}}{\Cobsnoise}$, with
      $\Cobsnoise\!\in\! \Rnum^{\Nobs\times\Nobs}$ 
      describing spatiotemporal correlations. 
      For simplicity, we assume that observation errors are uncorrelated, 
      with diagonal $\Cobsnoise$. 
      Observation noise variances are set to 5\% of the maximum absolute 
      contaminant concentration at each observation point, 
      estimated from a simulation over $[0, T=3.6]$ using the 
      true model parameter; see~\Cref{fig:AD_Setup}.
      
      In this case 
      the posterior is Gaussian $\GM{\ipara}{\Cparampostmat}$   
      (see, e.g., \cite{attia2018goal,stuart2010inverse}) with
      \begin{equation}\label{eqn:Posterior_Params}
        \Cparampostmat = \left(\F \adj \Cobsnoise^{-1} \F 
          + \Cparampriormat^{-1} \right)^{-1} \,, \quad
        \ipara = \Cparampostmat \left( \Cparampriormat^{-1} \iparb 
          + \F\adj \Cobsnoise^{-1}\, \obs \right) \,.
      \end{equation}
      For OED, both factors in~\eqref{eqn:Posterior_Params} inherit the design
      dependence: $\F\equiv\F(\design)$ through the design-dependent observation
      operator $\ObsOperind{k}(\design)$ introduced above, and the noise
      covariance is generally block-diagonal,
      \begin{equation*}
        \Cobsnoise(\design) = \directsum{k=1}{\nobstimes}{\mat{\Gamma}_{\rm noise, k}(\design)} \,,
      \end{equation*}
      where each diagonal block
      $\mat{\Gamma}_{\rm noise, k}(\design)\in\Rnum^{s\times s}$
      characterizes the spatial correlation between the $s$ moving sensors at
      observation time $t_k$. 
      For a full treatment of correlated observational errors in the context of OED, 
      see, e.g., \cite{attia2022optimal}.
      In all experiments in this paper we assume, for
      simplicity, uncorrelated observation errors, so each block
      $\mat{\Gamma}_{\rm noise, k}(\design)$ is diagonal and $\Cobsnoise(\design)$
      is therefore an overall diagonal matrix. The 
      design-dependent posterior covariance matrix $\Cparampost(\design)$ in general reads
      \begin{equation}\label{eqn:Posterior_Params_design_dependent}
        \Cparampost(\design) = \bigl(\F(\design)\adj\,\Cobsnoise(\design)^{-1}\,\F(\design)
          + \Cparampriormat^{-1} \bigr)^{-1} \,.
      \end{equation}
      %

    \paragraph{Initial location and initial policy parameter}
      The initial policy parameter $\hyperparamvec^{(0)}$ passed to 
      \Cref{alg:probabilistic_path_optimization} encodes both the initial point of the trajectory 
      and the admissible navigation directions.
      A fixed start at node $v_i$ is enforced by setting 
      $\hyperparam_j=\delta_{ij}$, while multiple admissible starts are handled by 
      assigning nonzero values (e.g., $0.5$) to the corresponding entries and zero elsewhere.
      Here we allow the trajectory to start anywhere on the navigation mesh by 
      setting the initial distribution parameter to $\hyperparam_{i}=0.5, i=1, \ldots, \Nsens$. 
      Additional experiments employing a fixed starting point are discussed in 
      the supplementary material in \Cref{sup:sec:numerical_results}.
      The transition probabilities $\hyperparam^{i}_{j}$ in all experiments 
      are initialized to $0.5$ for all arcs in the navigation mesh.
      For higher-order policies, the initial lag weights are initialized to 
      $\lambda_i=1/k, i=1, \ldots, k, $ when the weights are allowed to be optimized; 
      otherwise they are modeled by \eqref{eqn:decreasing_lag_weights}.
      %

  \subsection{Utility Function and the OED Optimization Problem}
  \label{subsec:the_path_optimization_problems}
    In Bayesian OED an optimal design $\design\opt$ aims to minimize the uncertainty 
    in the solution of the inverse problem~\eqref{eqn:Posterior_Params_design_dependent}, 
    by minimizing a scalar summary of $\Cparampost(\design)$.
    Popular optimality criteria include the trace (A-optimal), 
    determinant or log-determinant (D-optimal), and  maximum eigenvalue (E-optimal) of 
    the posterior covariance.
    Here we seek D-optimal paths, with experiments targeting A- and E-optimal 
    paths deferred to \Cref{sup:subsec:Other_Utility_Functions}.
    Thus, the probabilistic OED optimization 
    \Cref{prbl:probabilistic_path_oed_problem} becomes 
    \begin{equation}\label{eqn:D-opt_probabilistic_optimization}
      \hyperparamvec\opt \in \argmin_{\designvec\sim\CondProb{\designvec}{\hyperparamvec}} 
        \Expect{\designvec\sim\CondProb{\designvec}{\hyperparamvec}}{\utilityfunc(\designvec)} 
        \,; \qquad \utilityfunc(\designvec) = \logdet{\Cparampost(\designvec)}
          \,.
    \end{equation}

    Unless stated otherwise, we use the stochastic gradient with optimal baseline estimate
    \eqref{eqn:baseline_gradient_group} with batch size $N_{\rm b}=1$.
    Justification of this choice is 
    discussed in \Cref{subsec:importance_of_baseline}.
    All experiments use a maximum of $300$ iterations, and \Cref{alg:probabilistic_path_optimization}
    terminates early if the parameter update norm 
    $\wnorm{\hyperparamvec^{(i+1)}-\hyperparamvec^{(i)}}{2}$
    falls below $10^{-12}$.

  \subsection{Results with Coarse Navigation Mesh and $1$ Moving Sensor}
  \label{subsec:numerical_results_coarse}
    This experiment employs the coarse navigation mesh in \Cref{fig:navigation_meshes} (left).
    The trajectory length is set to $n=7$ nodes (observation time points), and 
    the observation frequency is set to $3\Delta t$, with $\Delta t=0.2$ 
    being the model simulation step.
    This is a simplified setup that 
    allows us to enumerate all \emph{feasible} paths and thus 
    enables sketching the distribution of the utility function $\utilityfunc$,  
    evaluating the exact objective $\Expect{}{\utilityfunc}$ and 
    the exact gradient $\nabla_{\hyperparamvec} \Expect{}{\utilityfunc}$, 
    and finding the global optima for benchmarking. 

    \subsubsection{Benchmark: Utility Function Distribution and Global Optimum}
    \label{subsubsec:benchmark}
      The support cardinality---total number of feasible paths---in this case is $307,200$, 
      with objective values $\utilityfunc$ displayed in \Cref{fig:coarse_bruteforce} (left).
      The utility evaluation for each trajectory 
      took (on average) $0.3$ seconds, totaling 
      approximately $27$ hours to generate brute-force results on a non-parallel platform.
      The global optimum path is unique (shown in \Cref{fig:coarse_bruteforce}, right) with minimum value $-21714.37$. The corresponding posterior variance field (diagonal of $\Cparampost$ in \eqref{eqn:Posterior_Params}) exhibits a substantial reduction in uncertainty relative to the prior variance field in \Cref{fig:AD_Setup}.

      \begin{figure}[!htbp]
        \centering
        \includegraphics[width=0.60\linewidth]{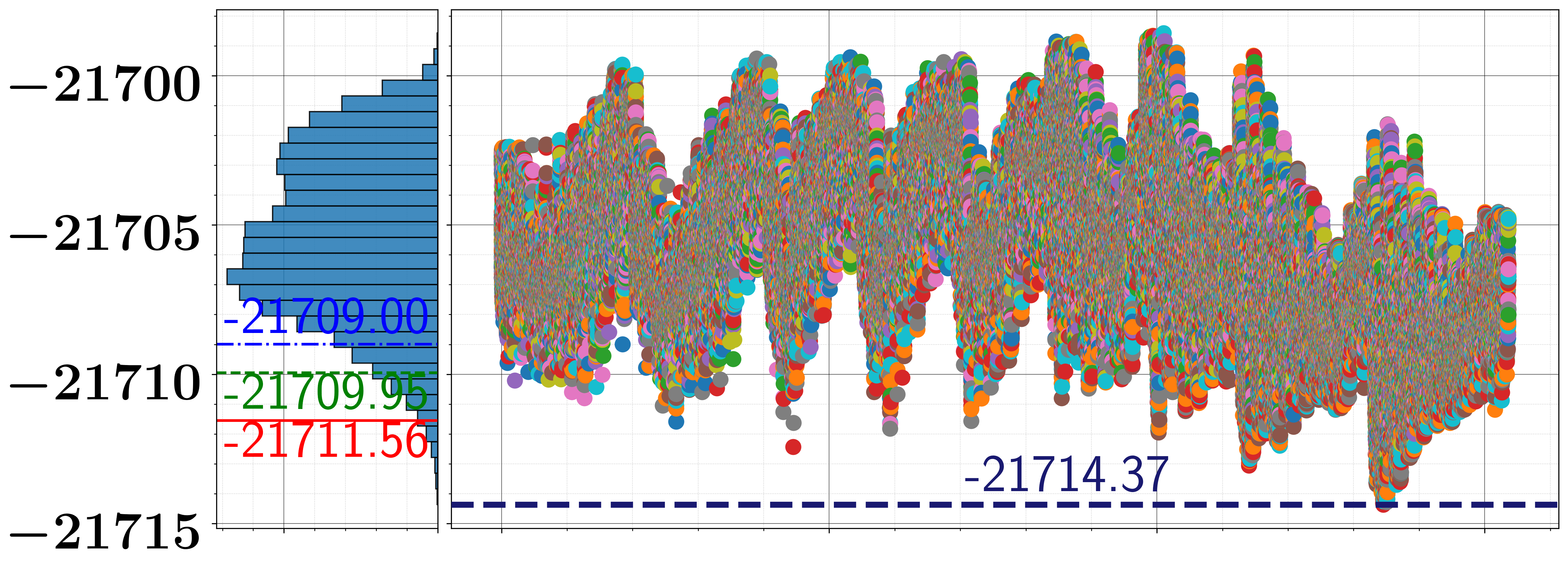}
          \quad
          \includegraphics[width=0.245\linewidth]{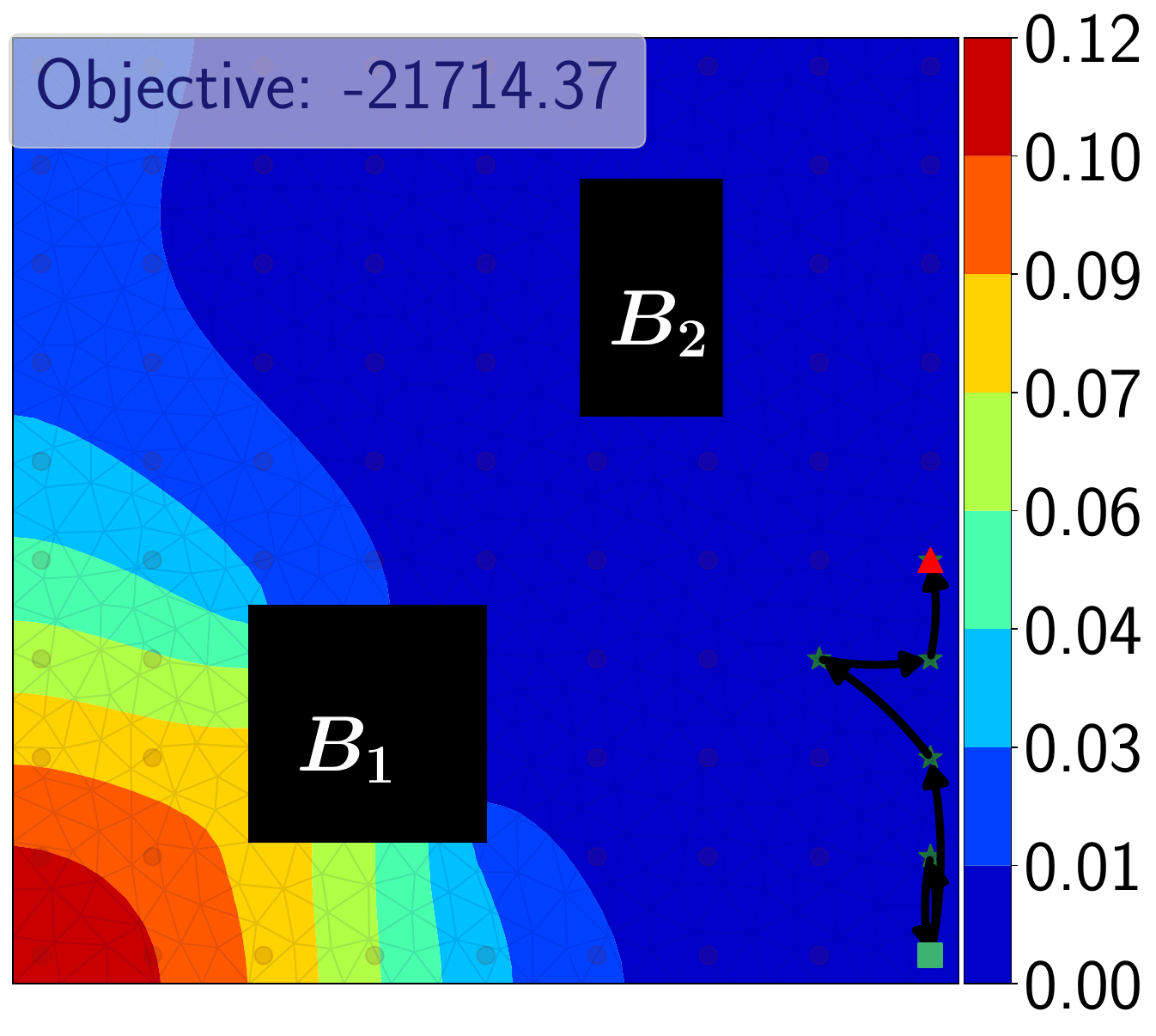}
        \caption{
          Brute-force results for the coarse experiment.
          Left: The marginal distribution of the utility function
          $\utilityfunc$ \eqref{eqn:D-opt_probabilistic_optimization} over all
          feasible paths, displayed as a horizontal histogram. The dashed
          horizontal lines mark the first, second, and third quartiles of the
          distribution.
          Middle: The utility values of all $307{,}200$ feasible paths, with
          each path assigned a unique index on the x-axis and its utility
          value plotted on the y-axis. The dashed horizontal line at
          $-21714.37$ marks the unique global optimum.
          Right: The global optimum path along with the corresponding
          posterior variance field.
        }
      \label{fig:coarse_bruteforce}
      \end{figure}

      The scatter plot in the middle panel of \Cref{fig:coarse_bruteforce}
      shows the utility values of all $307{,}200$ feasible paths against
      their path index. Although the global optimum (dashed line) is unique,
      a substantial number of paths attain near-optimal utility values, as
      also reflected by the concentration in the lower portion of the
      marginal distribution in the left panel. This observation motivates
      exploration of the lower tail (e.g., the lowest $1\%$ of values)
      rather than relying on a single best path, since sampling from a
      probabilistic policy can discover and exploit this rich set of
      near-optimal solutions.
      Accordingly, our goal is to learn a probabilistic policy that samples 
      trajectories with objective/utility values close to the global optimum.

    \subsubsection{Results with the First-Order Policy}
    \label{subsubsec:coarse_first_order_results}
      We start by discussing the results of 
      \Cref{alg:probabilistic_path_optimization} with the first-order policy 
      (\Cref{defn:first_order_path_model}).

      \Cref{fig:coarse_first_order_iterations_with_random_sample} (left) shows the 
      utility function $\utilityfunc$ value
      evaluated at the samples generated at each iteration of the optimization procedure.
      The algorithm identifies the rightmost side of the 
      domain---particularly near building $B_2$ and the lower-right corner---as  
       having higher chances of optimizing the objective when a path is initiated in these locations;
      see \Cref{fig:coarse_first_order_iterations_with_random_sample} (middle). 
      The optimal initial parameters (middle) are sparse, and the optimal transition 
      parameters (right) are likewise sparse across nodes, 
      highlighting preferred navigation directions.
      \Cref{fig:coarse_first_order_iterations_with_random_sample} (left) shows that  
      the policy updates rapidly reduce the average objective, 
      with most improvement in early iterations, suggesting early stopping is possible. 
      In contrast, uniform random sampling fails to capture the lower tail of the utility 
      distribution, even in this simplified setup.

      \begin{figure}[!htbp]
        \centering
        \includegraphics[width=0.52\linewidth]{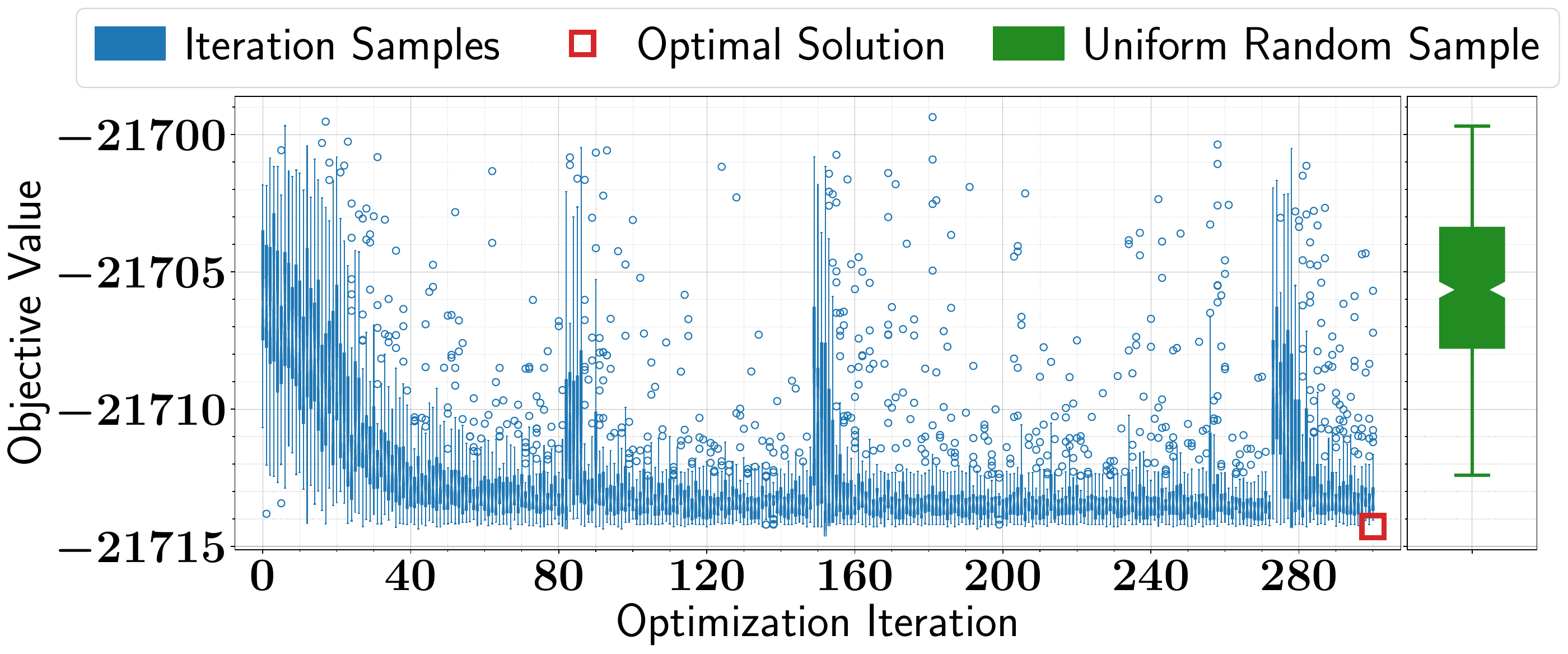}
        \includegraphics[width=0.23\linewidth]{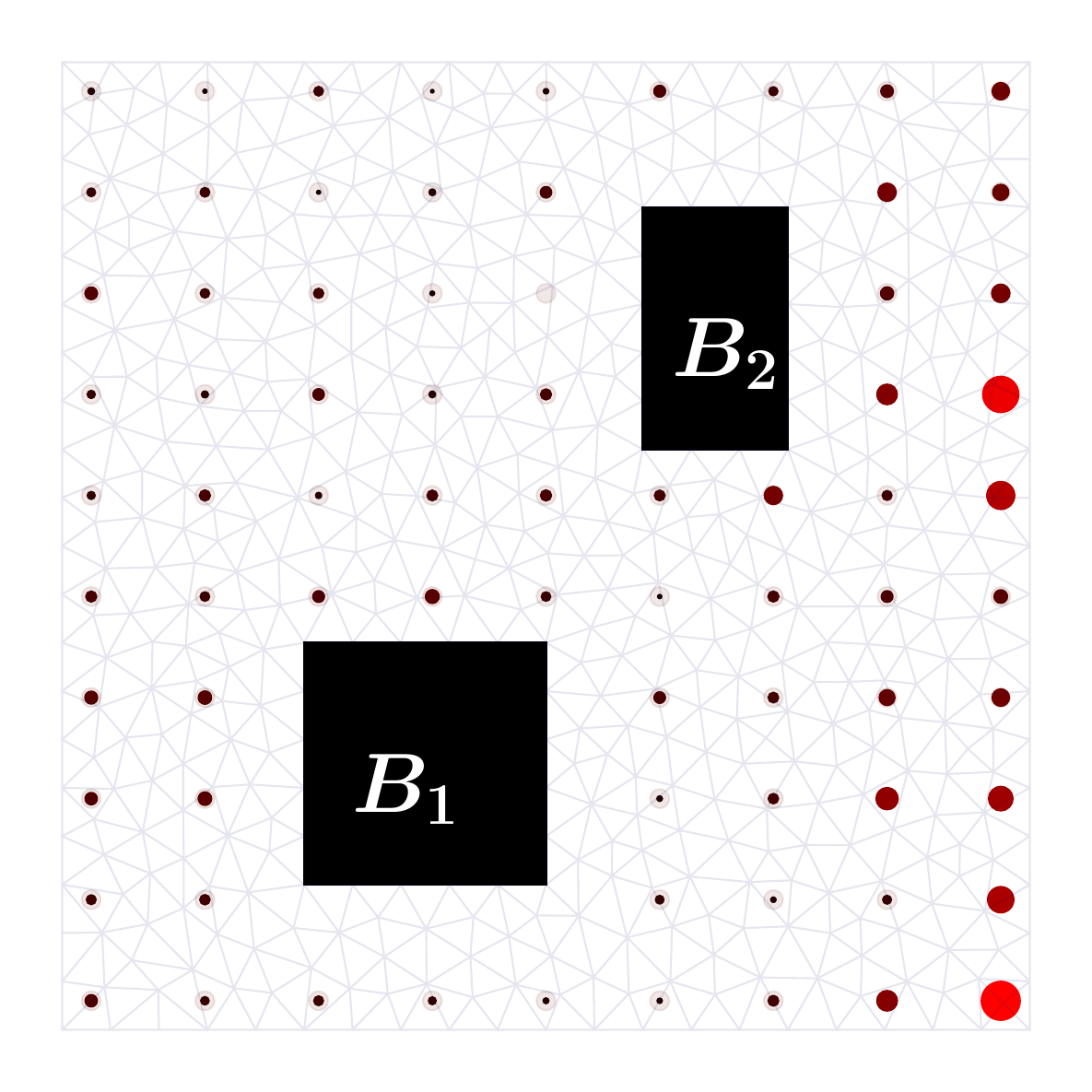}
        \includegraphics[width=0.23\linewidth]{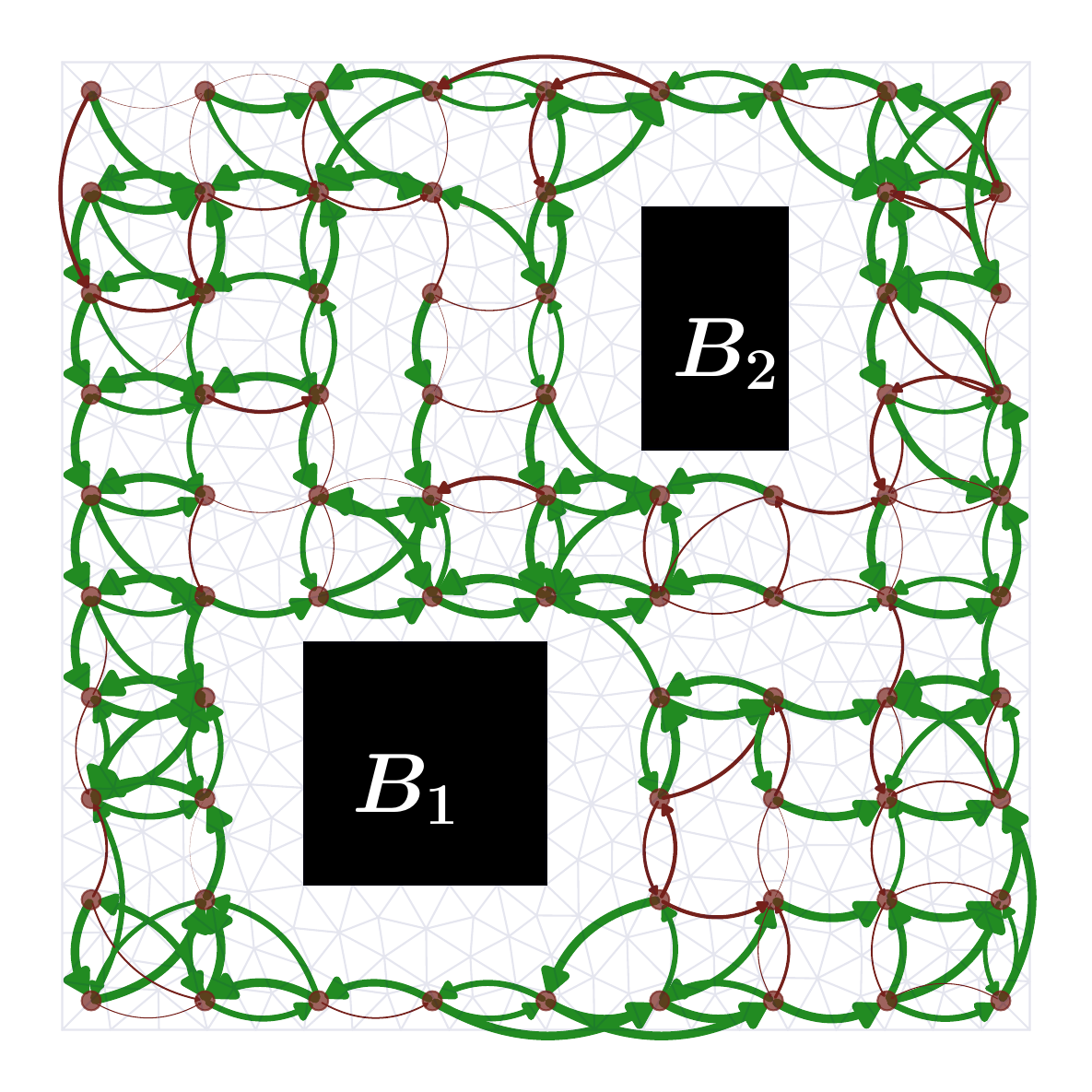}
        \caption{
          Results of \Cref{alg:probabilistic_path_optimization} with the first-order 
          policy (\Cref{defn:first_order_path_model}).
          Left: utility function $\utilityfunc$ 
            \eqref{eqn:D-opt_probabilistic_optimization} evaluated at the samples generated
            at each iteration,  
            with the optimal value 
            shown as a red circle.
            A boxplot of $\utilityfunc$ values corresponding to $500$ 
            uniformly sampled---with  parameters equal $0.5$; paths
            are also displayed.
          Middle: the optimal initial parameters $\pi_i$
            plotted as circles centered at the corresponding mesh nodes, with 
            the circle size proportional to the magnitude of the parameter. 
          Right: the optimal transition parameters $\hyperparam^{i}_{j}$ with 
            line widths proportional to the magnitude 
            of  $\hyperparam^{i}_{j}$ values. 
            Maroon indicates values below $0.5$, and green indicates otherwise.
        }\label{fig:coarse_first_order_iterations_with_random_sample}
      \end{figure}

      The solution (\Cref{fig:coarse_first_order_optimal_solution}, left) 
      returned by \Cref{alg:probabilistic_path_optimization} 
      is associated with 
      an objective value that is close (but not identical) to the global optimal 
      value as shown in \Cref{fig:coarse_first_order_optimal_solution} (left). 
      Moreover, the samples generated from the optimal policy explore the space 
      near the global optimal value, 
      as indicated by the samples shown in \Cref{fig:coarse_first_order_optimal_solution}.
      Additionally, the posterior velocity field given 
      the optimal trajectory and the other optimal policy samples 
      are almost identical to the posterior velocity 
      field associated with the global optimal solution shown in 
      \Cref{fig:coarse_bruteforce} (right).
      This 
      aligns well with our objective to generate a policy that explores the 
      lower tail of the utility distribution. 
      Although global optimality is not guaranteed, with more samples from the optimal policy 
      we were able to recover the global optimum; 
      the results are omitted for brevity. 

      \begin{figure}[!htbp]
        \includegraphics[width=0.245\linewidth]{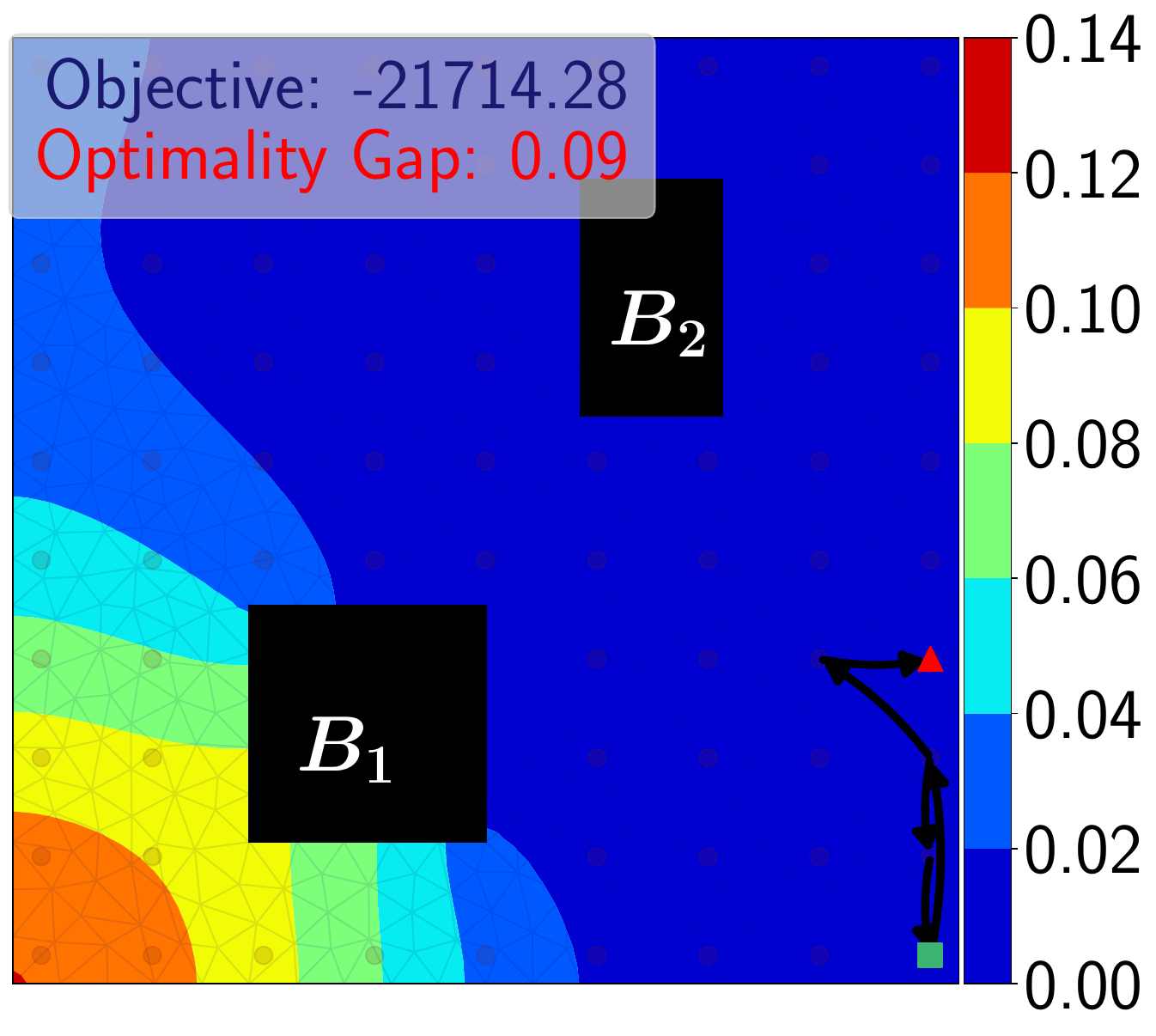}
        \includegraphics[width=0.245\linewidth]{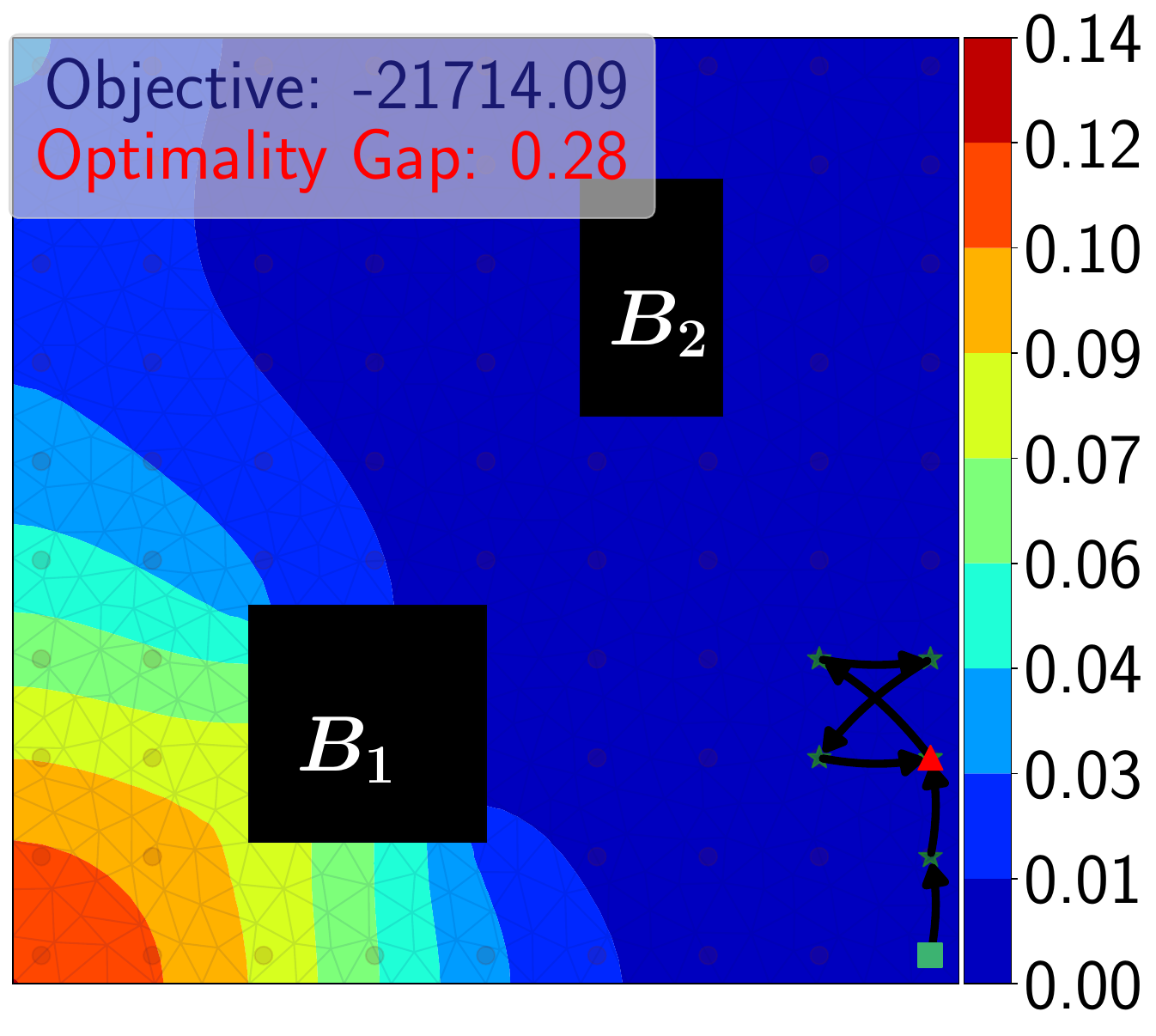}
        \includegraphics[width=0.245\linewidth]{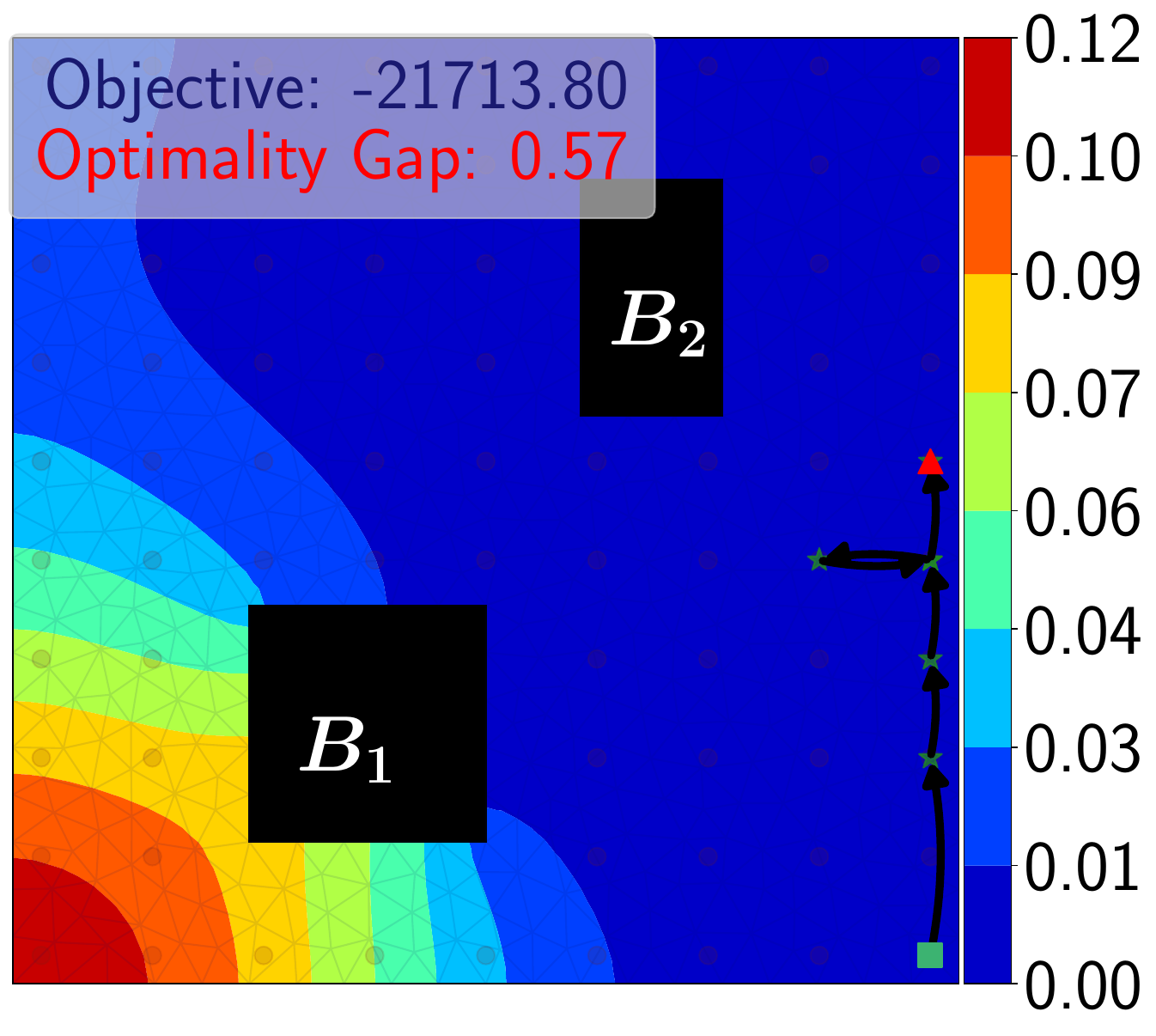}
        \includegraphics[width=0.245\linewidth]{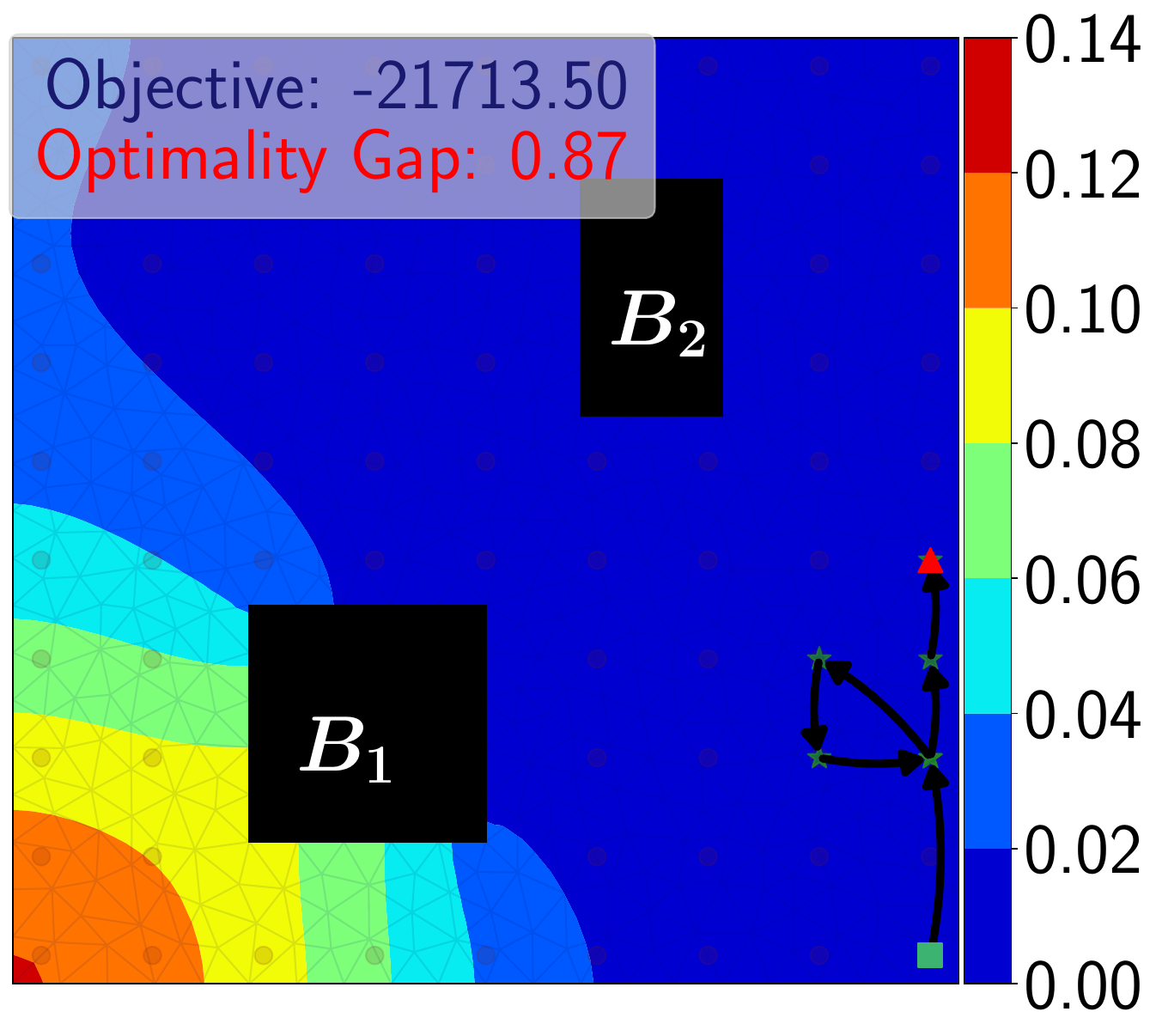}
        \caption{
          Optimal trajectory followed by $3$ samples from 
          the optimal policy in \Cref{fig:coarse_first_order_iterations_with_random_sample}.
        }\label{fig:coarse_first_order_optimal_solution}
      \end{figure}

    \subsubsection{Results with the Higher-Order Policies}
    \label{subsubsec:coarse_higher_order_results}
      We present results of \Cref{alg:probabilistic_path_optimization} with 
      higher-order policies 
      (\Cref{defn:higher_order_path_model}, \Cref{defn:generalized_higher_order_path_model}), 
      comparing optimized versus fixed lag weights (via \eqref{eqn:decreasing_lag_weights}) 
      for orders $k=3$ and $k=5$, respectively.  
      
      Results obtained by using 
      \Cref{defn:higher_order_path_model}
      with orders $k=3, k=5$, are shown in  
      \Cref{fig:coarse_higher_order_iterations_order_3} and 
      \Cref{fig:coarse_higher_order_iterations_order_5}, respectively.
      The results indicate that optimizing lag weights slightly increases 
      optimizer instability.
      While it explores the second percentile of the utility 
      distribution (see \Cref{fig:coarse_bruteforce}, left), 
      this higher-order model generally underperforms the first-order model
      except for higher-order $k=5$ and fixed lag weights.

      \begin{figure}[!htbp]
        \centering
        \includegraphics[width=0.53\linewidth]{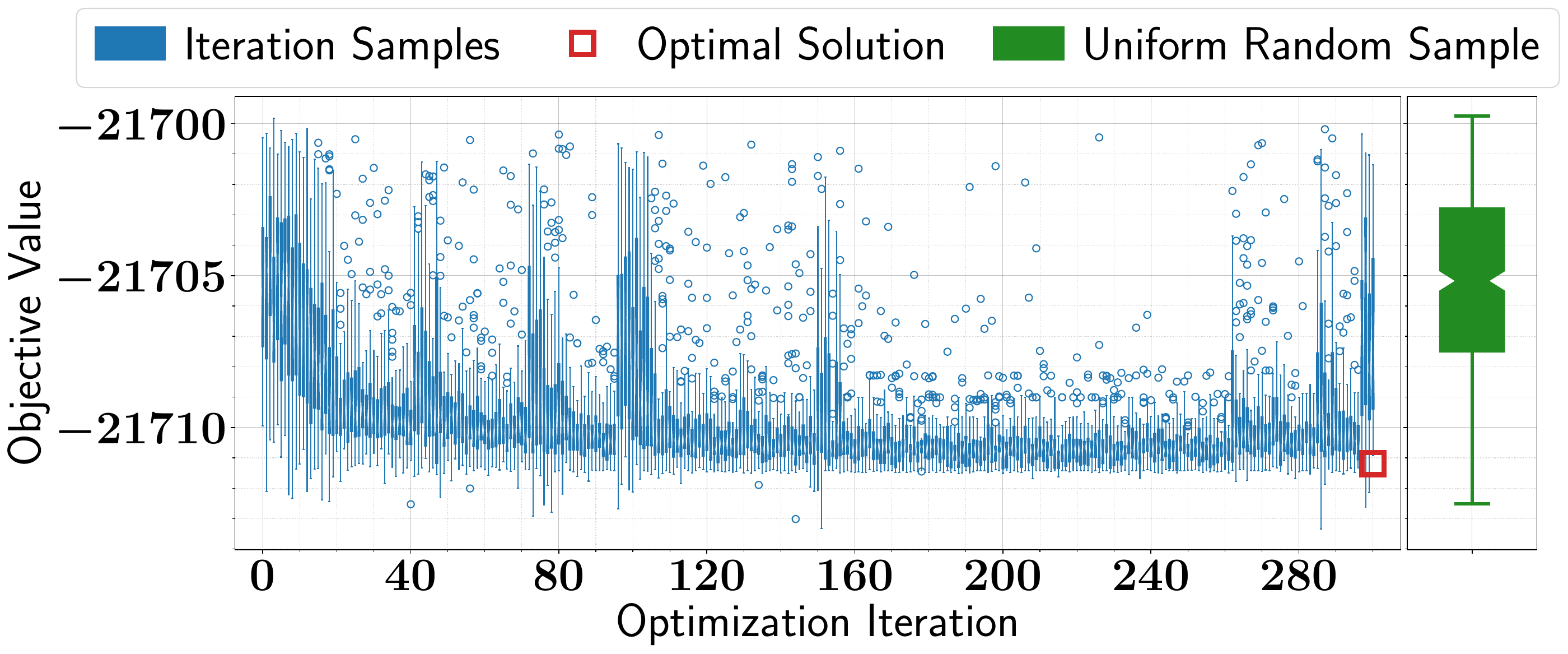}
        \includegraphics[width=0.215\linewidth]{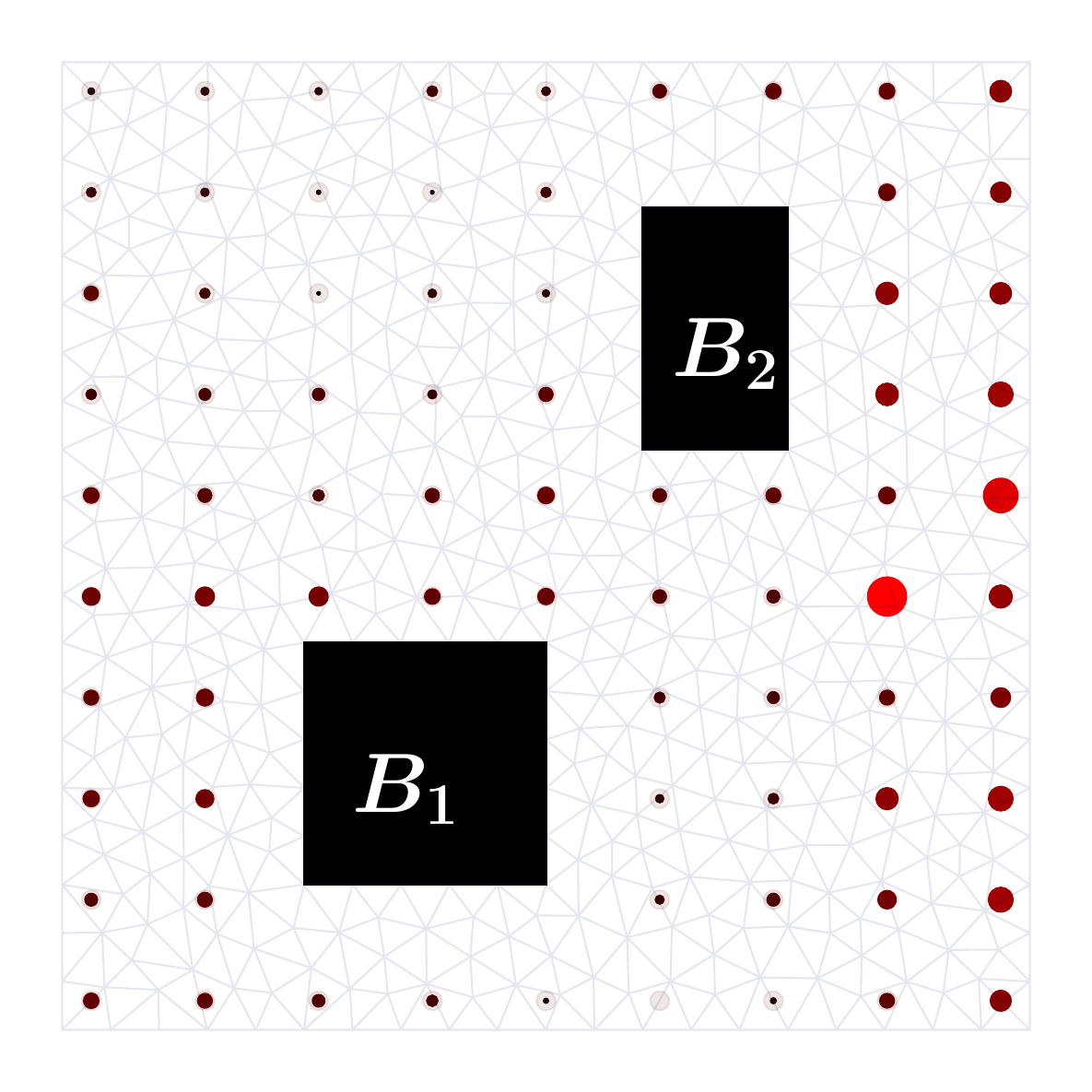}
        \includegraphics[width=0.235\linewidth]{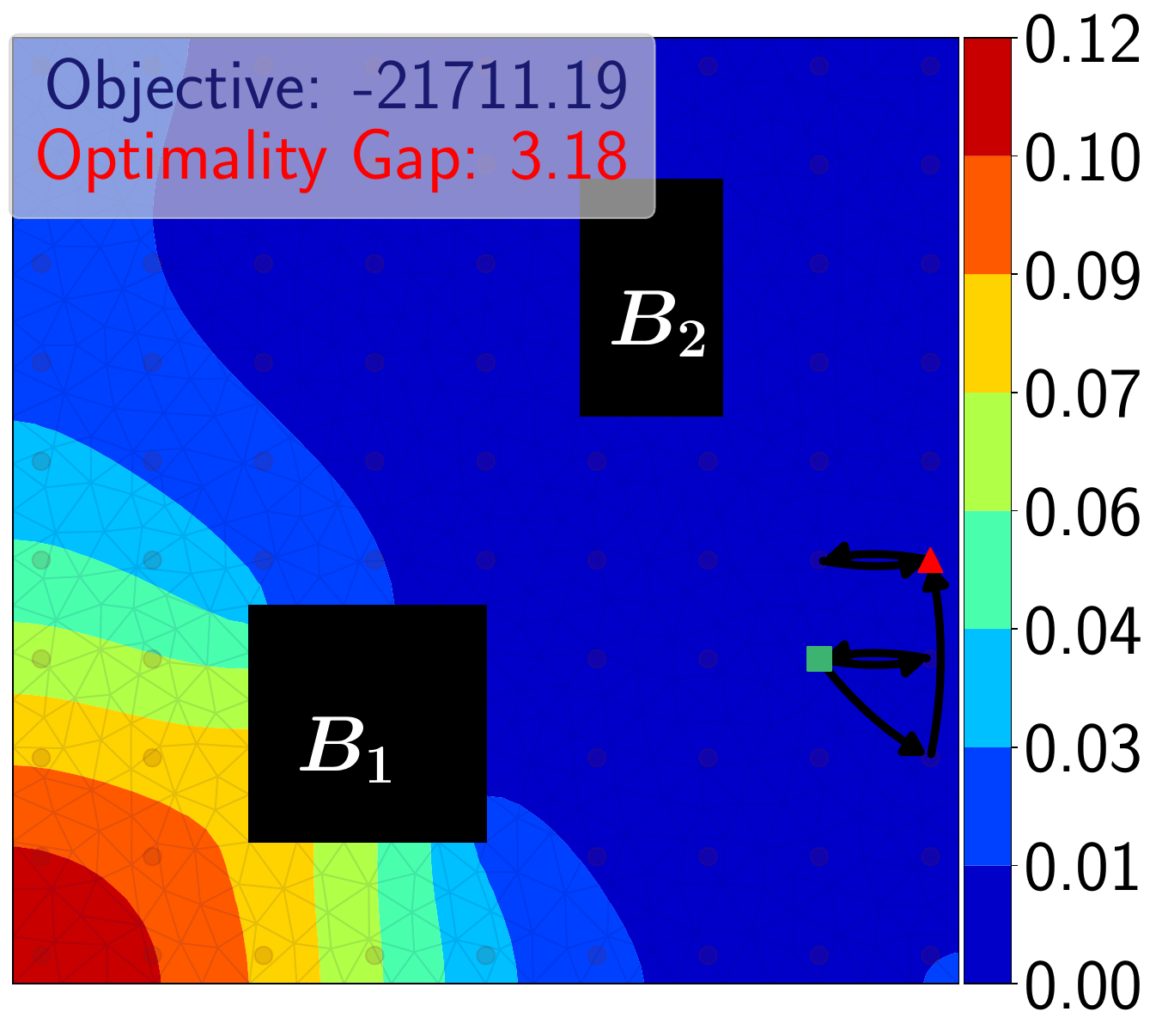}
        \includegraphics[width=0.53\linewidth]{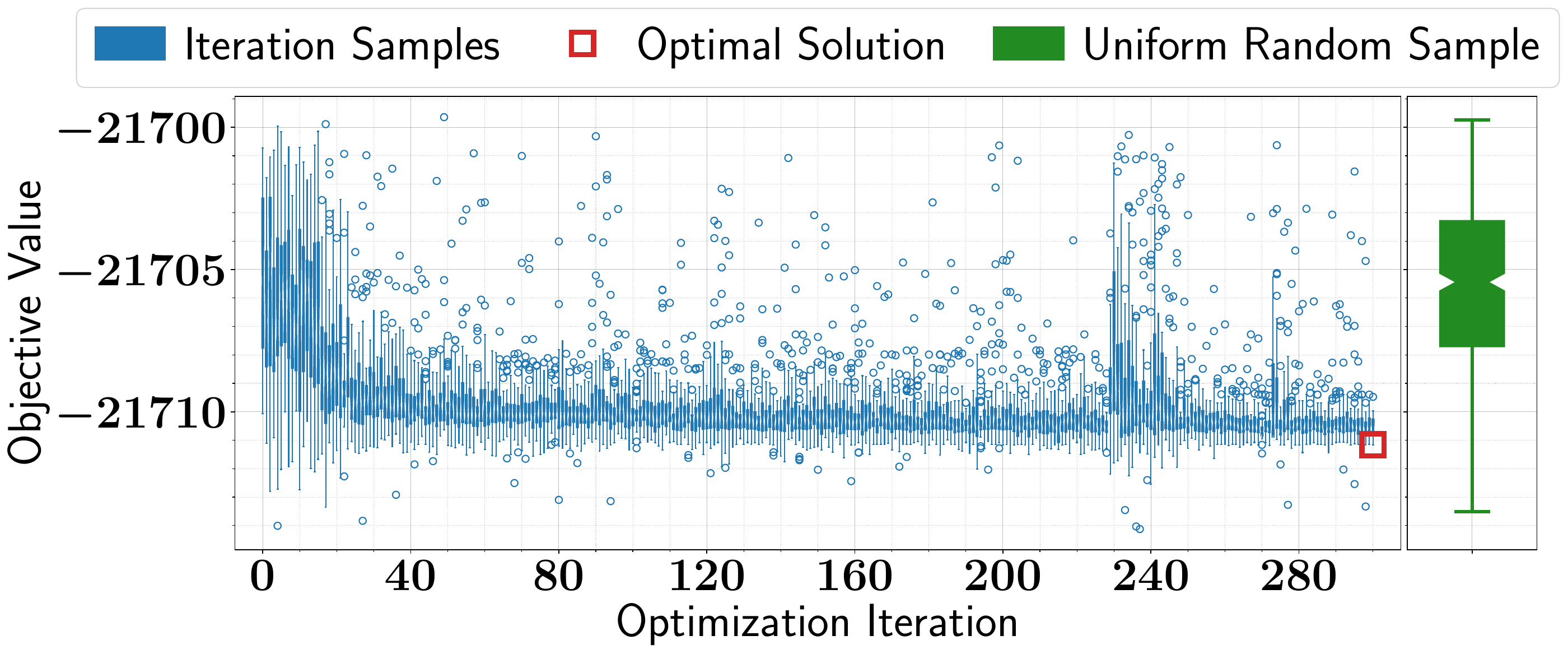}
        \includegraphics[width=0.215\linewidth]{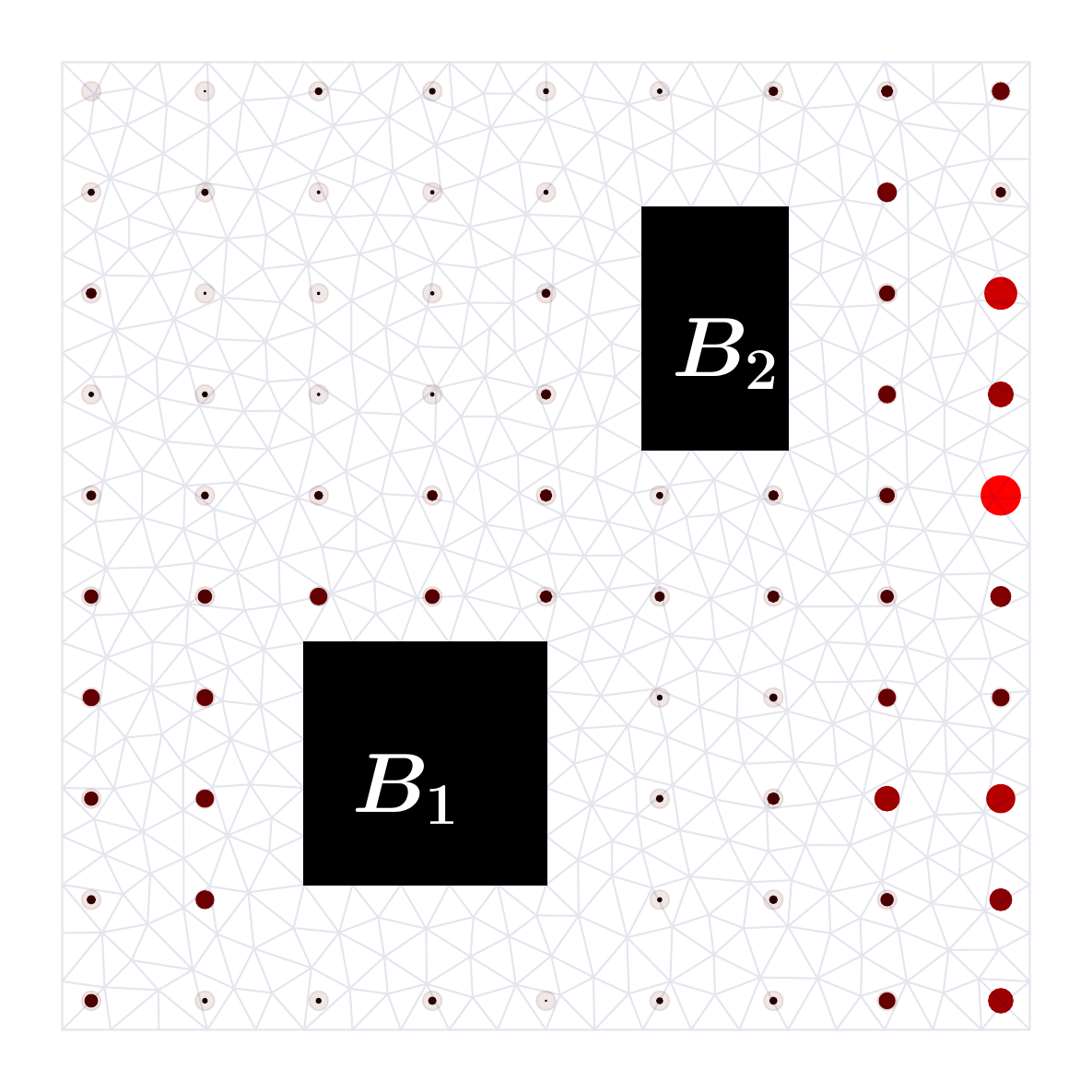}
        \includegraphics[width=0.235\linewidth]{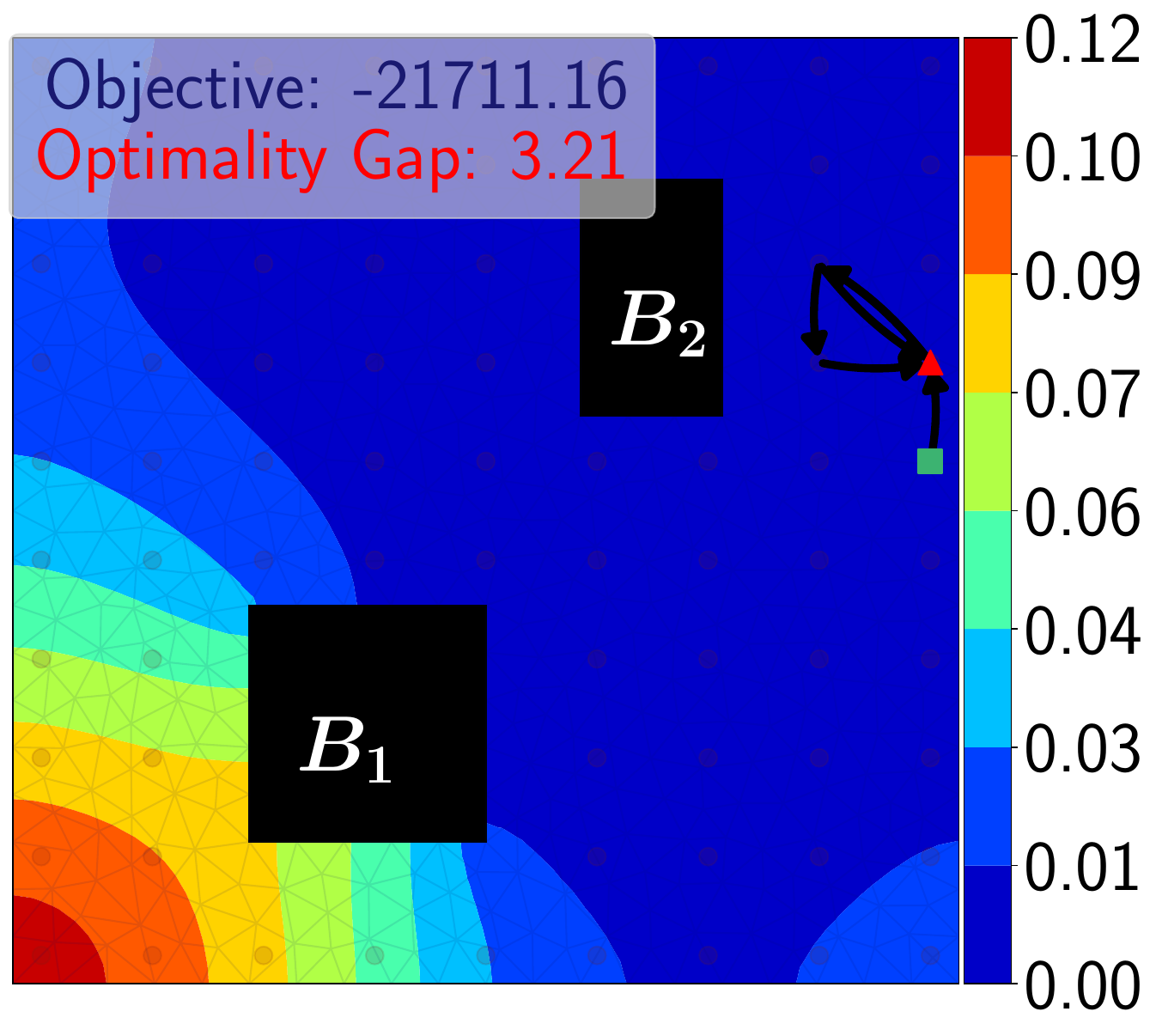}
        \caption{
          Results of \Cref{alg:probabilistic_path_optimization} with the higher-order policy 
          model \Cref{defn:higher_order_path_model} with order $k=3$.
          The first row shows results with lag weights being optimized, and 
          the second row shows results with lag weights modeled 
          by \eqref{eqn:decreasing_lag_weights}.
        }\label{fig:coarse_higher_order_iterations_order_3}
      \end{figure}
      \begin{figure}[!htbp]
        \centering
        \includegraphics[width=0.53\linewidth]{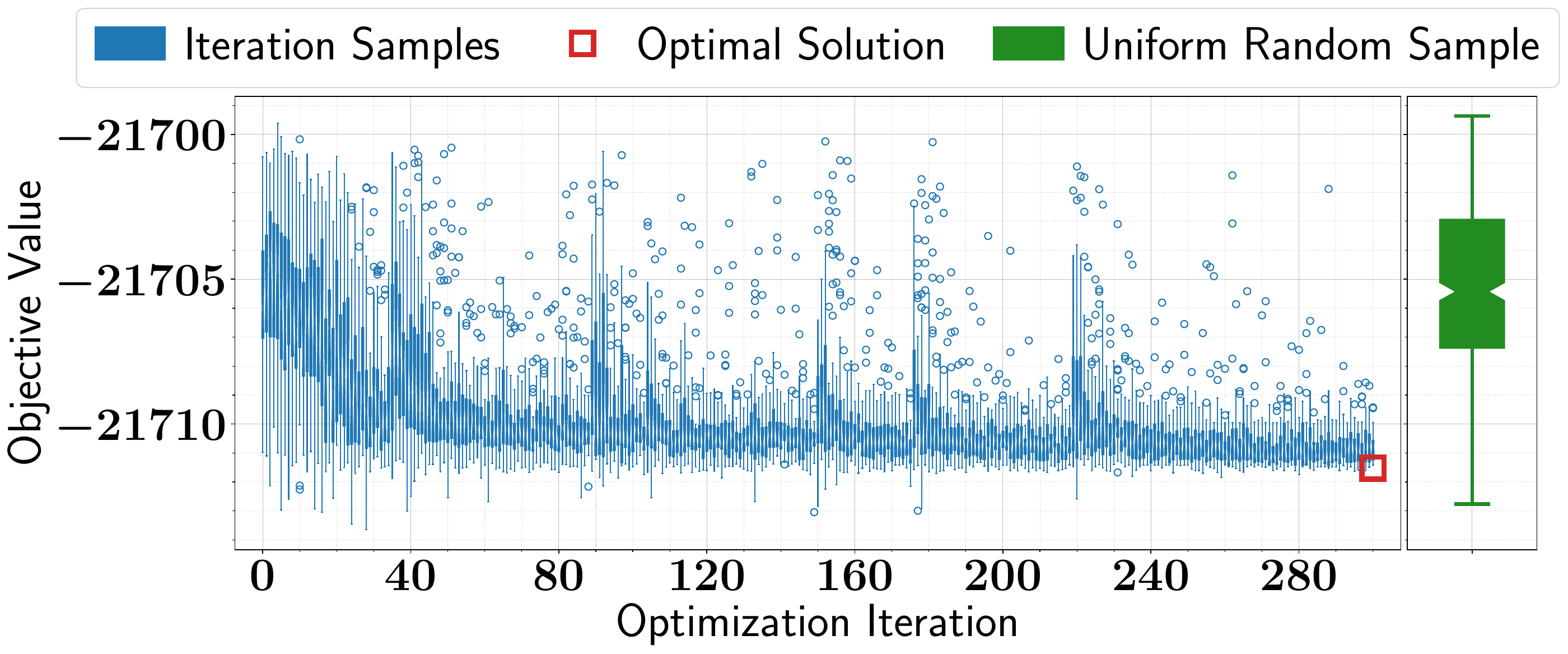}
        \includegraphics[width=0.215\linewidth]{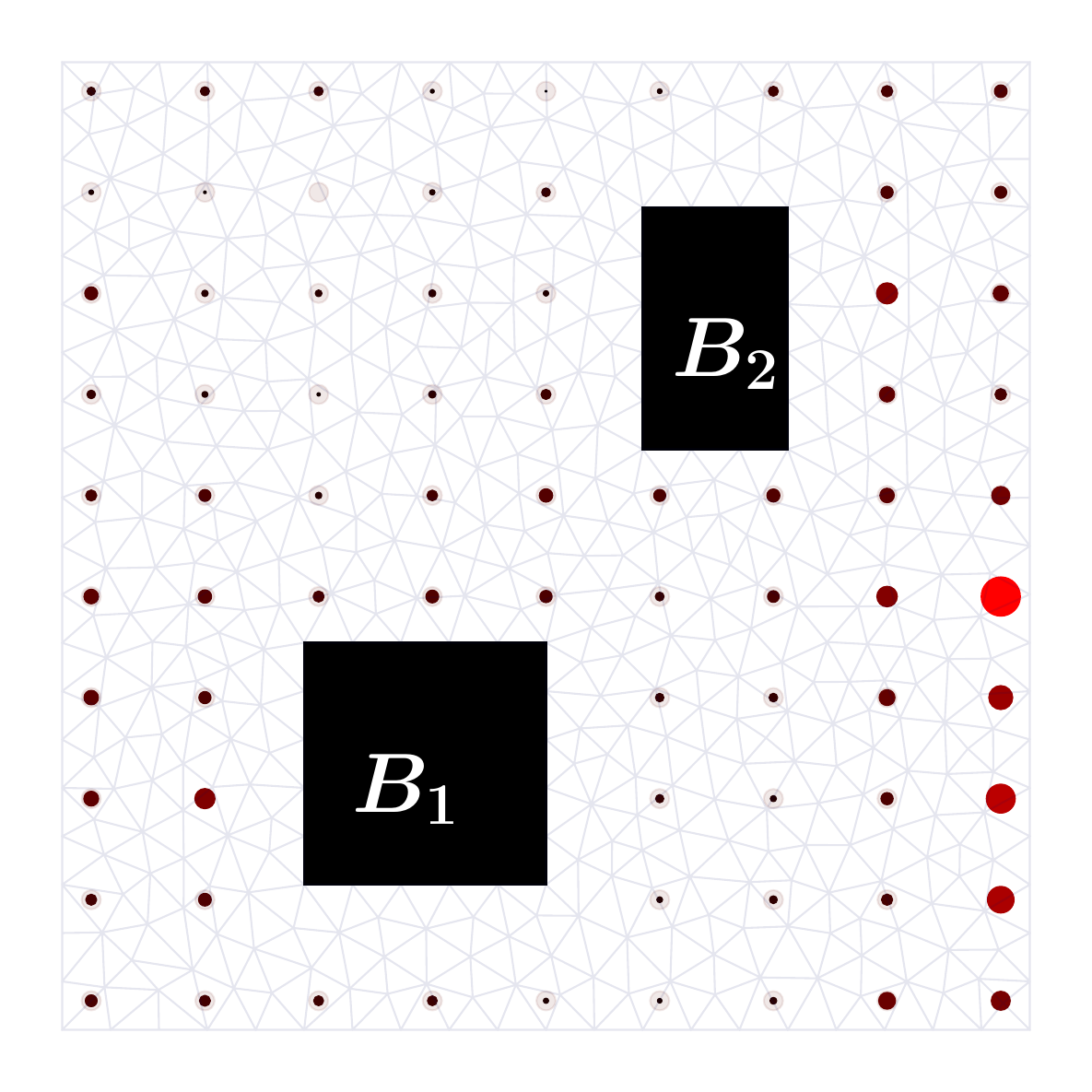}
        \includegraphics[width=0.235\linewidth]{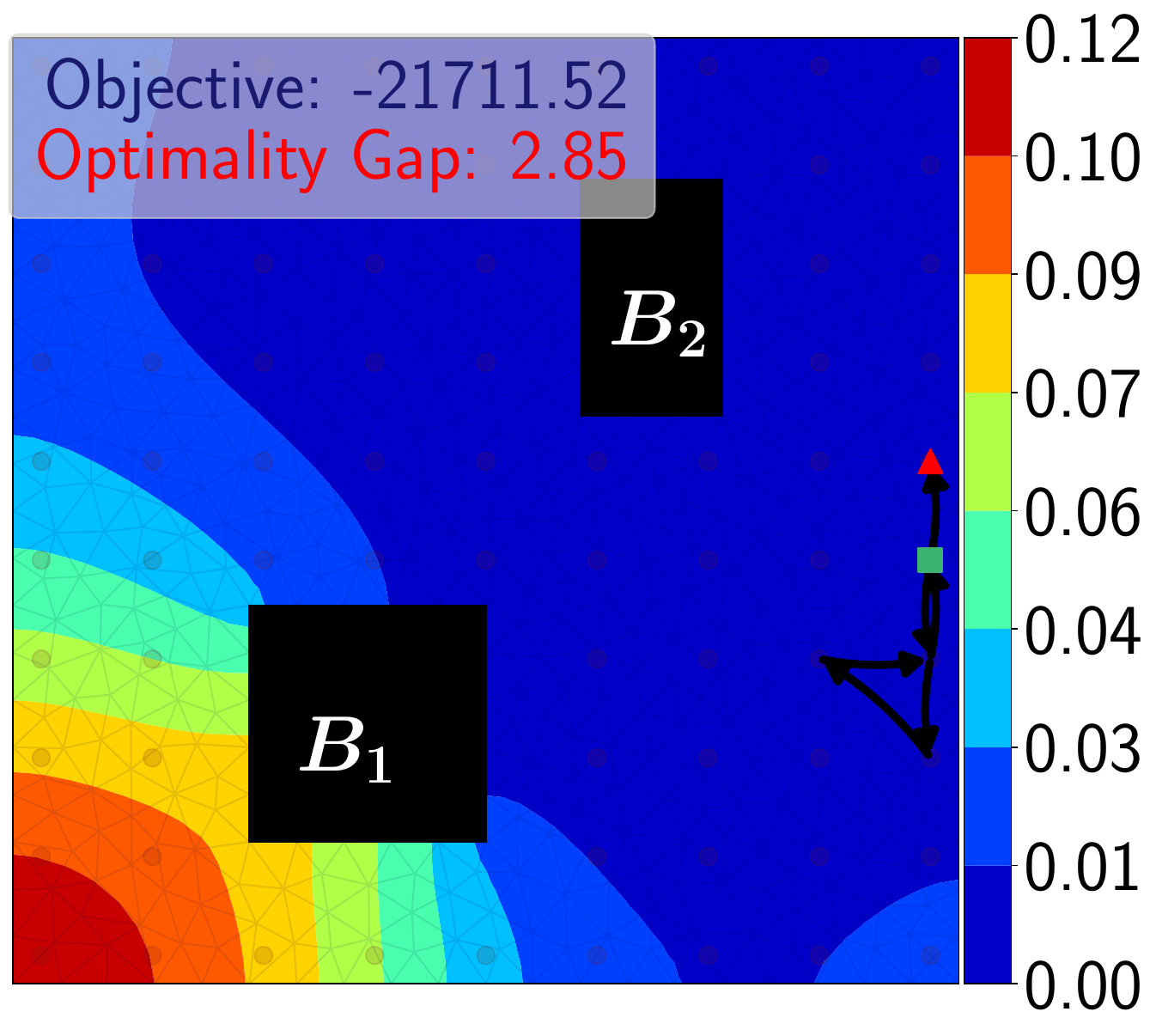}
        \includegraphics[width=0.53\linewidth]{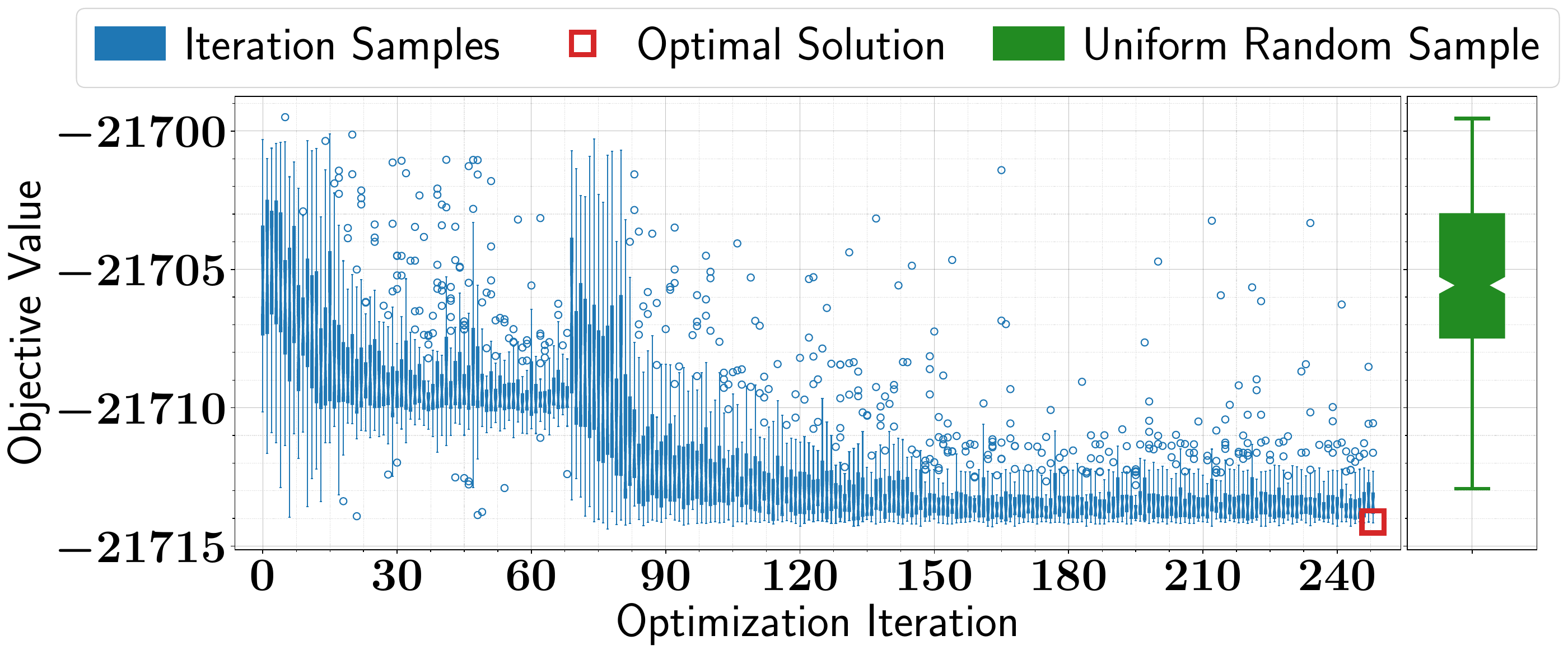}
        \includegraphics[width=0.215\linewidth]{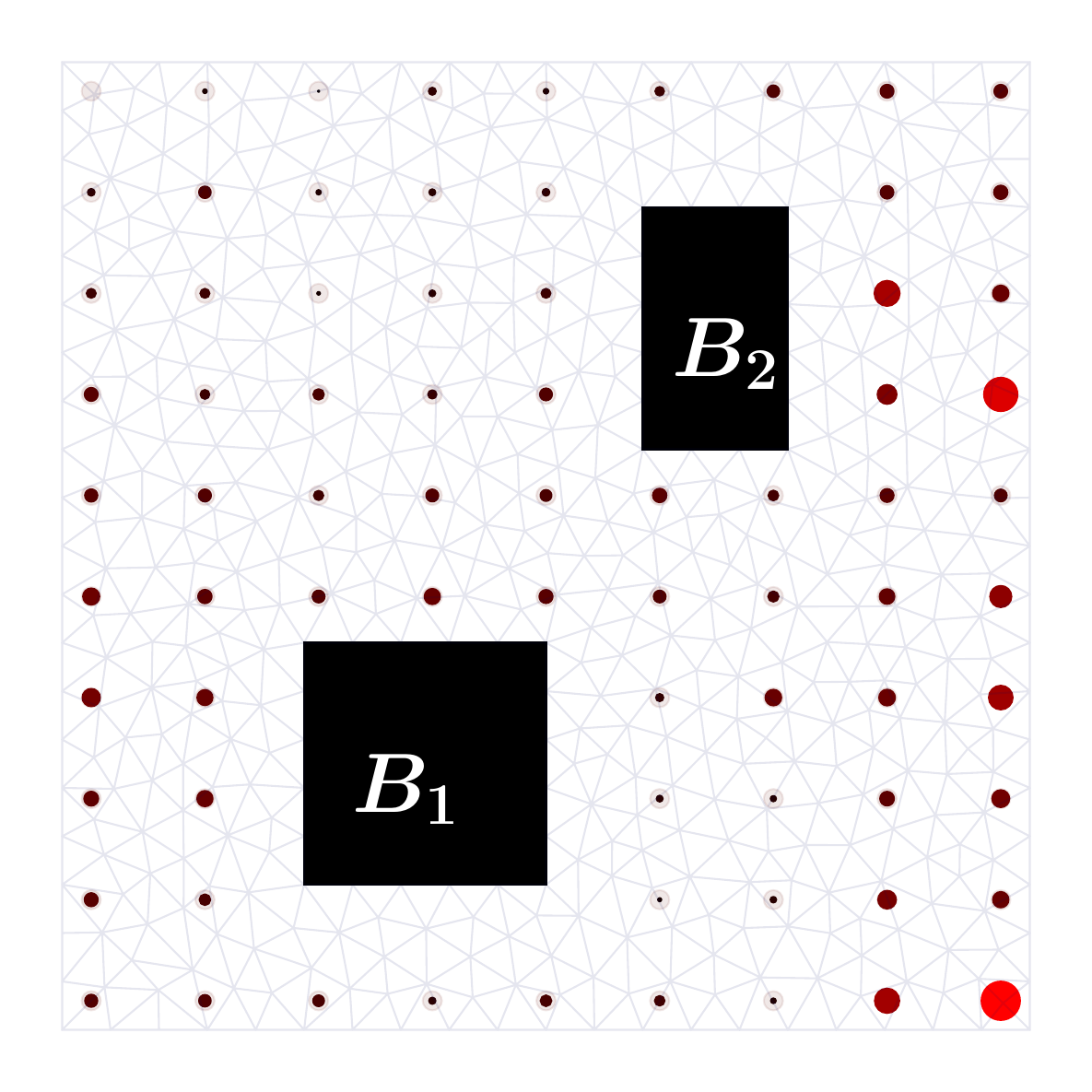}
        \includegraphics[width=0.235\linewidth]{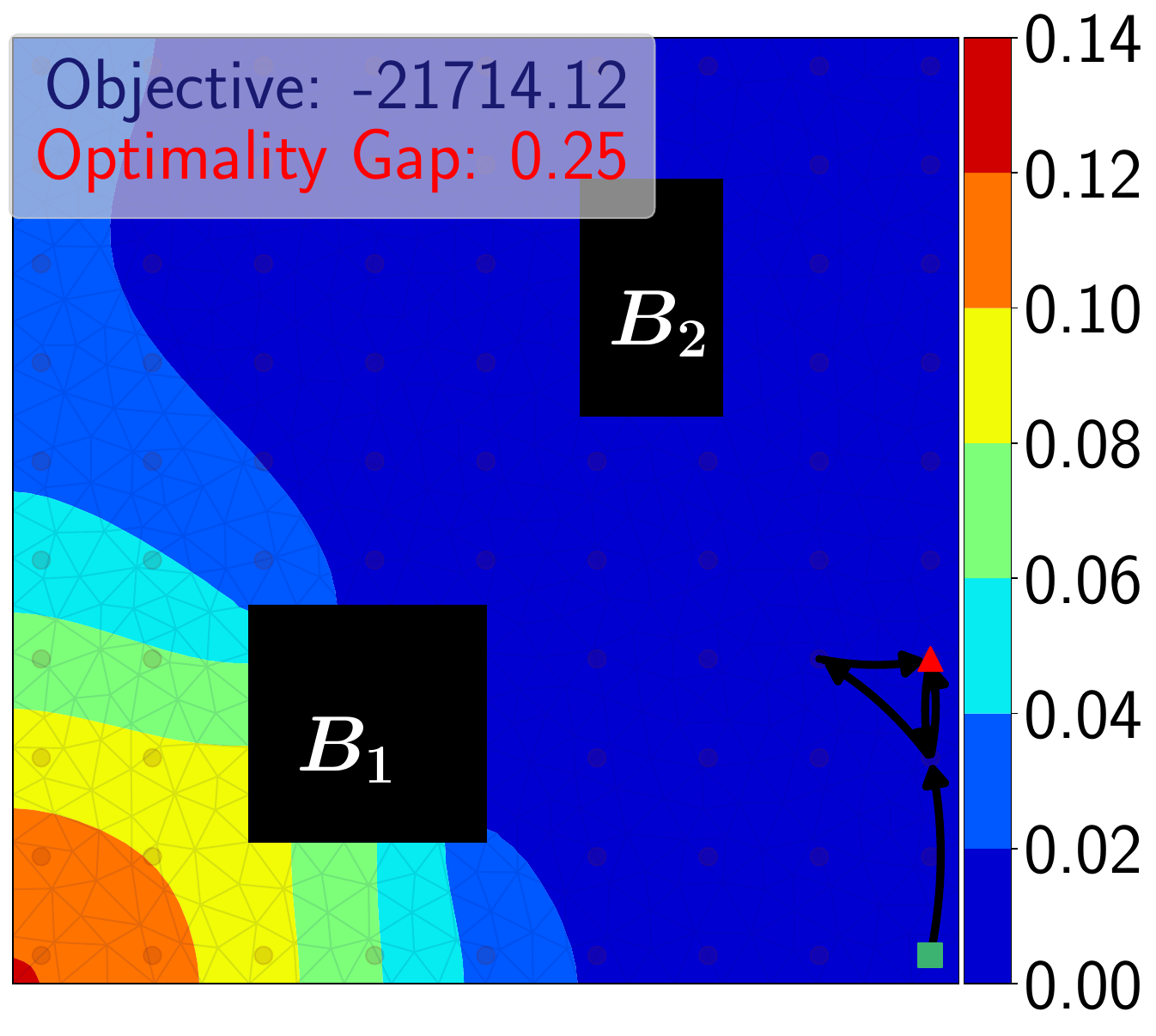}
        \caption{
          Similar to \Cref{fig:coarse_higher_order_iterations_order_3}.
          Here the policy order is set to $k=5$. 
        }\label{fig:coarse_higher_order_iterations_order_5}
      \end{figure}

      Results obtained by using 
      \Cref{defn:generalized_higher_order_path_model}
      with order $k=3$ and $k=5$ are shown in 
      \Cref{fig:coarse_generalized_higher_order_iterations_order_3} and 
      \Cref{fig:coarse_generalized_higher_order_iterations_order_5}, respectively.
      These results suggest that the two higher-order policy models proposed behave 
      similarly in this setup 
      and yield optimal paths associated with posterior 
      variance fields similar to that of the global optimum. 

      \begin{figure}[!htbp]
        \centering
        \includegraphics[width=0.53\linewidth]{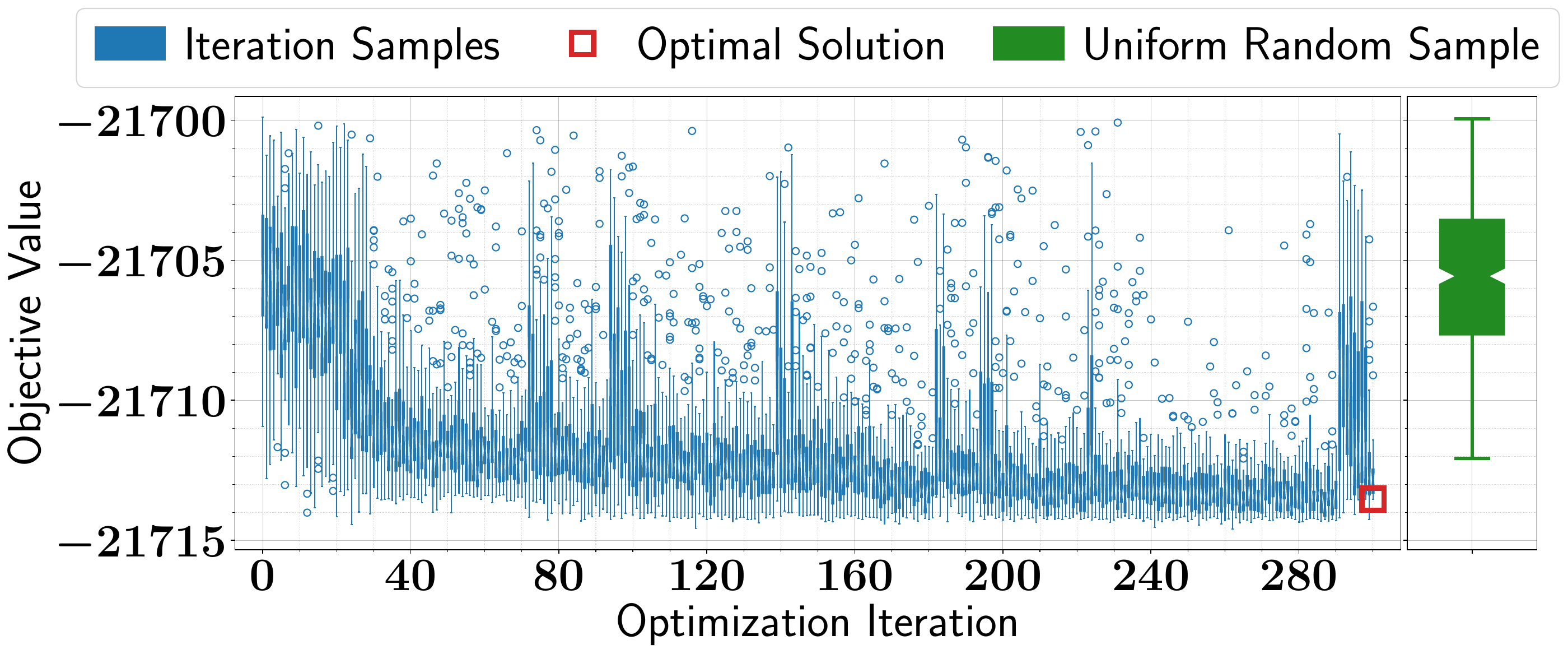}
        \includegraphics[width=0.215\linewidth]{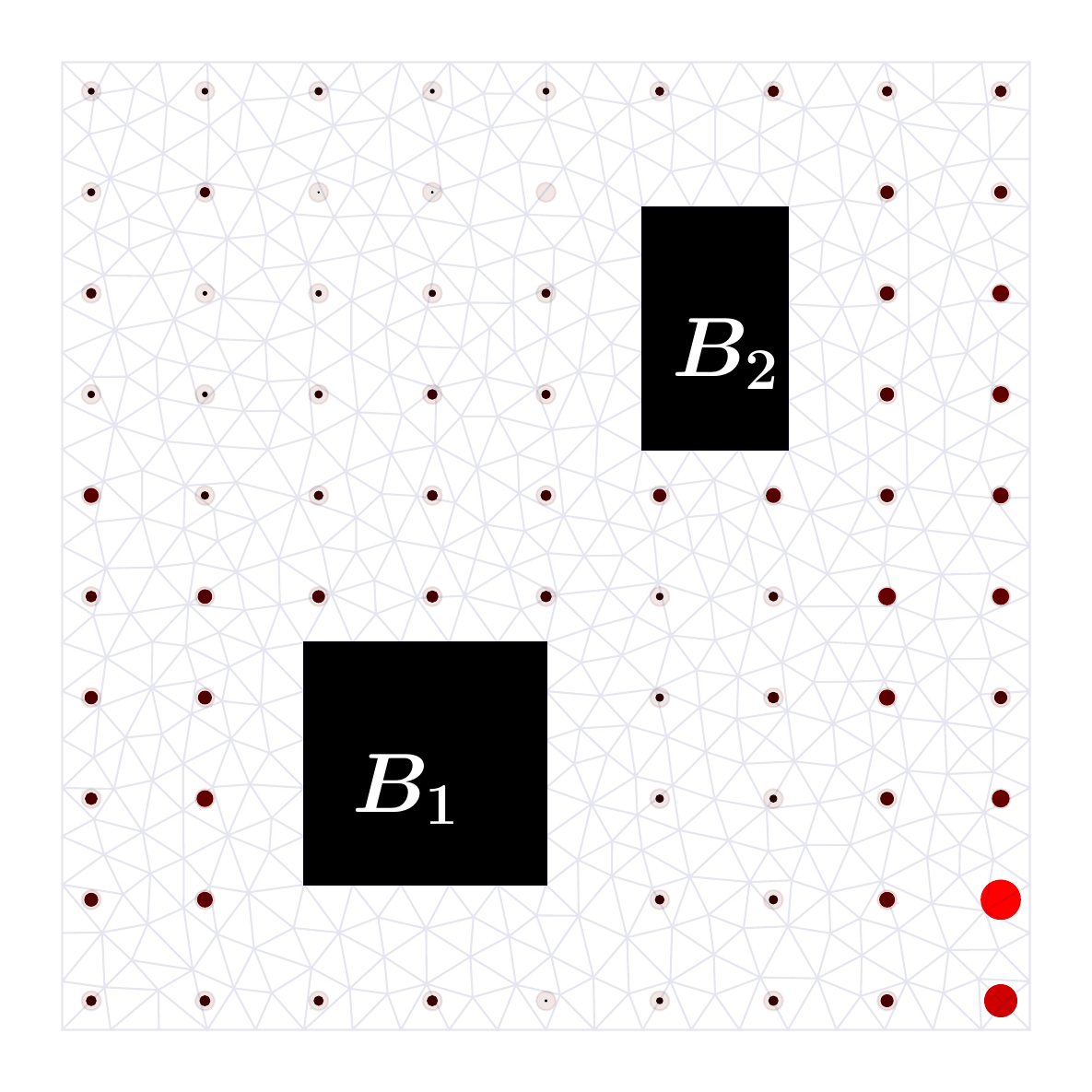}
        \includegraphics[width=0.235\linewidth]{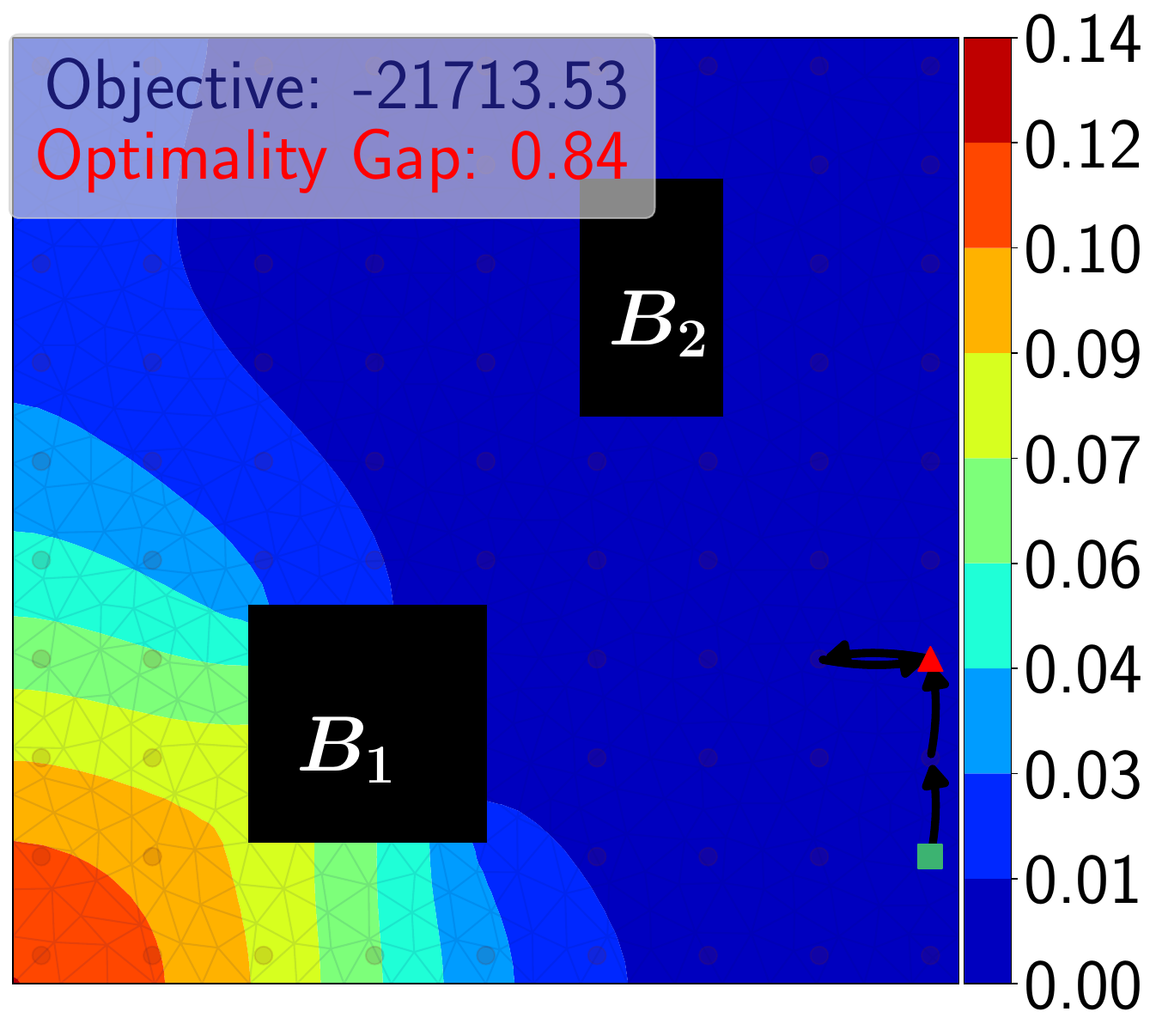}
        \includegraphics[width=0.53\linewidth]{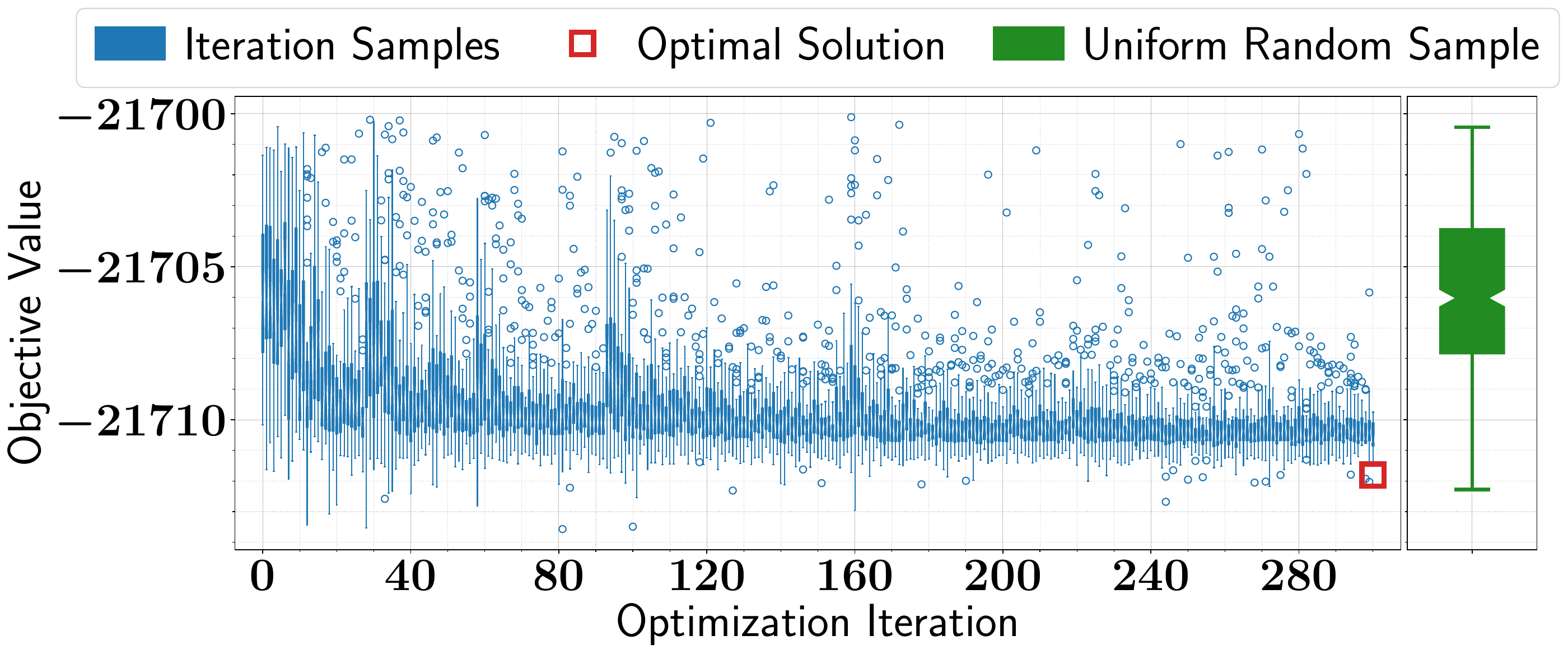}
        \includegraphics[width=0.215\linewidth]{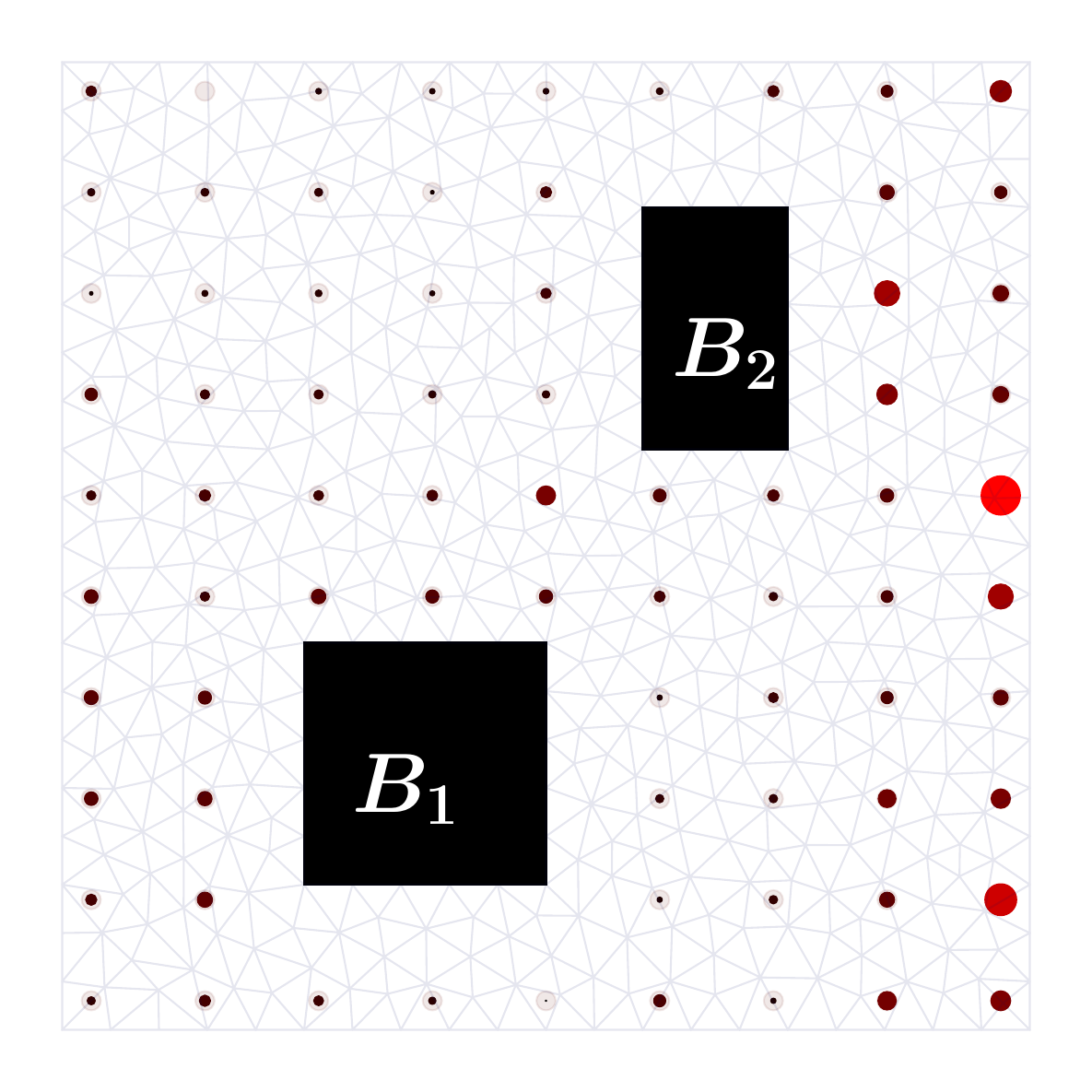}
        \includegraphics[width=0.235\linewidth]{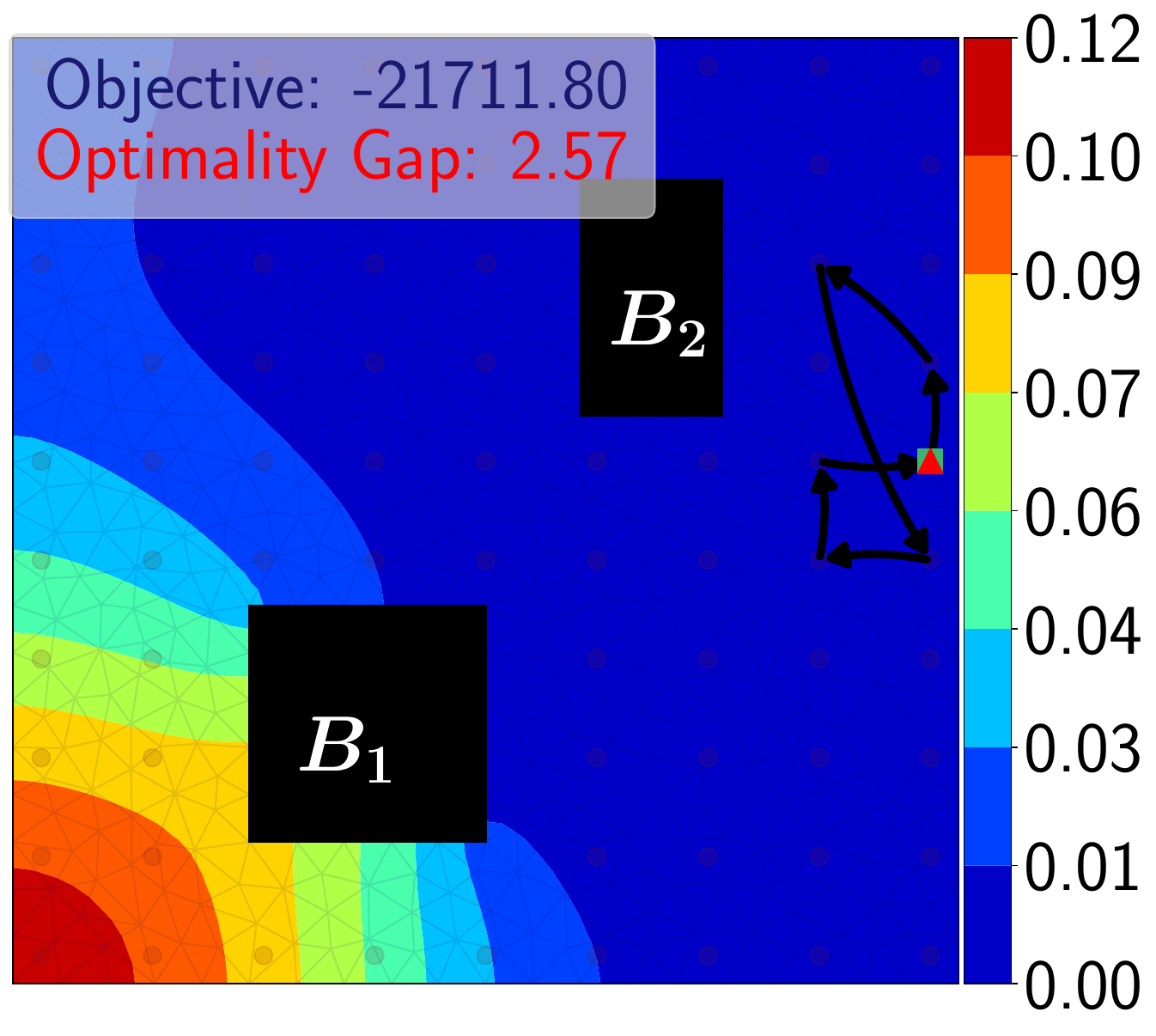}
        \caption{
          Similar to \Cref{fig:coarse_higher_order_iterations_order_3}.
          Here the higher-order policy (\Cref{defn:generalized_higher_order_path_model}) is used.
        }\label{fig:coarse_generalized_higher_order_iterations_order_3}
      \end{figure}
      \begin{figure}[!htbp]
        \centering
        \includegraphics[width=0.53\linewidth]{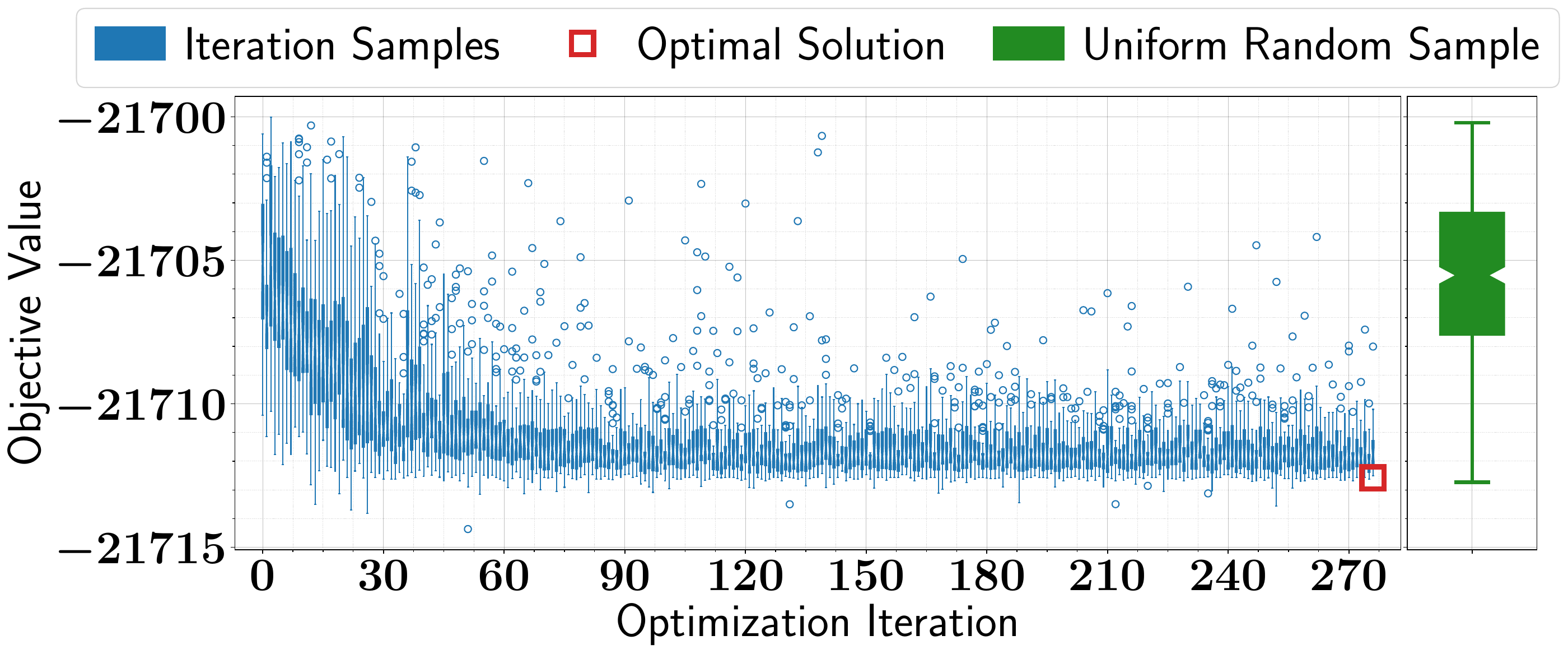}
        \includegraphics[width=0.215\linewidth]{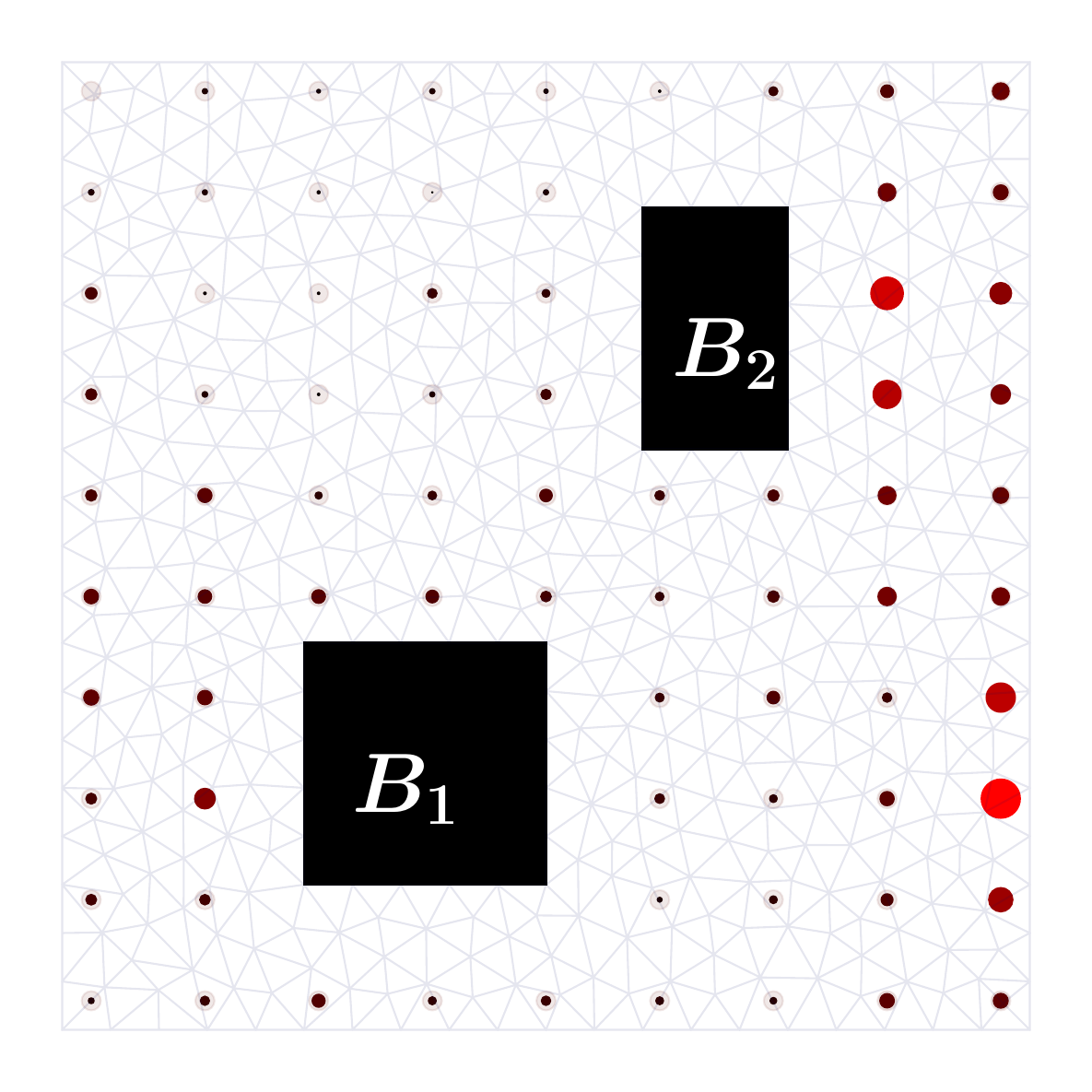}
        \includegraphics[width=0.235\linewidth]{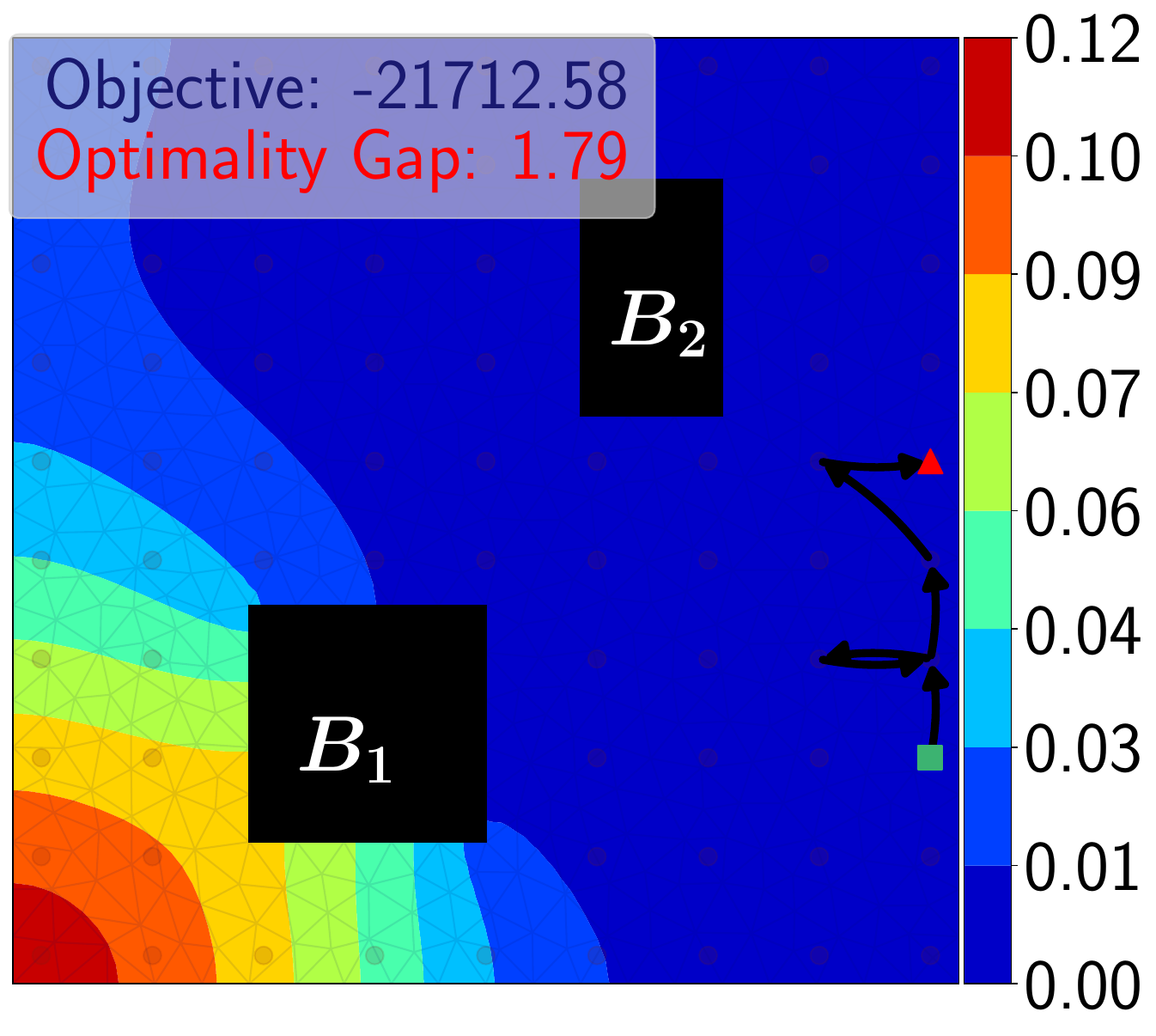}
        \includegraphics[width=0.53\linewidth]{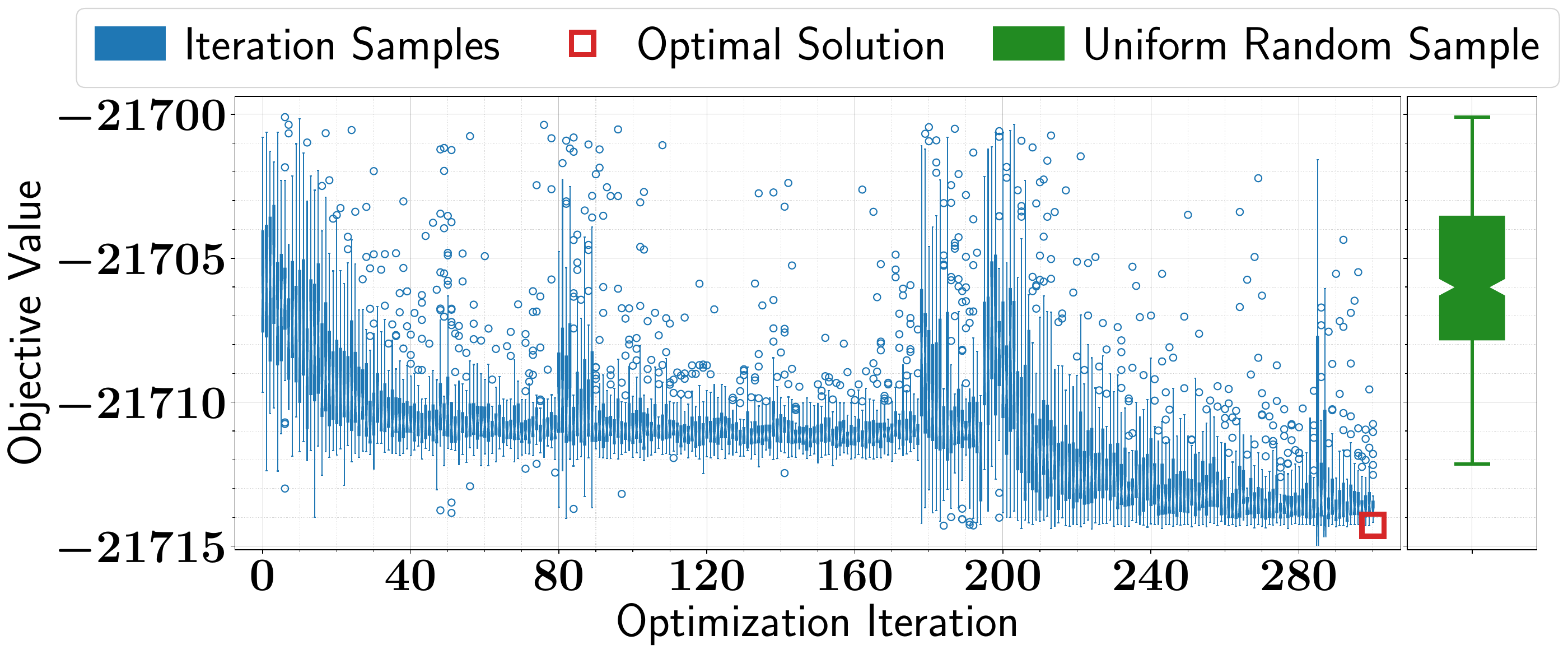}
        \includegraphics[width=0.215\linewidth]{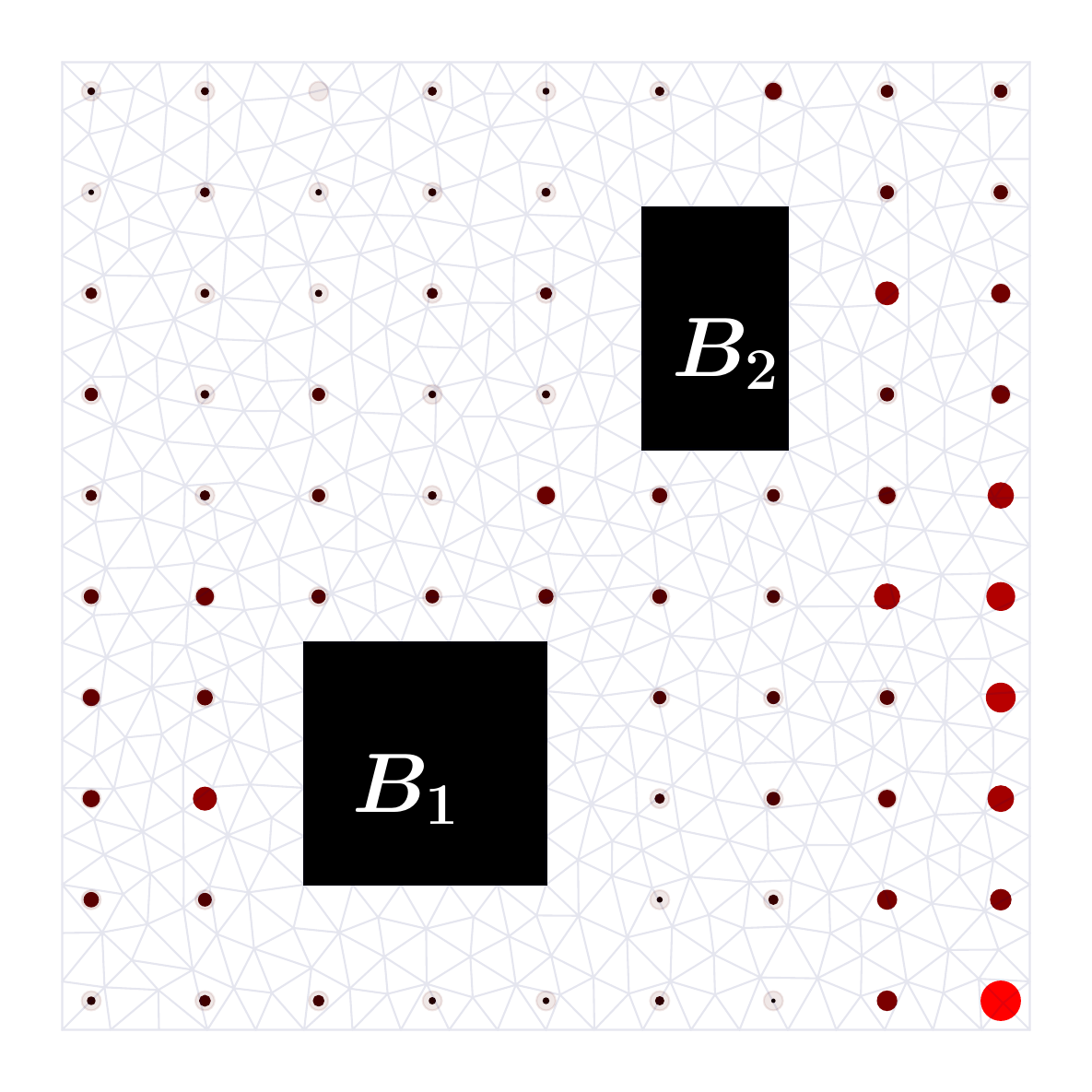}
        \includegraphics[width=0.235\linewidth]{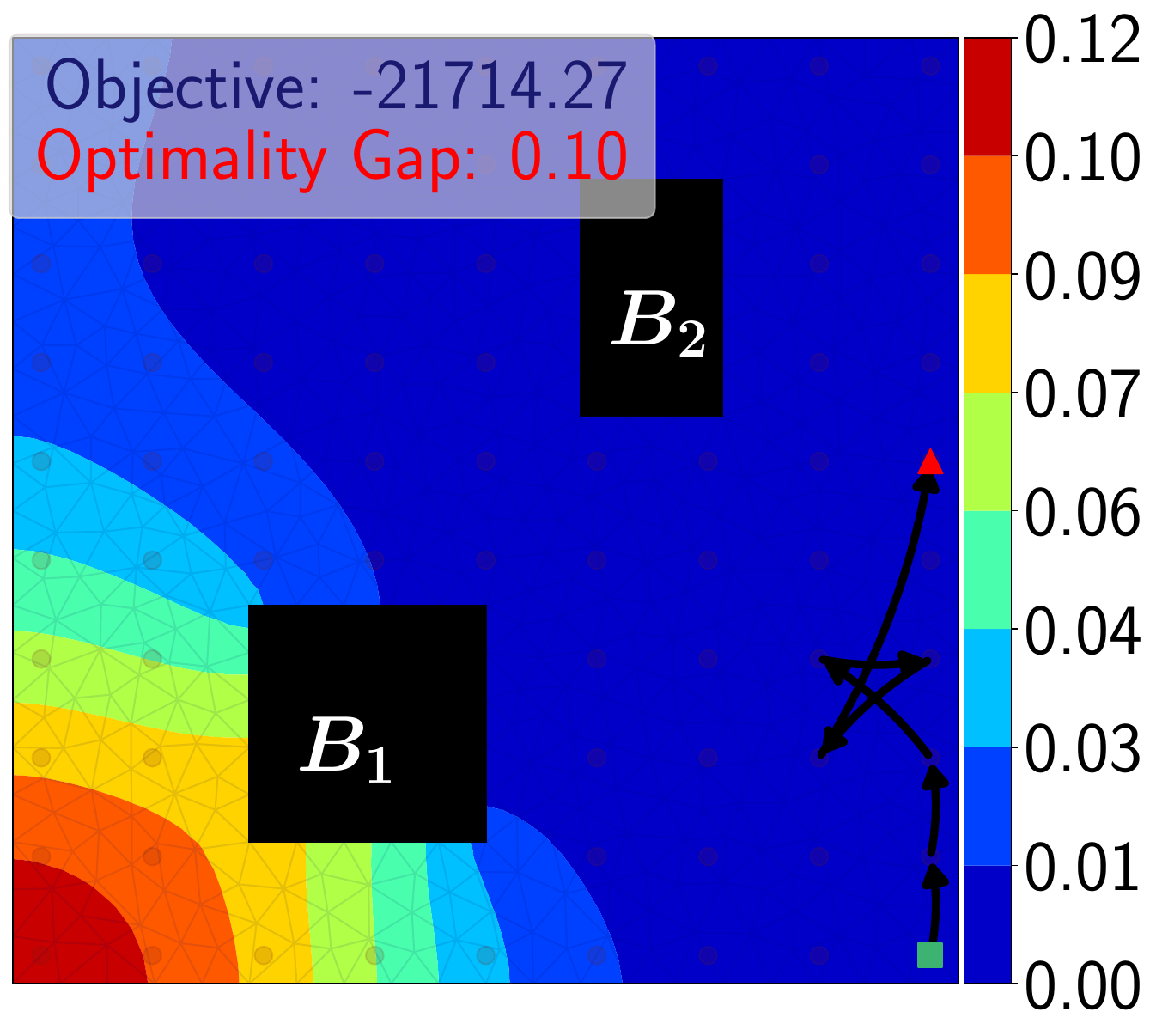}
        \caption{
          Similar to \Cref{fig:coarse_generalized_higher_order_iterations_order_3}.
          Here the policy order is set to $k=5$. 
        }\label{fig:coarse_generalized_higher_order_iterations_order_5}
      \end{figure}

  \subsection{On the Necessity of a Baseline}
  \label{subsec:importance_of_baseline}
    The baseline is critical for reducing the variability of the stochastic estimator 
    and enhancing optimization efficiency. 
    We empirically assess stochastic gradient quality, with and without a baseline. 
    We employ the same setup in \Cref{subsubsec:coarse_first_order_results}, and we 
    use the brute-force results in \Cref{subsubsec:benchmark} 
    to compute the exact gradient  $\vec{g}$ \eqref{eqn:full_gradient} 
    and the error of the stochastic gradient. 
    For fixed sample sizes, stochastic gradients $\widehat{\vec{g}}$ 
    with and without a baseline are 
    replicated $32$ times to estimate the errors relative to the 
    exact gradient $\vec{g}$, 
    across varying ensemble $\Nens$ and baseline batch $N_{\rm b}$ sizes.
    The errors and variances of the stochastic gradient estimator are summarized 
    in \Cref{fig:Gradient_Error_Analysis_UseSameSample} 
    and \Cref{fig:Gradient_Variance_Analysis_UseSameSample}, respectively.
    \Cref{fig:Gradient_Error_Analysis_UseSameSample} shows that the mean error is centered 
    around $0$ reflecting the unbiasedness of the stochastic gradient with or without a baseline.
    Increasing the sample size $\Nens$, by definition, reduces variability but at a 
    high computational cost. 
    In contrast, using a baseline greatly reduces variance even with 
    a batch size of $N_{\rm b}=1$, as shown in 
    \Cref{fig:Gradient_Variance_Analysis_UseSameSample}, and incurs no additional cost. 

    \begin{figure}[!htbp]
      \includegraphics[width=1.0\linewidth]{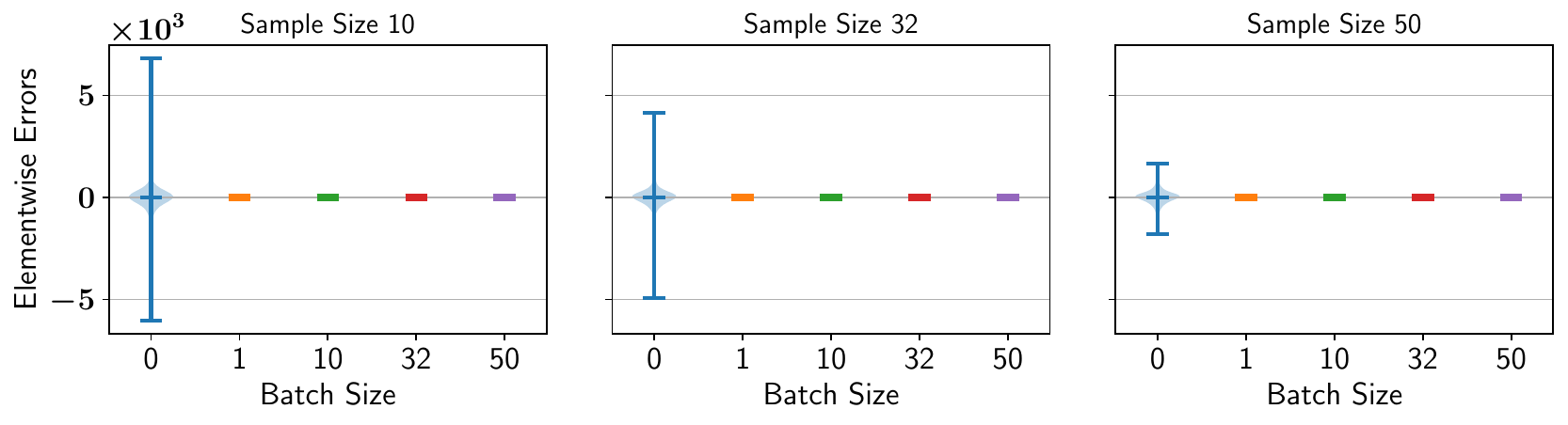}
      \caption{
        Stochastic gradient elementwise errors 
        $\widehat{\vec{g}}-\vec{g}$ averaged across $32$ replicas  
        for sample sizes $\Nens$ set to $10$ (left), $32$ (middle), 
        and $50$ (right), respectively.
        A uniformly sampled parameter instance $\hyperparamvec\in[0, 1]^{\Nparam}$ is used.
        For each sample size, the stochastic gradient is computed 
        without baseline $\baseline=0$ (batch size $0$) and 
        with optimal baseline estimate 
        \eqref{eqn:general_optimal_baseline_estimate} by using batch sizes 
        $N_{\rm b}=1, 10, 32, 50$, respectively.
      }\label{fig:Gradient_Error_Analysis_UseSameSample}
    \end{figure}
    \begin{figure}[!htbp]
      \centering
      \includegraphics[width=1.0\linewidth]{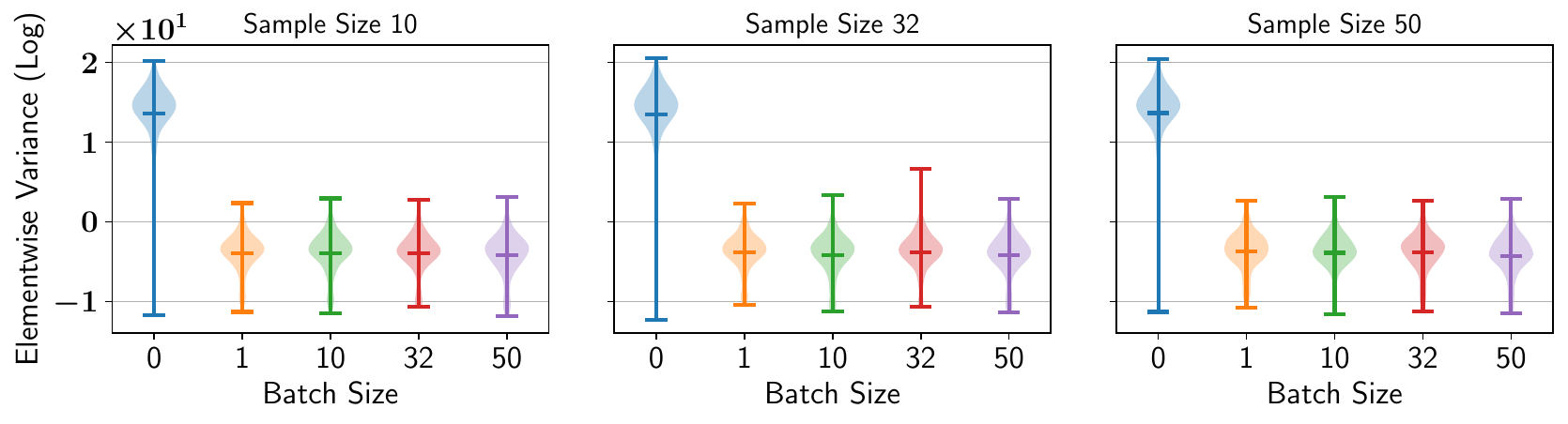}
      \caption{
        Similar to \Cref{fig:Gradient_Error_Analysis_UseSameSample}.
        Here elementwise variances (log scale) 
        are plotted. 
      }\label{fig:Gradient_Variance_Analysis_UseSameSample}
    \end{figure}

    To further illustrate the baseline’s impact, 
    \Cref{fig:coarse_first_order_iterations_with_random_sample_no_baseline} 
    shows results for the benchmark experiment 
    in \Cref{subsubsec:benchmark} without using a baseline. 
    With the small sample size $\Nens=32$, high gradient variance degrades policy updates, 
    preventing consistent reduction of the expected objective. 
    Although a near-optimal path (right) is retrieved as evident by the posterior variance, 
    this is incidental, as the initial distribution is poorly optimized (middle) 
    and the resulting policy remains highly variable (left).

    \begin{figure}
      \centering
      \includegraphics[width=0.53\linewidth]{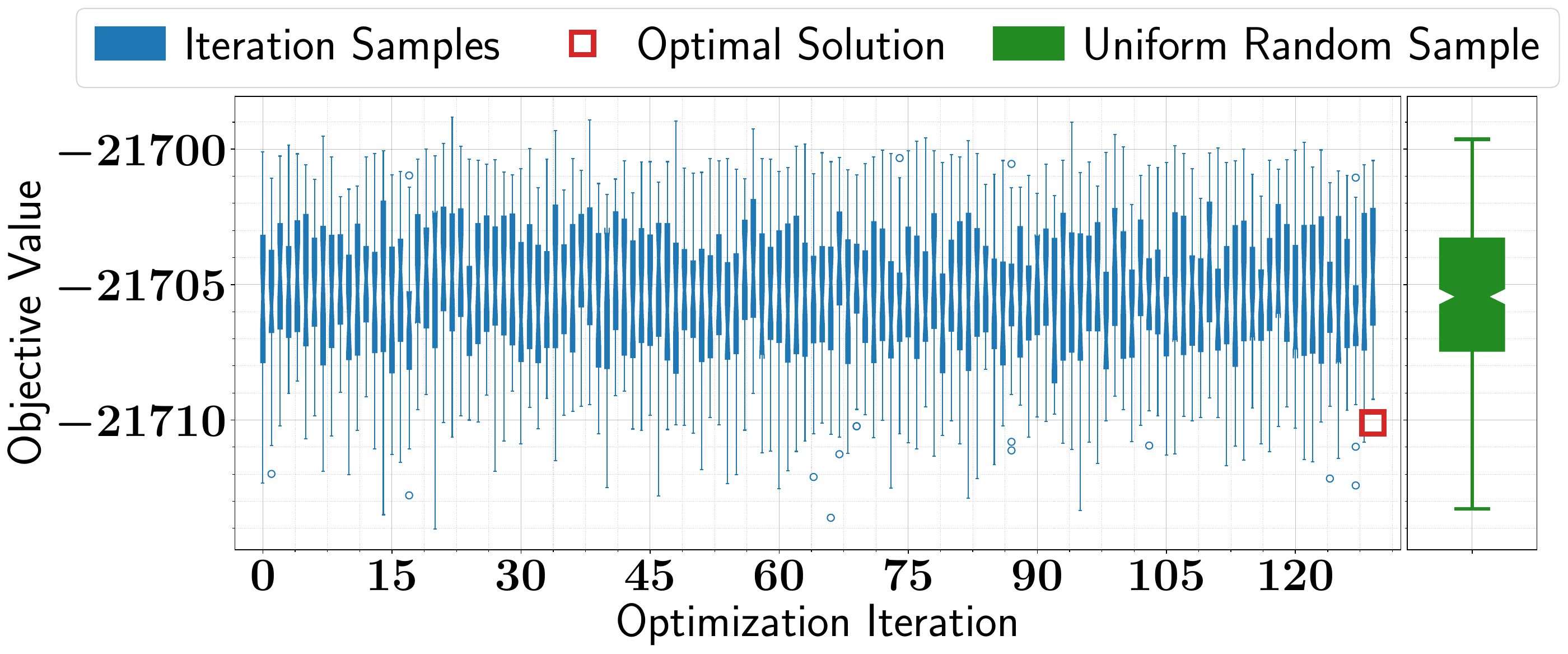}
      \includegraphics[width=0.215\linewidth]{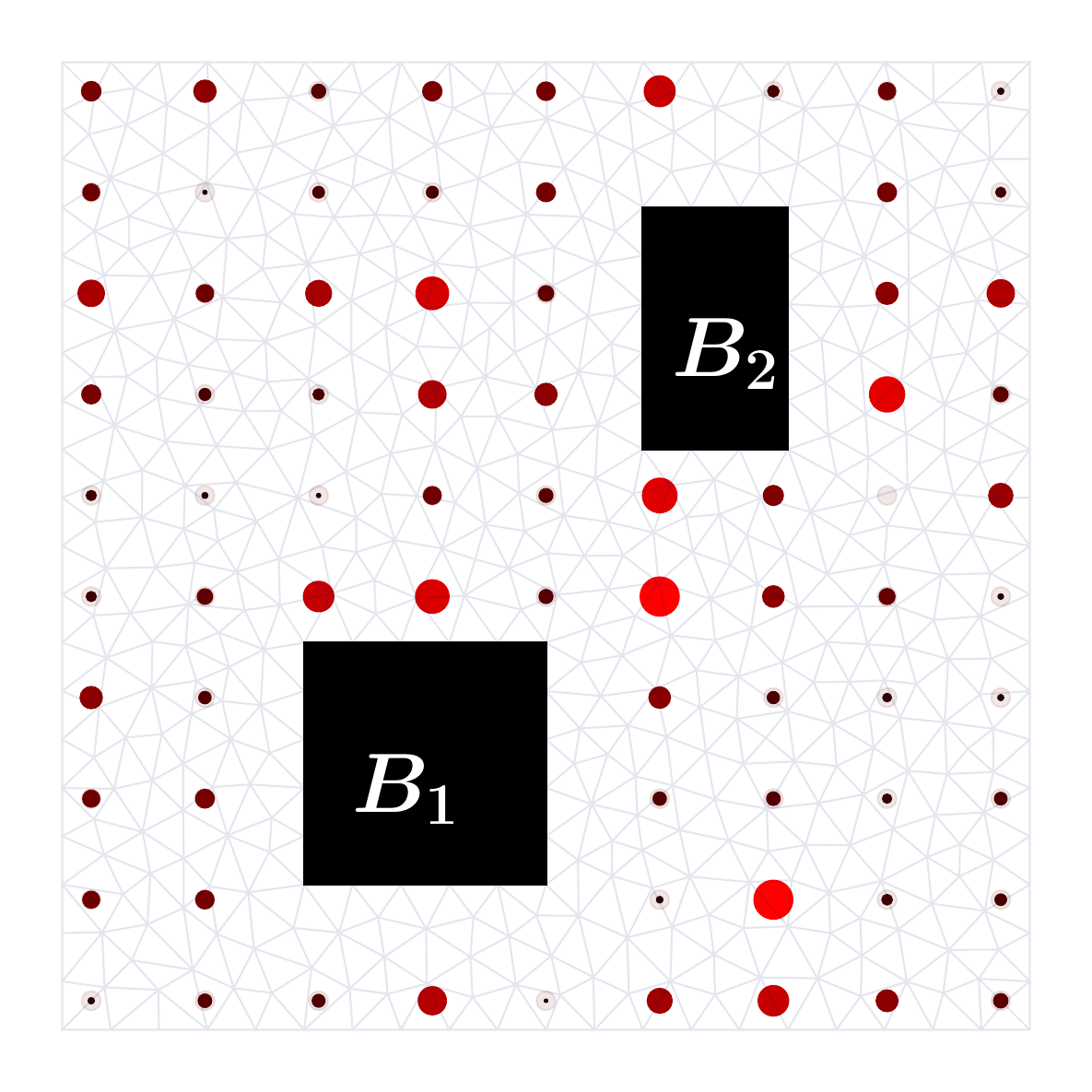}
      \includegraphics[width=0.235\linewidth]{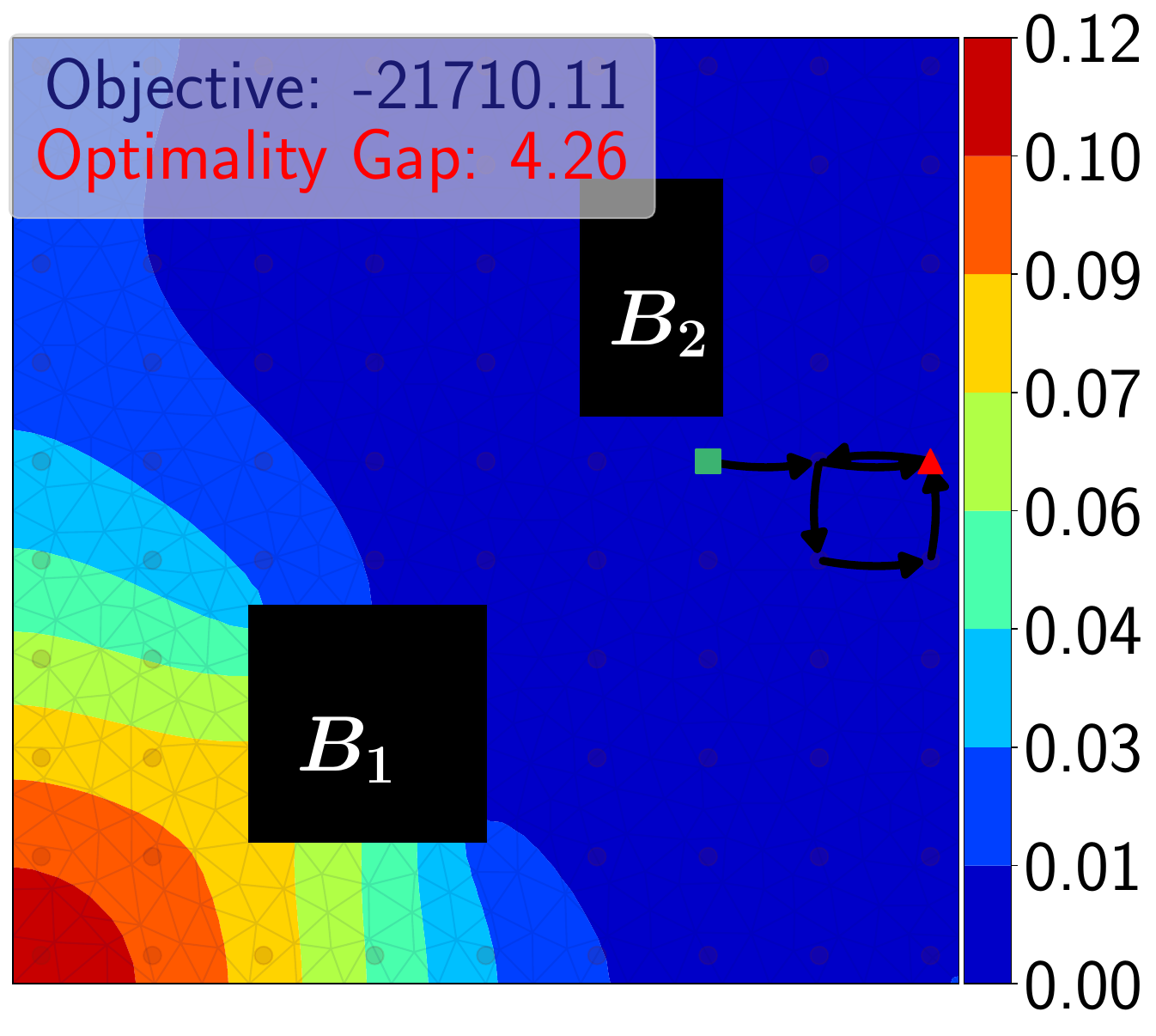}
      \caption{
        Results of the experiment in \Cref{subsubsec:coarse_first_order_results} without 
        employing a baseline.
      }\label{fig:coarse_first_order_iterations_with_random_sample_no_baseline}
    \end{figure}

    The results in \Cref{fig:Gradient_Variance_Analysis_UseSameSample} and 
    \Cref{fig:coarse_first_order_iterations_with_random_sample_no_baseline}
    clearly show that employing the baseline is critical for the performance of 
    the optimization procedure.

  \subsection{Results with Fine Navigation Mesh and $1$ Moving Sensor}
  \label{subsubsec:fine_mesh_results_unspecified_start_point}
    This experiment uses the fine navigation mesh (\Cref{fig:navigation_meshes}, right), 
    with trajectories of $n=19$ nodes and observation frequency $f=1$, yielding $19$ 
    scalar observations along the path.

    \begin{figure}[!htbp]
      \centering
      \includegraphics[width=0.53\linewidth]{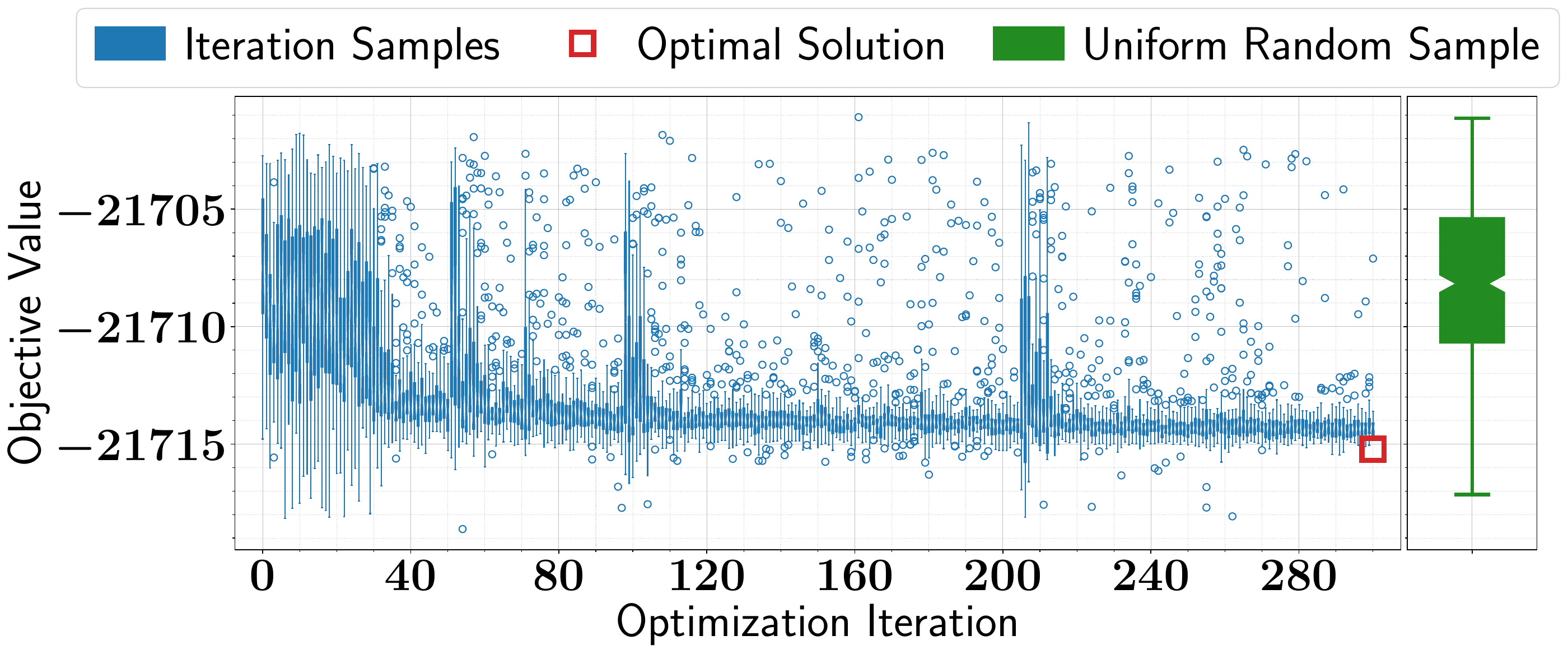}
      \includegraphics[width=0.215\linewidth]{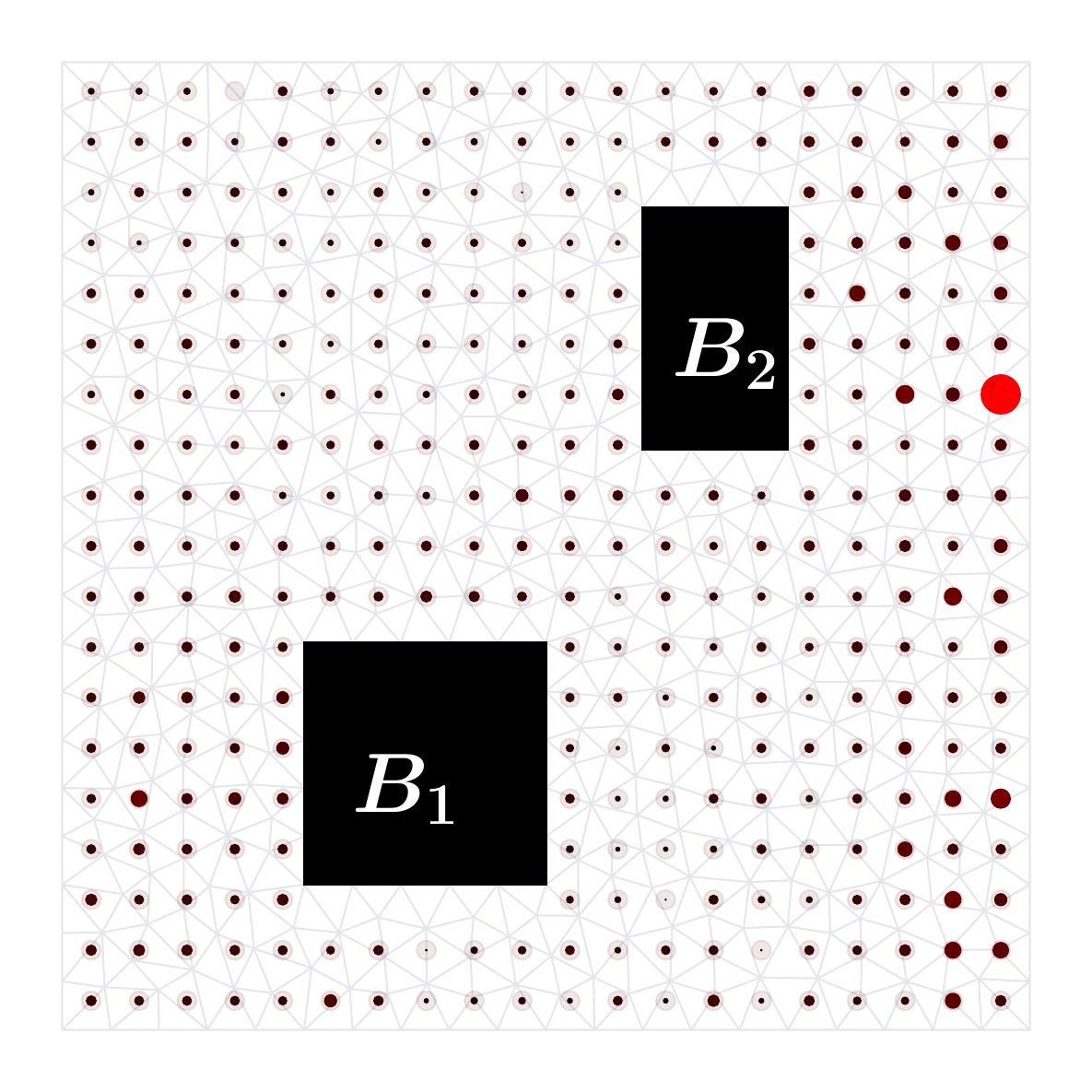}
      \includegraphics[width=0.235\linewidth]{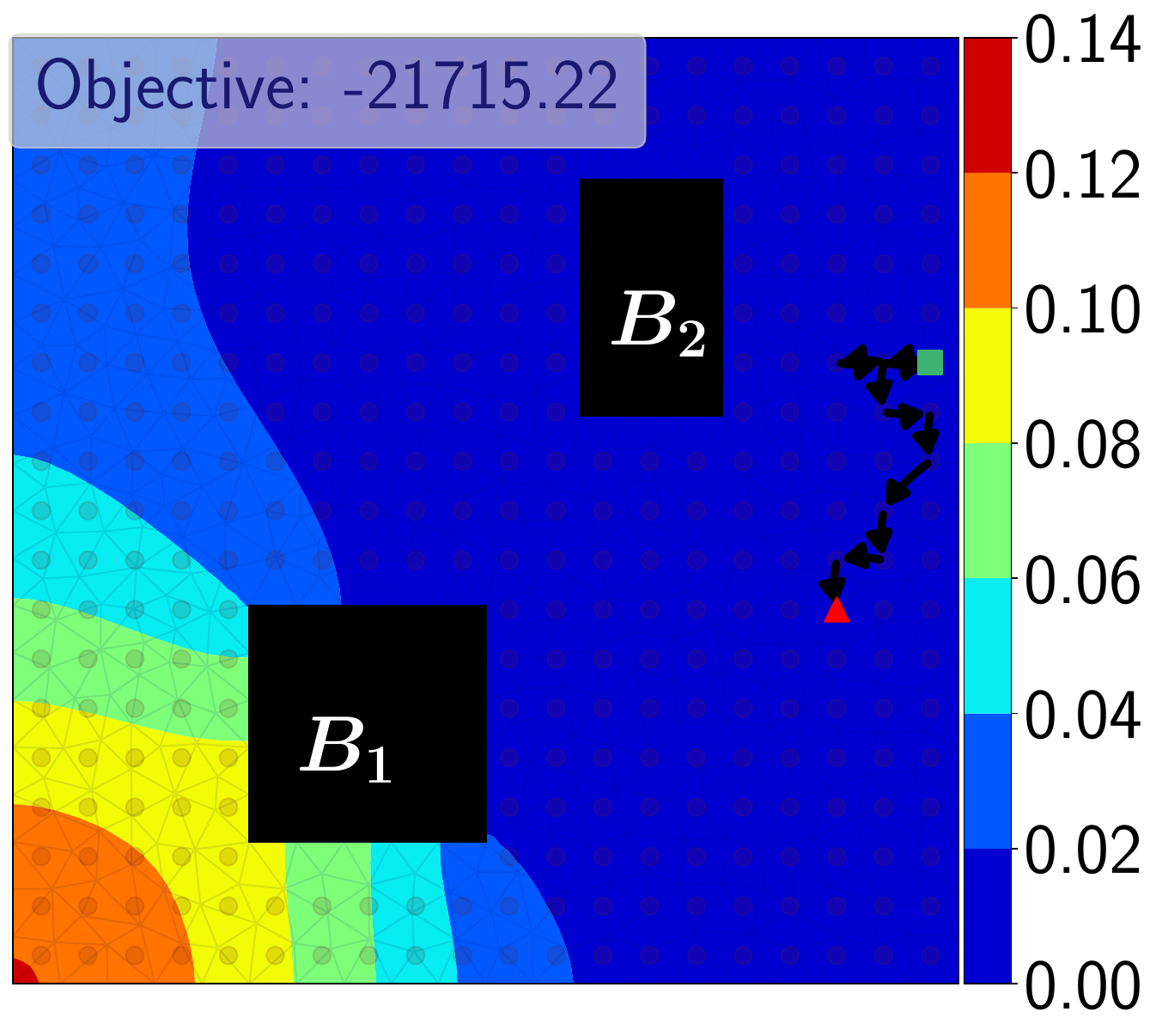}
      \caption{
        Results of \Cref{alg:probabilistic_path_optimization} with the first-order 
        policy (\Cref{defn:first_order_path_model}) applied to 
        the fine navigation mesh (\Cref{fig:navigation_meshes}, right) 
        with trajectory length of $n=19$.
      }\label{fig:fine_first_order_unspecified_start_point}
    \end{figure}

    Results obtained by using the 
    first-order policy (\Cref{defn:first_order_path_model}) are shown in 
    \Cref{fig:fine_first_order_unspecified_start_point}.
    Results obtained by using the higher-order policy (\Cref{defn:higher_order_path_model}) are shown in
    \Cref{fig:fine_higher_order_unspecified_start_point} 
    and \Cref{fig:fine_higher_order_unspecified_start_point_order_3_initial_param_and_traject}.
    Results obtained by using the generalized higher-order policy (\Cref{defn:generalized_higher_order_path_model}) are shown in
    \Cref{fig:fine_generalized_higher_order_unspecified_start_point} 
    and \Cref{fig:fine_generalized_higher_order_unspecified_start_point_order_3_initial_param_and_traject}.

    \begin{figure}[!htbp]
      \centering
      \includegraphics[width=0.495\linewidth]{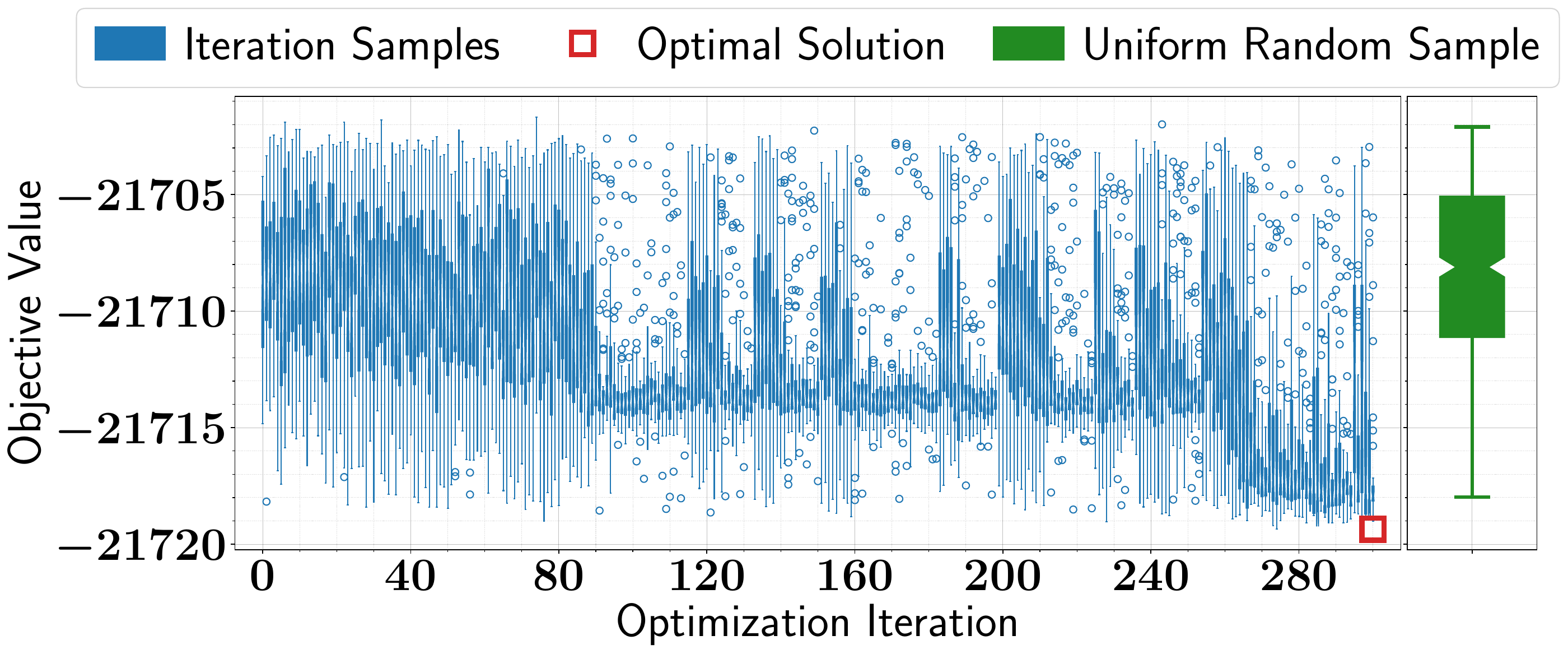}
      \includegraphics[width=0.495\linewidth]{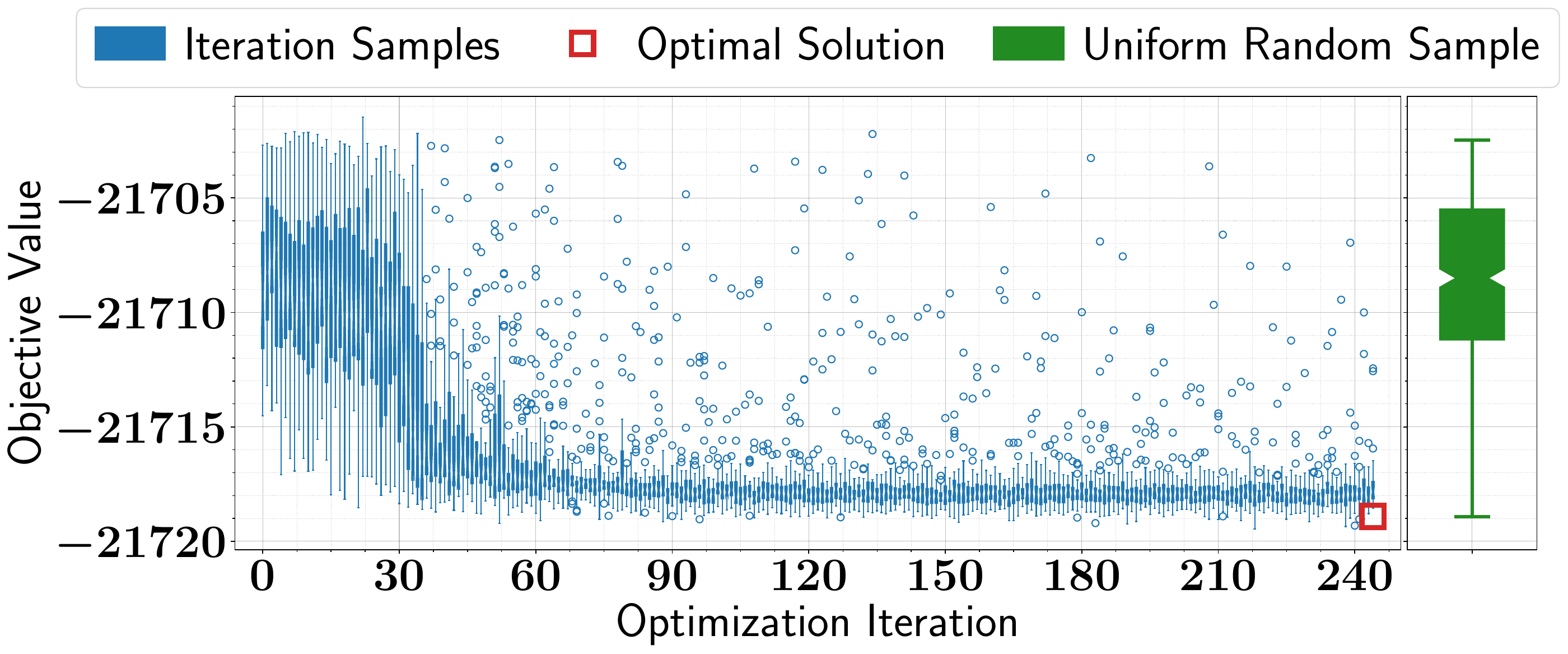}
      \includegraphics[width=0.495\linewidth]{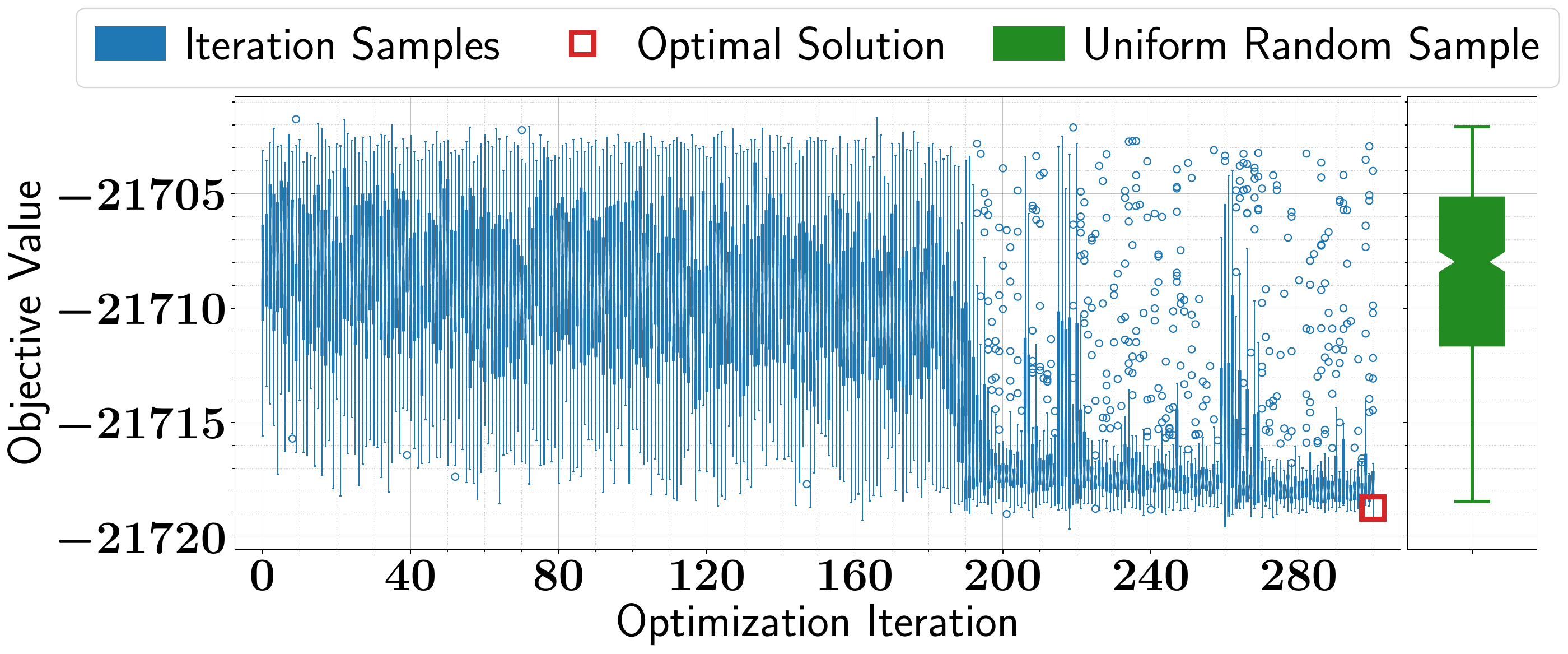}
      \includegraphics[width=0.495\linewidth]{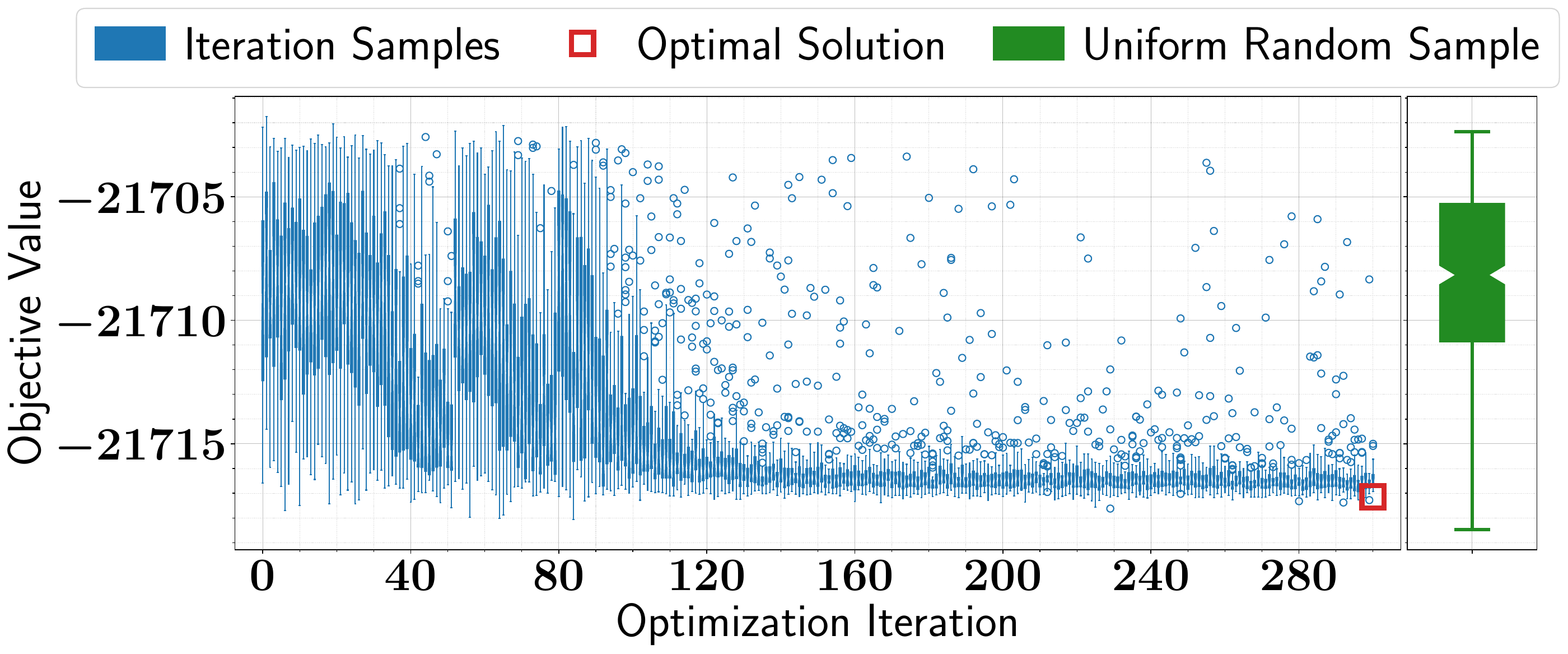}
      \caption{
        Results of \Cref{alg:probabilistic_path_optimization} with the 
        higher-order 
        policy (\Cref{defn:higher_order_path_model}) applied to 
        the fine navigation mesh (\Cref{fig:navigation_meshes}, right) 
        with trajectory length of $n=19$.
        Results are shown for policy order $k=3$ (first row) and $k=5$ (second row). 
        The first column shows results with lag weights being calibrated by the 
        optimization procedure, and 
        the second column shows results with lag weights modeled 
        by \eqref{eqn:decreasing_lag_weights}.
      }\label{fig:fine_higher_order_unspecified_start_point}
    \end{figure}
    \begin{figure}[!htbp]
      \begin{subfigure}[t]{0.24\linewidth}
        \includegraphics[width=0.91\linewidth]{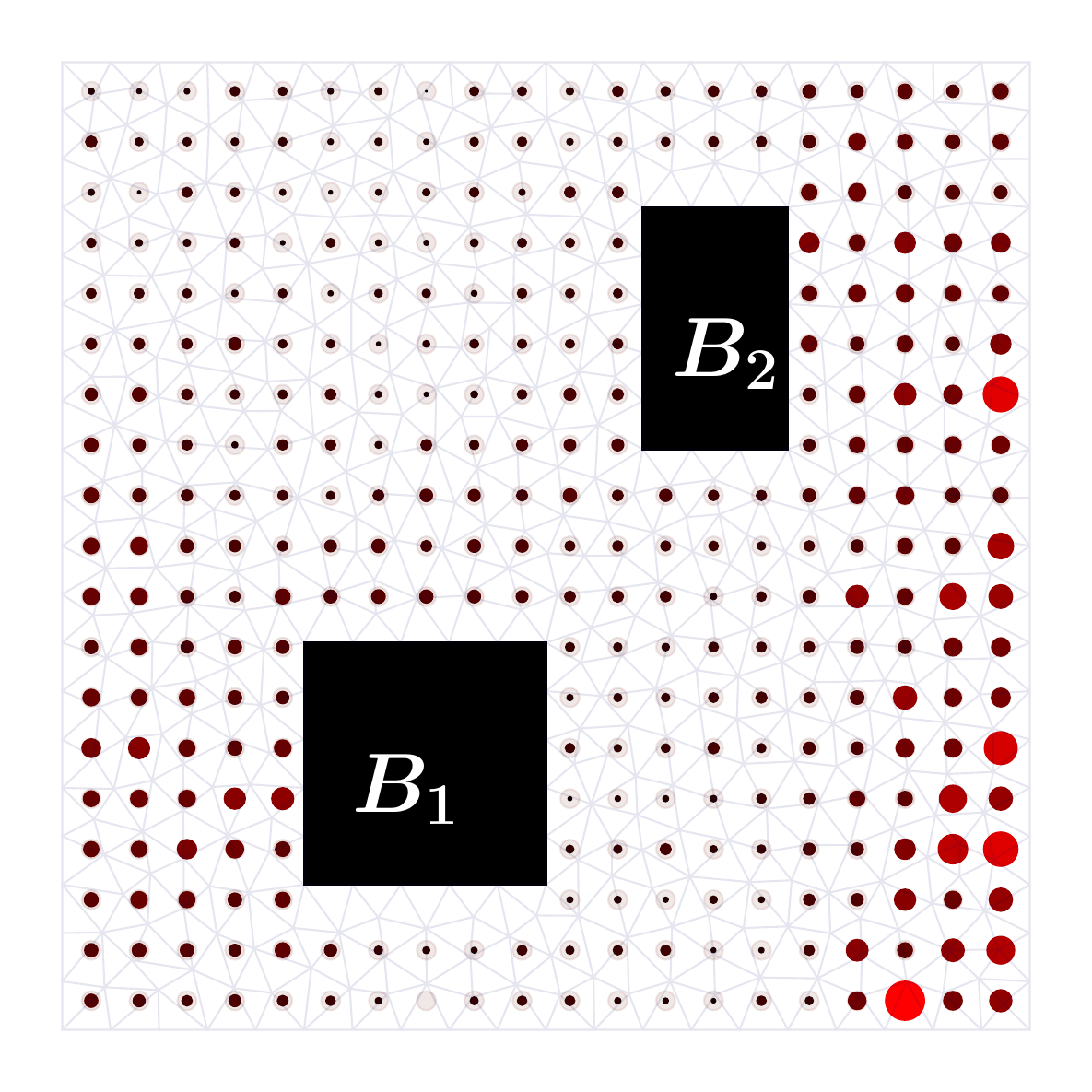}
      \end{subfigure}%
      \begin{subfigure}[t]{0.24\linewidth}
        \includegraphics[width=0.91\linewidth]{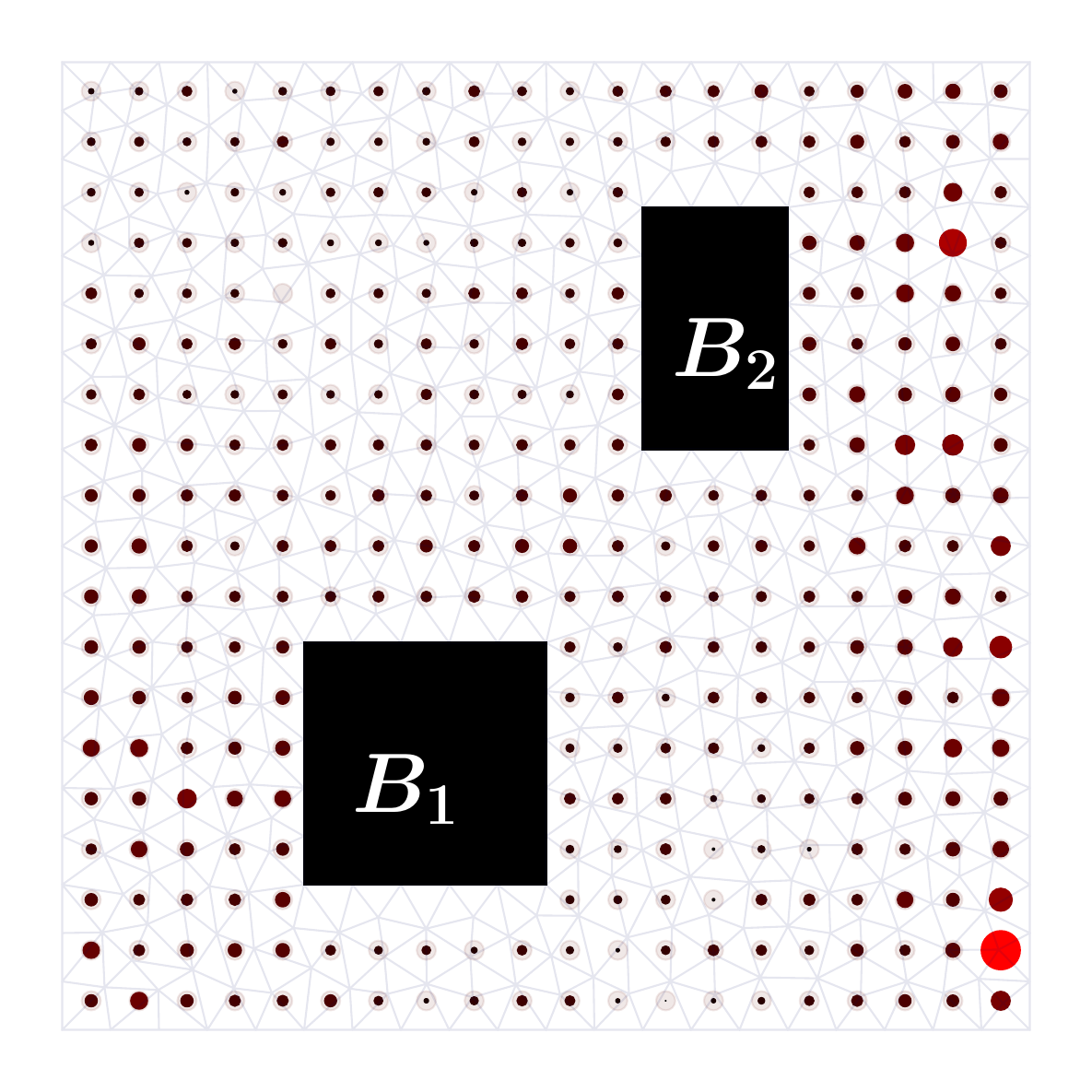}
      \end{subfigure}%
      \begin{subfigure}[t]{0.24\linewidth}
        \includegraphics[width=0.91\linewidth]{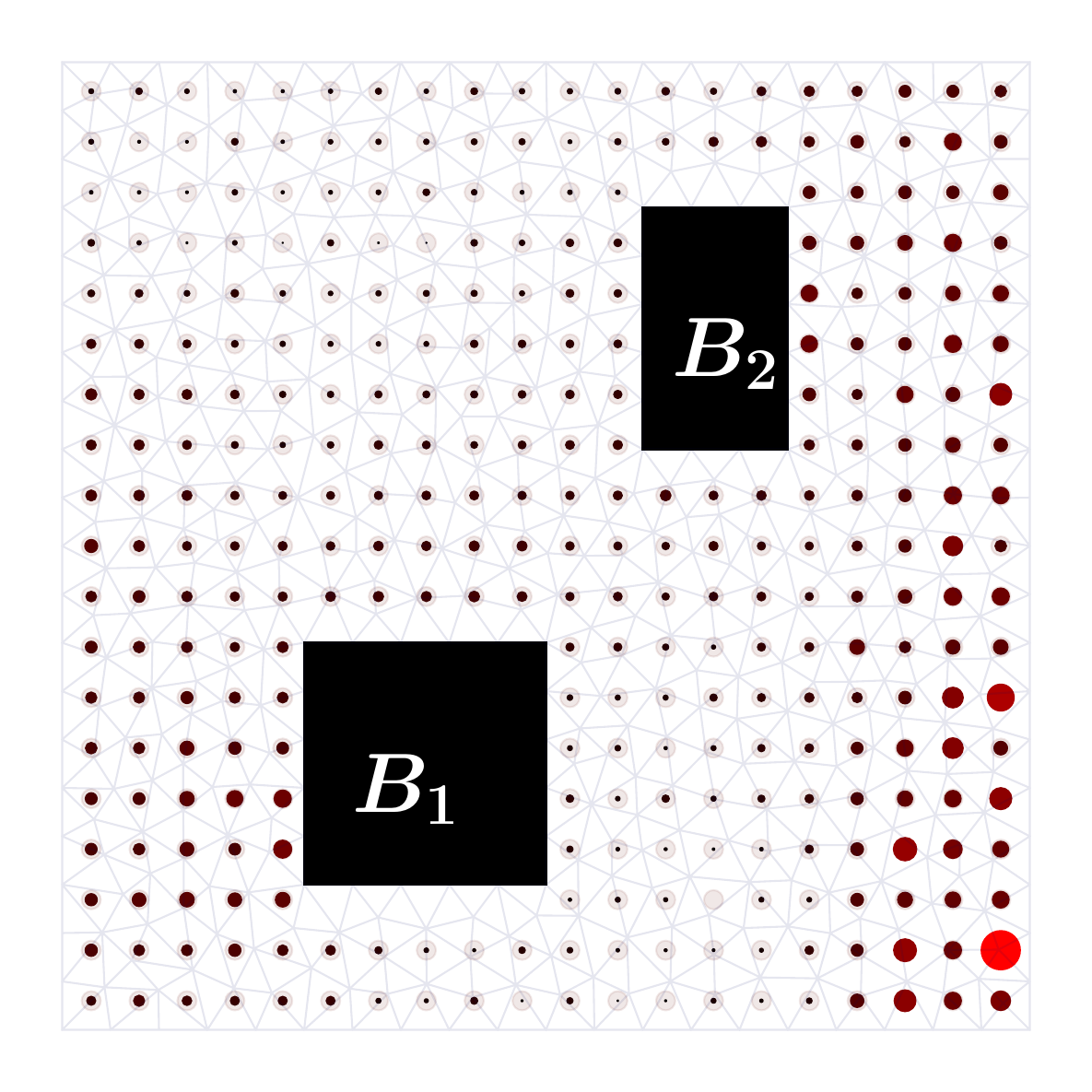}
      \end{subfigure}%
      \begin{subfigure}[t]{0.24\linewidth}
        \includegraphics[width=0.91\linewidth]{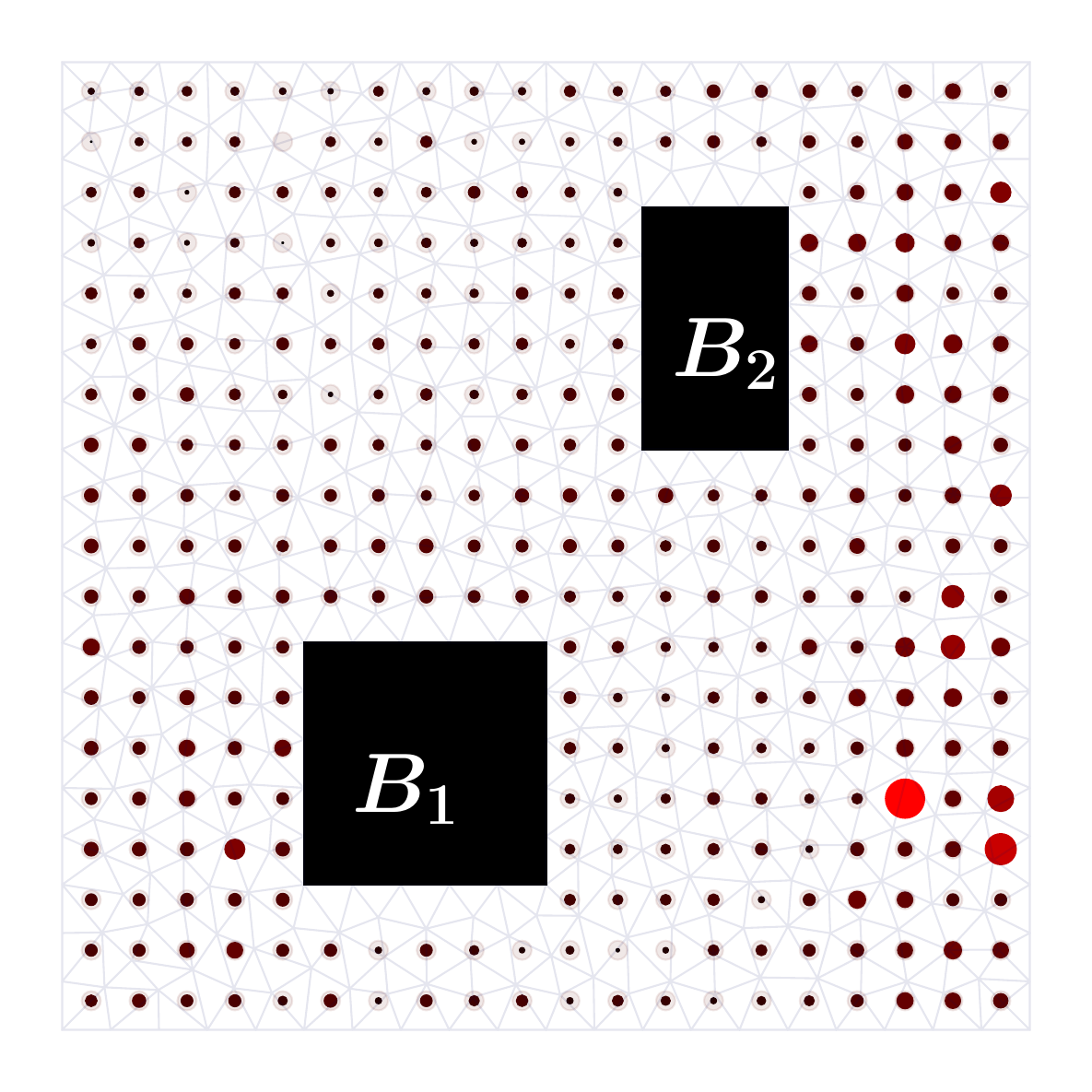}
      \end{subfigure}

      \begin{subfigure}[t]{0.245\linewidth}
        \includegraphics[width=\linewidth]{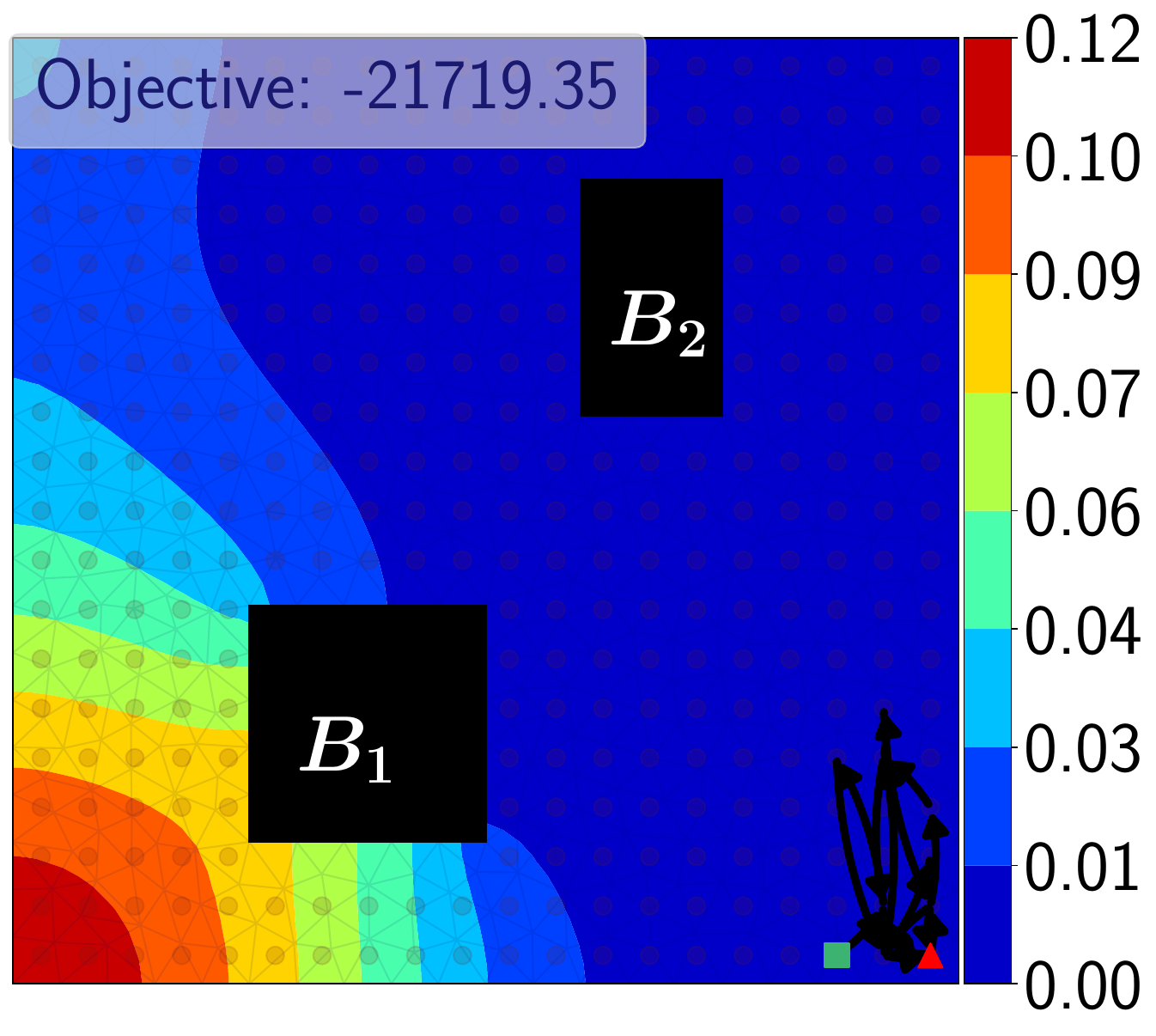}
      \end{subfigure}%
      \begin{subfigure}[t]{0.245\linewidth}
        \includegraphics[width=\linewidth]{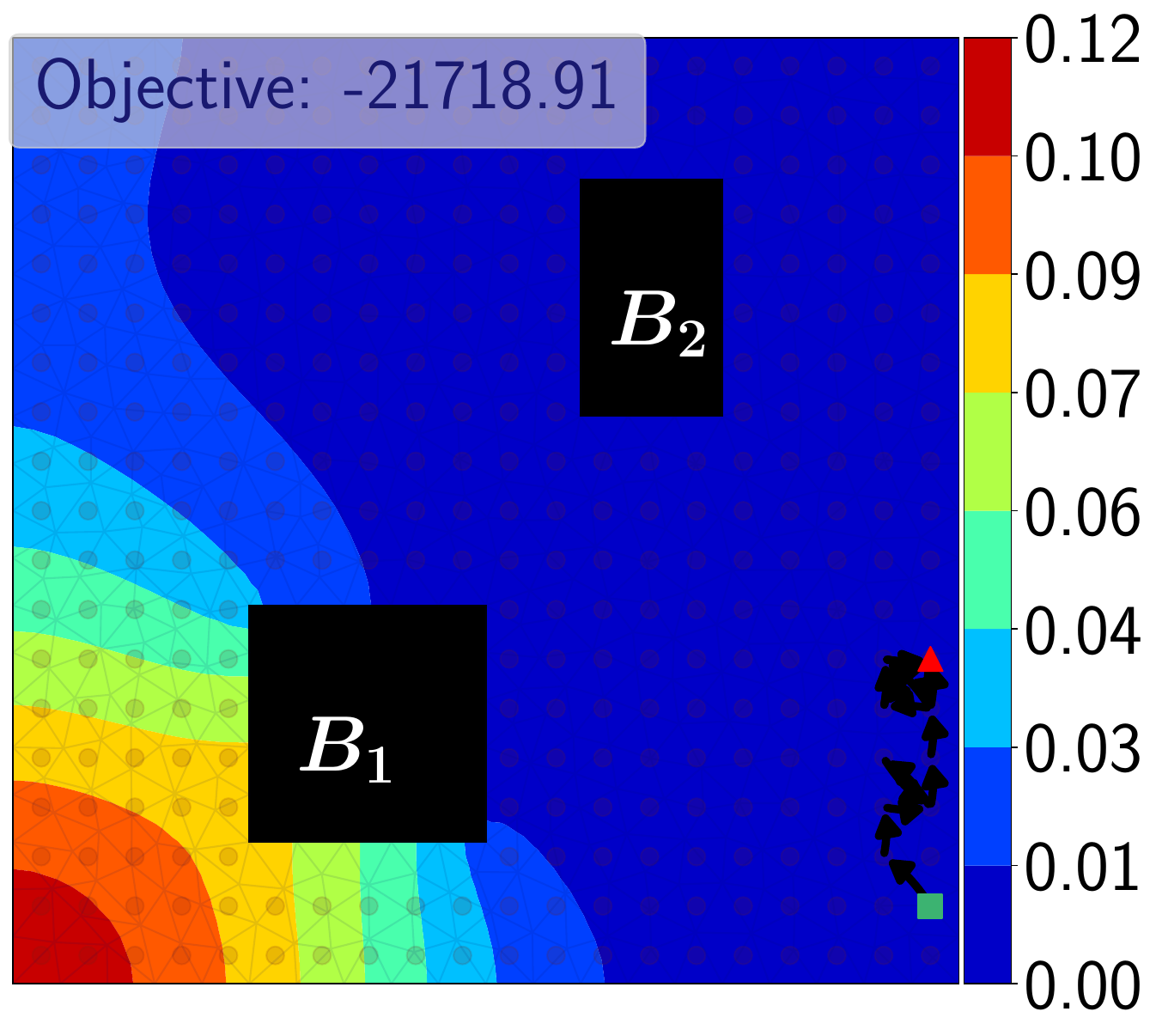}
      \end{subfigure}%
      \begin{subfigure}[t]{0.245\linewidth}
        \includegraphics[width=\linewidth]{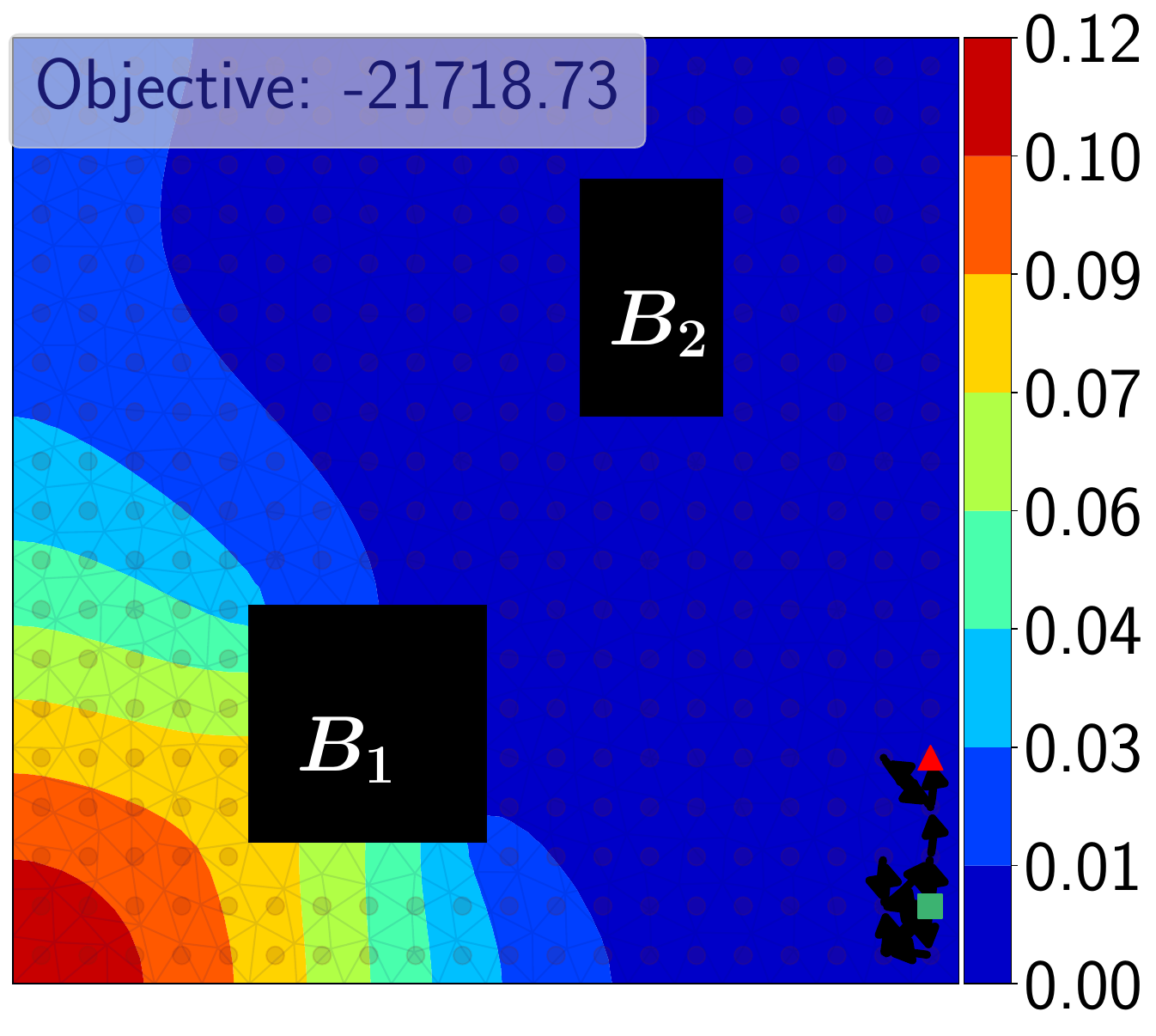}
      \end{subfigure}%
      \begin{subfigure}[t]{0.245\linewidth}
        \includegraphics[width=\linewidth]{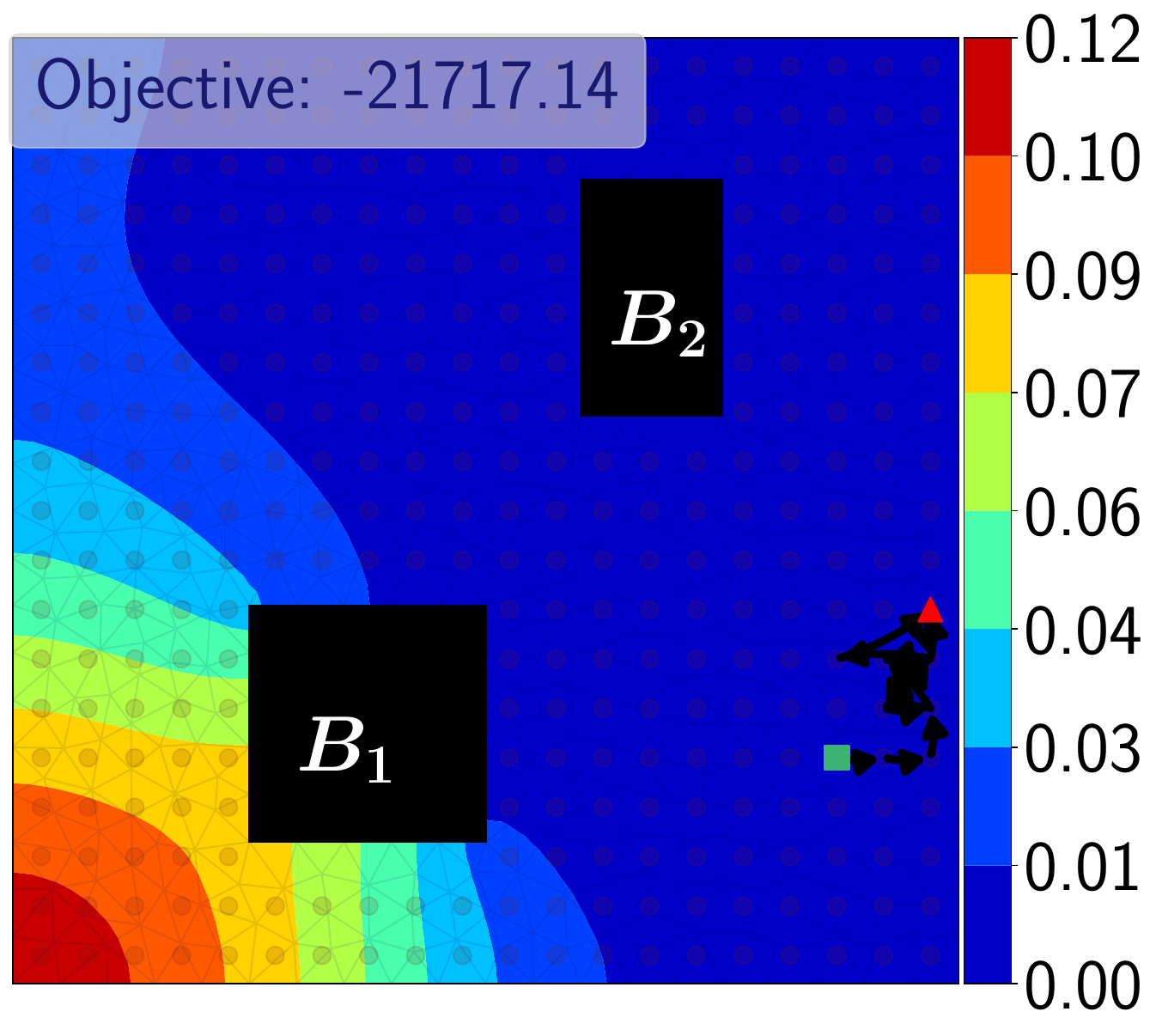}
      \end{subfigure}
      \caption{
        Optimal initial distribution parameters (top row) 
        and optimal trajectories (bottom row) 
        corresponding to \Cref{fig:fine_higher_order_unspecified_start_point}. 
        In each row, the first two panels match the first row of that figure, 
        followed by panels corresponding to its second row.
      }\label{fig:fine_higher_order_unspecified_start_point_order_3_initial_param_and_traject}
    \end{figure}
    \begin{figure}[!htbp]
      \centering
      \includegraphics[width=0.495\linewidth]{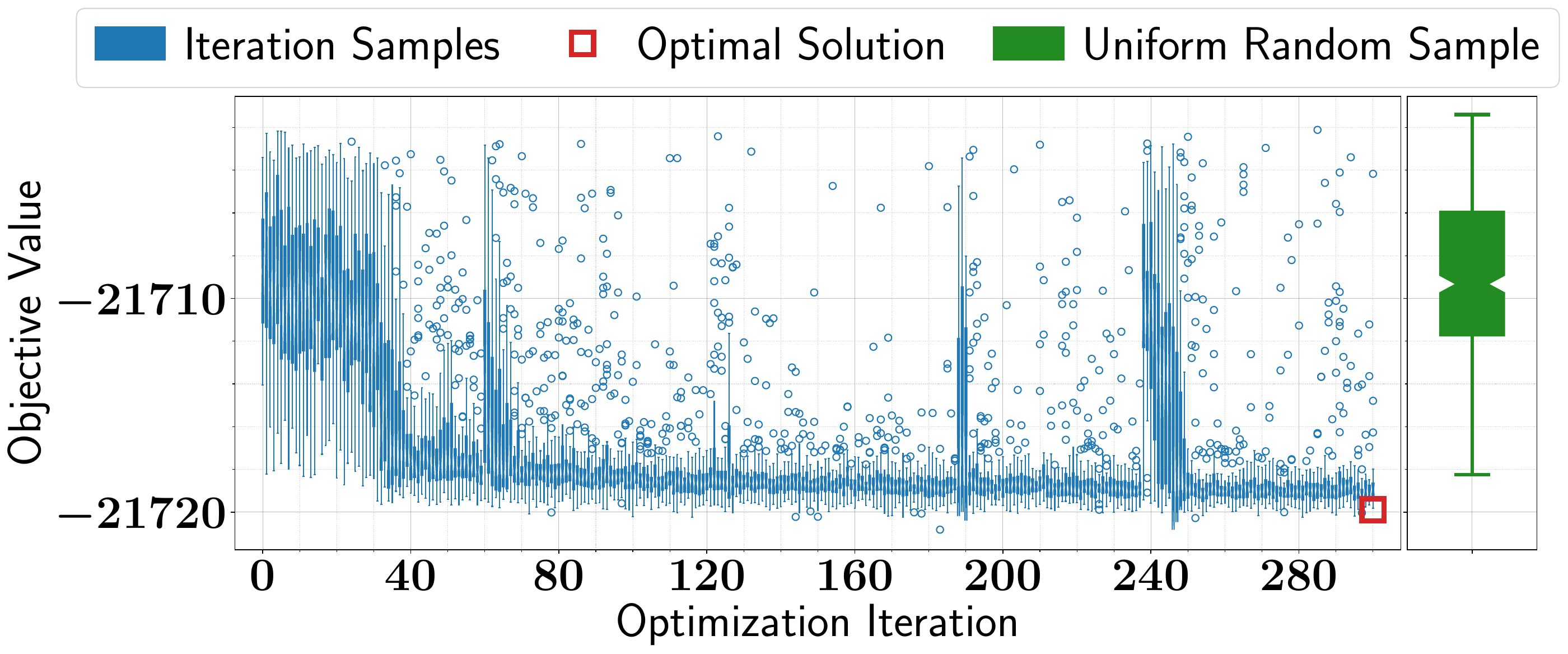}
      \includegraphics[width=0.495\linewidth]{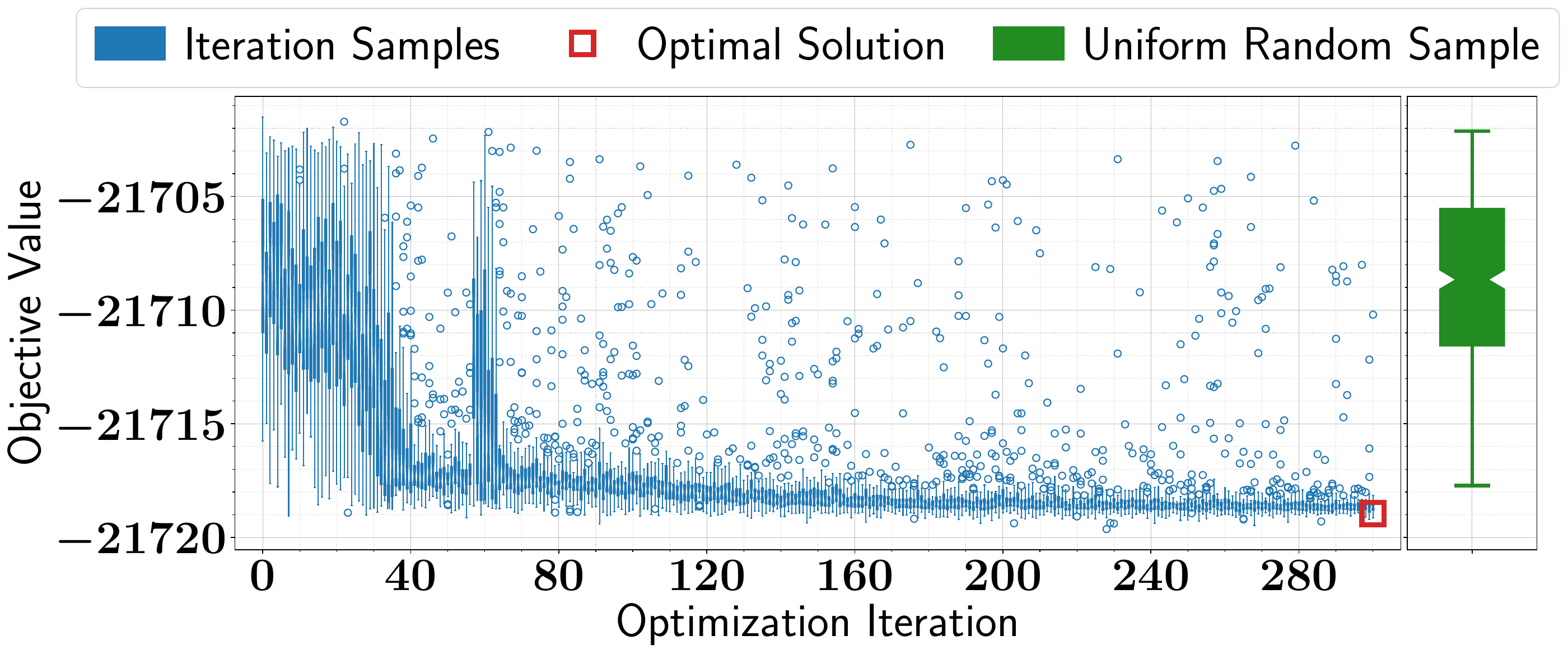}
      \includegraphics[width=0.495\linewidth]{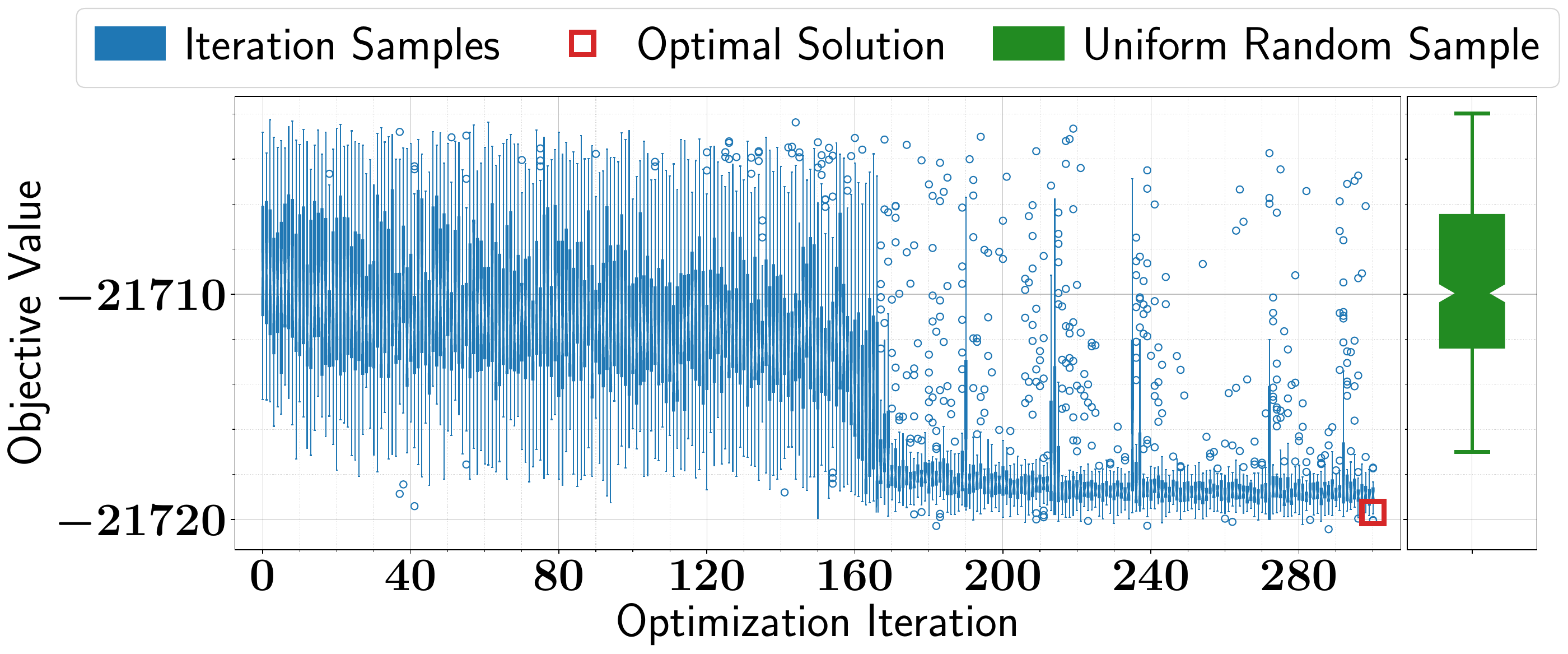}
      \includegraphics[width=0.495\linewidth]{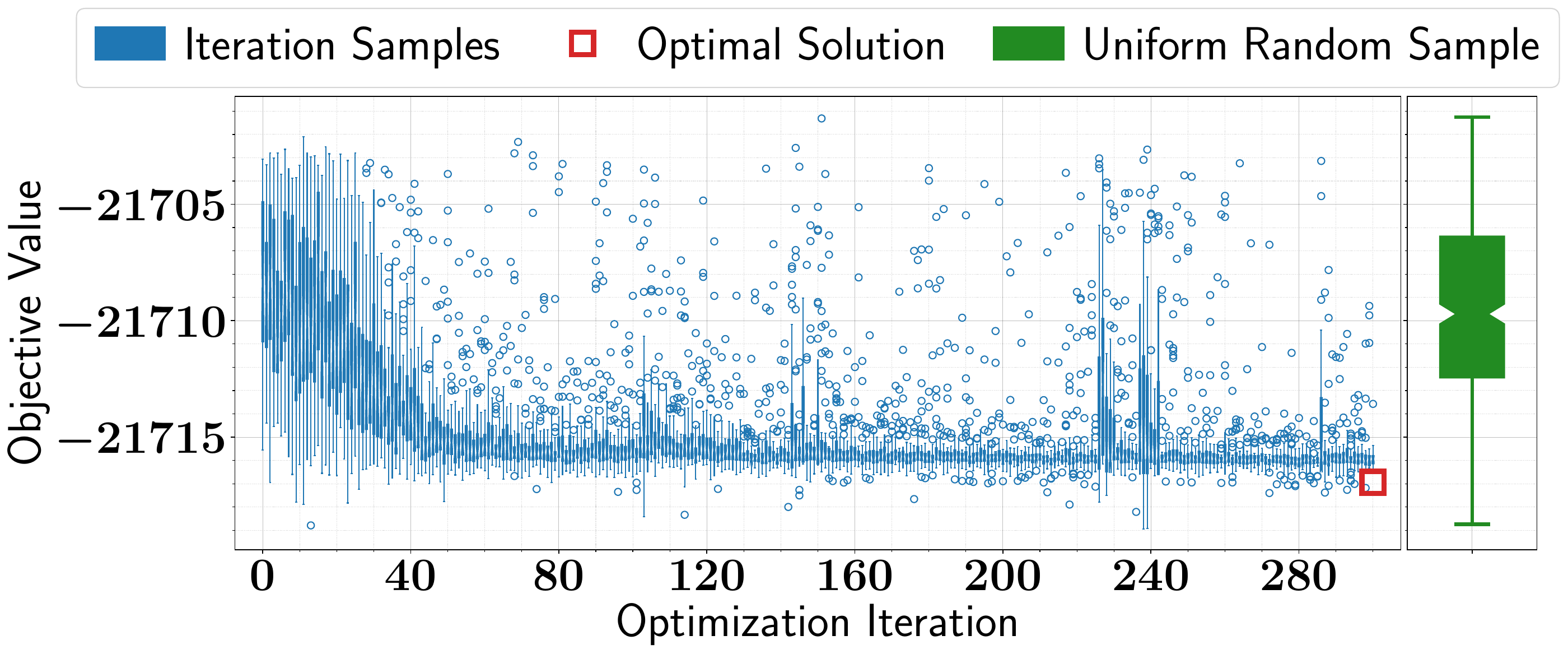}
      \caption{
        Similar to \Cref{fig:fine_higher_order_unspecified_start_point}.
        Here the policy given by \Cref{defn:generalized_higher_order_path_model} is used.
      }\label{fig:fine_generalized_higher_order_unspecified_start_point}
    \end{figure}
    \begin{figure}[!htbp]
      \begin{subfigure}[t]{0.24\linewidth}
        \includegraphics[width=0.91\linewidth]{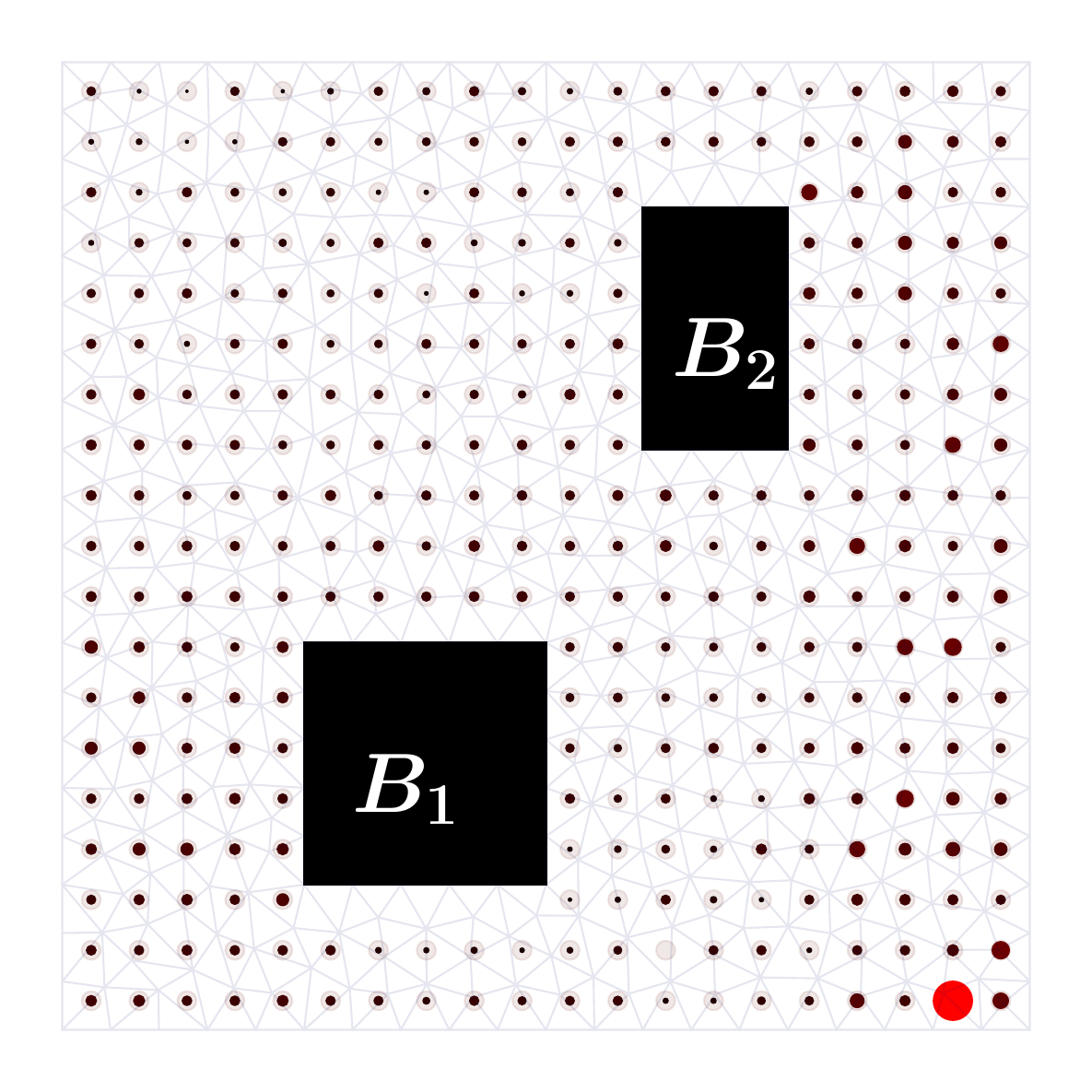}
      \end{subfigure}%
      \begin{subfigure}[t]{0.24\linewidth}
        \includegraphics[width=0.91\linewidth]{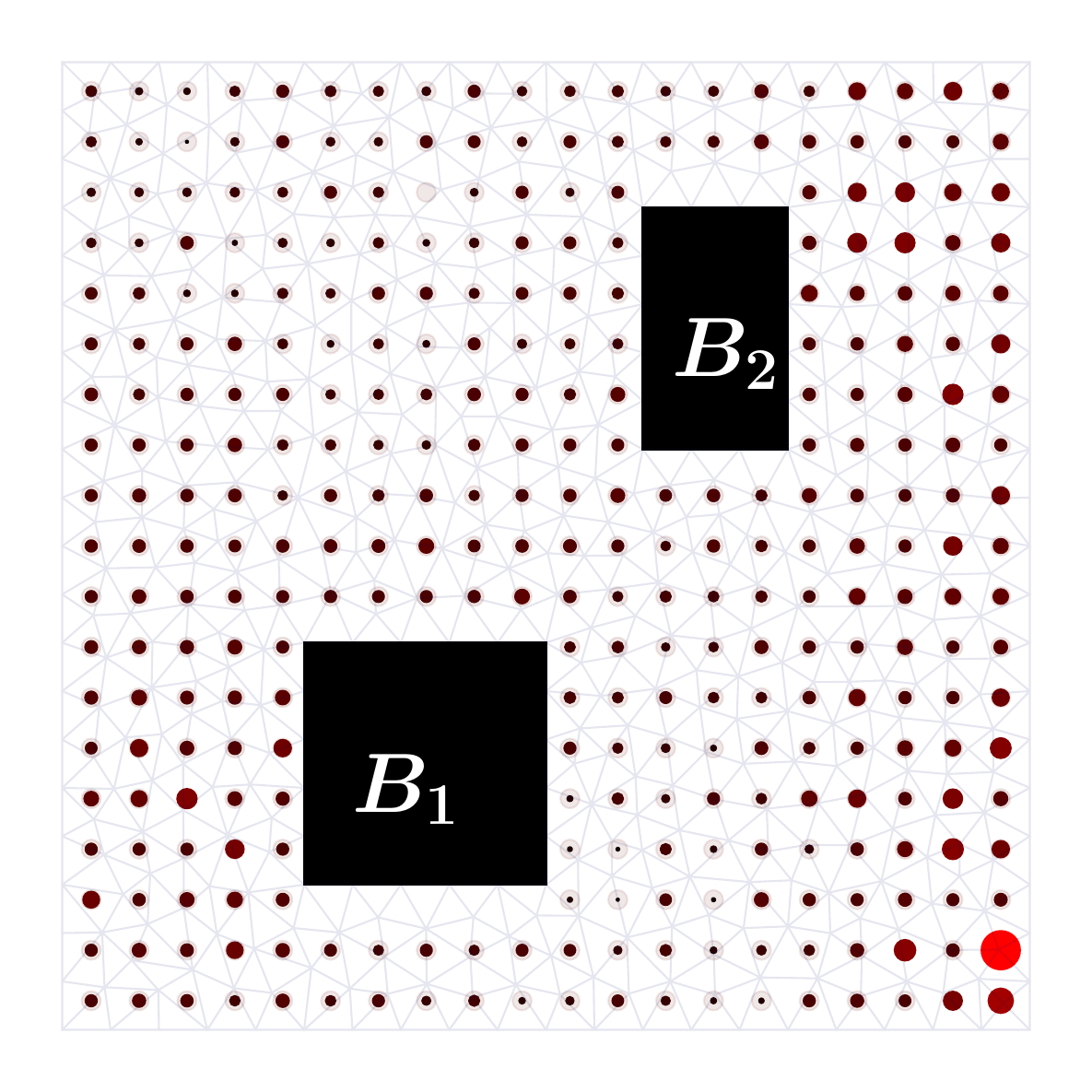}
      \end{subfigure}%
      \begin{subfigure}[t]{0.24\linewidth}
        \includegraphics[width=0.91\linewidth]{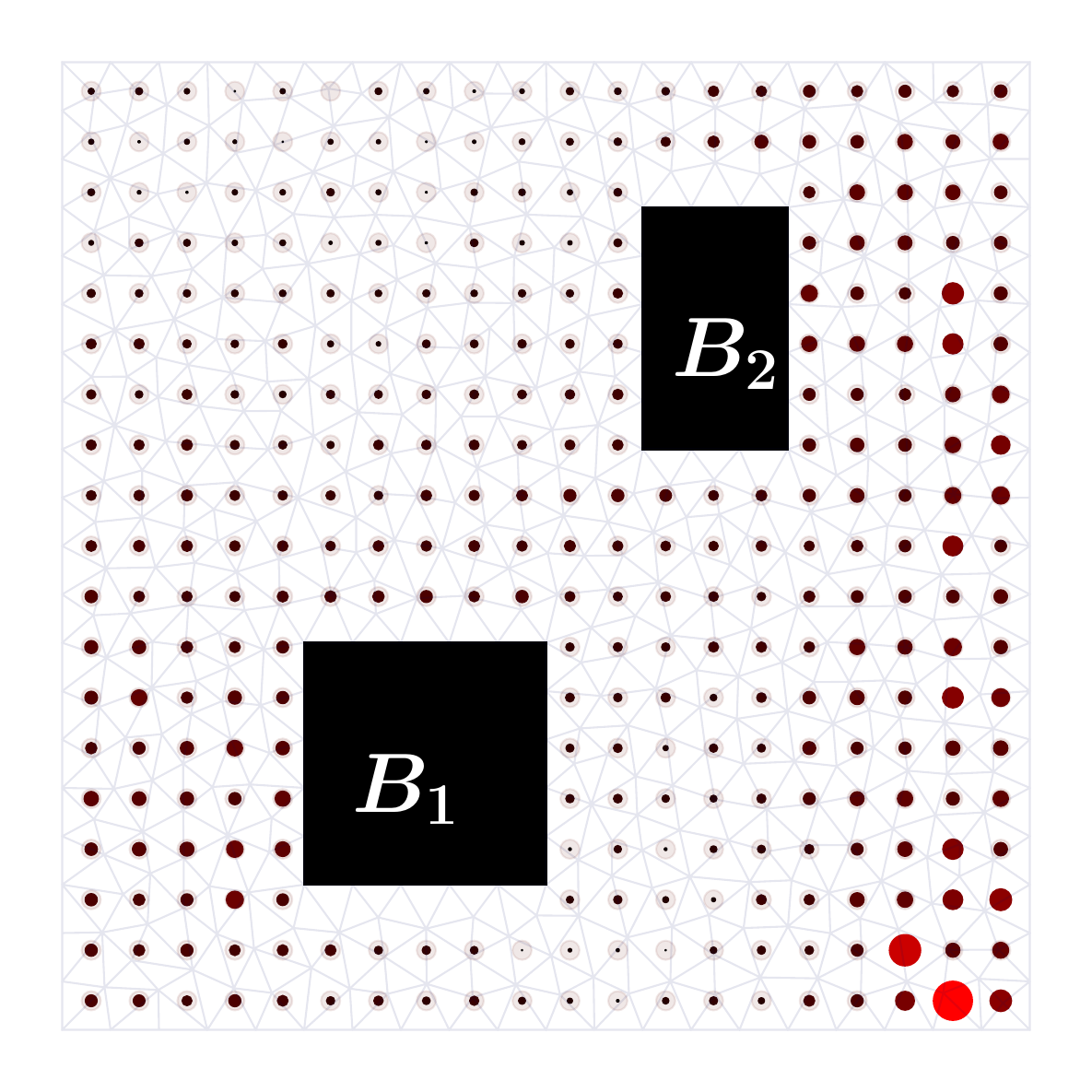}
      \end{subfigure}%
      \begin{subfigure}[t]{0.24\linewidth}
        \includegraphics[width=0.91\linewidth]{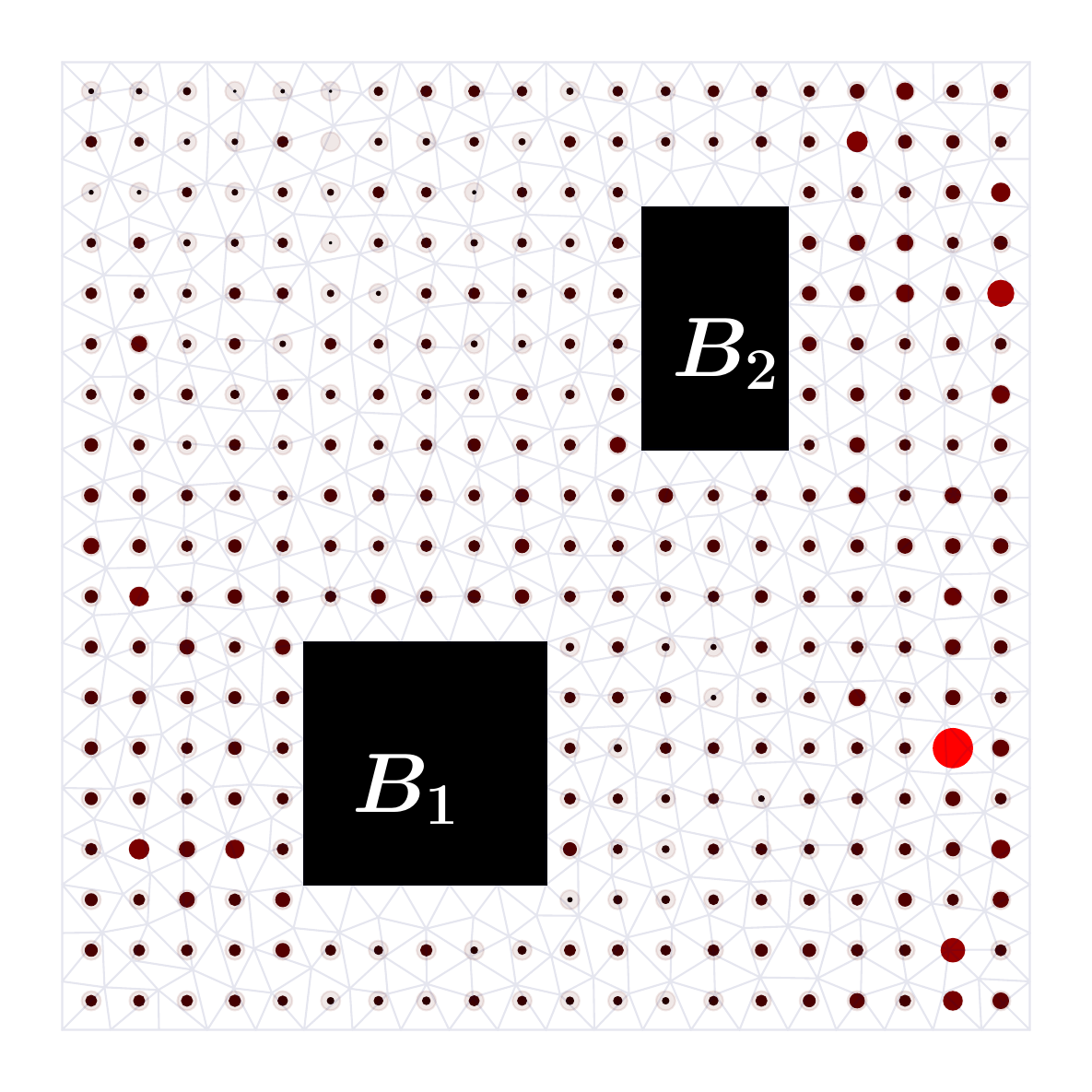}
      \end{subfigure}

      \begin{subfigure}[t]{0.245\linewidth}
        \includegraphics[width=\linewidth]{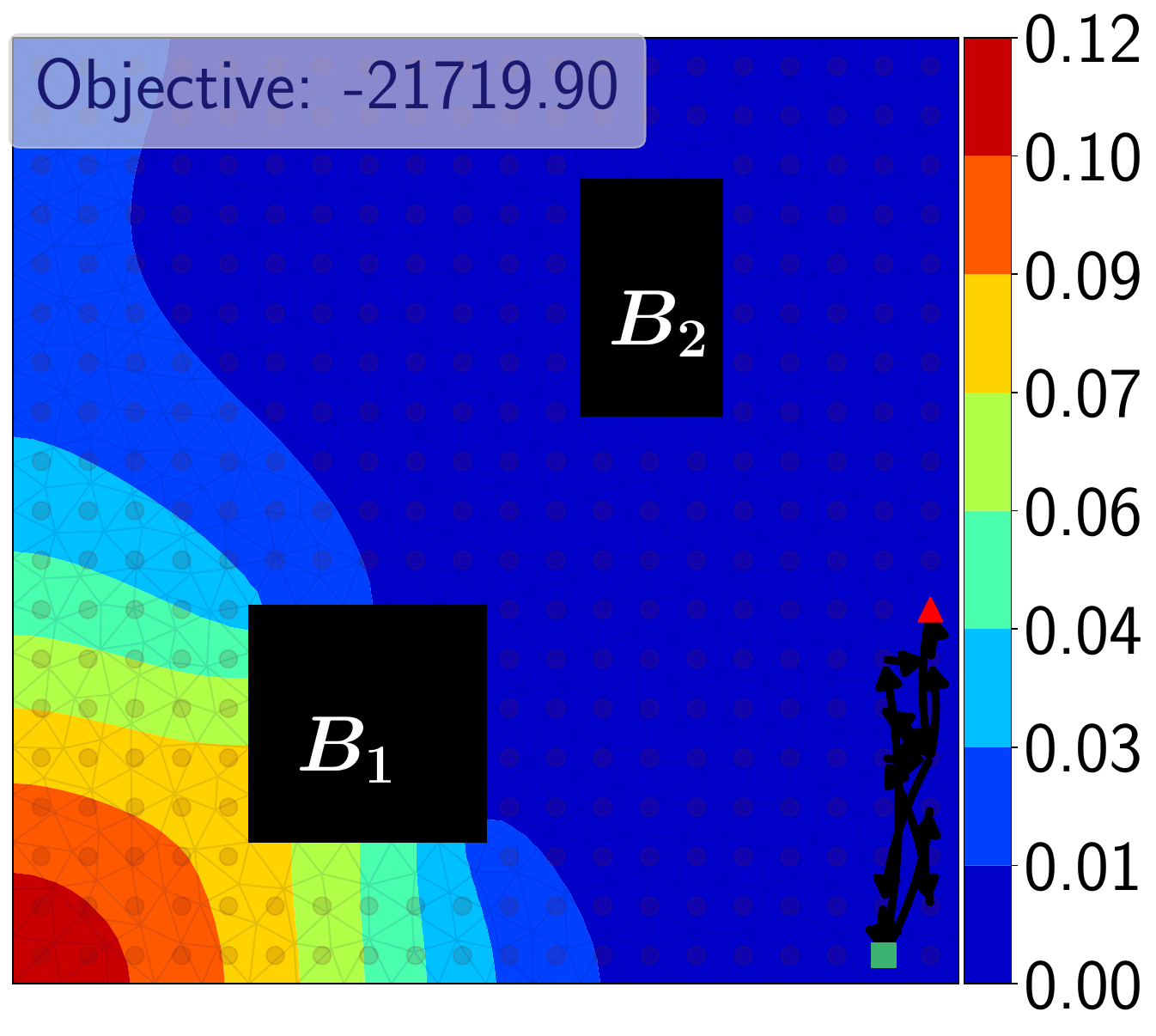}
      \end{subfigure}%
      \begin{subfigure}[t]{0.245\linewidth}
        \includegraphics[width=\linewidth]{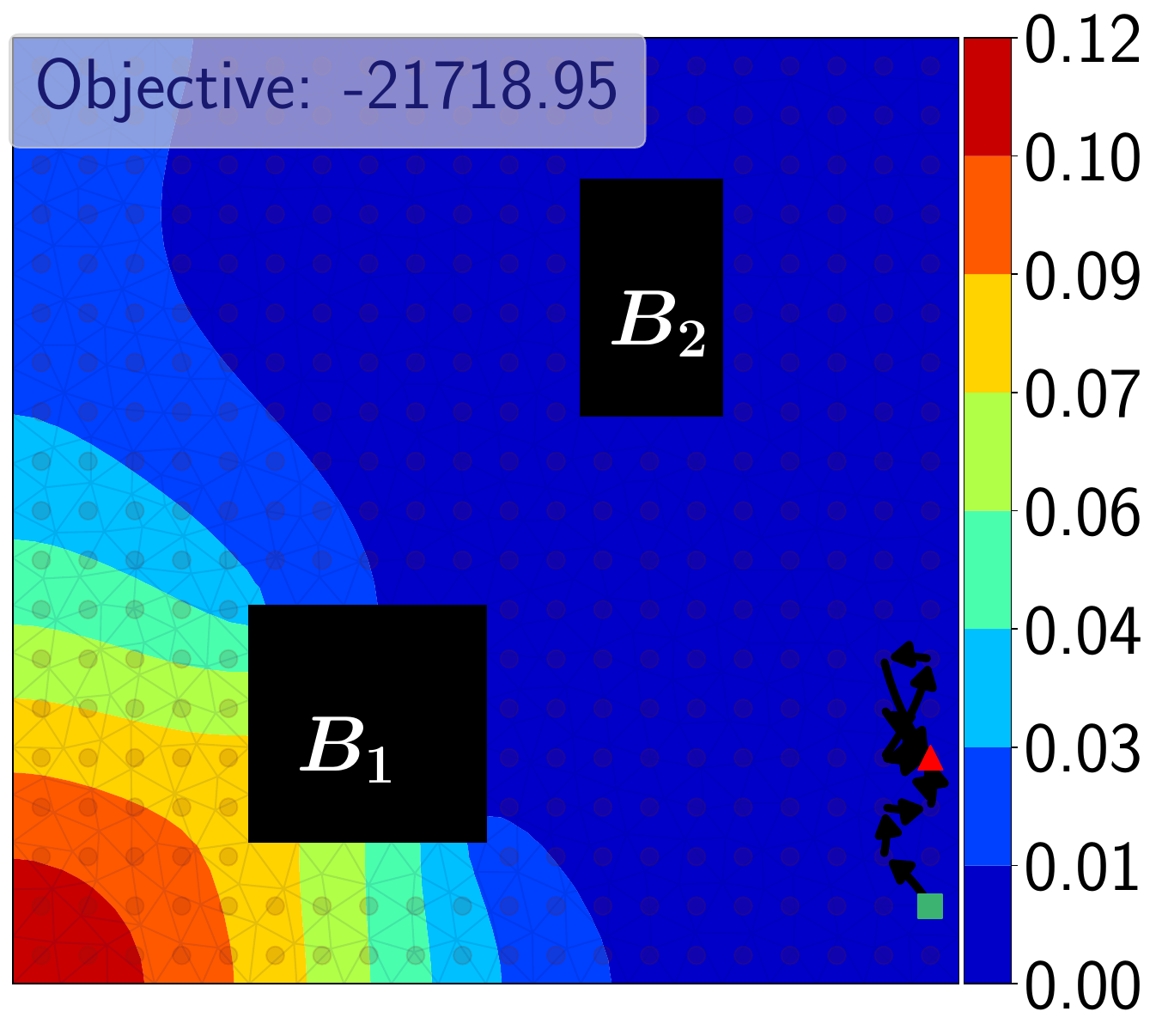}
      \end{subfigure}%
      \begin{subfigure}[t]{0.245\linewidth}
        \includegraphics[width=\linewidth]{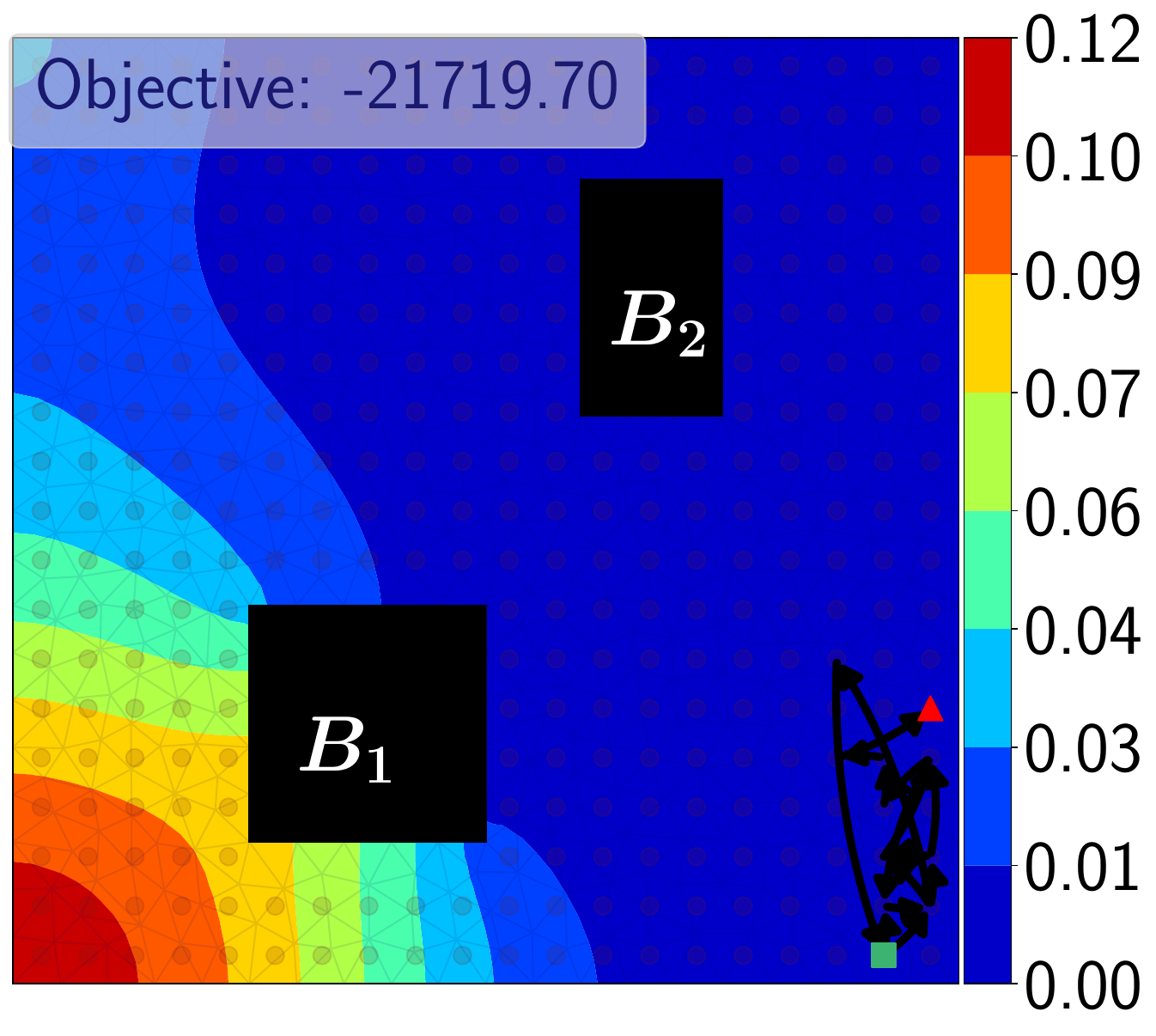}
      \end{subfigure}%
      \begin{subfigure}[t]{0.245\linewidth}
        \includegraphics[width=\linewidth]{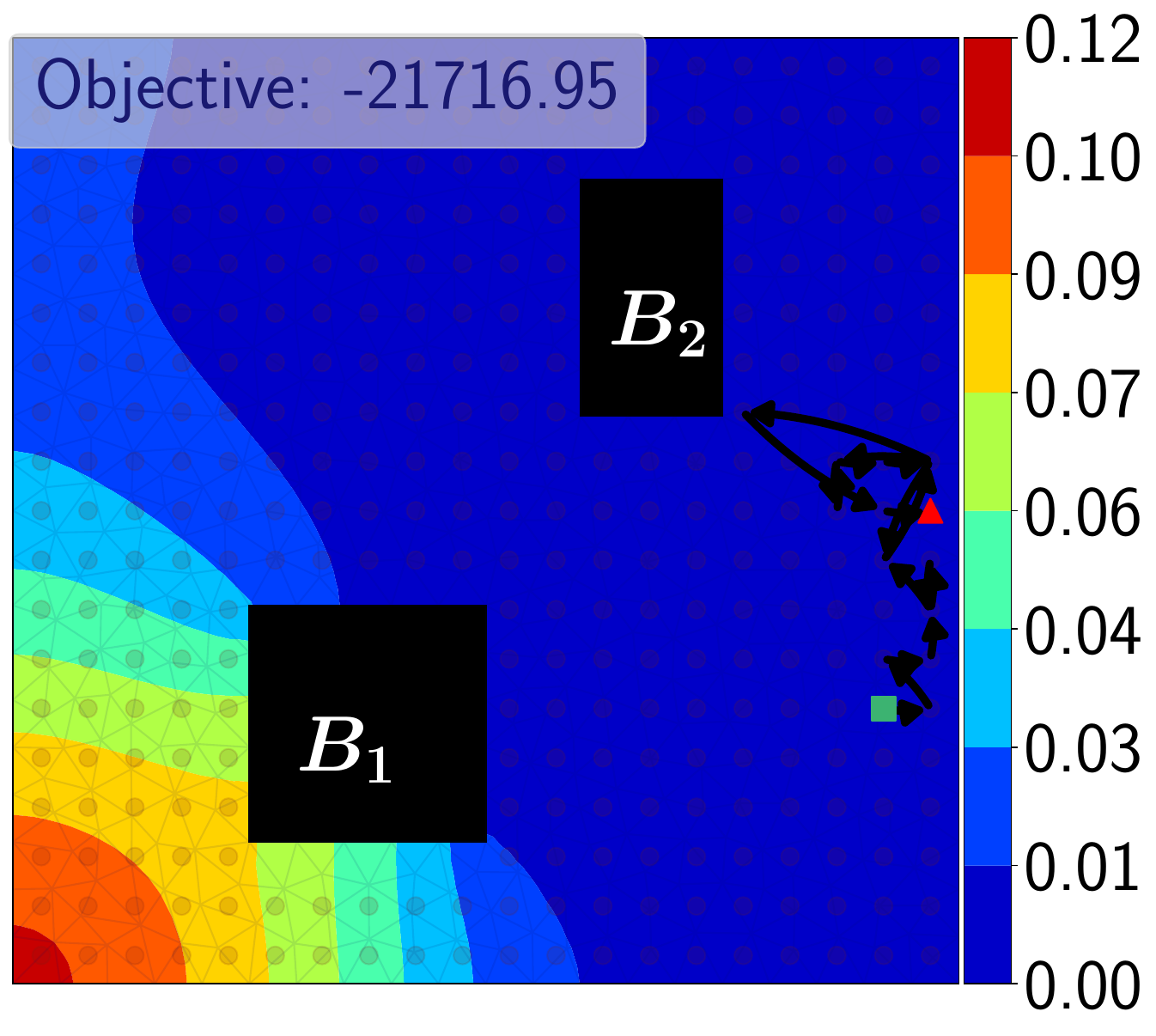}
      \end{subfigure}
      \caption{
        Optimal initial parameters and trajectories
        corresponding to 
        \Cref{fig:fine_generalized_higher_order_unspecified_start_point}. 
      }\label{fig:fine_generalized_higher_order_unspecified_start_point_order_3_initial_param_and_traject}
    \end{figure}

    The results shown here suggest that the higher-order policies both enable achieving 
    lower objective values than the first-order policy, with optimal solutions closer to 
    the global optimum of the experiment with coarse mesh; 
    see \Cref{fig:coarse_bruteforce} (right).
    Similar to the case of the coarse mesh in \Cref{subsec:numerical_results_coarse}, 
    for a higher order the optimization procedure takes more time to reduce 
    the average objective.
    Additionally, the modeled lag weights \eqref{eqn:decreasing_lag_weights}
    stabilize the optimizer more than when the 
    lag weights are calibrated. 
    All higher-order configurations here identify the bottom right corner of the domain 
    as the favorable initial region of the path 
    and return an optimal design with similar objective values.
    The generalized higher-order model with lag-dependent transitions 
    (\Cref{defn:generalized_higher_order_path_model}), however, 
    yields slightly better performance than does the case of the higher-order policy 
    (\Cref{defn:higher_order_path_model}). 
    This, however, comes at a higher computational cost, as discussed 
    in \Cref{subsubsec:computational_conciderations}.

    We conclude by noting that different OED utility functions can yield optimal trajectories with distinct shapes, as supported by the results here and in \Cref{sup:sec:numerical_results}. Nevertheless, across all choices of $\utilityfunc$, the optimal trajectories lie in regions of minimal activity, reflecting reduced uncertainty induced by the construction of the synthetic data and noise covariance here.
    Utility functions may also be designed to target specific 
    objectives, such as prediction accuracy for tracking 
    concentration modes, and their choice is independent 
    of the proposed path optimization framework.

\section{Discussion and Concluding Remarks}
  \label{sec:conclusions}
  This work presents an efficient probabilistic approach for path optimal experimental design, 
  such as optimizing moving sensors path.
  The path OED problem is formulated as discrete optimization on a navigation mesh, treating the utility function as a black box, and is reformulated as a probabilistic optimization over policy parameters. The approach is broadly applicable to path OED problems on navigation meshes and accommodates any utility function.

  A central feature of the approach is that it needs only pointwise evaluations of
  the utility. It requires no derivatives of the utility or of the observation
  operator with respect to the design. The method is therefore agnostic to the form
  of the utility. It applies unchanged to any utility function, to linear and
  nonlinear inverse problems, and to objectives beyond experimental design, while
  enforcing the navigation constraints by construction. A common alternative
  parameterizes the trajectory directly, for example through an ordinary
  differential equation, and optimizes the parameters with a gradient-based method.
  A parallel development pursues this route for infinite-dimensional Bayesian
  inverse problems governed by PDEs \cite{neuberger2026path}. Such methods can use
  fewer utility evaluations. They also require differentiable utility and
  observation operators, they rely on design- and utility-specific gradients whose
  derivation is involved, and their performance depends on the chosen parametric
  family. These parametrized-path formulations are viable alternatives, and a
  head-to-head comparison with them is outside the scope of this work. The potential
  price of the present approach is a larger number of black-box utility evaluations.
  When the utility is very expensive and design gradients are readily available, a
  gradient-based approach may be preferable.

  In this work we propose three parametric policies to model the probability distribution 
  of the design/path based on first- and higher-order Markov chains.
  The experiments discussed in \Cref{sec:numerical_experiments} 
  along with the additional results in the supplementary material in 
  \Cref{sup:sec:numerical_results} indicate that  
  the first-order policy (\Cref{defn:first_order_path_model}) yields acceptable 
  results and generally converges to a policy that explores the lower tail of 
  the OED utility function distribution.
  The generalized higher-order policy with lag-dependent transitions 
  (\Cref{defn:generalized_higher_order_path_model}) achieves 
  better results than does the higher-order policy with lag-independent transitions 
  (\Cref{defn:higher_order_path_model}) at a slightly higher computational cost. 
  The higher-order policies are generally more flexible; but without freezing the lag weights, 
  the optimizer can be unstable and may require more iterations to converge.
  Moreover, they can allow the generation of trajectories involving successive nodes
  violating connectivity in the navigation mesh.
  To mitigate such an effect, one can increase the fineness of the mesh, 
  use a small order (e.g., $k=3$), or use a fast-decaying series of lag weights to enforce 
  much lower weight on the past. 

  The empirical study shows that the optimization procedure behaves similarly for various choices 
  of the OED utility function: overall 
  the optimizer quickly reduces the average value of the utility function, converging
  to the tail of the utility function distribution.
  This enables exploring the space near global optima. Reaching the global  optimal value, however, is not guaranteed. 

  A central practical finding of this work is the role of the optimal baseline 
  in the REINFORCE-style stochastic gradient estimator. As shown empirically in 
  \Cref{subsec:importance_of_baseline}, the baseline substantially reduces the 
  variance of the stochastic gradient even with a batch size of one, at no 
  additional cost, and is critical for the convergence of the optimizer at the 
  modest sample sizes used in our experiments. Without a baseline, the 
  optimizer fails to consistently reduce the expected utility, and any 
  near-optimal trajectory recovered is largely incidental. We therefore 
  recommend that the baseline always be used when applying the proposed 
  framework, particularly in settings where utility evaluations are expensive 
  and small sample sizes are unavoidable.

  While we provide insight on the choice between the proposed policies,
  and the policy order, these remain as user choices. 
  Alternatively, one can 
  identify the best model and order, 
  for example, based on some model choice or information criterion. 
  This, however, is outside the scope of this work and will be explored in the future.
  The proposed policies do not explicitly enforce path smoothness, but this can be indirectly encouraged through penalizing the utility function or by fitting smooth curves between nodes on a coarse mesh.
  Adding explicit design constraints is another useful direction. Here it helps
  to separate the value of a path from its cost. The value of a path is how
  informative it is, and this is already measured by the OED utility. The
  optimizer therefore favors more informative paths on its own. Value is thus
  part of the objective and is not a constraint. What remains to be constrained
  is the feasibility and the cost of a path. Several of these constraints are
  already built into the model. The directed graph keeps every step on a feasible
  move. Fixed or disallowed nodes are handled by the degenerate parameters of
  \Cref{subsubsec:degenerate_case}. Each step is a budget-one conditional
  Bernoulli that picks a single move, so the trajectory length $n$ already acts as
  a budget on the number of observations.

  Other constraints are harder to impose. Asking the sensor not to revisit a node
  turns the trajectory into a self-avoiding walk \cite{madras1993self}. Such a
  walk is no longer Markovian, and even drawing samples from it is difficult and
  usually needs special Monte Carlo methods \cite{madras1988pivot}. We are not
  aware of a way to enforce it exactly within the present policy model. A budget on the
  total cost of a path is hard for a related reason. A natural cost here is the
  energy used by the moving sensor, which depends on the length of each move and
  on whether the sensor travels with or against the advective flow, so different
  paths spend different amounts of a shared energy budget. Staying within the
  budget makes the moves available at each step depend on how much budget is
  left, which in turn depends on the path taken so far. The policy then becomes
  history dependent rather than stationary, which is the setting studied for
  constrained Markov decision processes \cite{altman1999constrained}. A
  practical alternative is to discourage unwanted paths through a penalty in the
  utility. We leave the exact treatment of these constraints to future work.


  \appendix
\section{Proofs of Propositions in \Cref{subsec:the_probabilistic_models}}
\label{app:sec:gradients_proofs}
  \begin{proof}[Proof of \Cref{proposition:first_order_gradient}]  
    The log-policy gradient \eqref{eqn:first_order_trajectory_distribution_grad_log_prob_init}
    requires defining the gradients of \commentout{both} the initial distribution $\Prob{\design_1}$ (first term) 
    and the transition distribution $\CondProb{\design_{t+1}}{\design_{t}}$ (second term). 
    The parameter vector contains initial parameters $\hyperparam_{i}$ 
    and transition parameters $\hyperparam^{i}_j$. Moreover, the initial distribution 
    does not depend on the transition parameters, and 
    the transition distribution does not depend on the initial parameters.
    Thus, it suffices to develop the partial derivatives 
    $\del{\log \pi_i} {\hyperparam_j}$ and 
    $\del{\log \pi^{i}_{j}}{\hyperparam^{l}_{m}}$: 
    
    \begin{subequations}
    {\allowdisplaybreaks
      \begin{align}
        \del{\log \pi_i} {\hyperparam_j}
          &\equiv\del{\log \Prob{\design_1=v_i}} {\hyperparam_j}
           \stackrel{\eqref{eqn:initial_distribution_inclusion}}{=} 
           \del{}{\hyperparam_j}{\left(\log{w_i}-\brlog{\sum_{k=1}^{N}{w_k}} \right)} 
             \\
             &= \frac{1}{w_i} \del{w_i} {\hyperparam_j} 
               - \frac{\sum\limits_{k=1}^{N} \del{w_k}{\hyperparam_j} }{\sum\limits_{k=1}^{N}{w_k}} 
       = \frac{1}{w_i} \del{w_i} {\hyperparam_j} 
           - \frac{ \del{w_j}{\hyperparam_j} }{\sum\limits_{k=1}^{N}{w_k}} 
           \stackrel{\eqref{eqn:initial_distribution_inclusion}}{=} 
            \frac{\delta_{ij} }{w_j} \del{w_j} {\hyperparam_j} 
             - \frac{ \del{w_j}{\hyperparam_j} }{\sum\limits_{k=1}^{N}{w_k}}  \\
           &= \del{w_j} {\hyperparam_j} \left(
                \frac{\delta_{ij}}{w_j} 
               - \frac{ 1 }{\sum\limits_{k=1}^{N}{w_k}} 
             \right) 
           = \frac{1}{\left(1-\hyperparam_j\right)^2} \left(
                \frac{\delta_{ij}}{w_j} 
               - \frac{ 1 }{\sum\limits_{k=1}^{N}{w_k}} 
             \right) \,,
      \end{align}
    }
    \end{subequations}
    where 
    $\delta_{in}$ is the Kronecker delta function. Similarly we derive 
    \eqref{eqn:first_order_trajectory_distribution_grad_log_prob_transition}:
    \begin{equation}\label{eqn:transition_distribution_inclusion_log_prob_grad}
      \del{\log \pi^{i}_{j}}{\hyperparam^{l}_{m}}
         = \frac{\delta_{jm} }{w^{l}_{m}} \del{w^{l}_{m}} {\hyperparam^{l}_{m}} 
           - \frac{ \del{w^{l}_{m}}{\hyperparam^{l}_{m}} }{\sum\limits_{k=1}^{N}{w^{l}_{m}}} 
         = \frac{\delta_{il}}{\left(1-\hyperparam^{l}_{m}\right)^2} \left(
              \frac{\delta_{jm}}{w^{l}_{m}} 
             - \frac{ 1 }{\sum\limits_{k=1}^{N}{w^{l}_{k}}} 
           \right) \,.
    \end{equation}
  \end{proof}
  \begin{proof}[Proof of \Cref{proposition:higher_order_gradient}]
    The policy parameter \eqref{eqn:higher_order_trajectory_distribution_param} 
    consists of 
    the initial parameters, 
    the transition parameters, 
    and the lag parameters. 
    Thus, the gradient of the log-policy is given by the general 
    form \eqref{eqn:higher_order_path_distribution_grad_log_prob_operator}.
    From 
    \eqref{eqn:higher_order_trajectory_distribution_pmf}, it follows that
    \begin{equation}\label{app:eqn:kth-order-trajectory-log-probability}
      \log\Prob{\designvec} 
        = 
        \log\Prob{\design_1} 
          + \sum_{t=1}^{k-1}
          \log\CondProb{\designvec_{t+1}}{\designvec_{t}}
          +
          \sum_{t=k}^{n-1} \log 
          \left( 
            \sum_{i=1}^{k} {\lambda_{i}\, \CondProb{\designvec_{t+1}}{\designvec_{t+1-i}} }
            \right)
          \,.
    \end{equation}

    Only the first term in \eqref{app:eqn:kth-order-trajectory-log-probability} depends on the initial parameters, thus
    $
      \del{\log\Prob{\designvec}}{\hyperparam_j}
        = \del{\log\Prob{\design_{1}}}{\hyperparam_j} 
    $ which proves \eqref{eqn:higher_order_path_distribution_grad_log_prob_operator_initial}.
    Conversely, only the second and third terms in
    \eqref{app:eqn:kth-order-trajectory-log-probability} depend on the transition parameters, which proves \eqref{eqn:higher_order_path_distribution_grad_log_prob_operator_transitions}. 
    The derivative with respect to the lag parameters 
    \eqref{eqn:higher_order_path_distribution_grad_log_prob_operator_lag} is obtained from the third term in
    \eqref{app:eqn:kth-order-trajectory-log-probability} as follows
    \begin{equation}
      \begin{aligned}
        \del{\log\Prob{\designvec} }{\lambda_l}
        &=
        \del{}{\lambda_l}
          \sum_{t=k}^{n-1} \log 
          \left( 
            \sum_{i=1}^{k} {\lambda_{i}\, \CondProb{\designvec_{t+1}}{\designvec_{t+1-i}} }
            \right) 
        \\
        &=
        \sum_{t=k}^{n-1} 
        \frac{ 
              \sum_{i=1}^{k} {\lambda_{i}\, \del{\CondProb{\designvec_{t+1}}{\designvec_{t+1-i} }}{\lambda_l} }
            }{
              \sum_{i=1}^{k} {\lambda_{i}\, \CondProb{\designvec_{t+1}}{\designvec_{t+1-i}} }
            }
        = 
        \sum_{t=k}^{n-1} 
        \frac{
              \CondProb{\designvec_{t+1}}{\designvec_{t+1-l}} 
            }{
              \sum_{i=1}^{k} {\lambda_{i}\, \CondProb{\designvec_{t+1}}{\designvec_{t+1-i}} }
            }
        \,.
      \end{aligned}
    \end{equation}
  \end{proof}
  \begin{proof}[Proof of \Cref{proposition:generalized_higher_order_gradient}]
    The proof is similar to that of \Cref{proposition:generalized_higher_order_gradient} with  
    higher-order transitions used.
    Specifically, from \eqref{eqn:generalized_higher_order_trajectory_distribution_pmf} 
    it follows that
    \begin{equation}\label{app:eqn:generalized_kth-order-trajectory-log-probability}
      \log\Prob{\designvec} 
        = 
        \log\Prob{\design_1} 
          + \sum_{t=1}^{k-1}
          \log\CondProb{\designvec_{t+1}}{\designvec_{t}}
          +
          +
          \sum_{t=k}^{n-1} \log 
          \left( 
            \sum_{i=1}^{k} {\lambda_{i}\, \nCondProb{i}{\designvec_{t+1}}{\designvec_{t+1-i}} }
            \right)
          \,.
    \end{equation}

    Because the first terms in 
    \eqref{app:eqn:kth-order-trajectory-log-probability} and 
    \eqref{app:eqn:generalized_kth-order-trajectory-log-probability}
    are identical, 
    the derivative of the log-policy with respect to the initial parameters 
    holds here as well, which proves 
    \eqref{eqn:generalized_higher_order_path_distribution_grad_log_prob_operator_initial}.
    Conversely, the third term in 
    \eqref{app:eqn:generalized_kth-order-trajectory-log-probability} replaces 
    the first-order transitions 
    $
      \CondProb{\designvec_{t+1}}{\designvec_{t+1-i}}
    $ 
    with the higher-order transitions 
    $
      \nCondProb{i}{\designvec_{t+1}}{\designvec_{t+1-i}}
    $. 
    This modifies the derivative of the log-policy with respect to the transition and the lag parameters as follows. 
    The derivative with respect to transition 
    parameters \eqref{eqn:generalized_higher_order_path_distribution_grad_log_prob_operator_transitions} is
    \begin{equation}
      \begin{aligned}
        \del{\log\Prob{\designvec} }{\hyperparam^{l}_{m}}
          &=  \sum_{t=1}^{k-1} 
            \del{\log \CondProb{\designvec_{t+1}}{\designvec_{t}}}{\hyperparam^{l}_{m}}
          +
          \del{}{\hyperparam^{l}_{m}} 
            \sum_{t=k}^{n-1} \log 
            \left( 
              \sum_{i=1}^{k} {\lambda_{i}\, \nCondProb{i}{\designvec_{t+1}}{\designvec_{t+1-i}} }
              \right)
        \\
        &=
          \sum_{t=1}^{k-1} 
            \del{\log \CondProb{\designvec_{t+1}}{\designvec_{t}}}{\hyperparam^{l}_{m}}
            +
            \sum_{t=k}^{n-1} 
            \frac{ 
              \sum_{i=1}^{k} {\lambda_{i}\, \del{}{\hyperparam^{l}_{m}} \nCondProb{i}{\designvec_{t+1}}{\designvec_{t+1-i}} }
            }{
              \sum_{i=1}^{k} {\lambda_{i}\, \nCondProb{i}{\designvec_{t+1}}{\designvec_{t+1-i}} }
            } \,.
      \end{aligned}
    \end{equation}

    The derivative with respect to the lag parameters 
    \eqref{eqn:generalized_higher_order_path_distribution_grad_log_prob_operator_lag} 
    is obtained directly by differentiating the third term in 
    \eqref{app:eqn:generalized_kth-order-trajectory-log-probability} 
    since the first two terms are independent of these parameters.
    Finally \eqref{eqn:generalized_higher_order_path_distribution_grad_log_prob_operator_higher_order_gradient}
    follows by differentiating the higher-order transition probabilities
    \eqref{eqn:generalized_higher_order_transition}
    with respect to the first-order transition parameters. 
  \end{proof}

%
\clearpage

\def\siamprelabel{SM}

\renewcommand{\thesection}{\siamprelabel\arabic{section}}

\setcounter{section}{0}
\setcounter{subsection}{0}
\setcounter{equation}{0}
\setcounter{figure}{0}
\setcounter{table}{0}
\setcounter{page}{1}

\renewcommand{\theHsection}{SM.\arabic{section}}
\renewcommand{\theHsubsection}{SM.\arabic{section}.\arabic{subsection}}
\renewcommand{\theHequation}{SM.\arabic{equation}}
\renewcommand{\theHfigure}{SM.\arabic{figure}}
\renewcommand{\theHtable}{SM.\arabic{table}}

\begingroup
  \centering
  \rule{\linewidth}{0.8pt}\par
  \medskip
  {\bfseries\Large Supplementary Materials}\par
  \medskip
  {\itshape \fulltitle}\par
  \smallskip
  {\itshape Ahmed Attia}\par
  \medskip
  \rule{\linewidth}{0.8pt}\par
  \bigskip
\endgroup

  This supplementary material provides
a detailed explanation of the computations 
under the probabilistic policies proposed in 
\Cref{subsec:the_probabilistic_models}, 
and provides additional numerical experiments 
to complement the 
empirical analysis discussed in \Cref{sec:numerical_experiments}.

\section{Illustrative Example}
\label{app:sec:Example}
  Here we 
  provide an example with detailed computations under the parametric probabilistic policies
  proposed by 
  \Cref{defn:first_order_path_model}, 
  \Cref{defn:higher_order_path_model}, and
  \Cref{defn:generalized_higher_order_path_model}, respectively.

  Consider the simplified navigation mesh (transition graph) shown in \Cref{fig:illustrative_example_DMC}, 
  with the first-order transition parameters $\hyperparam^{i}_{j}$ indicated on the 
  corresponding arcs.

  \begin{figure}[H]
    \centering
    \includegraphics[width=0.26\textwidth]{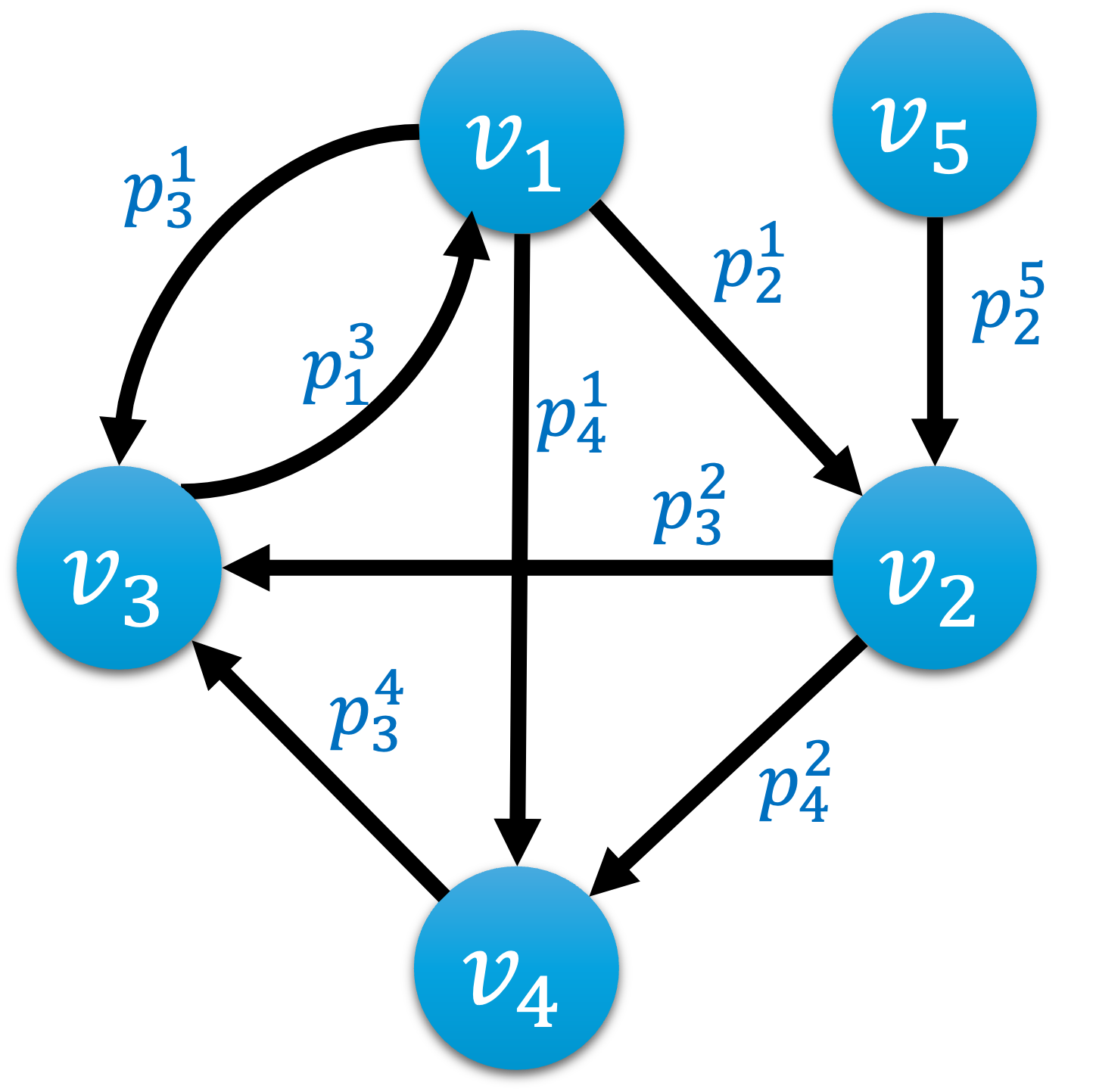}
    \caption{
      Directed graph describing a simple navigation mesh. 
      The graph consists of $\Nsens=5$ nodes 
      $v_1, \ldots, v_5$ with transition parameters $p^{i}_{j}$ describing 
      the probability of moving from node $v_i$ to node $v_j$ displayed 
      on the respective arcs $(v_i, v_j)$.
    }\label{fig:illustrative_example_DMC}
  \end{figure}

  The proposed probabilistic policies require the initial parameters $\hyperparam_i$ 
  that define the probability of a path starting from node $v_i$ and 
  the transition parameters $\hyperparam^{i}_{j}$ that encode the probability of 
  moving from node $v_i$ to node $v_j$. 
  In this example these parameters  are given by
  %
  \begin{subequations}\label{supp:eqn:parameters}
    \begin{align}
      \left[\hyperparam_i\right]_{i=1, \ldots, 5 } 
        &:= \left(
              \hyperparam_1, \, 
              \hyperparam_2, \, 
              \hyperparam_3, \, 
              \hyperparam_4, \, 
              \hyperparam_5
            \right) \,, 
        \label{supp:eqn:initial_parameters}
        \\ 
      \left[\hyperparam^{i}_{j}\right]_{i, j=1, \ldots, 5 } 
        &:= 
          \begin{blockarray}{cccccc} 
          & v_1 & v_2 & v_3 & v_4 & v_5 \\ 
          \begin{block}{c[ccccc]}
            v_1 & 0 & \hyperparam^{1}_{2} & \hyperparam^{1}_{3} & \hyperparam^{1}_{4} & 0 \\
            v_2 & 0  & 0  & \hyperparam^{2}_{3} & \hyperparam^{2}_{4} & 0 \\
            v_3 & \hyperparam^{3}_{1}   & 0  & 0 & 0  & 0 \\
            v_4 & 0   & 0  & \hyperparam^{4}_{1} & 0 & 0  \\
            v_5 & 0   & \hyperparam^{5}_{2}  & 0 & 0 & 0  \\
          \end{block}
        \end{blockarray} 
        \label{supp:eqn:transition_parameters}
        \,,  
    \end{align}
  \end{subequations}
  where the transition parameters (as well as connectivity)
  are induced from the navigation mesh in 
  \Cref{fig:illustrative_example_DMC}. 
  According to \eqref{eqn:first_order_trajectory_distribution_param}, 
  for example, the parameter vector $\hyperparamvec$ of the first-order policy takes the form
  \begin{equation}\label{supp:eqn:example_parameter}
    \hyperparamvec
      = 
          \Bigl(
          \underbrace{
            \hyperparam_{1},\,
            \hyperparam_{2},\,
            \hyperparam_{3},\,
            \hyperparam_{4},\,
            \hyperparam_{5}
          }_{\text{initial parameters}}; \, \,\, 
          \underbrace{
            \hyperparam^{1}_2,\, \hyperparam^{1}_3,\, \hyperparam^{1}_4;\,\,  \, 
            \hyperparam^{2}_3, \hyperparam^{2}_4                   ;\, \, \, 
            \hyperparam^{3}_1                                      ;\, \, \, 
            \hyperparam^{4}_1                                      ;\, \, \, 
            \hyperparam^{5}_2                     
          }_{\text{transition parameters}}
        \,\, \Bigr)  \,,
  \end{equation}
  where all zero-valued transition parameters are removed from the parameter vector because 
  they never change.
  
  When implemented, 
  however, this efficient representation \eqref{supp:eqn:example_parameter} requires keeping a connectivity matrix
  and an indexing scheme to enable efficient development of the gradient of the log-policy 
  with respect to the parameter vector.
  This approach is followed in the implementations available through the PyOED 
  framework \cite{chowdhary2025pyoed}. 
  Specifically, the code---publicly available through the PyOED package repository 
  \cite{pyoedrepo}---with implementations of the proposed policies is available 
  under the subpackage
  \texttt{pyoed.stats.distributions}.
   For clarity, however, here we keep the parameters in their original forms 
  \eqref{supp:eqn:parameters}.
 
  The policy parameters \eqref{supp:eqn:parameters} 
  are 
  used to define initial $w_{i}$ and transition $w^{i}_{j}$ weights, respectively:  
  \begin{subequations}  
    \label{supp:eqn:initial_and_transition_weights}
    \begin{align}
      w_i 
        &= \frac{\hyperparam_i}{1-\hyperparam_i}  \,, \qquad i=1, \ldots, 5 \,, \\
      w^{i}_{j} 
        &= \frac{\hyperparam^{i}_{j}}{1-\hyperparam^{i}_{j}} \,, \quad 
        \begin{matrix}
          & i = 1,\ldots,5 \,, \\ 
          & j = 1,\ldots,5 \,. 
        \end{matrix}
    \end{align}
  \end{subequations}

  From \eqref{supp:eqn:parameters} and \eqref{supp:eqn:initial_and_transition_weights}, the 
  initial and transition weights are given by
  \begin{subequations}
    \label{supp:eqn:initial_and_transition_weights_evaluated}
    \begin{align}
      \left[w_i\right]_{i=1, \ldots, 5} 
        &= \left(
            \frac{\hyperparam_1}{1-\hyperparam_1}, 
            \frac{\hyperparam_2}{1-\hyperparam_2}, 
            \frac{\hyperparam_3}{1-\hyperparam_3}, 
            \frac{\hyperparam_4}{1-\hyperparam_4}, 
            \frac{\hyperparam_5}{1-\hyperparam_5}
          \right) \,,
          \label{supp:eqn:initial_and_transition_weights_evaluated_initial}
          \\
      \left[w^{i}_{j}\right]_{i,j=1, \ldots, 5}
        &=
        \begin{blockarray}{cccccc} 
          & v_1 & v_2 & v_3 & v_4 & v_5 \\ 
          \begin{block}{c[ccccc]}
            v_1 & 0 & \frac{\hyperparam^{1}_{2}}{1-\hyperparam^{1}_{2}} 
                & \frac{\hyperparam^{1}_{3}}{1-\hyperparam^{1}_{3}} 
                & \frac{\hyperparam^{1}_{4}}{1-\hyperparam^{1}_{4}} & 0 \\
            v_2 & 0  & 0  & \frac{\hyperparam^{2}_{3}}{1-\hyperparam^{2}_{3}} 
                & \frac{\hyperparam^{2}_{4}}{1-\hyperparam^{2}_{4}} & 0 \\
            v_3 & \frac{\hyperparam^{3}_{1}}{1-\hyperparam^{3}_{1}}   & 0  & 0 & 0  & 0\\
            v_4 & 0   & 0  & \frac{\hyperparam^{4}_{3}}{1-\hyperparam^{4}_{3}} & 0 \\
            v_5 & 0   & \frac{\hyperparam^{5}_{2}}{1-\hyperparam^{5}_{2}}  & 0 & 0 & 0 \\ 
          \end{block}
        \end{blockarray} \,,
          \label{supp:eqn:initial_and_transition_weights_evaluated_transition}
    \end{align}
  \end{subequations}

  Common to the three policies proposed in this work are the initial 
  (inclusion) probabilities 
  $\pi^{i}\equiv \Prob{\design_{1}=v_i}$ \eqref{eqn:initial_distribution_inclusion}  
  and the first-order transition probabilities 
  $\pi^{i}_{j} \equiv \CondProb{\design_{t+1}=v_j}{\design_{t}=v_i}$ 
  \eqref{eqn:transition_probabilities} for any positive integer (time) $t>0$.
  From \eqref{supp:eqn:initial_and_transition_weights_evaluated} it follows that
  \begin{subequations}
    \label{supp:eqn:initial_and_transition_probabilities_evaluated}
    \begin{align}
      [\pi_i]_{i=1, \ldots, 5} 
        &= \frac{1}{\sum_{i=1}^{5}{\frac{\hyperparam_i}{1-\hyperparam_i}} }
          \left(
            \frac{\hyperparam_1}{1-\hyperparam_1}, 
            \ldots, 
            \frac{\hyperparam_5}{1-\hyperparam_5}
          \right)
          \label{supp:eqn:initial_probabilities_evaluated}
          \\
      %
      [\pi^{i}_{j}]_{i, j=1, \ldots, 5}  
        &=  
          \begin{blockarray}{cccccc} 
            & v_1 & v_2 & v_3 & v_4 & v_5 \\ 
            \begin{block}{c[ccccc]}
              v_1 & 0 & \frac{w^{1}_{2}}{w^{1}_{2}+w^{1}_{3}+w^{1}_{4}} 
                  & \frac{w^{1}_{3}}{w^{1}_{2}+w^{1}_{3}+w^{1}_{4}} 
                  & \frac{w^{1}_{4}}{w^{1}_{2}+w^{1}_{3}+w^{1}_{4}} & 0 \\
              v_2 & 0 & 0 & \frac{w^{2}_{3}}{w^{2}_{3}+w^{2}_{4}} 
                  & \frac{w^{2}_{4}}{w^{2}_{3}+w^{2}_{4}} & 0 \\
              v_3 & 1  & 0  & 0 & 0 & 0 \\
              v_4 & 0   & 0 & 1 & 0 & 0 \\
              v_5 & 0   & 1 & 0 & 0 & 0 \\ 
            \end{block}
          \end{blockarray} \,.
          \label{supp:eqn:transition_probabilities_evaluated}
    \end{align}
  \end{subequations}

  \noindent{}\textbf{Trajectory and Distribution Parameter.}
    Without loss of generality and for simplicity we consider the 
    case of a trajectory of length $n=3$ nodes.
    For the higher-order policy models, 
    we set the policy order to $k=2$.
    Since the path consists of only $n=3$ nodes, the memory here involves the full path. 
    We also assume the lag parameters are modeled by 
    $\lambda_1=\frac{2}{3}, \lambda_2=\frac{1}{3}$
    as defined by \eqref{eqn:decreasing_lag_weights}.

    As a  numerical example, we assign the following values to the initial and 
    the first-order transition parameters: 
    %
    \begin{subequations}\label{supp:eqn:numerical_parameter}
    \begin{align}
      \left[p_i\right]_{i=1, \ldots, 5} 
        &= \left(1/2,\, 1/2,\, 1/2,\, 1/2,\, 1/2\right) \,, 
        \label{supp:eqn:numerical_parameter_initial}
        \\
      \left[\hyperparam^{i}_{j}\right]_{i, j=1, \ldots, 5} 
        &= 
        \begin{blockarray}{cccccc} 
          & v_1 & v_2 & v_3 & v_4 & v_5 \\ 
          \begin{block}{c[ccccc]}
            v_1 & 0    & 1/2  & 1/2  &  4/5  & 0 \\
            v_2 & 0    & 0    & 1/2  &  1/2  & 0 \\
            v_3 & 1/2  & 0    & 0    &  0    & 0 \\
            v_4 & 0    & 0    & 1/2  &  0    & 0  \\
            v_5 & 0    & 1/2  & 0    &  0    & 0  \\
          \end{block}
        \end{blockarray} \,,
        \label{supp:eqn:numerical_parameter_transitions}
    \end{align}
    \end{subequations}
    resulting in the following initial and transition weights:
    \begin{subequations}
    \begin{align}
      \left[w_i\right]_{i=1, \ldots, 5} 
        &\stackrel{\eqref{supp:eqn:initial_and_transition_weights_evaluated_initial}}{=} 
          \left(1,\, 1,\, 1,\, 1,\, 1 \right) \,,  \\
      \left[w^{i}_{j}\right]_{i, j=1, \ldots, 5} 
      &\stackrel{\eqref{supp:eqn:initial_and_transition_weights_evaluated_transition}}{=} 
      \begin{blockarray}{cccccc} 
        & v_1 & v_2 & v_3 & v_4 & v_5 \\ 
        \begin{block}{c[ccccc]}
          v_1 & 0  & 1  & 1  &  4  & 0 \\
          v_2 & 0  & 0  & 1  &  1  & 0 \\
          v_3 & 1  & 0  & 0  &  0  & 0 \\
          v_4 & 0  & 0  & 1  &  0  & 0  \\
          v_5 & 0  & 1  & 0  &  0  & 0  \\
        \end{block}
      \end{blockarray} \,.
    \end{align}
    \end{subequations}

    The initial and transition probabilities are thus given by
    \begin{subequations}
      \begin{align}
      \left[\pi_i\right]_{i=1, \ldots, 5} 
      &\stackrel{\eqref{supp:eqn:initial_probabilities_evaluated}}{=} 
      \left(\frac{1}{5},\,  \frac{1}{5},\, \frac{1}{5},\, \frac{1}{5},\, \frac{1}{5}\right) \,, 
      \\
      \left[\pi^{i}_{j}\right]_{i, j=1, \ldots, 5} 
        &\stackrel{\eqref{supp:eqn:transition_probabilities_evaluated}}{=} 
        \begin{blockarray}{cccccc} 
          & v_1 & v_2 & v_3 & v_4 & v_5 \\ 
          \begin{block}{c[ccccc]}
            v_1 & 0  & 1/6  & 1/6  &  4/6  & 0 \\
            v_2 & 0  & 0  & 1/2  &  1/2  & 0 \\
            v_3 & 1  & 0  & 0  &  0  & 0 \\
            v_4 & 0  & 0  & 1  &  0  & 0  \\
            v_5 & 0  & 1  & 0  &  0  & 0  \\
          \end{block}
        \end{blockarray} \,.
    \end{align}
    \end{subequations}

    As expected, the initial probabilities add up to $1$, and each row of the transition probability matrix
    also sums up to $1$.
    Therefore, for any set of Bernoulli parameters 
    satisfying the box constraints $\pi_i\in[0, 1],\, \pi^{i}_j\in[0, 1],\, \forall i,j=1, \ldots, \Nsens$, 
    we can evaluate the initial or the transition 
    distribution/probabilities with automatic satisfaction of the probability distribution 
    constraints.

    In what follows we evaluate the probability distribution for each 
    of the policy models proposed in this work.

  \subsection{First-Order Policy}
    \label{app:subsec:example_first_order_model}
    We first discuss the generation of the support and then evaluate the probabilities for 
    each path in the support.
    
    \paragraph{Distribution support}
      Since the initial parameter \eqref{supp:eqn:numerical_parameter_initial} 
      does not include any zeros, the trajectory can start at any of the 
      $5$ nodes on the navigation mesh (\Cref{fig:illustrative_example_DMC}).
      We begin by listing all trajectories starting with $\design_1=v_1$.
      Since this is a first-order model, we use the first-order transitions 
      to define the next candidate nodes on the graph.
      Specifically, starting at $v_1$, the next 
      node---by inspecting the transition matrix or the navigation mesh---can 
      be $v_2$, $v_3$, or $v_4$. 
      All these are allowed
      because the conditional transition parameters 
      $\hyperparam^{1}_{2}, \hyperparam^{1}_{3}, \hyperparam^{1}_{4}$ are all nonzero.
      Nonzero parameters guarantee nonzero probabilities, and thus
      $\pi^{1}_{2}, \pi^{1}_{3}, \pi^{1}_{4}$ are all positive.

      Thus, all trajectories of length $n=2$ starting with $\design_1=v_1$ are 
      \begin{equation}
        (\design_1=v_1,\, \design_2) \in \{(v_1, v_2),\, (v_1, v_3),\, (v_1, v_4)\} \,, 
      \end{equation}
      which are then expanded recursively to generate all paths of length $3$ starting with $\design_1=v_1$.

      From $v_2$ a sensor is allowed to move to $v_3, v_4$. 
      Thus the trajectory $(v_1, v_2)$ can be expanded to two feasible trajectories 
      $\{(v_1, v_2, v_3),\, (v_1, v_2, v_4)\}$. 
      Similarly, moving from $v_3$ is allowed to $v_1$, and 
      $v_4$ is allowed to transition only to $v_3$.
      Adding these possible choices expands the trajectories of length $2$ starting with $\design_1=v_1$ 
      to the following trajectories of length $3$ starting with the same node $\design_1=v_1$:
      \begin{equation}
        (\design_1=v_1, \design_2, \design_3) \in 
          \{
            (v_1, v_2, v_3),\, 
            (v_1, v_2, v_4),\, 
            (v_1, v_3, v_1),\, 
            (v_1, v_4, v_3) 
          \} \,.
      \end{equation}

      By following the same strategy for the other starting points 
      $\design_1 \in \{v_2, v_3, v_4, v_5\}$, respectively,
      the distribution support consisting of $12$ feasible paths 
      $(\design_1, \design_2, \design_3)$ of length $n=3$ can be generated.

    \paragraph{Probability mass function}
      From \Cref{defn:first_order_path_model}, the 
      probability distribution (mass function) evaluated at all feasible trajectories of 
      length $3$ is given by 
      \begin{subequations}
      {\allowdisplaybreaks
        \begin{align}
          \Prob{\designvec\!=\!\left(v_1, v_2, v_3\right)} 
            &= \pi_1\, \pi^{1}_2\, \pi^{2}_3 
             = \frac{1}{60} \,, \\ 
          \Prob{\designvec\!=\!\left(v_1, v_2, v_4\right)} 
            &= \pi_1\, \pi^{1}_2\, \pi^{2}_4 
              = \frac{1}{60} \,,  \\
          \Prob{\designvec\!=\!\left(v_1, v_3, v_1\right)} 
           &= \pi_1\, \pi^{1}_3\, \pi^{3}_1 
             = \frac{2}{60} \,,  \\
          \Prob{\designvec\!=\!\left(v_1, v_4, v_3\right)} 
            &= \pi_1\, \pi^{1}_4\, \pi^{4}_3 
              = \frac{8}{60} \,,  \\ 
          \Prob{\designvec\!=\!\left(v_2, v_3, v_1\right)} 
            &= \pi_2\, \pi^{2}_3\, \pi^{3}_1 
              = \frac{6}{60} \,,  \\
          \Prob{\designvec\!=\!\left(v_2, v_4, v_3\right)} 
            &= \pi_2\, \pi^{2}_4\, \pi^{4}_3 
              = \frac{6}{60} \,,  \\
          \Prob{\designvec\!=\!\left(v_3, v_1, v_2\right)} 
            &= \pi_3\, \pi^{3}_1\, \pi^{1}_2 
              = \frac{2}{60} \,,  \\
          \Prob{\designvec\!=\!\left(v_3, v_1, v_3\right)} 
            &= \pi_3\, \pi^{3}_1\, \pi^{1}_3 
              = \frac{2}{60} \,,  \\
          \Prob{\designvec\!=\!\left(v_3, v_1, v_4\right)} 
            &= \pi_3\, \pi^{3}_1\, \pi^{1}_4 
              = \frac{8}{60} \,,  \\
          \Prob{\designvec\!=\!\left(v_4, v_3, v_1\right)} 
            &= \pi_4\, \pi^{4}_3\, \pi^{3}_1 
              = \frac{12}{60} \,,  \\
          \Prob{\designvec\!=\!\left(v_5, v_2, v_3\right)} 
            &= \pi_5\, \pi^{5}_2\, \pi^{2}_3 
              = \frac{6}{60} \,,  \\
          \Prob{\designvec\!=\!\left(v_5, v_2, v_4\right)} 
            &= \pi_5\, \pi^{5}_2\, \pi^{2}_4 
              = \frac{6}{60} \,,  
        \end{align}
      }
      \end{subequations}
      where path probabilities satisfy $\Prob{\designvec[k]}\in[0, 1]$ 
      and $\sum_{k=1}^{12}{\Prob{\designvec[k]}}=1$ for an index $k=1, \ldots, 12$ uniquely associated 
      with the different paths in the support.

      Note that the support (all feasible trajectories with nonzero probabilities) involves only 
      consecutive nodes with nonzero transition probabilities.
      These feasible trajectories are generated by following the connectivity defined by the 
      navigation mesh in \Cref{fig:illustrative_example_DMC}.
      Alternatively, one could enumerate all possible permutations of $3$ nodes out of the $5$ mesh 
      nodes, resulting in $\Perm{5}{3}=60$ paths, of which only $12$ belong to the policy support.
      For example, the probability of $\design=\left(v_3, v_2, v_4\right)$ 
      is given by $\pi_{3} \pi^{3}_{2} \pi^{2}_{4} = \frac{1}{2}\times 0 \times \frac{1}{2}=0$.
      This of course leads to unneeded waste of computational resources if followed in practical 
      implementations of the policy, for example for sampling. 

  \subsection{Higher-Order Policy}
    \label{app:subsec:example_higher_order_model}
    The higher-order policy given by \Cref{defn:higher_order_path_model} employs 
    the conditional transition probability \eqref{eqn:Raftery_model} 
    to determine which nodes a moving sensor is allowed to transition to.
    Thus, given a partial path $(\design_{n-k}, \ldots, \design_{n-1})$, 
    the conditional probability of $\design_{n}$ is nonzero if the first-order
    transition probability $\CondProb{\design_{n}}{\design_{n-i}}$ 
    is nonzero for any $i=1, \ldots, k$, where $k$ is the predefined policy order.
    Note that we assume all lag weights $\lambda_i,\, i=1, \ldots, k$ 
    are nonzero.
   
    \paragraph{Distribution support}
      For clarity, we repeat the same exercise in 
      \Cref{app:subsec:example_higher_order_model} for support generation. 
      Specifically,
      we generate all paths in the support with length $n=3$ starting at $\design_1=v_1$.
      With a partial path containing only one node $(\design_1=v_1)$ with policy order $k=2$, 
      only the first-order transition probabilities are needed to calculate the 
      conditional transitions probabilities.
      Thus, similar to the first-order model, the trajectories of length $n=2$ 
      starting with $\design_1=v_1$ here 
      are
      \begin{equation}
        (\design_1=v_1, \design_2) \in \{ (v_1, v_2),\, (v_1, v_3),\, (v_1, v_4) \} \,.
      \end{equation}

      Now, the third node on the path depends on the last $k=2$ nodes on the path.
      If the first-order transition probabilities $\pi^{i}_{j}$ from any of those nodes to 
      a given node $\design_{n}$ are nonzero, the node $\design_{n}$ can appear next on the path.
      Let us consider the partial path $(v_1, v_2)$. 
      Since $v_2$ can transition to $v_3$ and $v_4$ can transition 
      to any of $v_2, v_3, v_4$, the next node on the path can be 
      any of $v_2, v_3, v_4$. 
      This results in the following possible paths:
      \begin{equation}\label{supp:eqn:higher_order_path_expansion_1}
        (\design_1=v_1, \design_2=v_2, \design_3)
        \in \{
          (v_1, v_2, v_2),\, 
          (v_1, v_2, v_3),\, 
          (v_1, v_2, v_4)
        \}
        \,.
      \end{equation}

      Note that these paths allow transitioning from $v_2$ to $v_1$ even though 
      the navigation mesh does not contain an arc 
      $(v_2, v_1)$. 
      This shows that the proposed higher-order models allow jumps 
      between nodes with joining paths of length up to the policy order $k$. 
      Similarly, $(v_1, v_3)$ expands to 
      \begin{equation}\label{supp:eqn:higher_order_path_expansion_2}
        (\design_1=v_1, \design_2=v_3, \design_3)
        \in
        \{
          (v_1, v_3, v_1),\, 
          (v_1, v_3, v_2),\, 
          (v_1, v_3, v_3),\, 
          (v_1, v_3, v_4) 
        \} \,
      \end{equation}
      and $(v_1, v_4)$ expands to
      \begin{equation}\label{supp:eqn:higher_order_path_expansion_3}
        (\design_1=v_1, \design_2=v_4, \design_3)
        \{
          (v_1, v_4, v_2),\, 
          (v_1, v_4, v_3),\, 
          (v_1, v_4, v_4)
        \} \,.
      \end{equation}

      By combining 
      \eqref{supp:eqn:higher_order_path_expansion_1},
      \eqref{supp:eqn:higher_order_path_expansion_2}, and 
      \eqref{supp:eqn:higher_order_path_expansion_3} we obtain all paths of length $n=3$ 
      in the support
      starting from $\design_1=v_1$. 
      By following the same strategy with other starting points, 
      the support consisting of $24$ trajectories can be constructed.
    
    \paragraph{Probability computations}
      We show an example of the computations of the path probabilities and leave the rest as 
      an exercise to the reader, since they follow similarly:
      \begin{equation}
        \begin{aligned}
          \Prob{\designvec=\left(v_1, v_2, v_2\right)} 
            &= \Prob{v_1} \,\, \CondProb{v_2}{v_1}  \,\, \CondProb{v_2}{v_1, v_2} \\
            &\stackrel{(\ref{eqn:higher_order_trajectory_initial_state}, 
              \ref{eqn:higher_order_trajectory_distribution_transition})}{=} 
              \pi_{1}\, \, \pi^{1}_{2} \,\,
              \CondProb{\design_{n}=v_2}{\design_{n-1}=v_2, \design_{n-2}=v_1} \\
            &\stackrel{\eqref{eqn:Raftery_model}}{=} 
              \pi_{1}\,\, \pi^{1}_{2} \,\, 
              \left(
                \lambda_1 \CondProb{v_2}{v_2} + 
                \lambda_2 \CondProb{v_2}{v_1} 
              \right) \\
            &\stackrel{\eqref{eqn:Raftery_model}}{=} 
              \pi_{1}\,\, \pi^{1}_{2} \,\, 
              \left(
                \lambda_1 \, \pi^{2}_{2} + 
                \lambda_2 \,  \pi^{1}_{2} 
              \right) \\ 
            &= \frac{1}{5} \times \frac{1}{6} 
              \left(\frac{2}{3} \times 0 + \frac{1}{3} \times \frac{1}{6}\right)
            = \frac{1}{540} \,. 
        \end{aligned}
      \end{equation}
      %

  \subsection{Generalized Higher-Order Policy}
    \label{app:subsec:example_generalized_higher_order_model}
    %
    \paragraph{Distribution support}
      Here the support consists of $19$ trajectories with nonzero probabilities.
      We follow the same strategy as in \Cref{app:subsec:example_generalized_higher_order_model}.
      Expanding a trajectory starting with $\design_1=v_1$ to a path of the same length here 
      is the same as in \Cref{app:subsec:example_higher_order_model}.
      Thus we start by discussing the next node allowed for a partial path $(v_1, v_2)$ under 
      the generalized higher-order policy given by \Cref{defn:generalized_higher_order_path_model}.

      The conditional transition probability \eqref{eqn:generalized_Raftery_model} 
      is used to determine the next nodes allowed on the path starting with 
      $(\design_1=v_1, \design_2=v_2)$.
      Starting from $v_2$, one-step transition is allowed to either $v_3$ or $v_4$.
      Starting from $v_1$, two-step transition (path consisting of two edges) is allowed to 
      $v_1$ (via $v_3$). 
      Thus, given the partial path $(\design_1=v_1, \design_2=v_2)$, 
      the conditional transition probability is  nonzero only for $v_1, v_3, v_4$.
      Thus the following paths belong to the distribution support:
      \begin{equation}
        (v_1, v_2, v_1),\, 
        (v_1, v_2, v_3),\, 
        (v_1, v_2, v_4)
        \,.
      \end{equation}
      
      We note that the $m$-step transition probabilities can be obtained, for example, by
      multiplying the first-step transition probability matrix by itself $m$ times. 
      Thus the two-step transition probability matrix is given by
      \begin{equation}
        \left[ \nCondProb{2}{v_j}{v_i} \right]_{i=1, \ldots, 5}
        %
          =
          \begin{blockarray}{cccccc} 
            & v_1 & v_2 & v_3 & v_4 & v_5 \\ 
            \begin{block}{c[ccccc]}
              v_1 & 2/12  & 0  & 9/12  &  1/12  & 0 \\
              v_2 & 1/2   & 0  & 1/2  &  0  & 0 \\
              v_3 & 0     & 2/12  & 2/12  &  8/12  & 0 \\
              v_4 & 1  & 0  & 0  &  0  & 0  \\
              v_5 & 0  & 0  & 1/2  &  1/2  & 0  \\
            \end{block}
          \end{blockarray} \,,
      \end{equation}
      which can be used to enumerate possible 2-step transitions.
      For example, from the third row we know that there are paths with two edges joining 
      $v_3$ to and only to each of $v_2, v_3, v_4$.

    \paragraph{Probability computations}
      An example of probability computations is given for the path $(v_1, v_2, v_1)$ as follows:
      \begin{equation}
        \begin{aligned}
          \Prob{\designvec=\left(v_1, v_2, v_1\right)} 
            &= \Prob{v_1} \,\, \CondProb{v_2}{v_1} \,\,  \CondProb{v_1}{v_1, v_2} \\
            &\stackrel{(\ref{eqn:generalized_higher_order_trajectory_initial_state}, 
              \ref{eqn:generalized_higher_order_trajectory_distribution_transition})}{=} 
              \pi_{1}\,\,  \pi^{1}_{2} \,\,  
              \CondProb{\design_{n}=v_1}{\design_{n-1}=v_2, \design_{n-2}=v_1} \\
            &\stackrel{\eqref{eqn:Raftery_model}}{=} 
              \pi_{1}\,\,  \pi^{1}_{2} \,\,  
              \left(
                \lambda_1 \CondProb{v_1}{v_2} + 
                \lambda_2 \nCondProb{2}{v_1}{v_1} 
              \right) \\
            &\stackrel{(\ref{eqn:generalized_Raftery_model}, \ref{eqn:generalized_higher_order_transition})}{=} 
            \frac{1}{5} \times \frac{1}{6} 
              \left(\frac{2}{3} \times 0 + \frac{1}{3} \times \frac{2}{12}\right)
            = \frac{1}{540} \,. 
        \end{aligned}
      \end{equation}

      The rest of the support---consisting of $19$ possible paths---and the full 
      probability distribution under this model 
      can be constructed by following the same steps above.

      Note that unlike the case of \Cref{defn:higher_order_path_model}, here
      the trajectory $(v_1, v_2, v_2)$ does not belong to the support of the generalized 
      higher-order policy \Cref{defn:generalized_higher_order_path_model}.
      Specifically, $\Prob{\designvec=(v_1, v_2, v_2)}=0$ because both $\CondProb{v_2}{v_2}=0$ 
      and $\nCondProb{2}{v_2}{v_1}=0$, yielding the conditional transition probability
      $\CondProb{v_2}{v_1, v_2}=0$.

\section{Additional Numerical Results}
  \label{sup:sec:numerical_results}
  Here we provide numerical experiments additional to those presented in 
  \Cref{sec:numerical_experiments}.
  Specifically, in \Cref{sup:subsubsec:fine_mesh_results_unspecified_start_point_7_sensors} 
  we complement the results presented in 
  \Cref{subsubsec:fine_mesh_results_unspecified_start_point} 
  with results obtained by employing the fine navigation mesh 
  \Cref{fig:navigation_meshes} (right) 
  with $s=7$ moving sensors, where the starting point of the trajectory is unknown.
  \Cref{sup:subsec:fine_mesh_results_fixed_start_point} presents results 
  for an experiment in which the starting point of a trajectory is fixed.
  %
  Finally, in \Cref{sup:subsec:Other_Utility_Functions} we show results 
  for experiments carried out with A- and E-optimality 
  OED criteria, respectively, used to define the utility/objective function.

  \subsection{Results with Fine Navigation Mesh and $7$ Moving Sensors}
  \label{sup:subsubsec:fine_mesh_results_unspecified_start_point_7_sensors}
    This section utilizes the same setup used in  
    \Cref{subsubsec:fine_mesh_results_unspecified_start_point} and replaces 
    the one moving sensor with a group of $s=7$ moving sensors.

    \Cref{sup:fig:fine_first_order_unspecified_start_point_7_sensors} 
    shows the results of \Cref{alg:probabilistic_path_optimization} with the first-order 
    policy defined by \Cref{defn:first_order_path_model}.
    \Cref{sup:fig:fine_higher_order_unspecified_start_point_7_sensors} 
    shows the performance of the optimization procedure with the higher-order 
    policy given by \Cref{defn:higher_order_path_model} with order $k=3$ and $k=5$, respectively.
    The resulting optimal initial parameter (first row) and optimal trajectories (second row)
     under this policy are shown in 
    \Cref{sup:fig:fine_higher_order_unspecified_start_point_order_3_initial_param_and_traject_7_sensors}.
    Similarly, results obtained by employing the generalized higher-order policy given by  
    \Cref{defn:generalized_higher_order_path_model} 
    are shown in \Cref{sup:fig:fine_generalized_higher_order_unspecified_start_point_7_sensors} 
    and 
    \Cref{sup:fig:fine_generalized_higher_order_unspecified_start_point_order_3_initial_param_and_traject_7_sensors}.

    These results support the results in \Cref{subsubsec:fine_mesh_results_unspecified_start_point} and 
    together show that the proposed approach behaves similarly for different numbers of moving sensors 
    $s\geq 1$. 
    With more sensors, however, the optimization landscape becomes more challenging, 
    with many more designs close to the global optima because of the overlap of the regions containing the moving sensors.
    This situation is demonstrated by the behavior of the optimization procedure over the early iterations.

    \begin{figure}[H]
      \centering
      \includegraphics[width=0.53\linewidth]{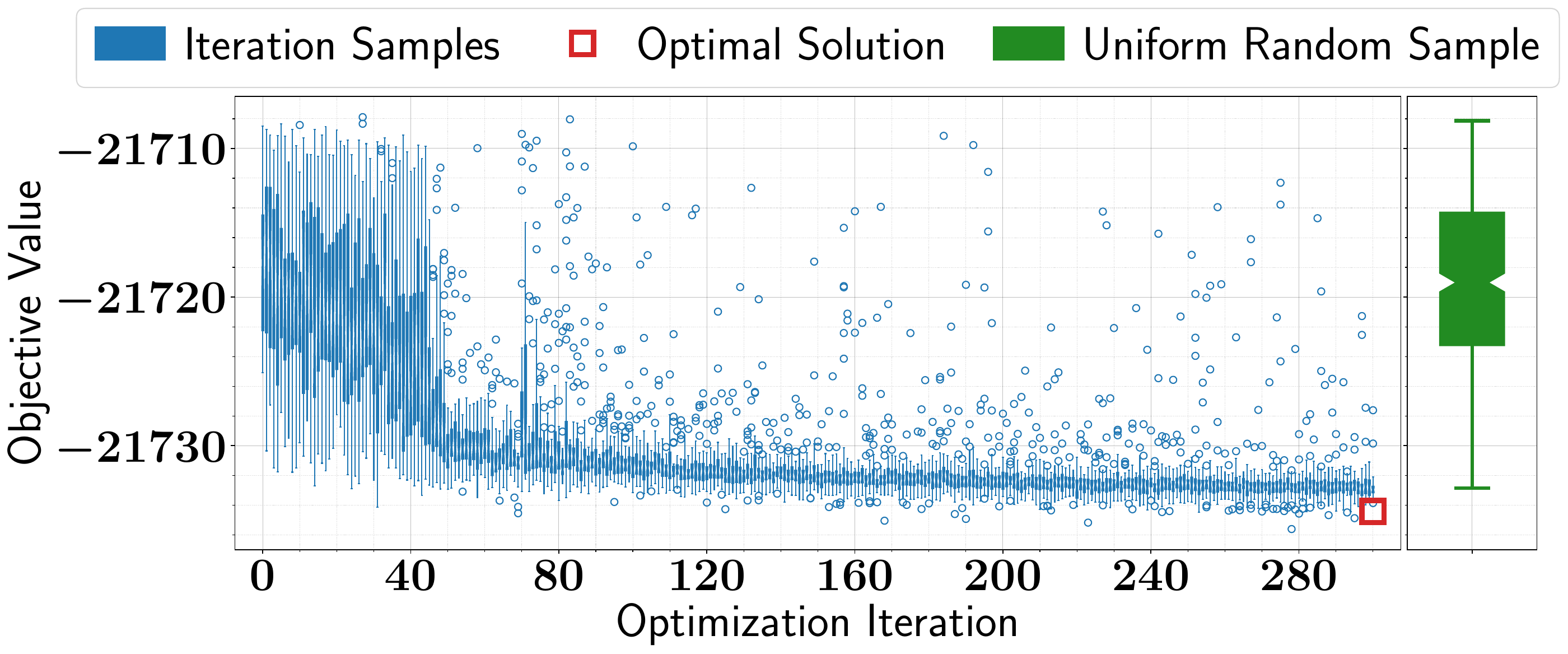}
      \includegraphics[width=0.215\linewidth]{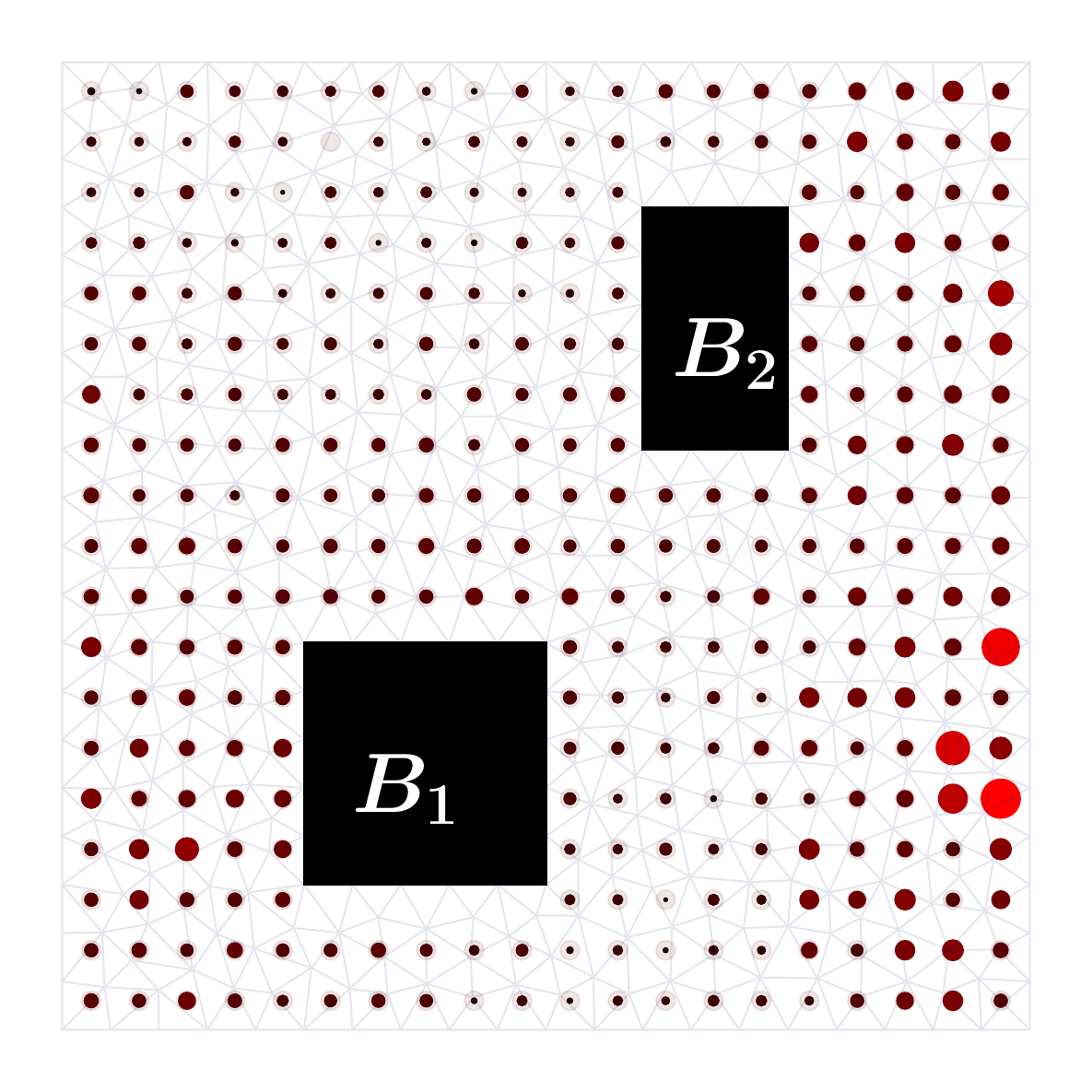}
      \includegraphics[width=0.235\linewidth]{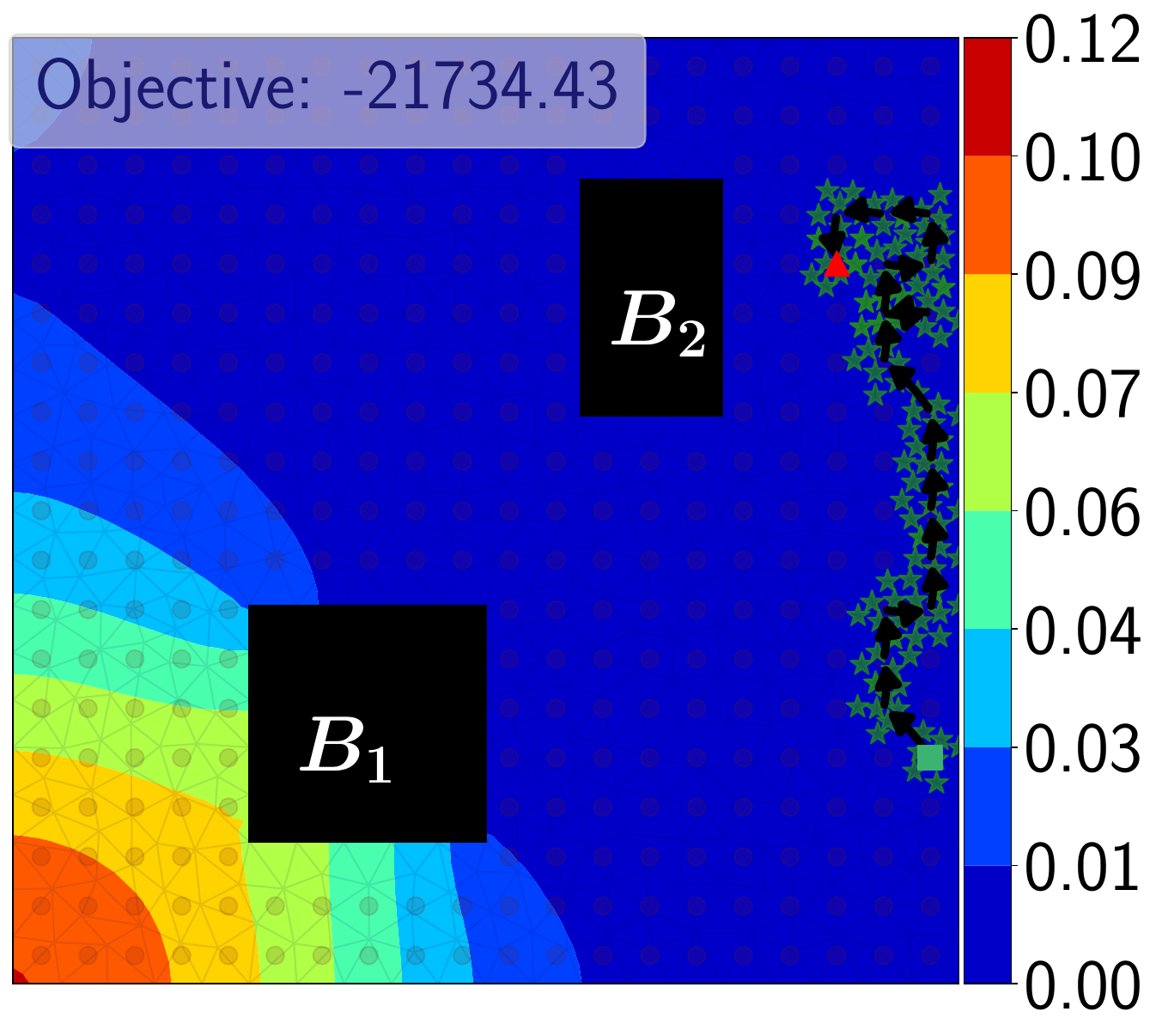}
      \caption{
        Results of \Cref{alg:probabilistic_path_optimization} with the first-order 
        policy (\Cref{defn:first_order_path_model}) applied to 
        the fine navigation mesh (\Cref{fig:navigation_meshes}, right) 
        with trajectory length of $n=19$ and a group of $s=7$ moving sensors.
      }\label{sup:fig:fine_first_order_unspecified_start_point_7_sensors}
    \end{figure}
    \begin{figure}[H]
      \centering
      \includegraphics[width=0.495\linewidth]{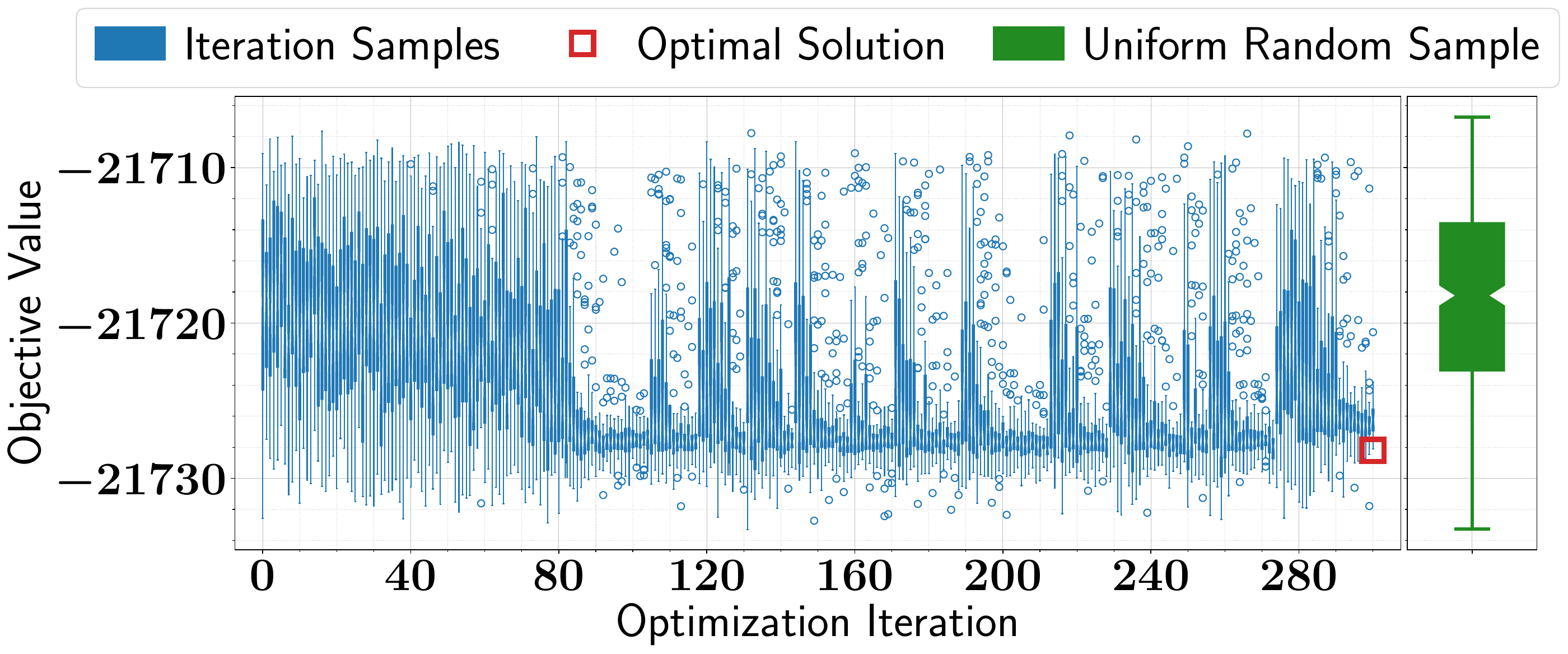}
      \includegraphics[width=0.495\linewidth]{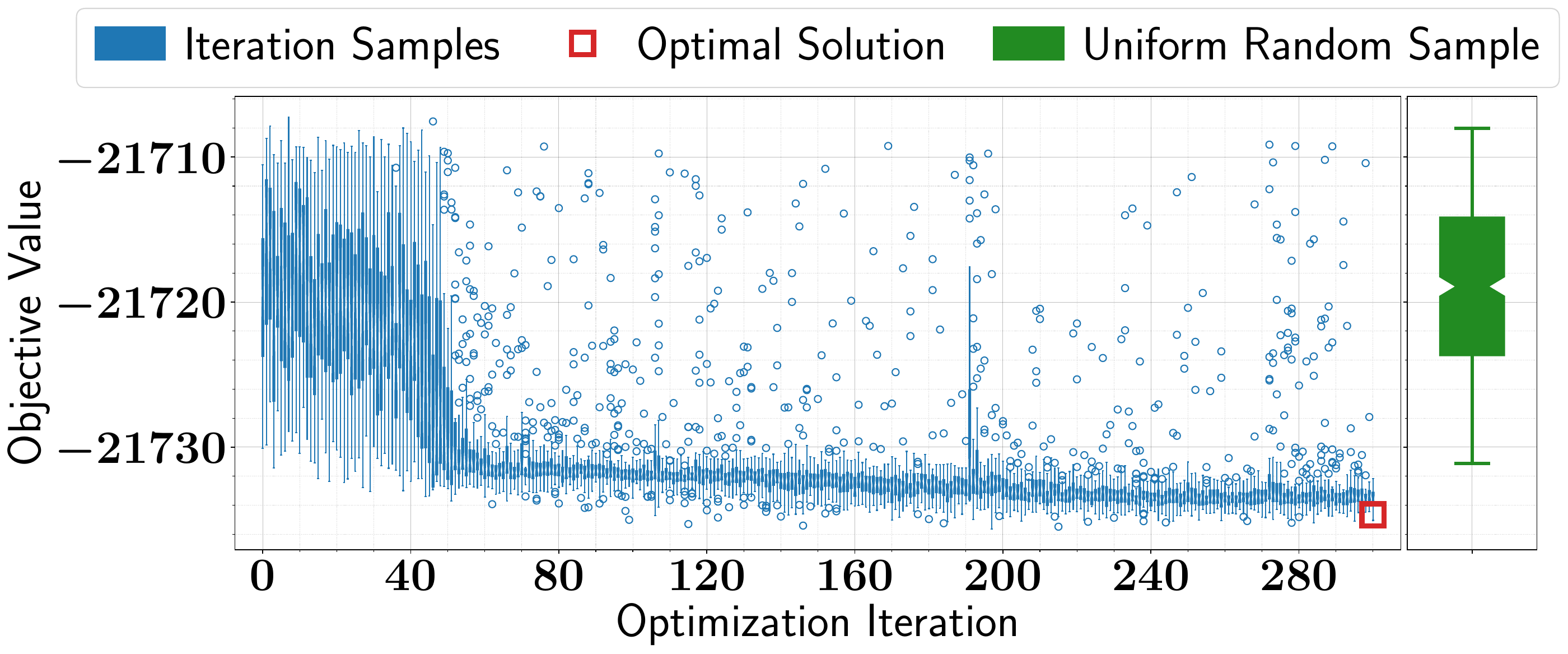}
      \includegraphics[width=0.495\linewidth]{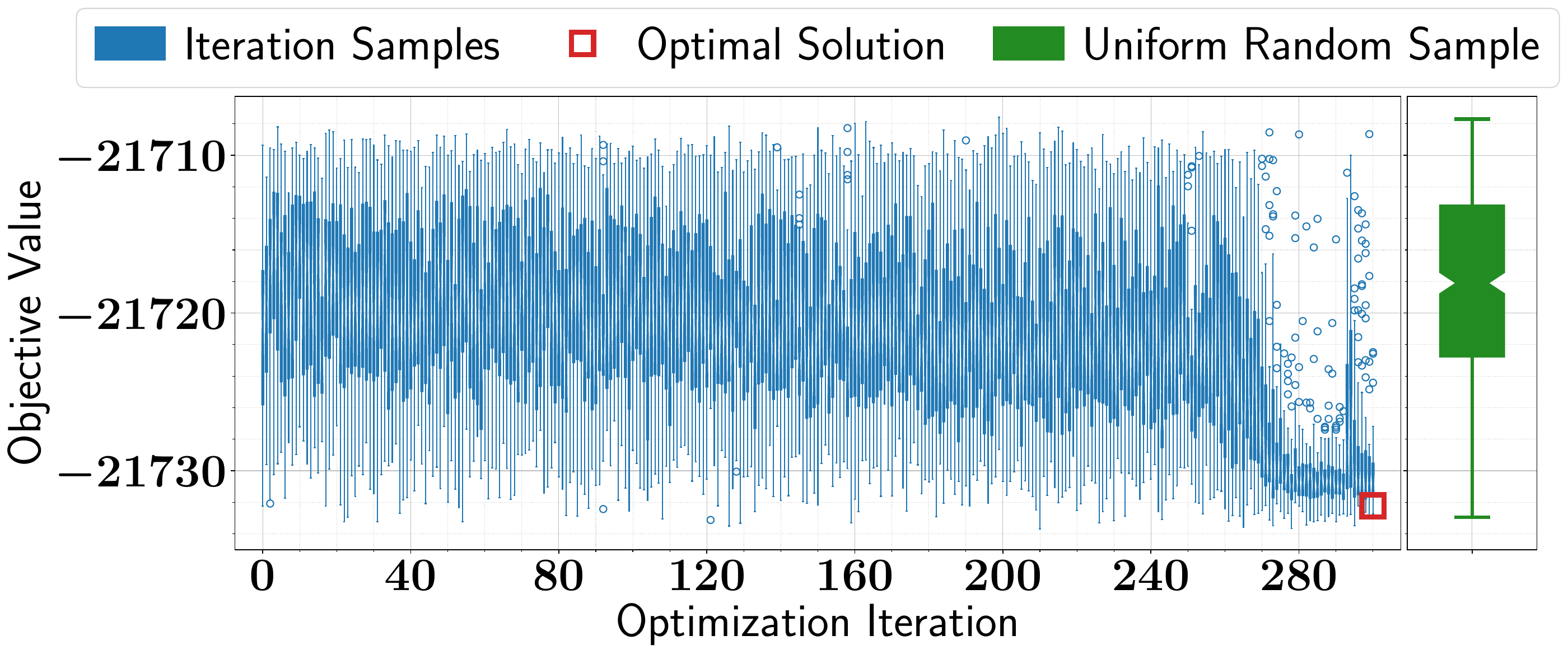}
      \includegraphics[width=0.495\linewidth]{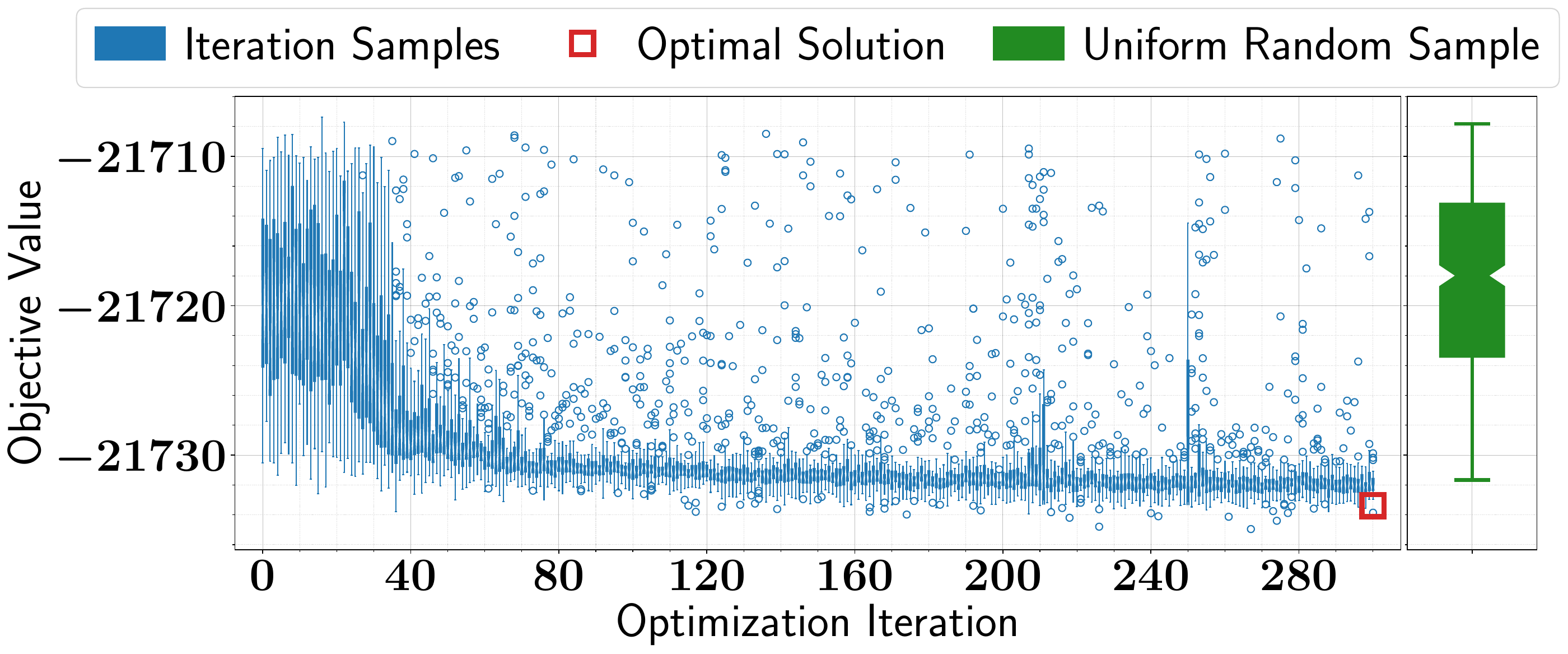}
      \caption{
        Results of \Cref{alg:probabilistic_path_optimization} with the 
        higher-order 
        policy (\Cref{defn:higher_order_path_model}) applied to 
        the fine navigation mesh (\Cref{fig:navigation_meshes}, right) 
        with trajectory length of $n=19$ and $s=7$ moving sensors.
        Results are shown for policy order $k=3$ (first row) and $k=5$ (second row). 
        The first column shows results with lag weights being calibrated by the 
        optimization procedure, and 
        the second column shows results with lag weights modeled 
        by \eqref{eqn:decreasing_lag_weights}.
      }\label{sup:fig:fine_higher_order_unspecified_start_point_7_sensors}
    \end{figure}
    \begin{figure}[H]
      \begin{subfigure}[t]{0.24\linewidth}
        \includegraphics[width=0.91\linewidth]{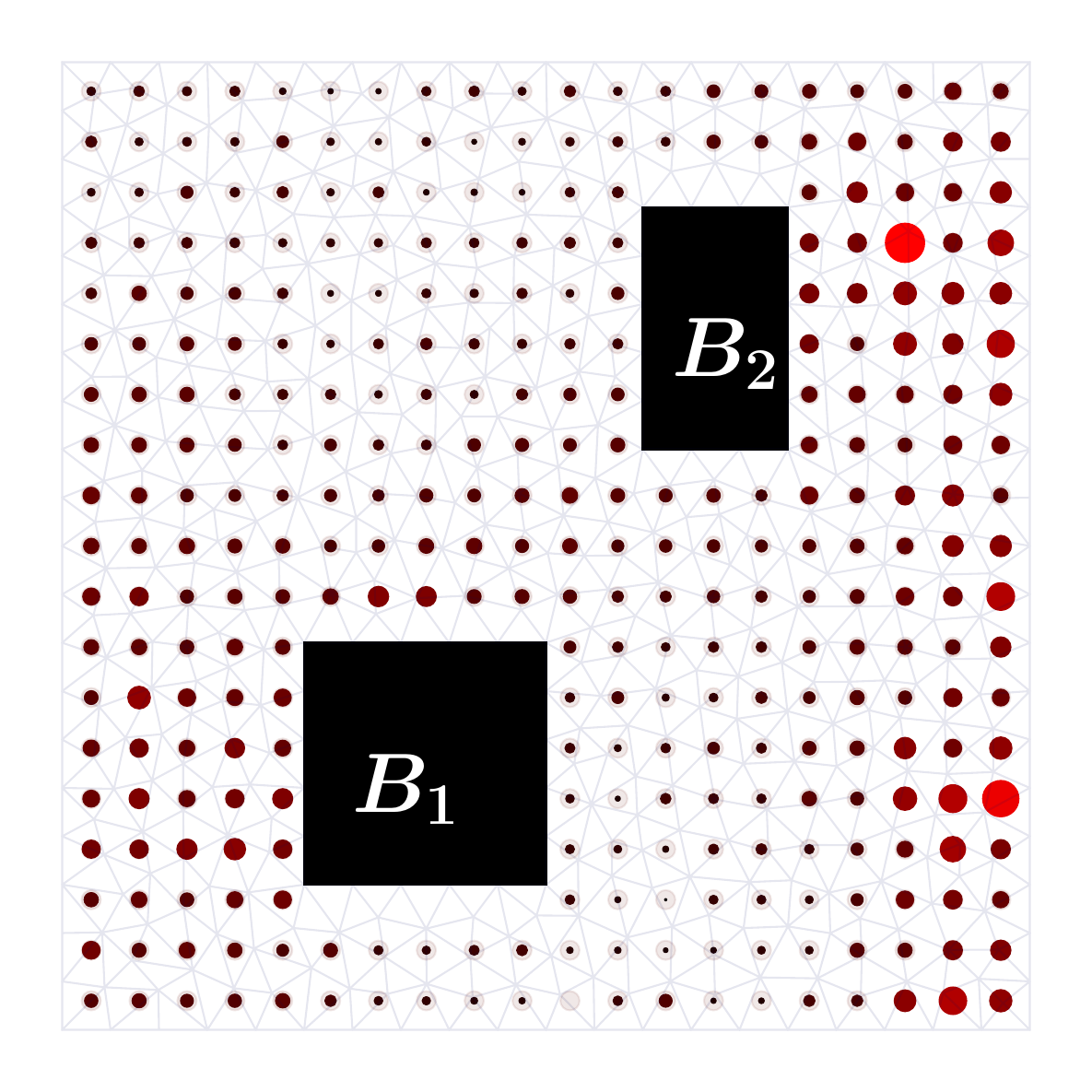}
      \end{subfigure}%
      \begin{subfigure}[t]{0.24\linewidth}
        \includegraphics[width=0.91\linewidth]{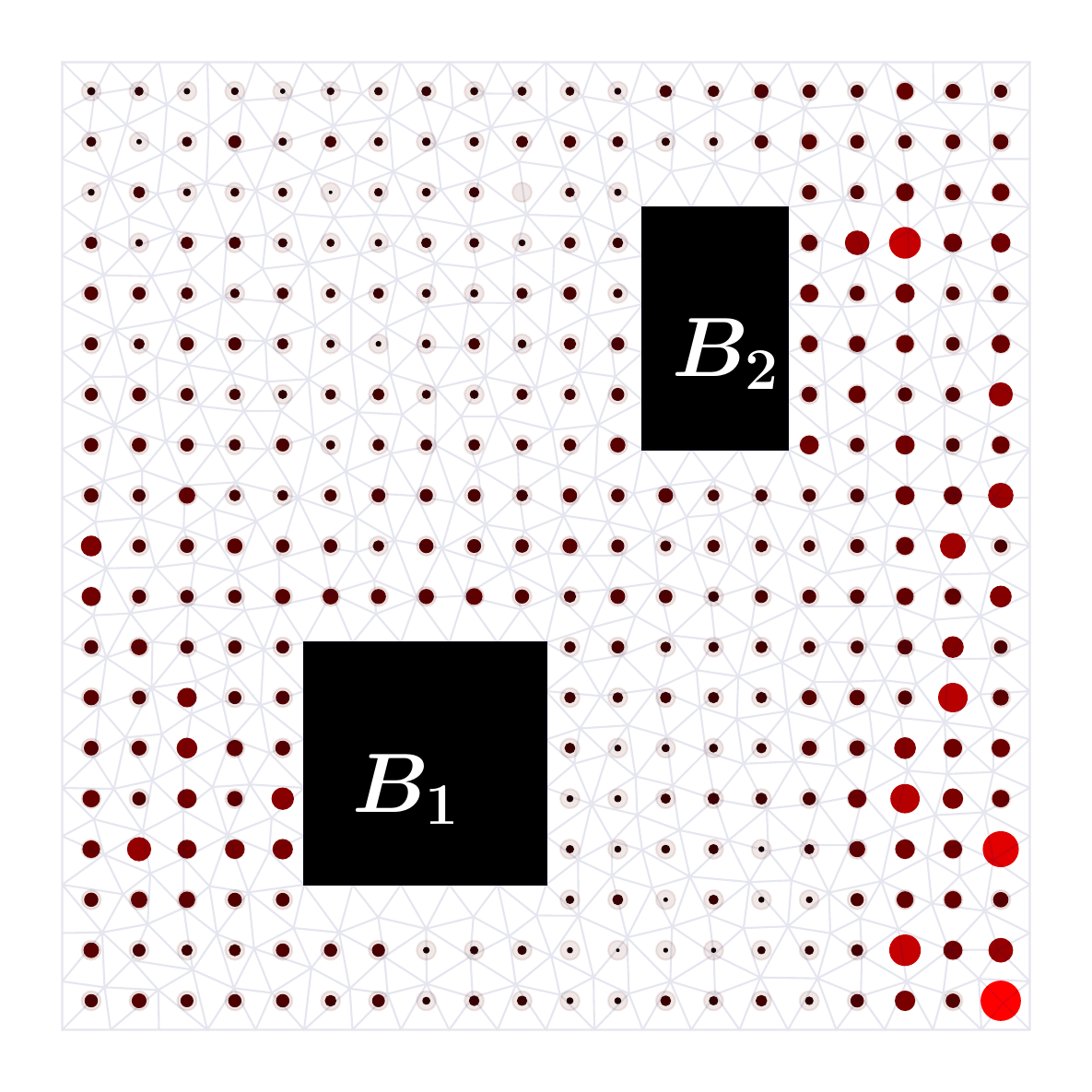}
      \end{subfigure}%
      \begin{subfigure}[t]{0.24\linewidth}
        \includegraphics[width=0.91\linewidth]{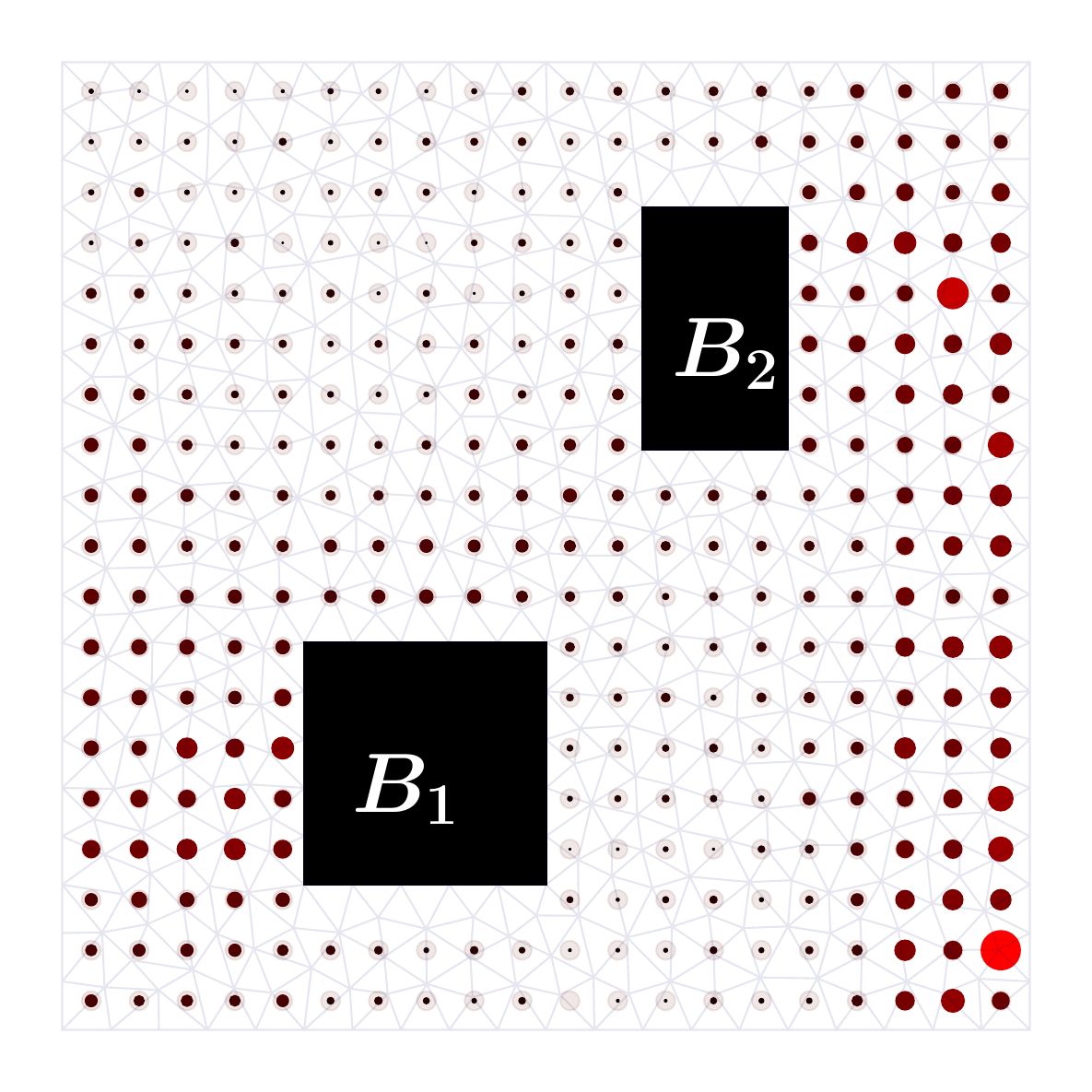}
      \end{subfigure}%
      \begin{subfigure}[t]{0.24\linewidth}
        \includegraphics[width=0.91\linewidth]{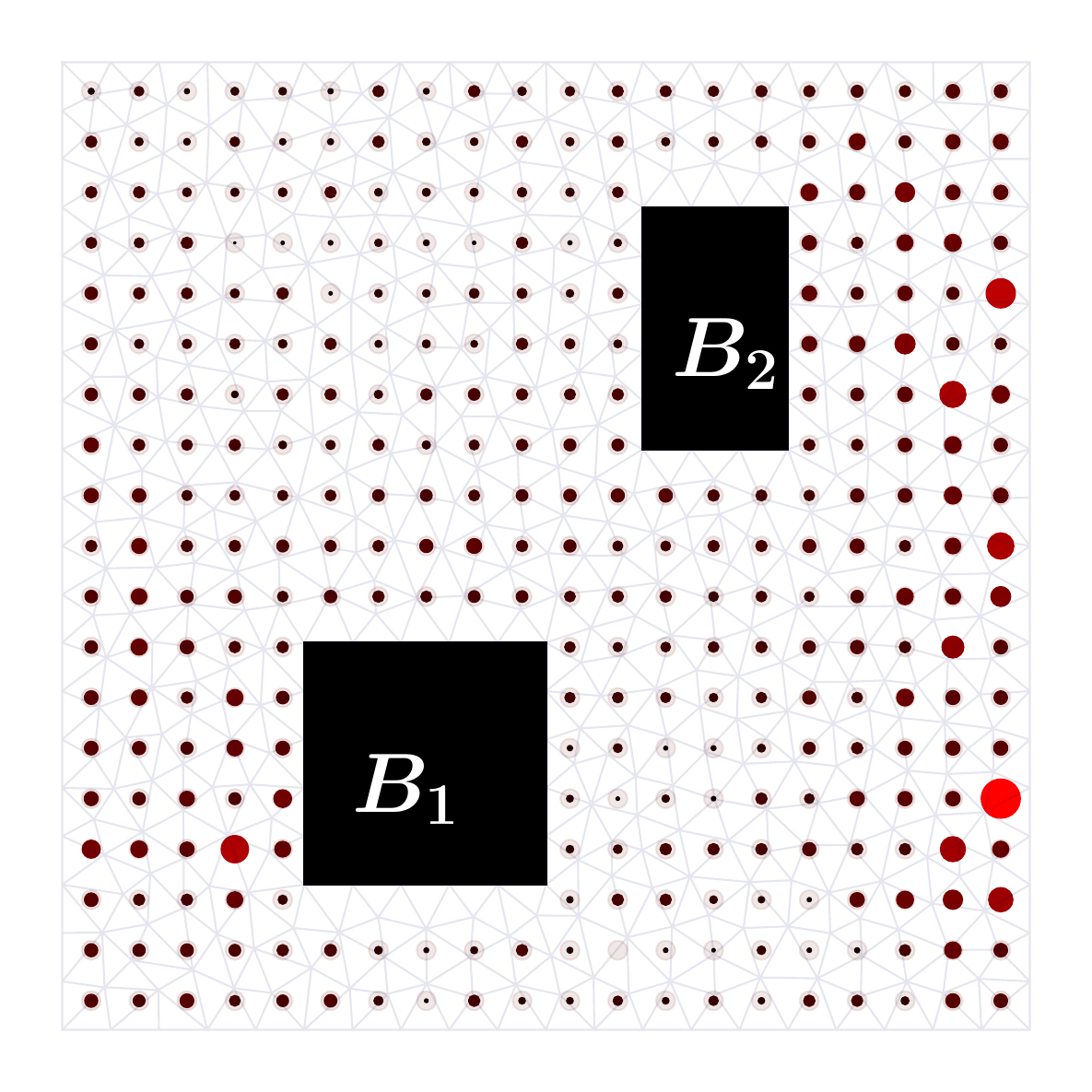}
      \end{subfigure}

      \begin{subfigure}[t]{0.245\linewidth}
        \includegraphics[width=\linewidth]{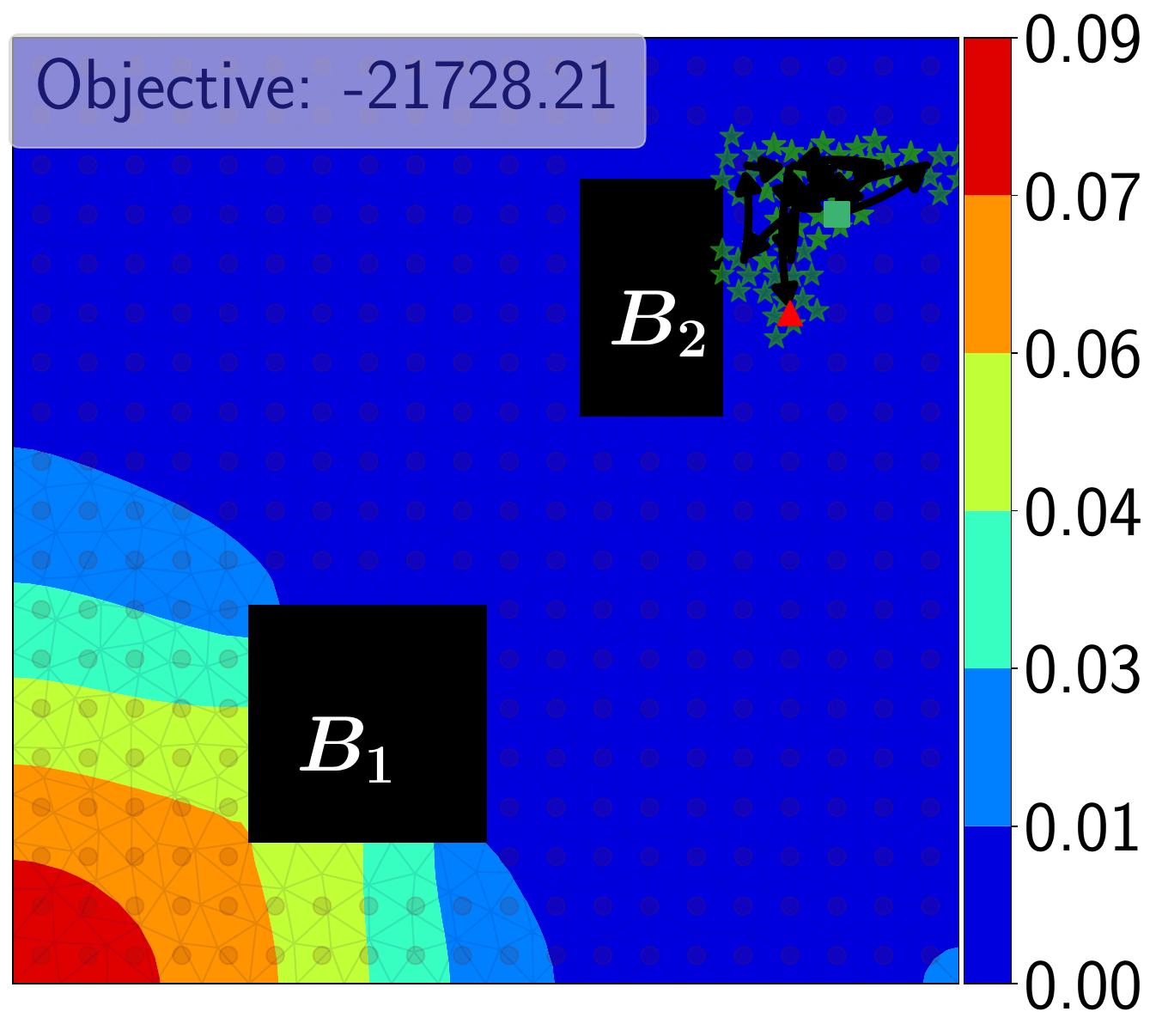}
      \end{subfigure}%
      \begin{subfigure}[t]{0.245\linewidth}
        \includegraphics[width=\linewidth]{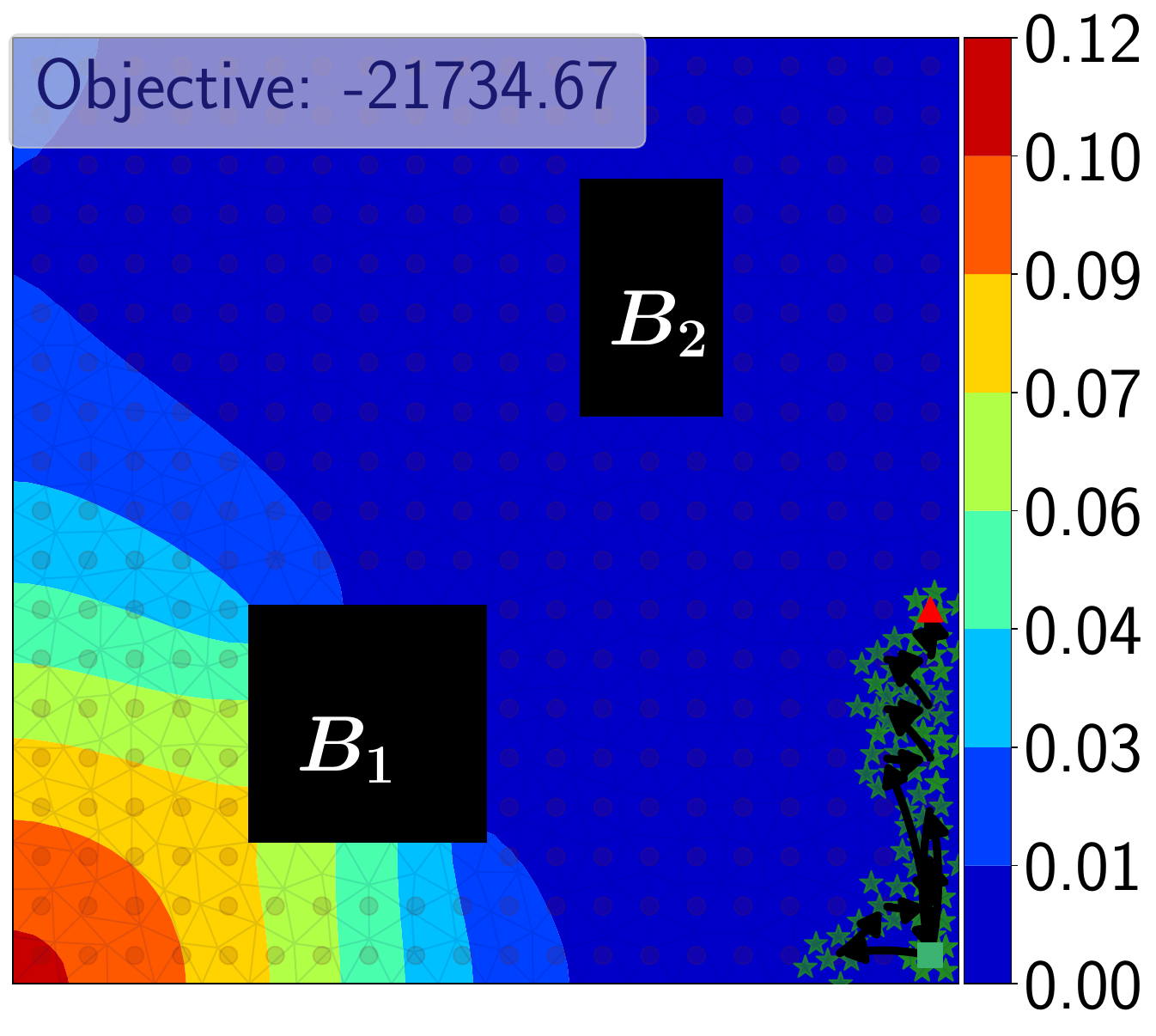}
      \end{subfigure}%
      \begin{subfigure}[t]{0.245\linewidth}
        \includegraphics[width=\linewidth]{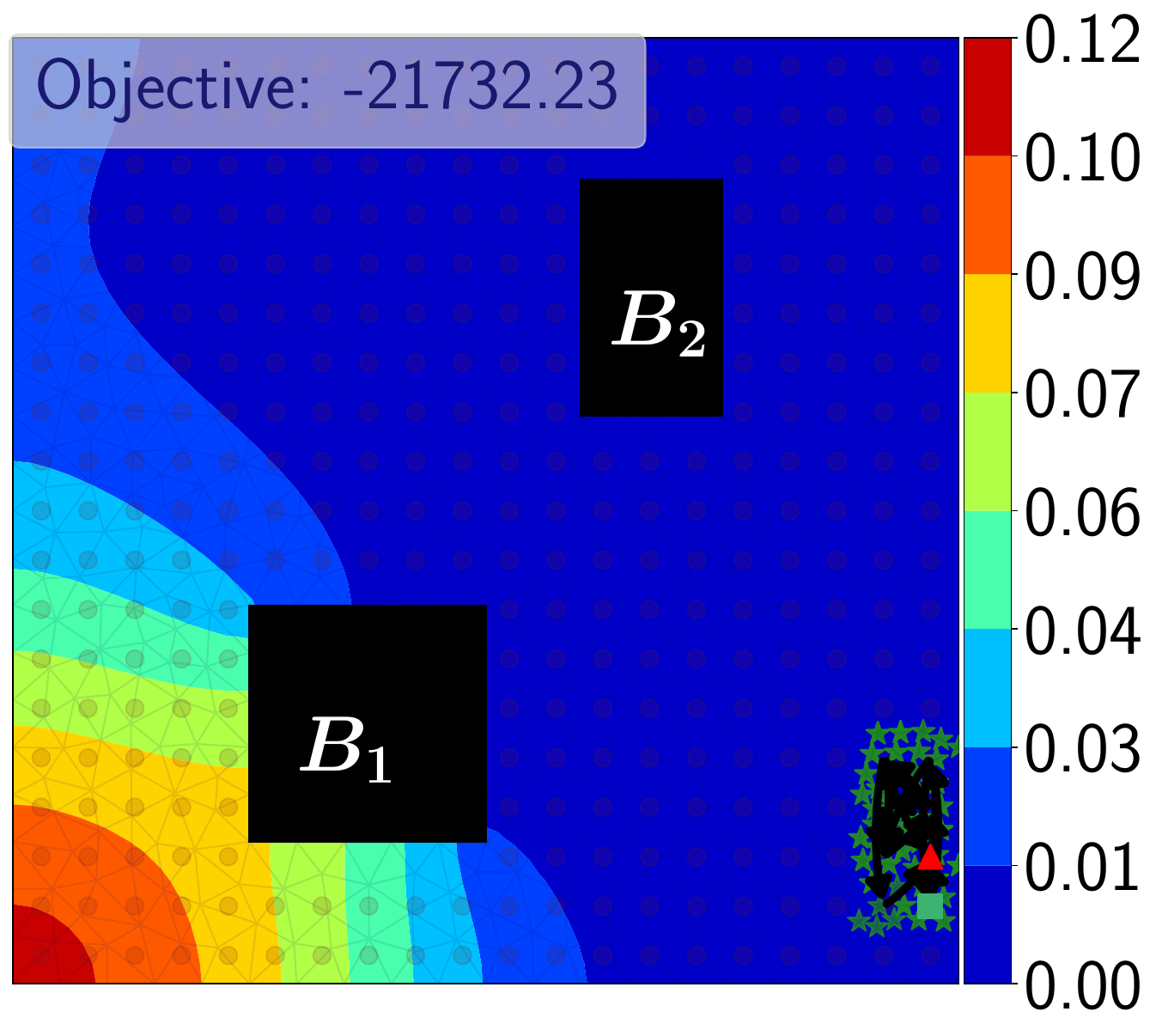}
      \end{subfigure}%
      \begin{subfigure}[t]{0.245\linewidth}
        \includegraphics[width=\linewidth]{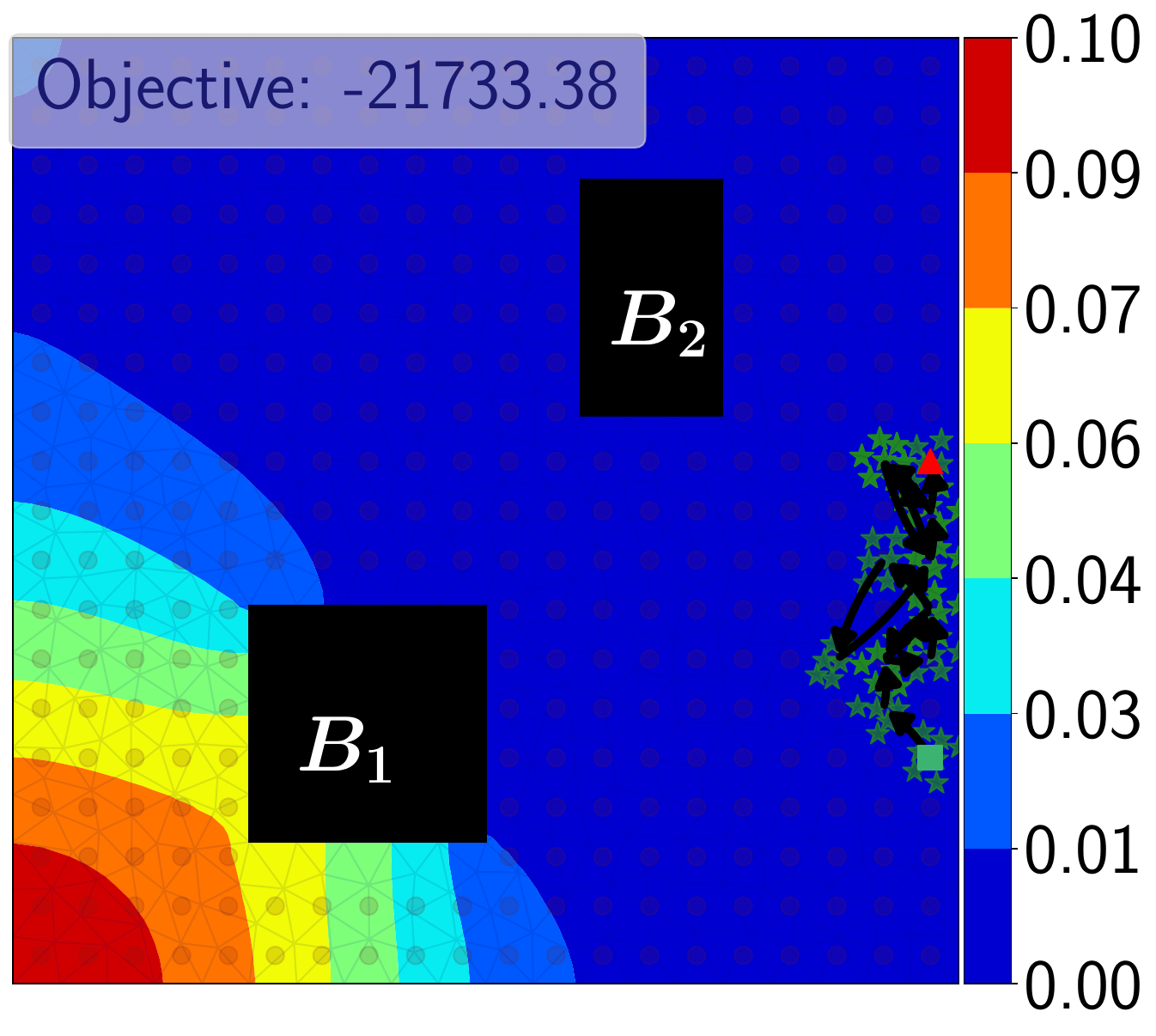}
      \end{subfigure}
      \caption{
        Optimal initial distribution parameters (top row) 
        and optimal trajectories (bottom row) 
        corresponding to \Cref{sup:fig:fine_higher_order_unspecified_start_point_7_sensors}. 
        In each row the first two panels match the first row of that figure, 
        followed by panels corresponding to its second row.
      }\label{sup:fig:fine_higher_order_unspecified_start_point_order_3_initial_param_and_traject_7_sensors}
    \end{figure}
    \begin{figure}[H]
      \centering
      \includegraphics[width=0.495\linewidth]{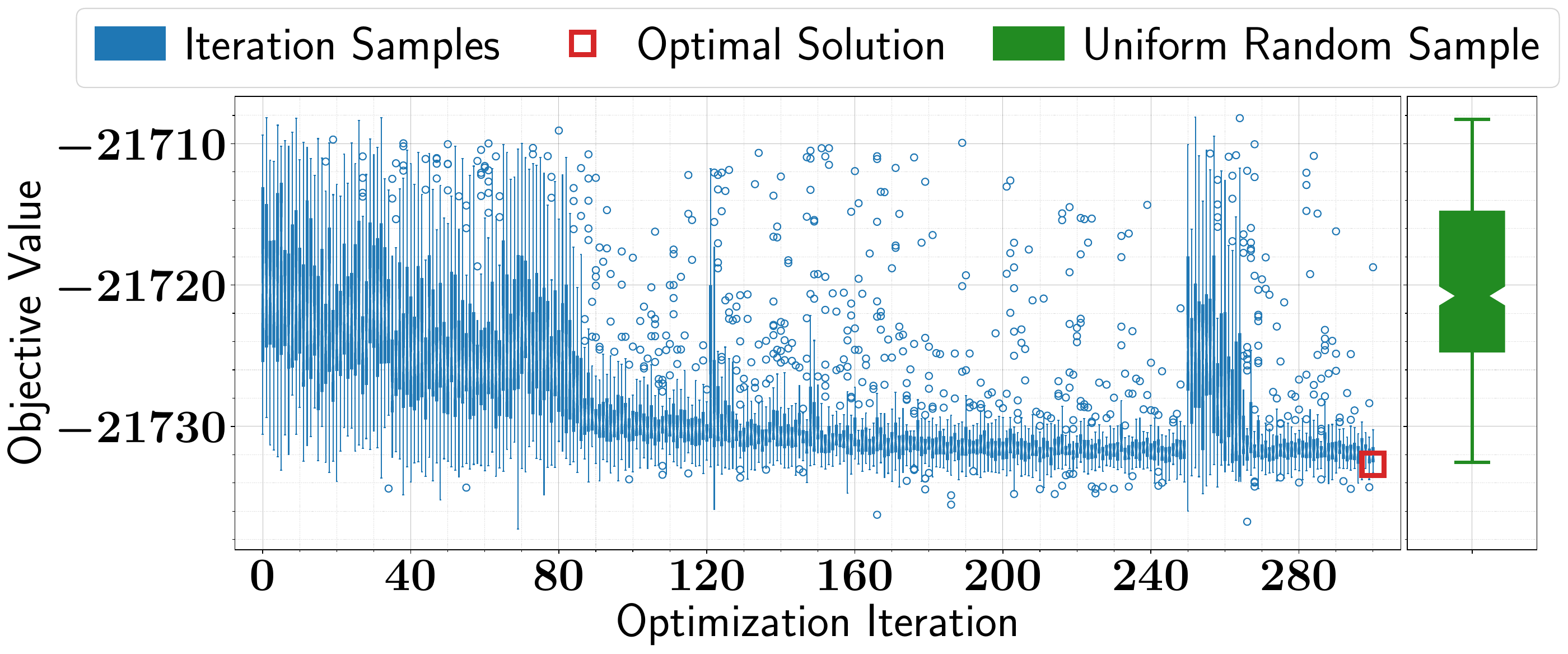}
      \includegraphics[width=0.495\linewidth]{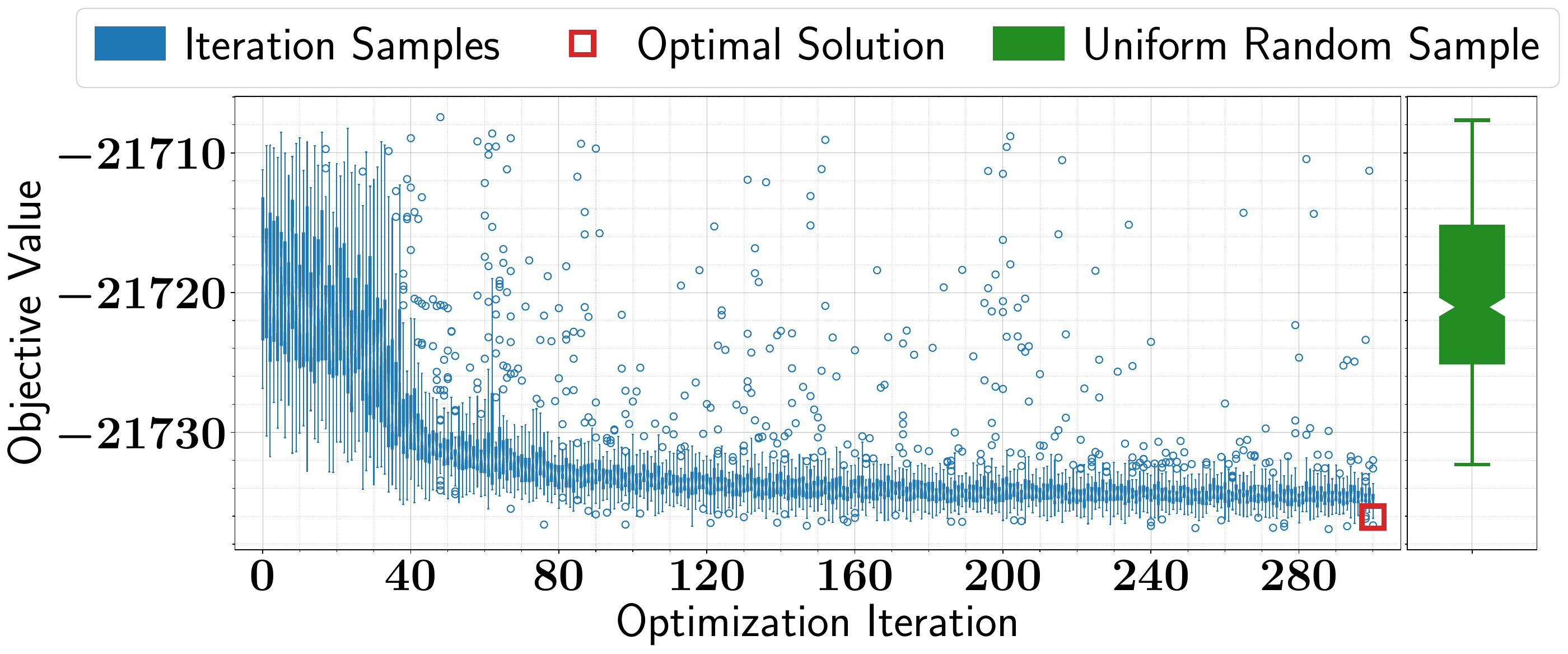}
      \includegraphics[width=0.495\linewidth]{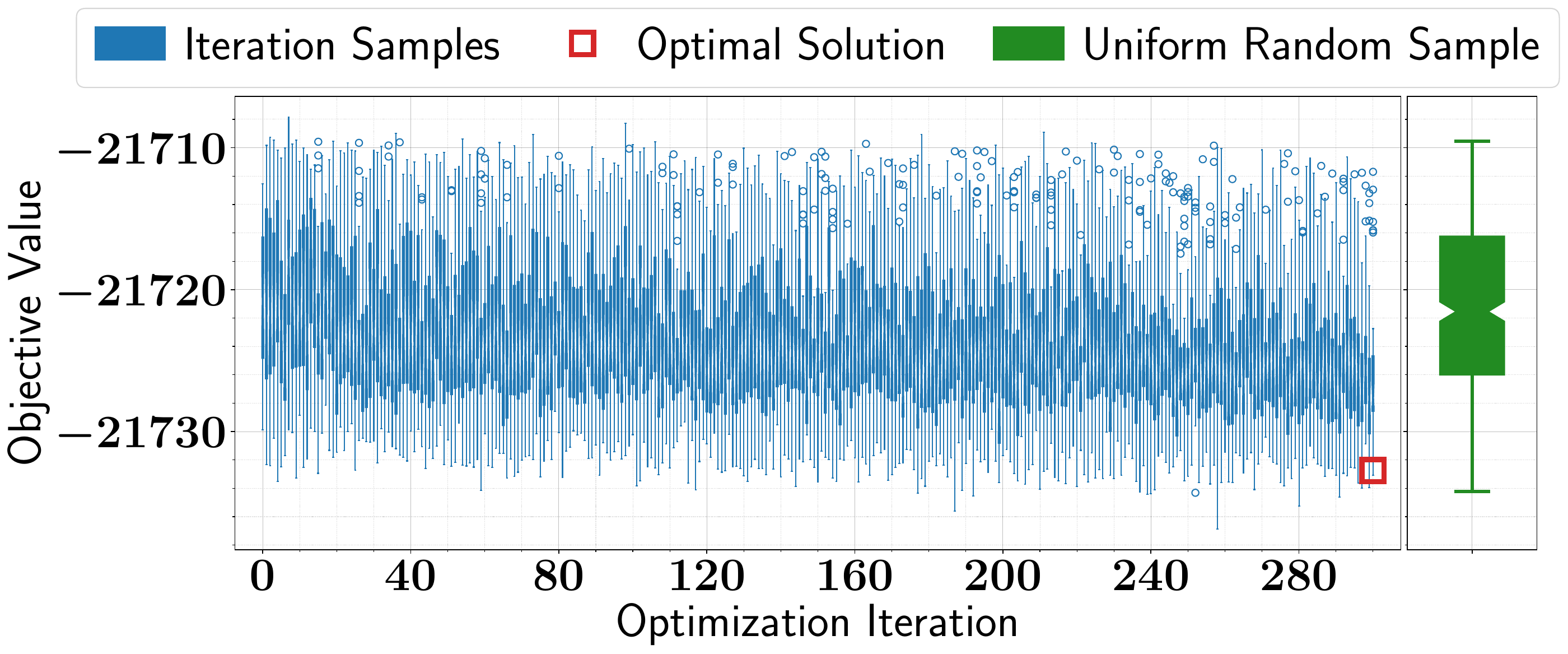}
      \includegraphics[width=0.495\linewidth]{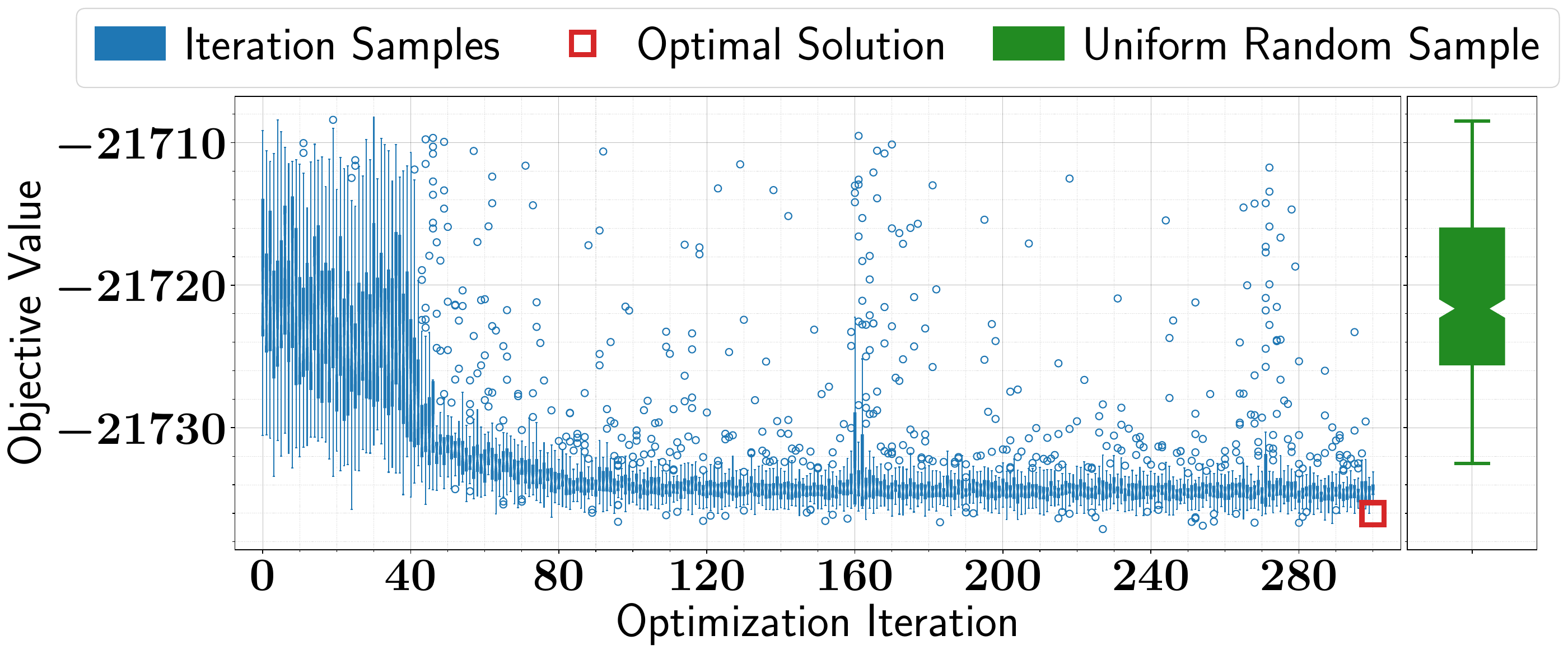}
      \caption{
        Results of \Cref{alg:probabilistic_path_optimization} with the 
        generalized higher-order 
        policy (\Cref{defn:generalized_higher_order_path_model}) applied to 
        the fine navigation mesh (\Cref{fig:navigation_meshes}, right) 
        with trajectory length of $n=19$ and $s=7$ moving sensors.
        Results are shown for policy order $k=3$ (first row) and $k=5$ (second row). 
        The first column shows results with lag weights being calibrated by the 
        optimization procedure, and 
        the second column shows results with lag weights modeled 
        by \eqref{eqn:decreasing_lag_weights}.
      }\label{sup:fig:fine_generalized_higher_order_unspecified_start_point_7_sensors}
    \end{figure}
    \begin{figure}[H]
      \begin{subfigure}[t]{0.24\linewidth}
        \includegraphics[width=0.91\linewidth]{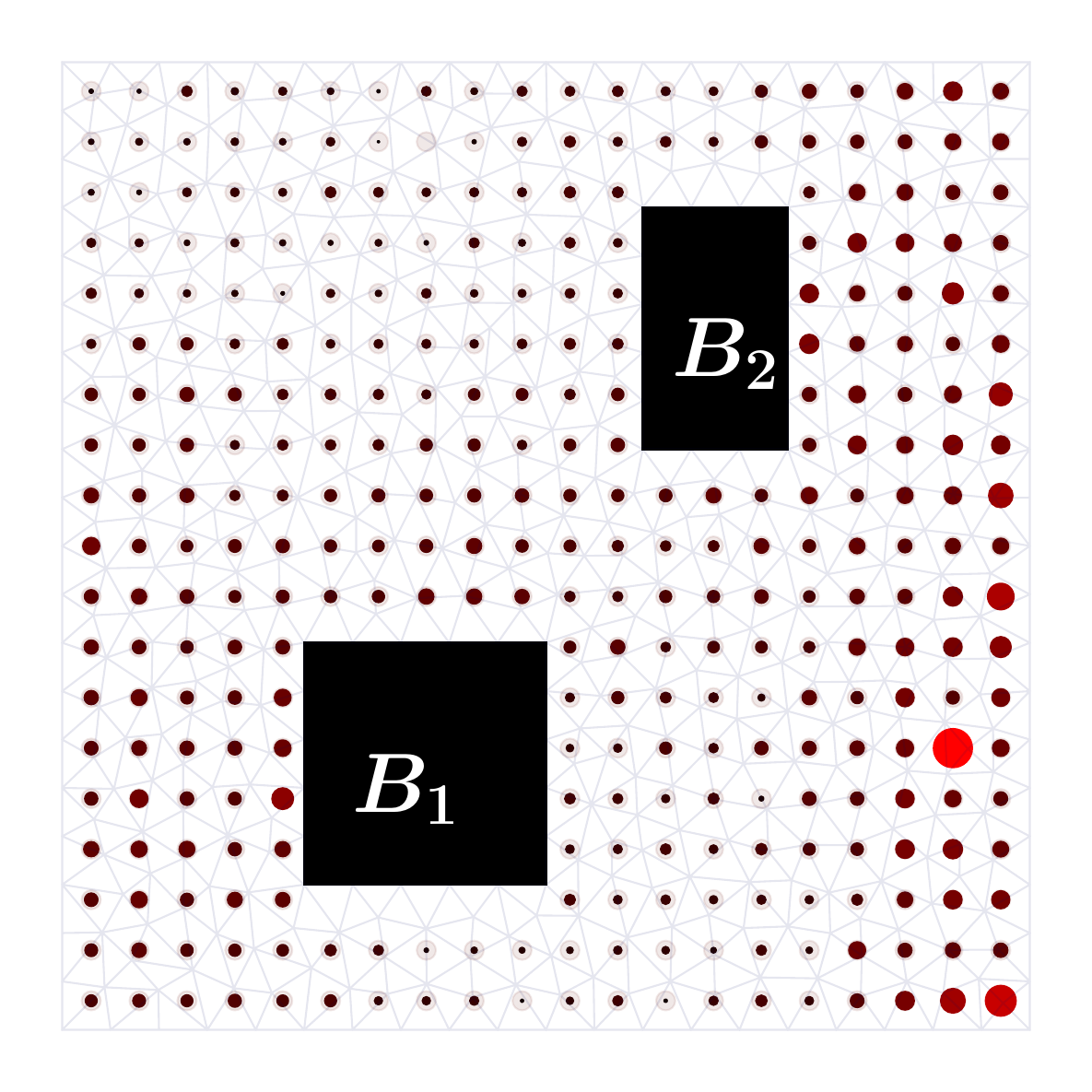}
      \end{subfigure}%
      \begin{subfigure}[t]{0.24\linewidth}
        \includegraphics[width=0.91\linewidth]{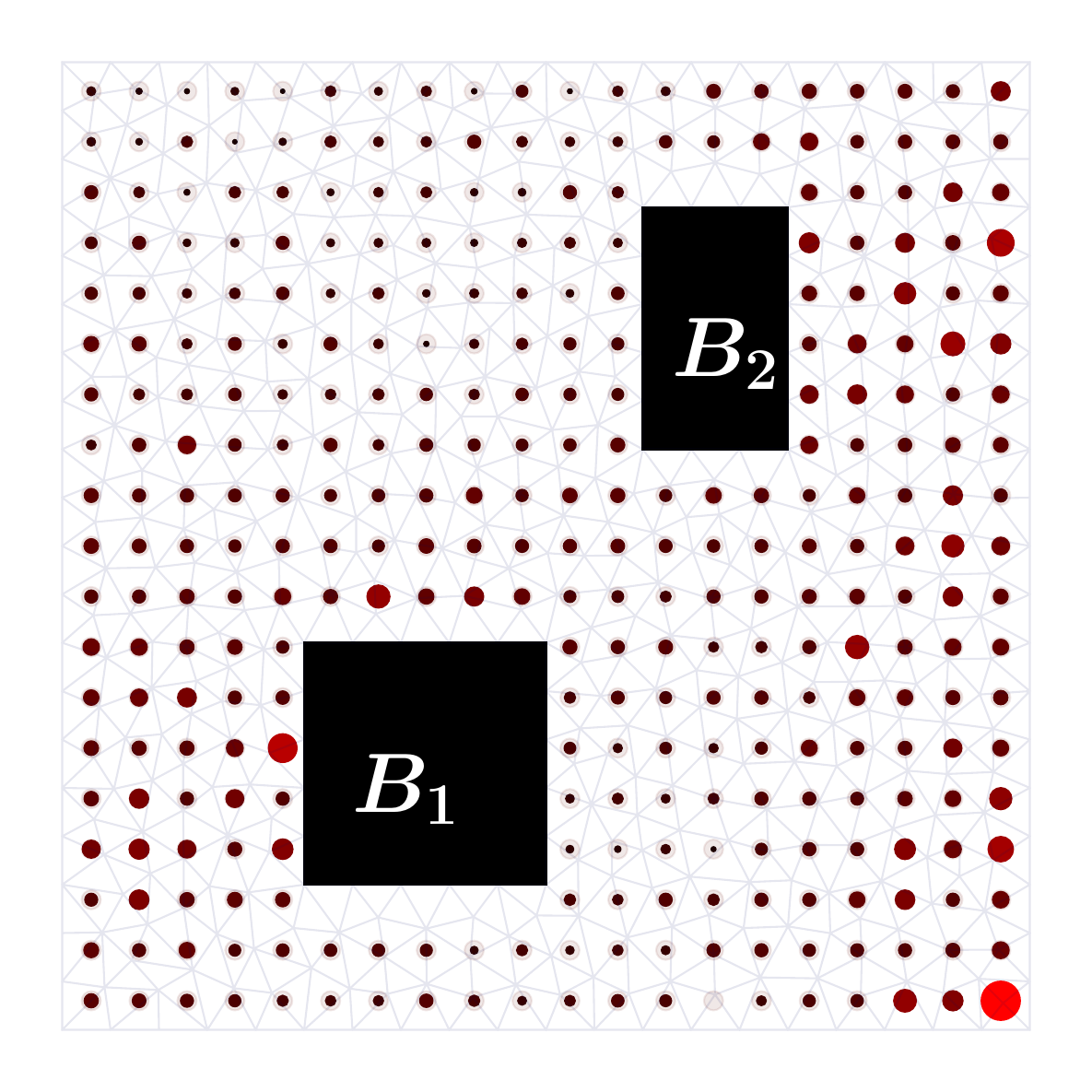}
      \end{subfigure}%
      \begin{subfigure}[t]{0.24\linewidth}
        \includegraphics[width=0.91\linewidth]{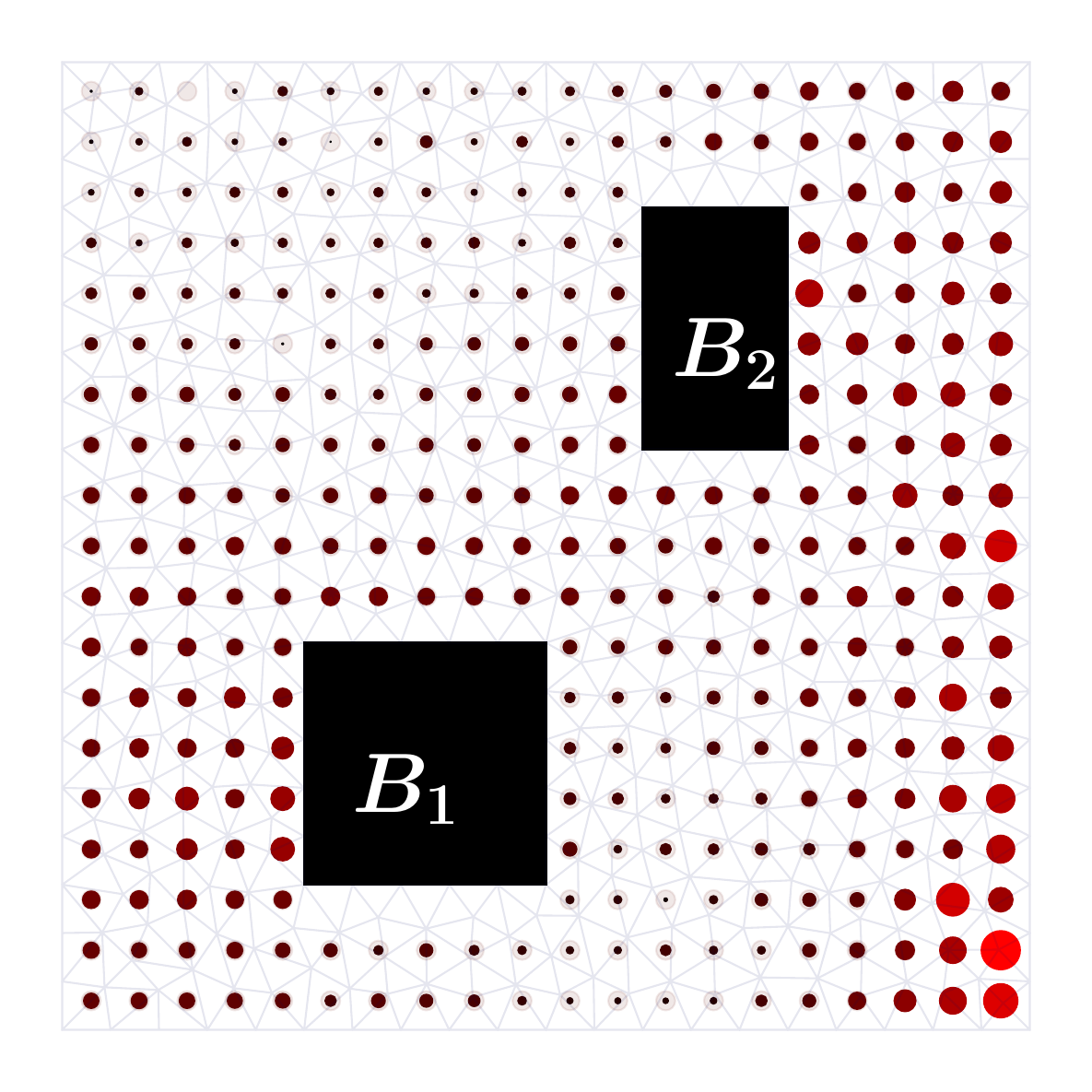}
      \end{subfigure}%
      \begin{subfigure}[t]{0.24\linewidth}
        \includegraphics[width=0.91\linewidth]{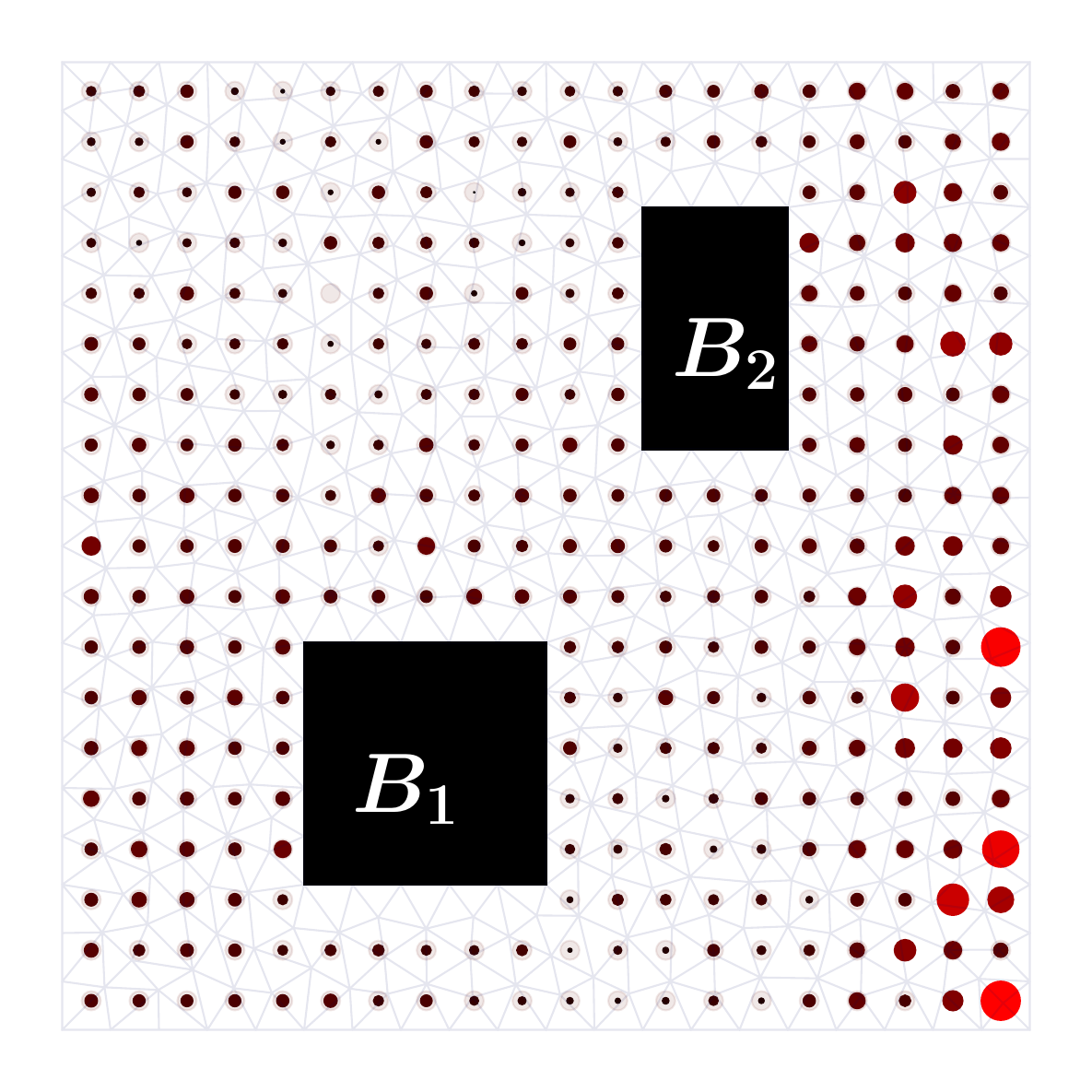}
      \end{subfigure}

      \begin{subfigure}[t]{0.245\linewidth}
        \includegraphics[width=\linewidth]{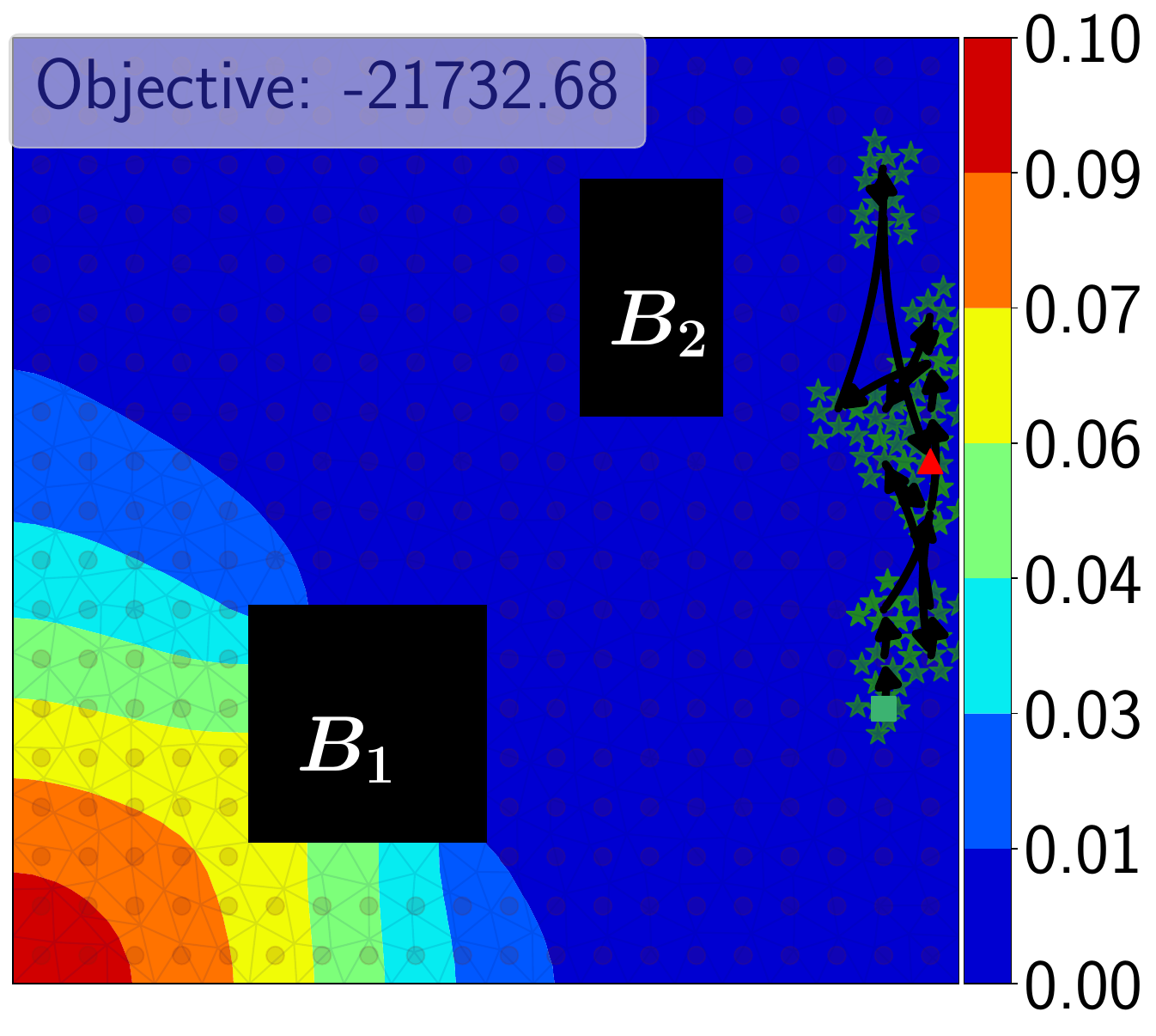}
      \end{subfigure}%
      \begin{subfigure}[t]{0.245\linewidth}
        \includegraphics[width=\linewidth]{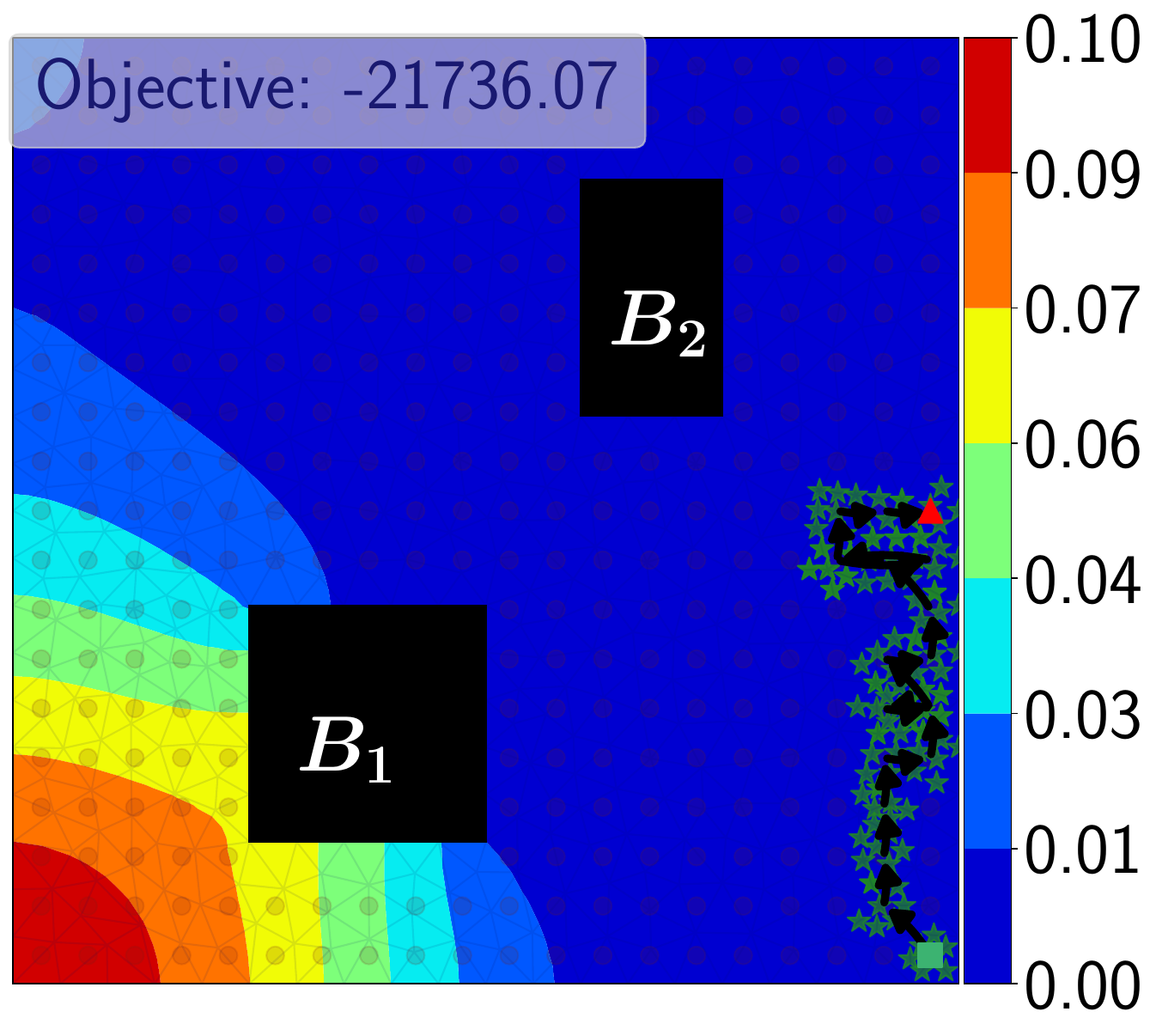}
      \end{subfigure}%
      \begin{subfigure}[t]{0.245\linewidth}
        \includegraphics[width=\linewidth]{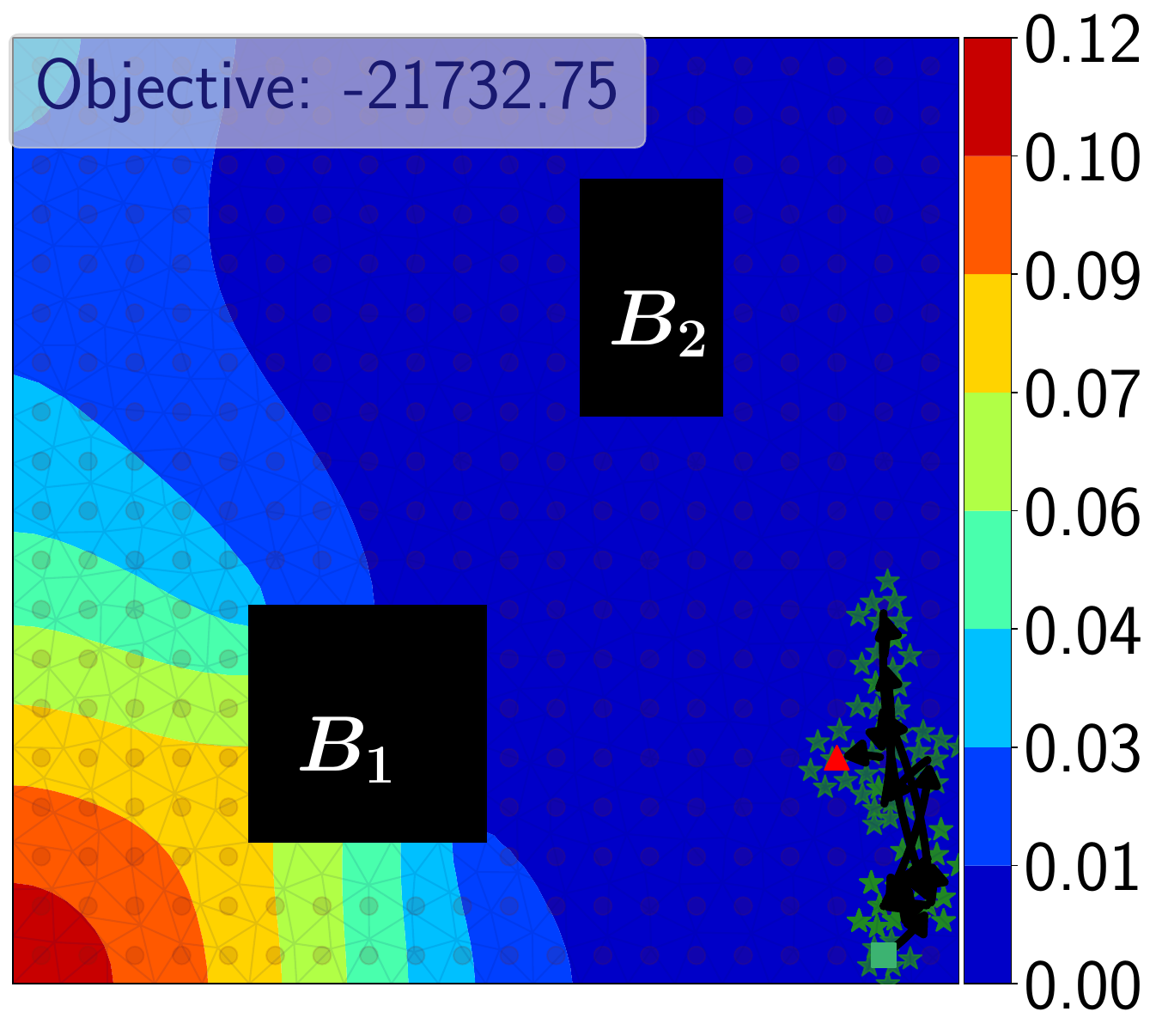}
      \end{subfigure}%
      \begin{subfigure}[t]{0.245\linewidth}
        \includegraphics[width=\linewidth]{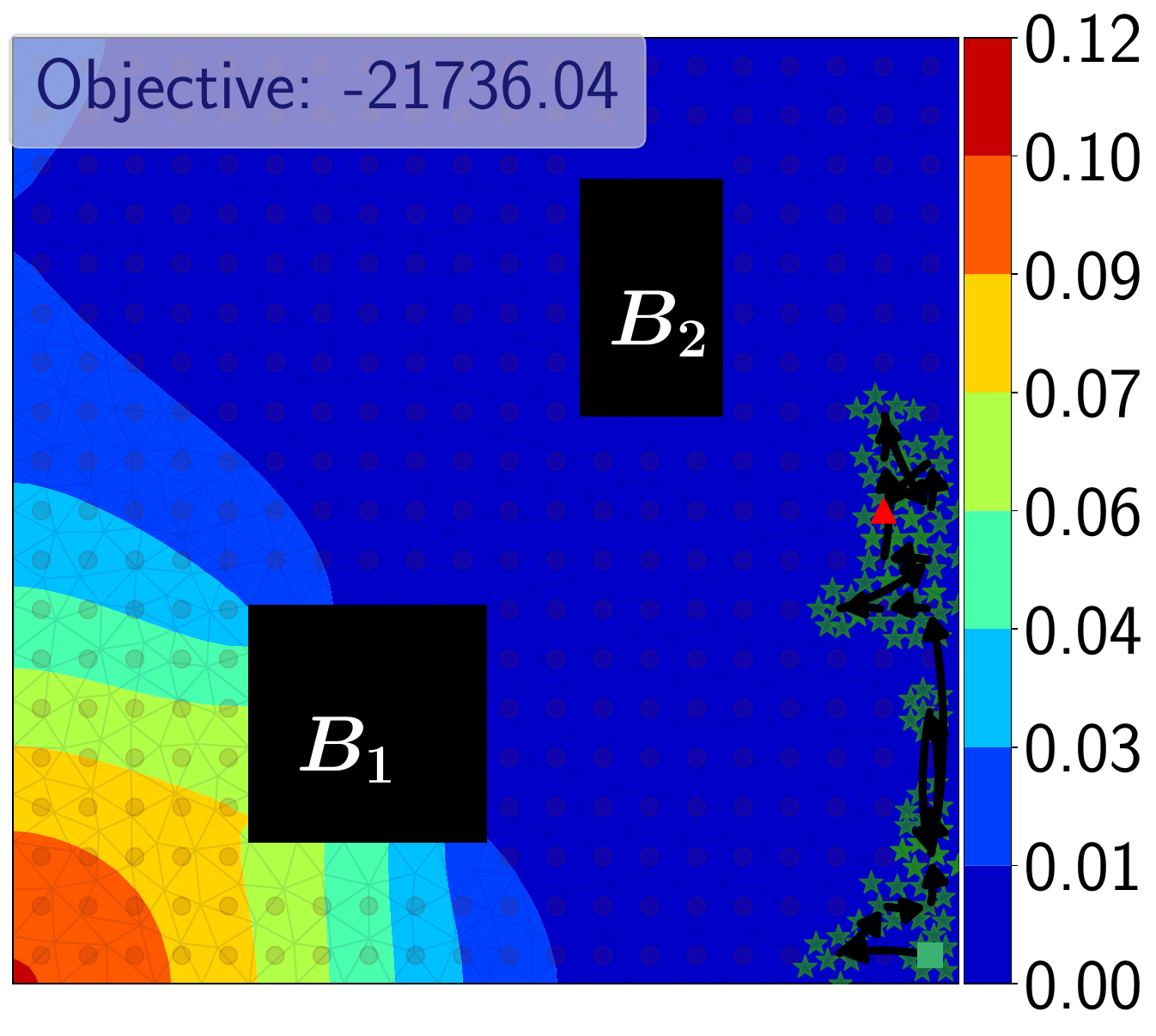}
      \end{subfigure}
      \caption{
        Optimal initial parameters and trajectories
        corresponding to 
        \Cref{sup:fig:fine_generalized_higher_order_unspecified_start_point_7_sensors}. 
        In each row, the first two panels match the first row of that figure, 
        followed by panels corresponding to its second row.
      }\label{sup:fig:fine_generalized_higher_order_unspecified_start_point_order_3_initial_param_and_traject_7_sensors}
    \end{figure}

  \subsection{Results with Fine Navigation Mesh and Fixed Starting point}
  \label{sup:subsec:fine_mesh_results_fixed_start_point}
    Here we discuss a different set of experiments in which 
    the path must start at a fixed starting point.
    Specifically, we enforce the starting point to be at the coordinates 
    $(x, y)=(0.2, 0.7)$ by setting the initial parameter $p_i$ corresponding to that coordinate on the navigation mesh to $1$; 
    all other initial parameters are set to $0$. 
    Other than the fixed starting point, all settings here are the same as
    in \Cref{subsubsec:fine_mesh_results_unspecified_start_point}. 
    Results obtained with $s=1$ moving sensors are given 
    in \Cref{sup:subsubsec:fine_mesh_results_fixed_start_point_1_sensor}, 
    and results with $s=7$ sensors are given in
    \Cref{sup:subsubsec:fine_mesh_results_fixed_start_point_7_sensors}

    \subsubsection{Results with $1$ Moving Sensor}
      \label{sup:subsubsec:fine_mesh_results_fixed_start_point_1_sensor}
      \Cref{sup:fig:fine_first_order_fixed_start_point_1_sensor} 
      shows the results of \Cref{alg:probabilistic_path_optimization} with the first-order 
      policy defined by \Cref{defn:first_order_path_model}.

      \begin{figure}[H]
        \centering
        \includegraphics[width=0.53\linewidth]{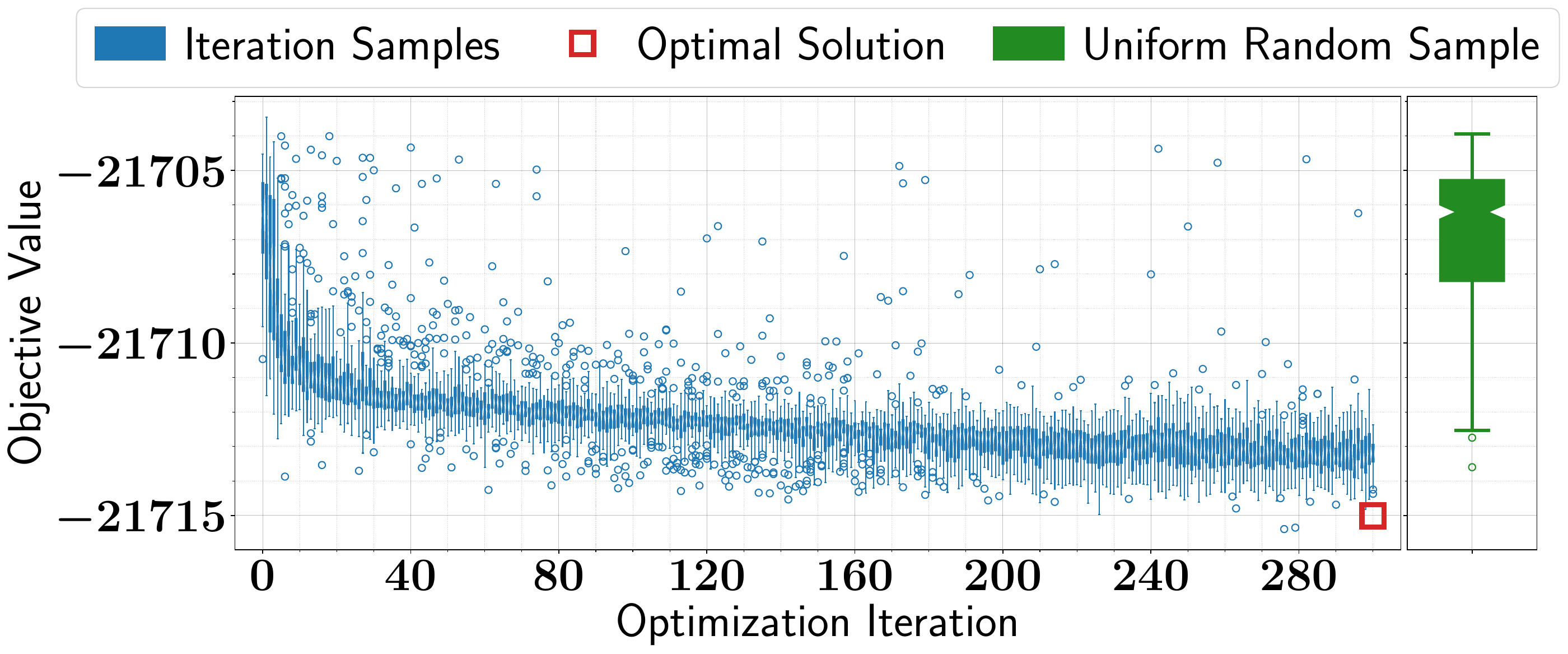}
        \includegraphics[width=0.215\linewidth]{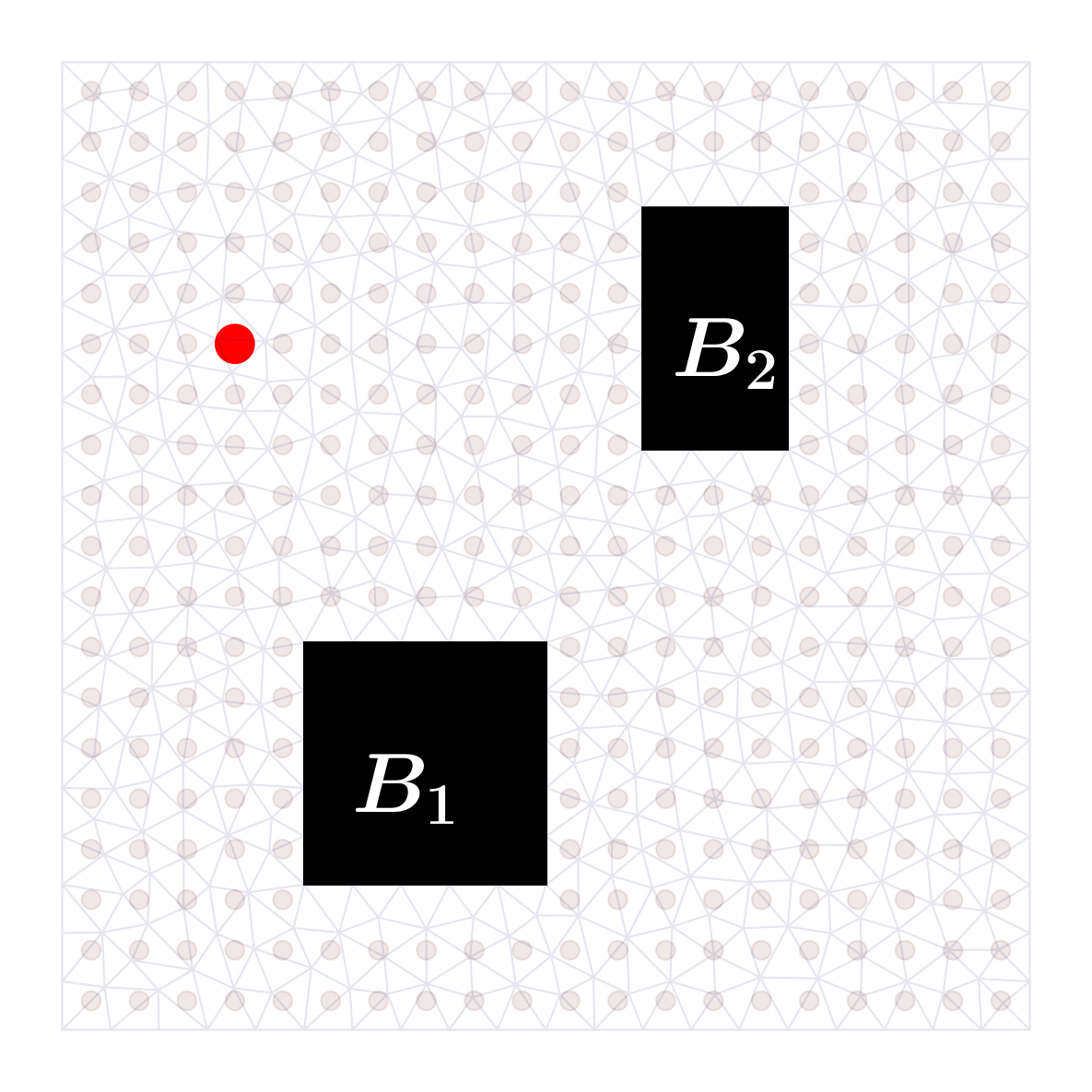}
        \includegraphics[width=0.235\linewidth]{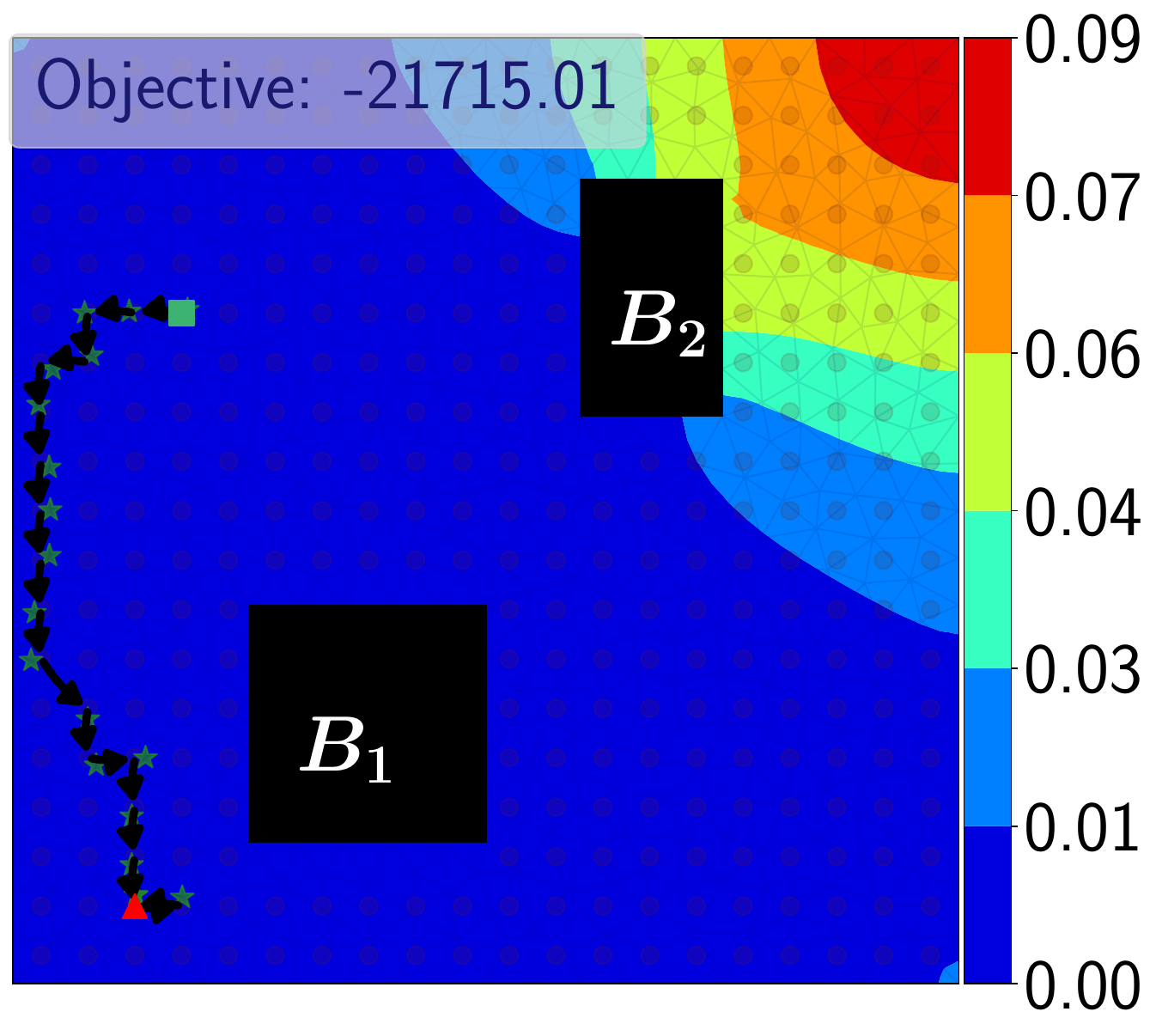}
        \caption{
          Results of \Cref{alg:probabilistic_path_optimization} with the first-order 
          policy model (\Cref{defn:first_order_path_model}) applied to the fine
          navigation mesh (\Cref{fig:navigation_meshes}, right) with 
          the starting point of the trajectory fixed to $(x, y)=(0.2, 0.7)$.
          Trajectory length is $n=19$, and the sensor size is $s=1$.
        }\label{sup:fig:fine_first_order_fixed_start_point_1_sensor}
      \end{figure}

      \Cref{sup:fig:fine_higher_order_fixed_start_point_1_sensor} 
      shows the performance of the optimization procedure with the higher-order 
      policy given by \Cref{defn:higher_order_path_model} with order $k=3$ and $k=5$, respectively.
      The resulting optimal initial parameter (first row) and optimal trajectories (second row)
      are shown in 
      \Cref{sup:fig:fine_higher_order_fixed_start_point_1_sensor_optimal_trajectories}.

      \begin{figure}[H]
        \centering
        \includegraphics[width=0.495\linewidth]{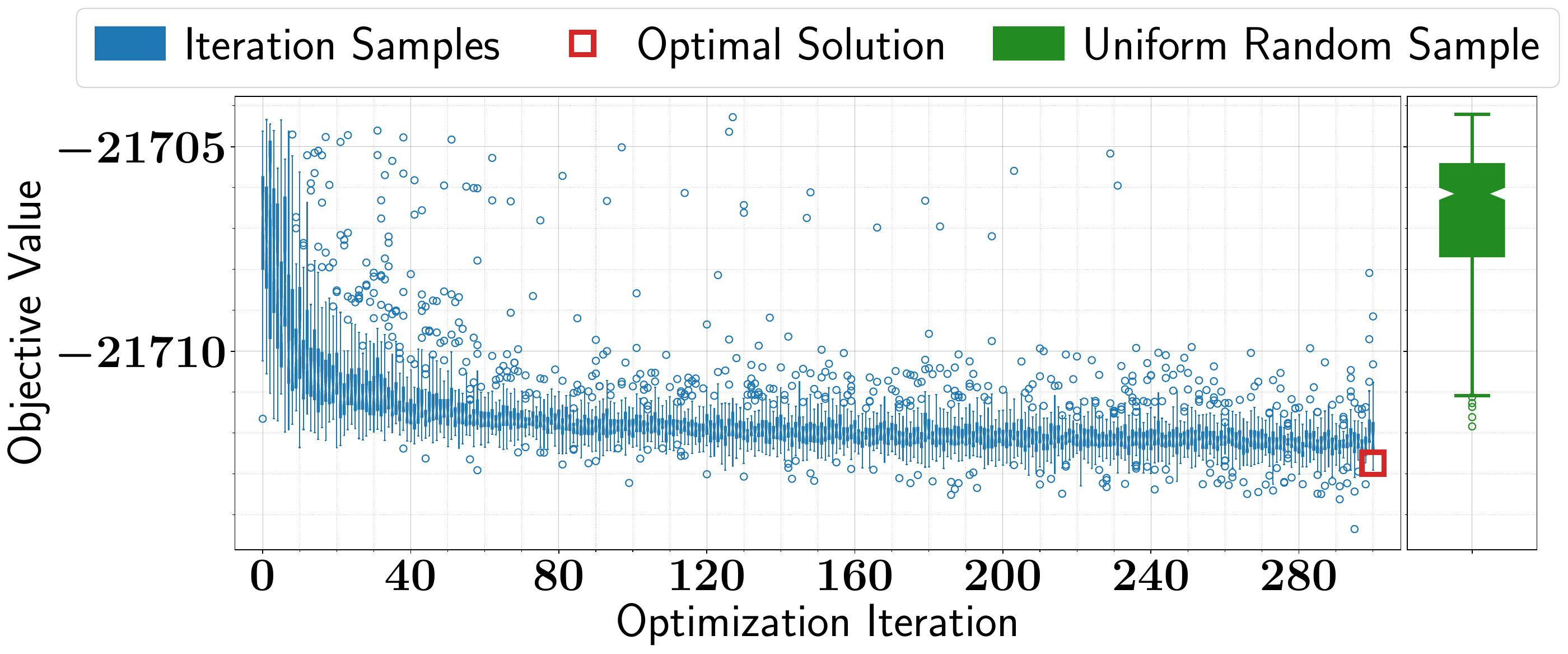}
        \includegraphics[width=0.495\linewidth]{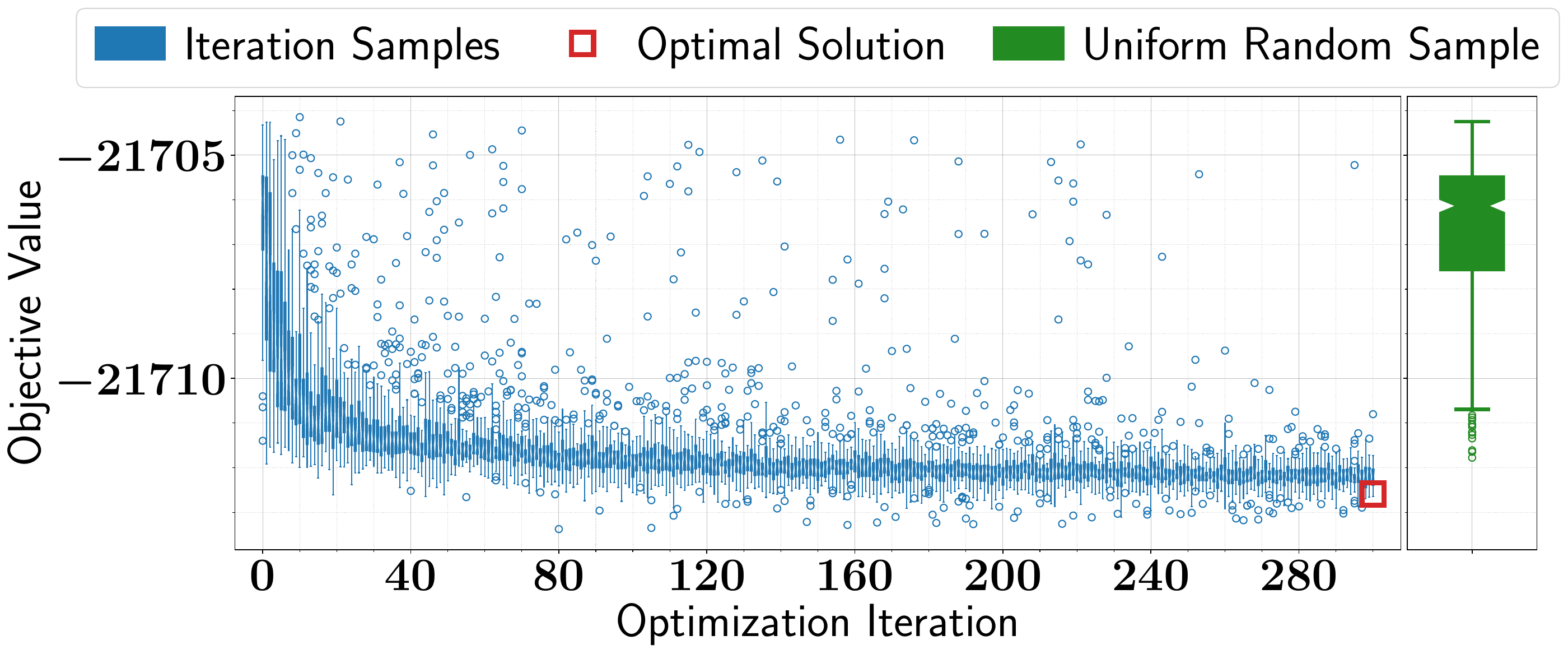}
        \includegraphics[width=0.495\linewidth]{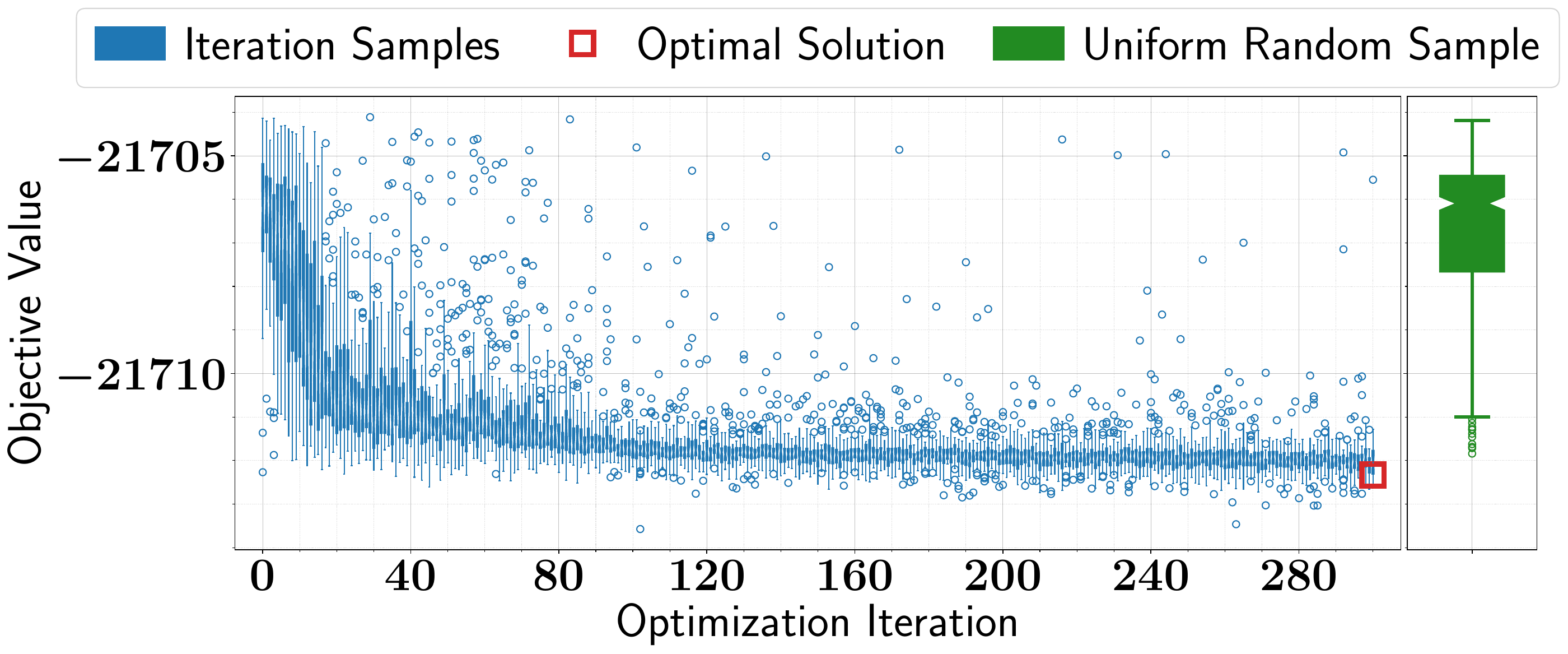}
        \includegraphics[width=0.495\linewidth]{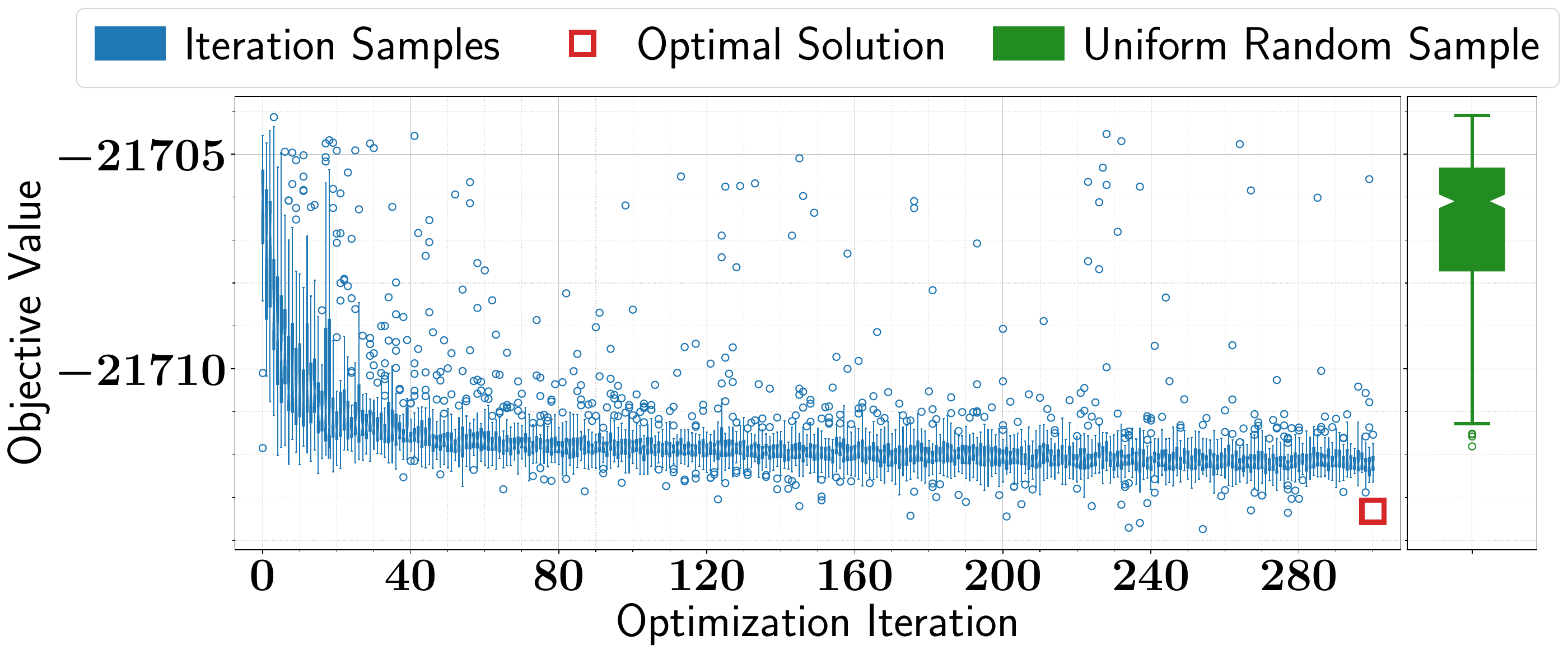}
        \caption{
          Results of \Cref{alg:probabilistic_path_optimization} with the higher-order 
          policy model (\Cref{defn:higher_order_path_model}) applied to the fine
          navigation mesh (\Cref{fig:navigation_meshes}, right) with 
          the starting point of the trajectory fixed to $(x, y)=(0.2, 0.7)$.
          Trajectory length is $n=19$, and the sensor size is $s=1$.
          The policy order is set to  $k=3$ (first row) and to $k=5$ (second row).
          The first column shows results with lag weights being optimized, and 
          the second column shows results with lag weights modeled 
          by \eqref{eqn:decreasing_lag_weights}.
        }\label{sup:fig:fine_higher_order_fixed_start_point_1_sensor}
      \end{figure}
      Similarly, results obtained by the policy given by  \Cref{defn:higher_order_path_model} 
      are shown in 
      \Cref{sup:fig:fine_generalized_higher_order_fixed_start_point_1_sensor} 
      and 
      \Cref{sup:fig:fine_generalized_higher_order_fixed_start_point_1_sensor_optimal_trajectories}.

      \begin{figure}[H]
        \centering
        \includegraphics[width=0.235\linewidth]{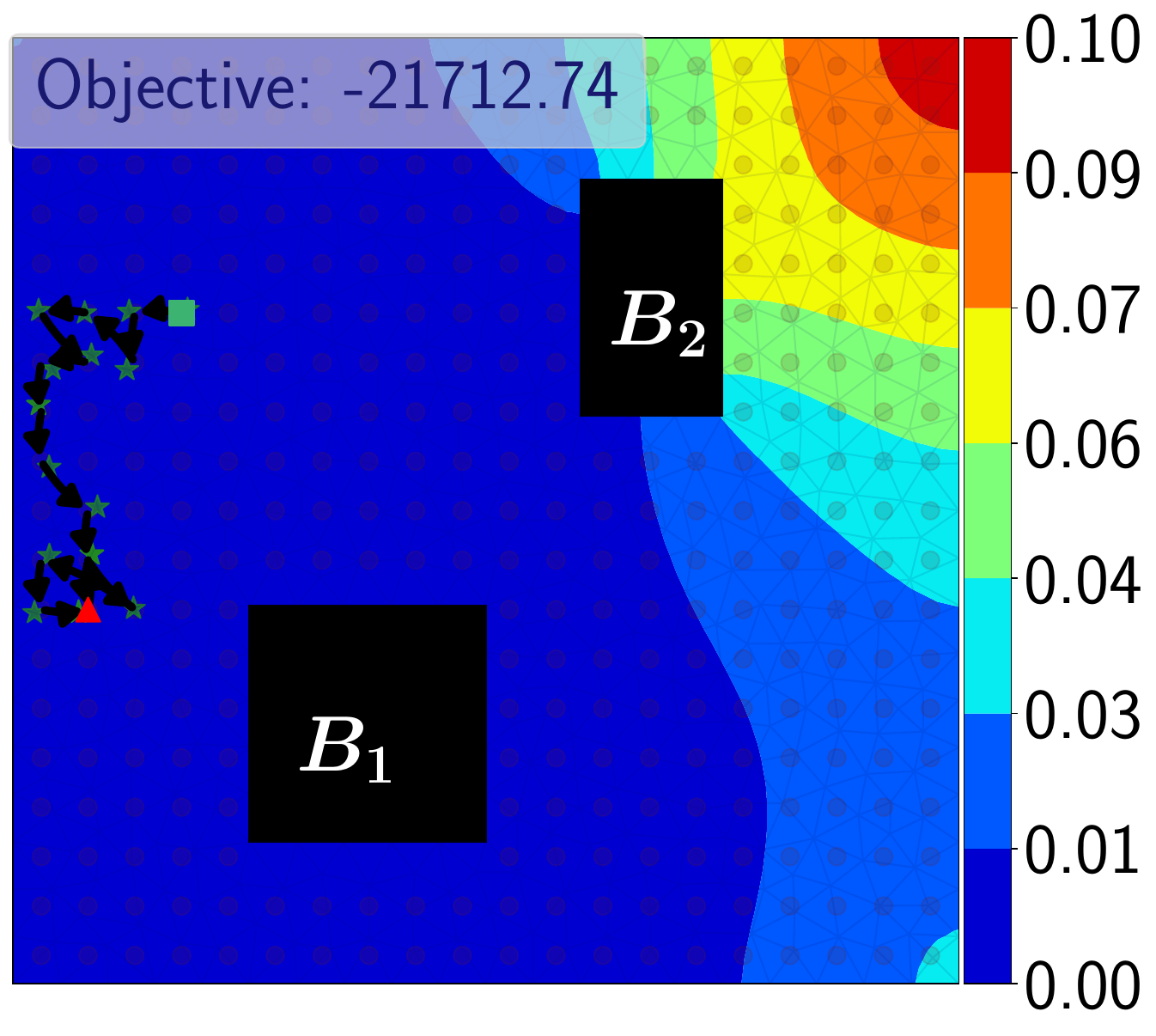}
        \includegraphics[width=0.235\linewidth]{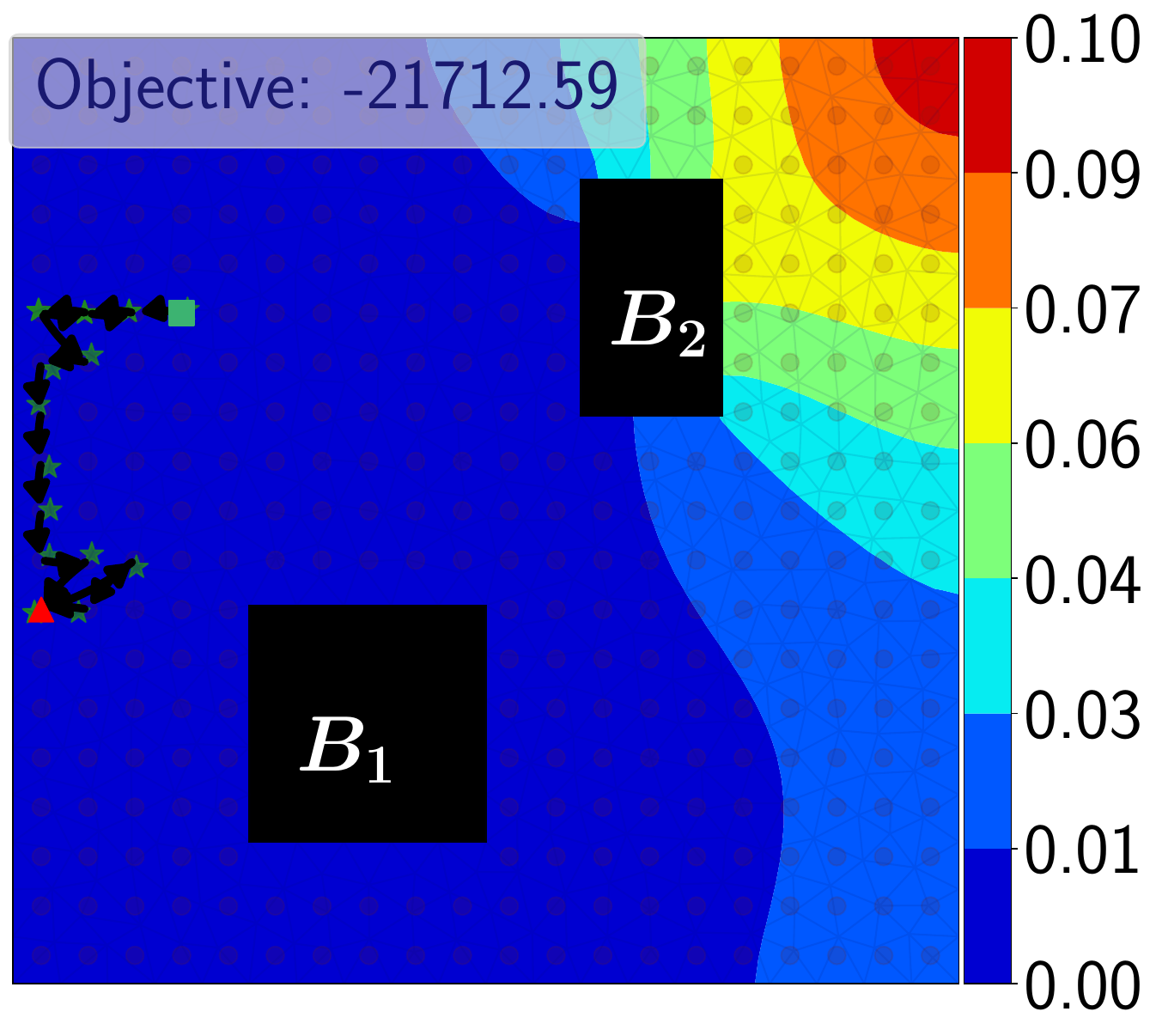}
        \\
        \includegraphics[width=0.235\linewidth]{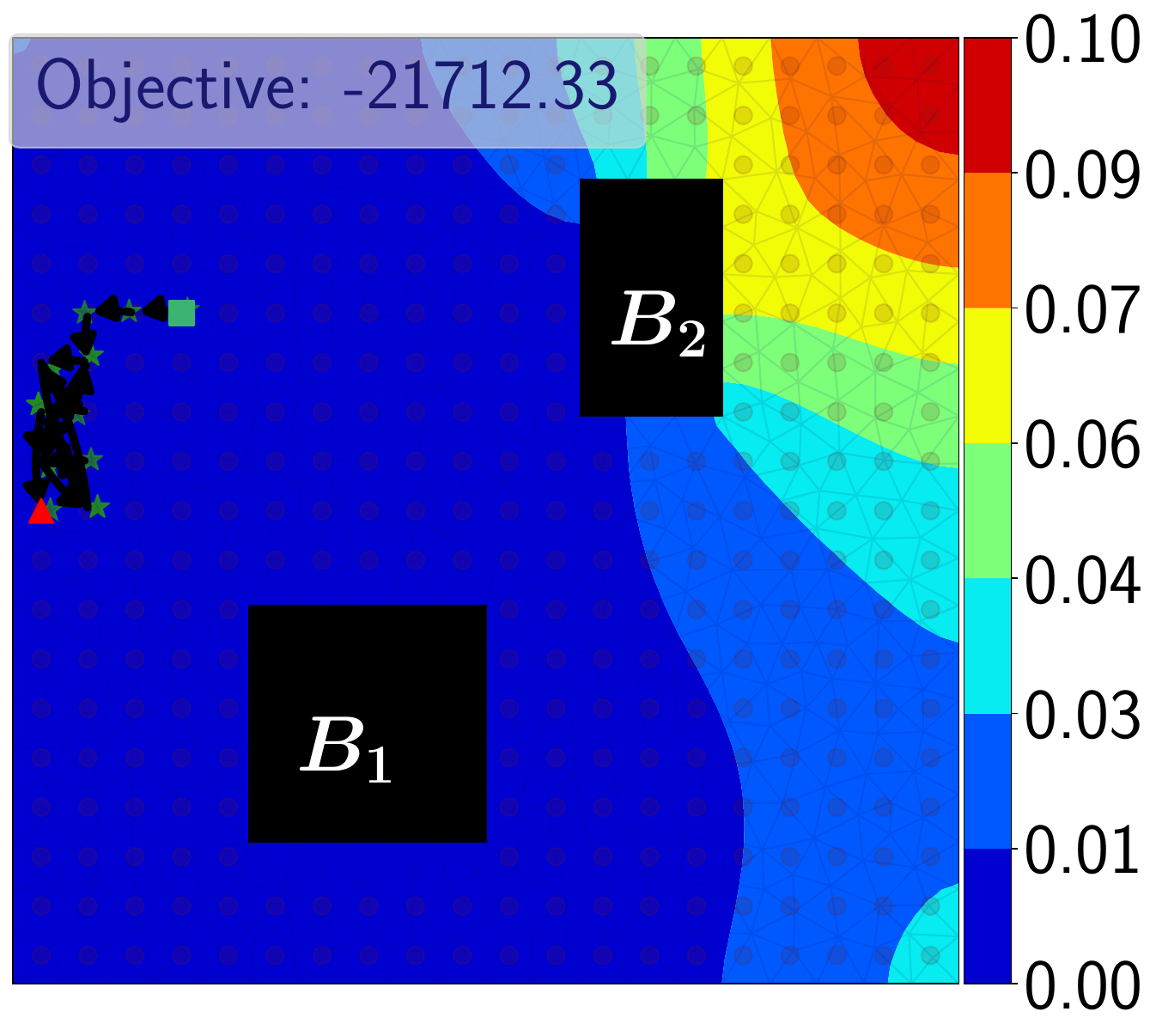}
        \includegraphics[width=0.235\linewidth]{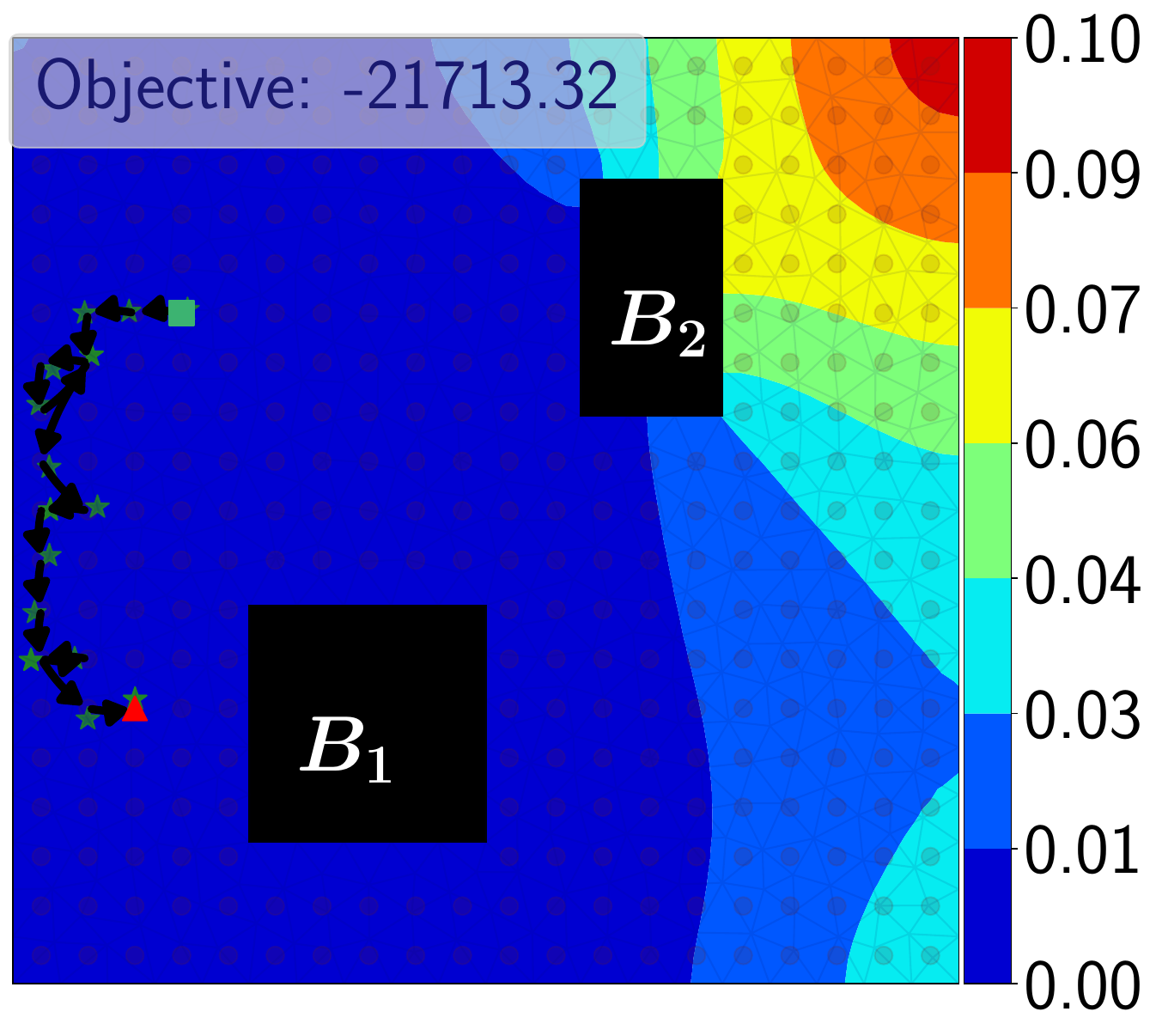}
        \caption{
          Optimal trajectories corresponding to results shown in 
          \Cref{sup:fig:fine_higher_order_fixed_start_point_1_sensor}.
        }\label{sup:fig:fine_higher_order_fixed_start_point_1_sensor_optimal_trajectories}
      \end{figure}
      \begin{figure}[H]
        \centering
        \includegraphics[width=0.495\linewidth]{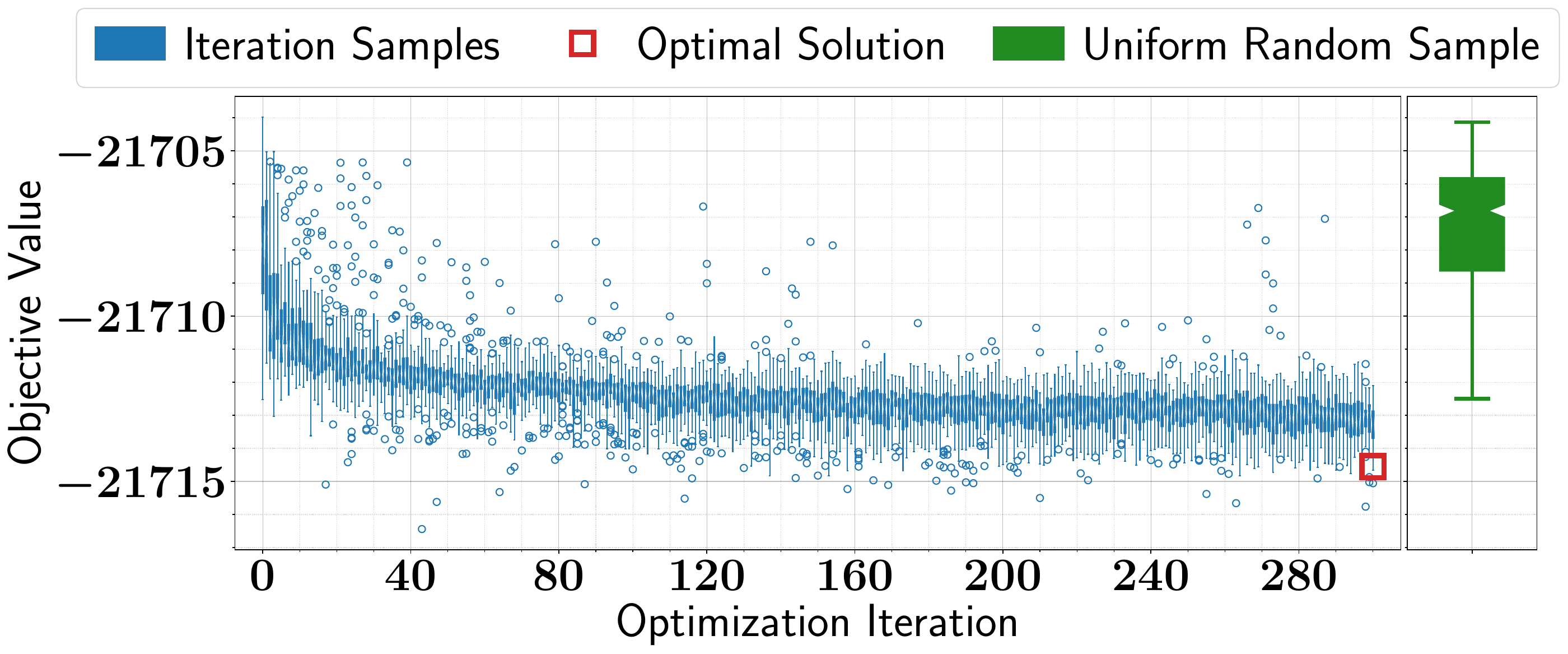}
        \includegraphics[width=0.495\linewidth]{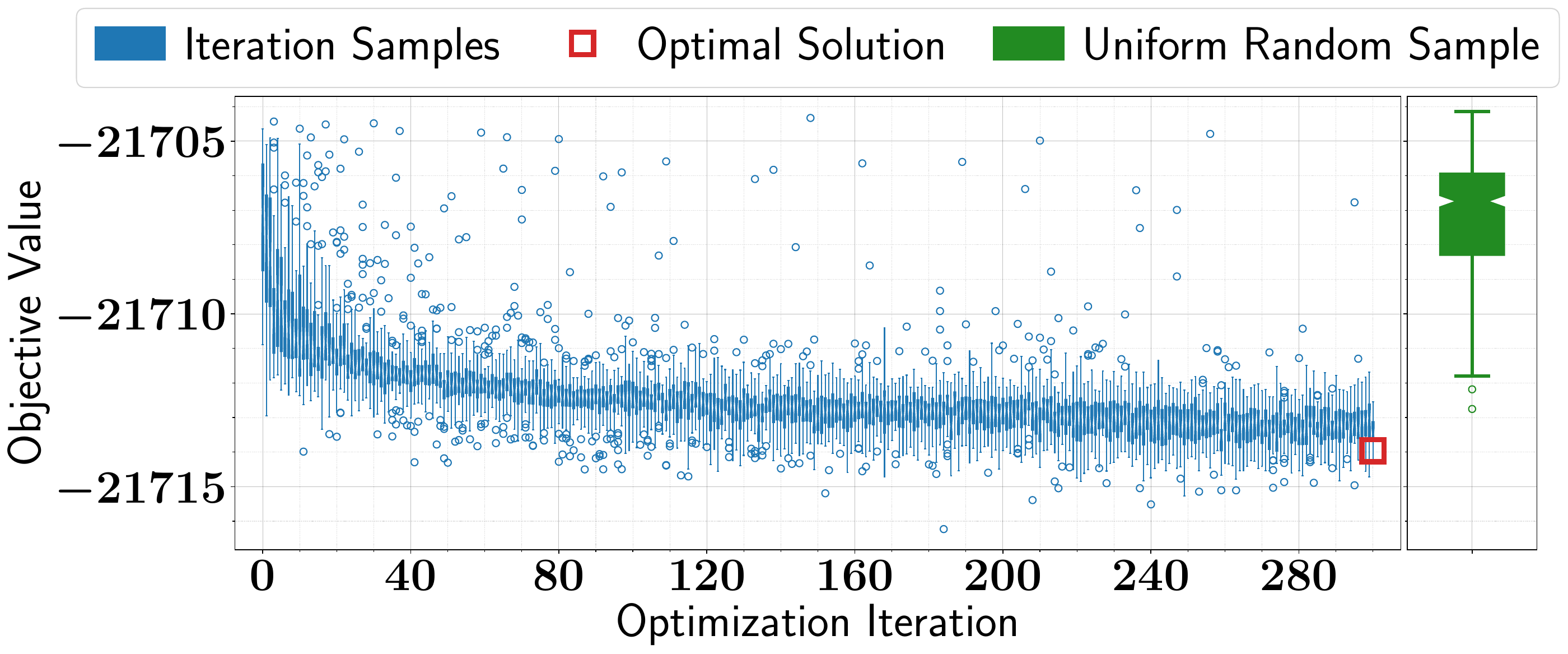}
        \includegraphics[width=0.495\linewidth]{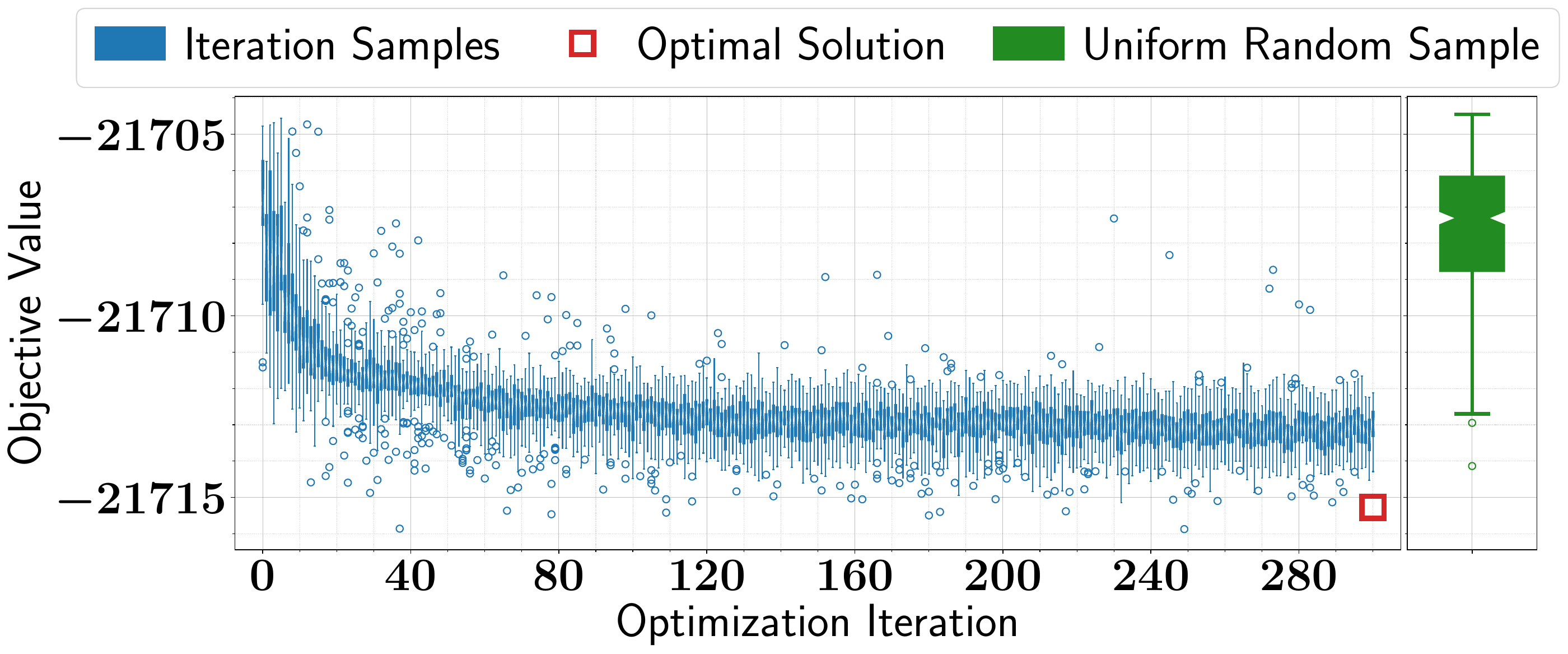}
        \includegraphics[width=0.495\linewidth]{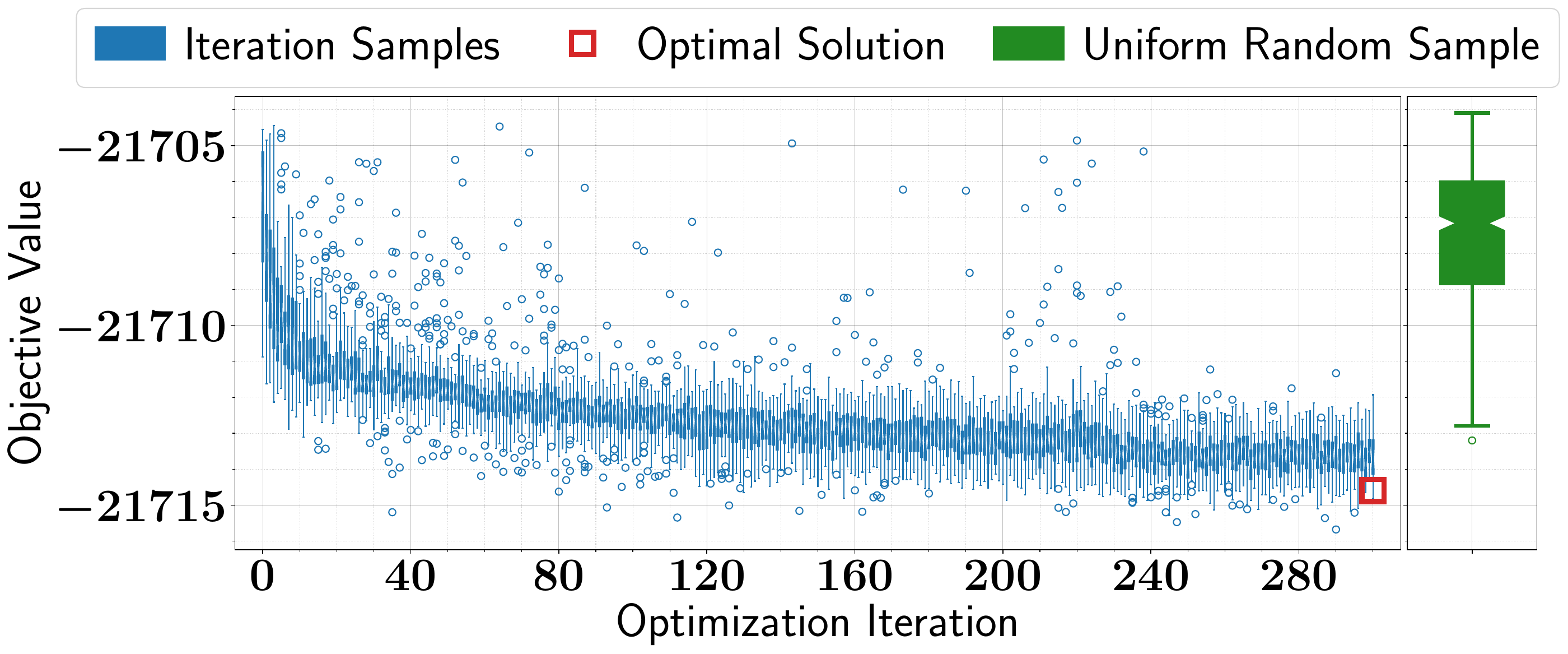}
        \caption{
          Similar to \Cref{sup:fig:fine_higher_order_fixed_start_point_1_sensor}.
          Here the generalized higher-order policy (\Cref{defn:generalized_higher_order_path_model}) is used. 
        }\label{sup:fig:fine_generalized_higher_order_fixed_start_point_1_sensor}
      \end{figure}
      \begin{figure}[H]
        \centering
        \includegraphics[width=0.235\linewidth]{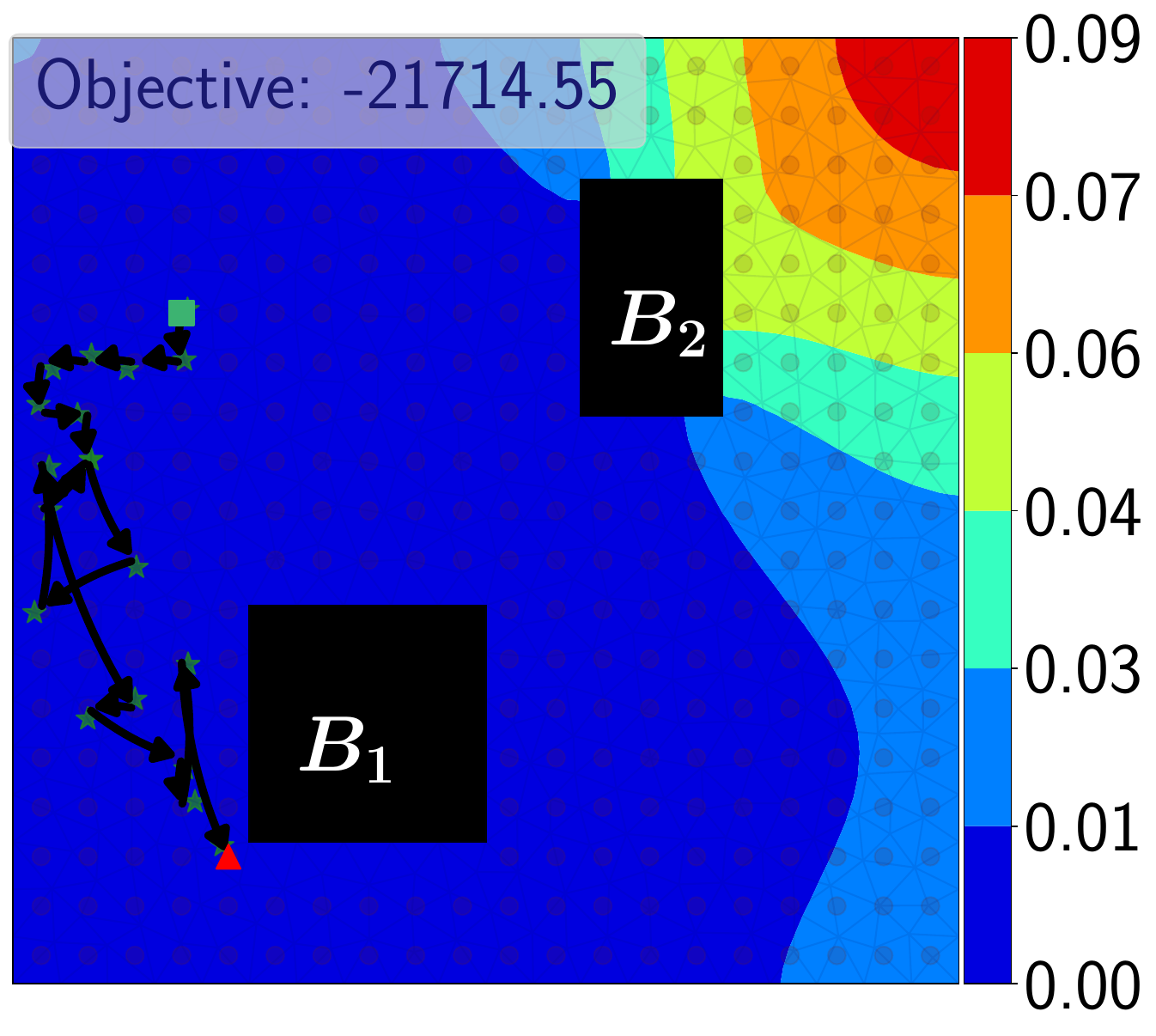}
        \includegraphics[width=0.235\linewidth]{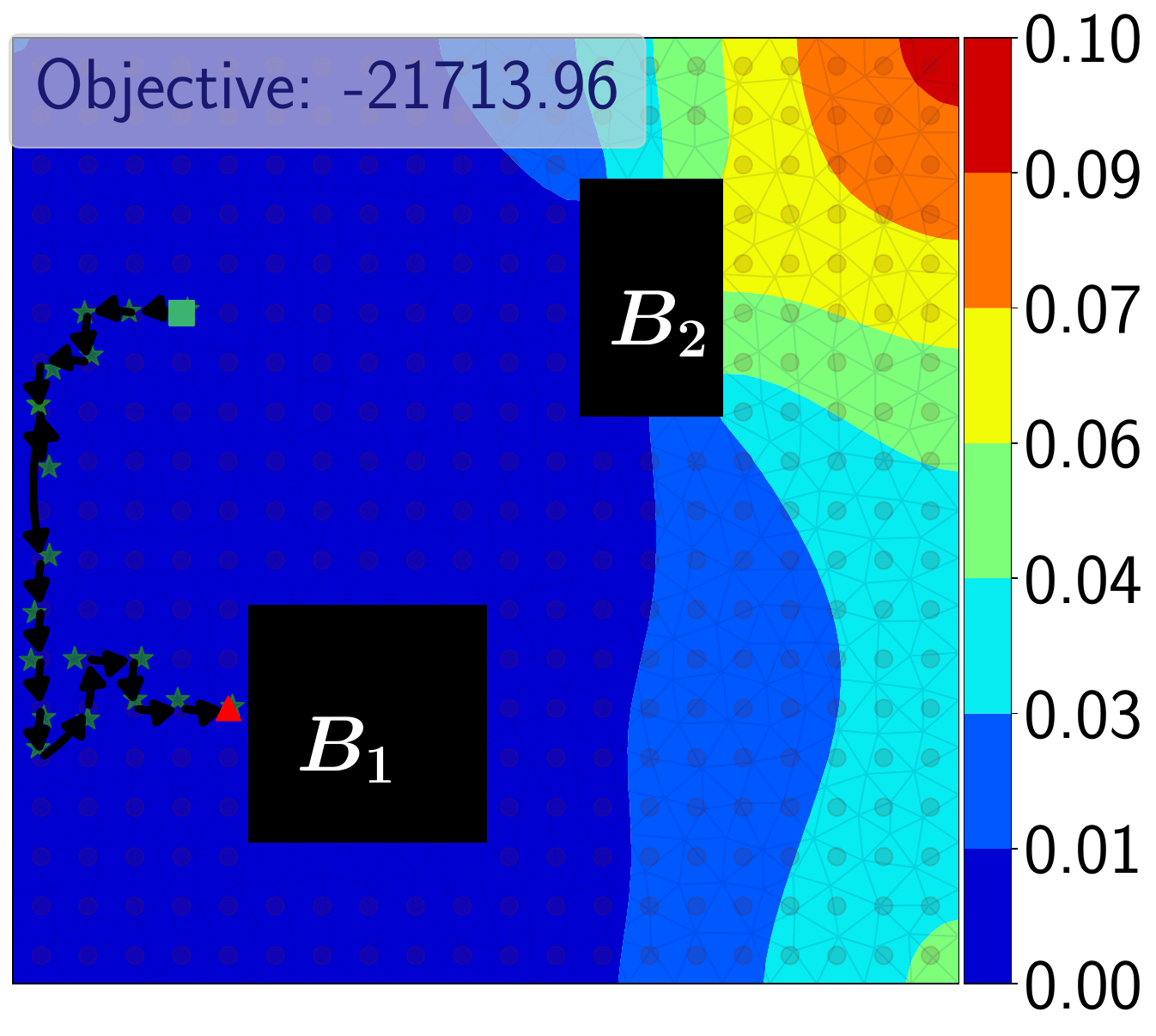}
        \\
        \includegraphics[width=0.235\linewidth]{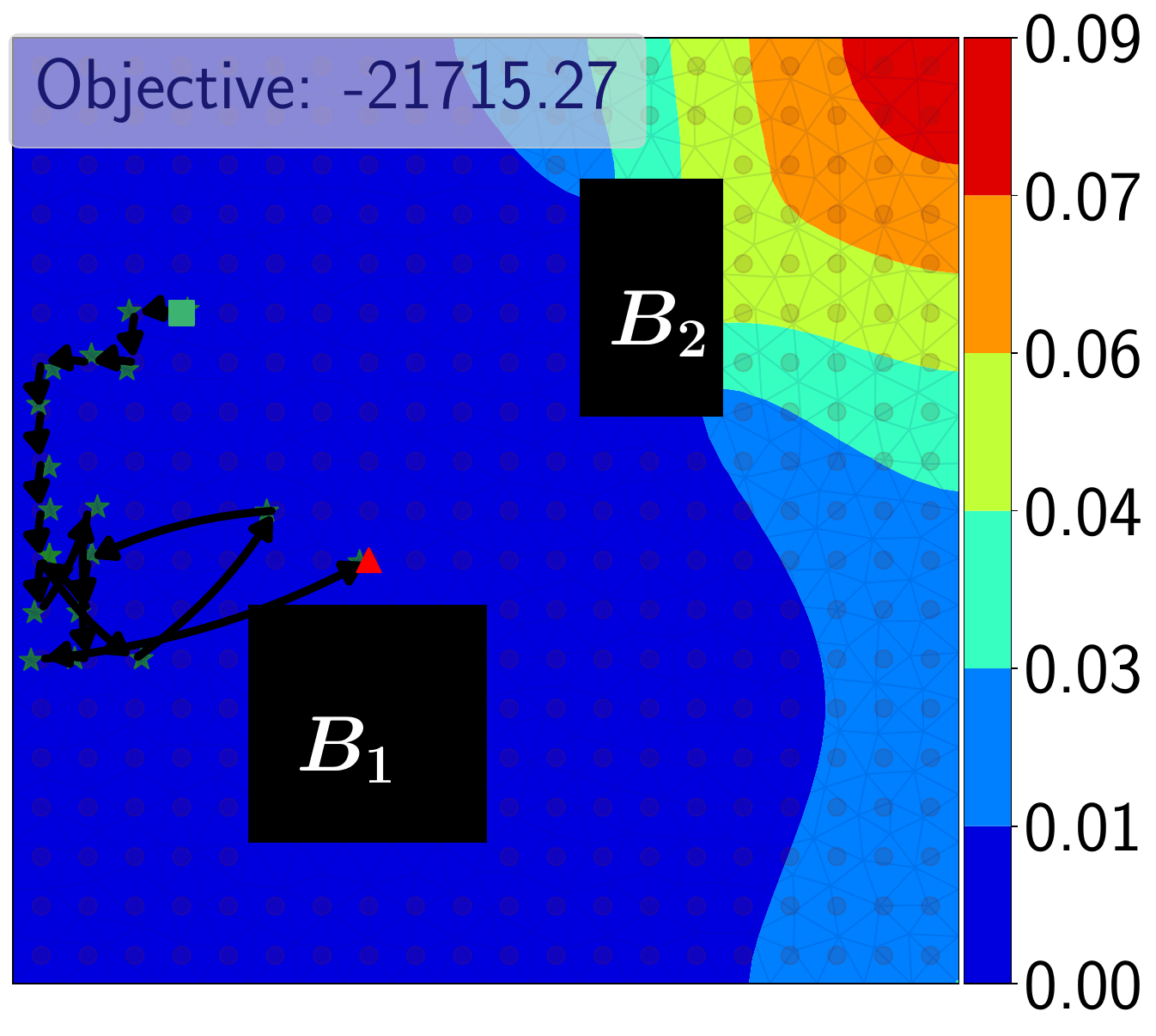}
        \includegraphics[width=0.235\linewidth]{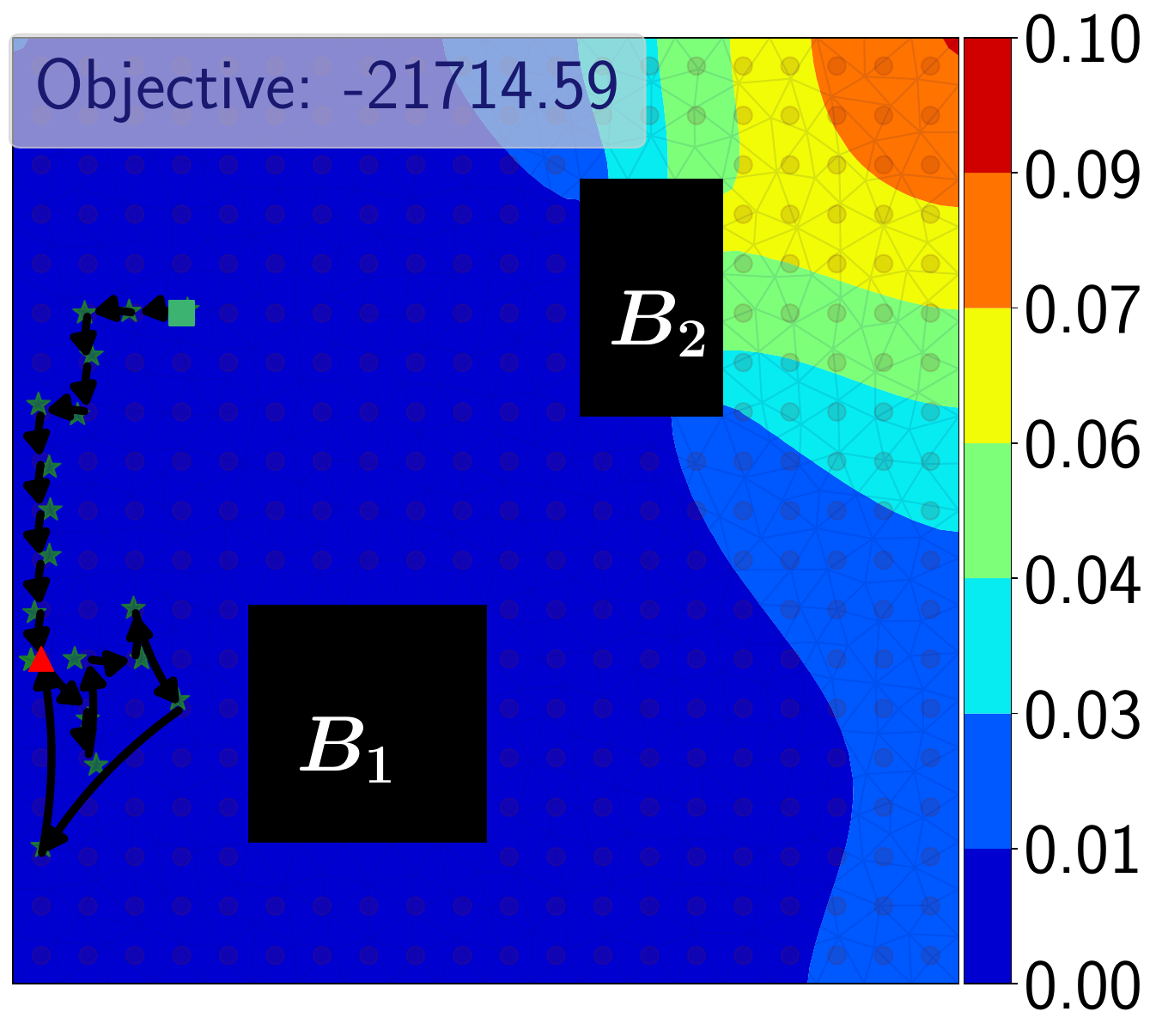}
        \caption{
          Optimal trajectories corresponding to results shown in 
          \Cref{sup:fig:fine_generalized_higher_order_fixed_start_point_1_sensor}.
        }\label{sup:fig:fine_generalized_higher_order_fixed_start_point_1_sensor_optimal_trajectories}
      \end{figure}

    \subsubsection{Results with $7$ Moving Sensors}
      \label{sup:subsubsec:fine_mesh_results_fixed_start_point_7_sensors}
      This case simulates $s=7$ sensors (e.g., drones) moving together. 

      \Cref{sup:fig:fine_first_order_fixed_start_point_7_sensors} 
      shows the results of \Cref{alg:probabilistic_path_optimization} with the first-order 
      policy defined by \Cref{defn:first_order_path_model}.
      \begin{figure}[H]
        \centering
        \includegraphics[width=0.53\linewidth]{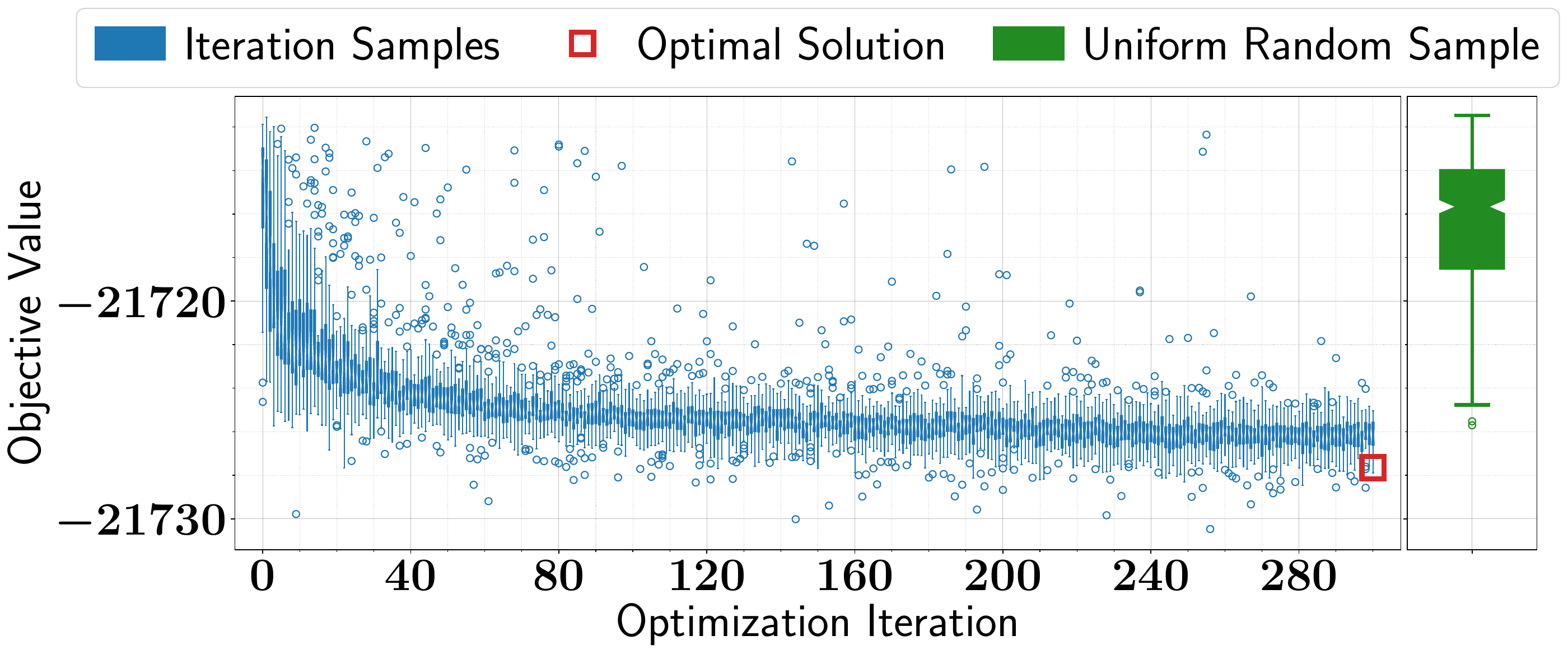}
        \includegraphics[width=0.215\linewidth]{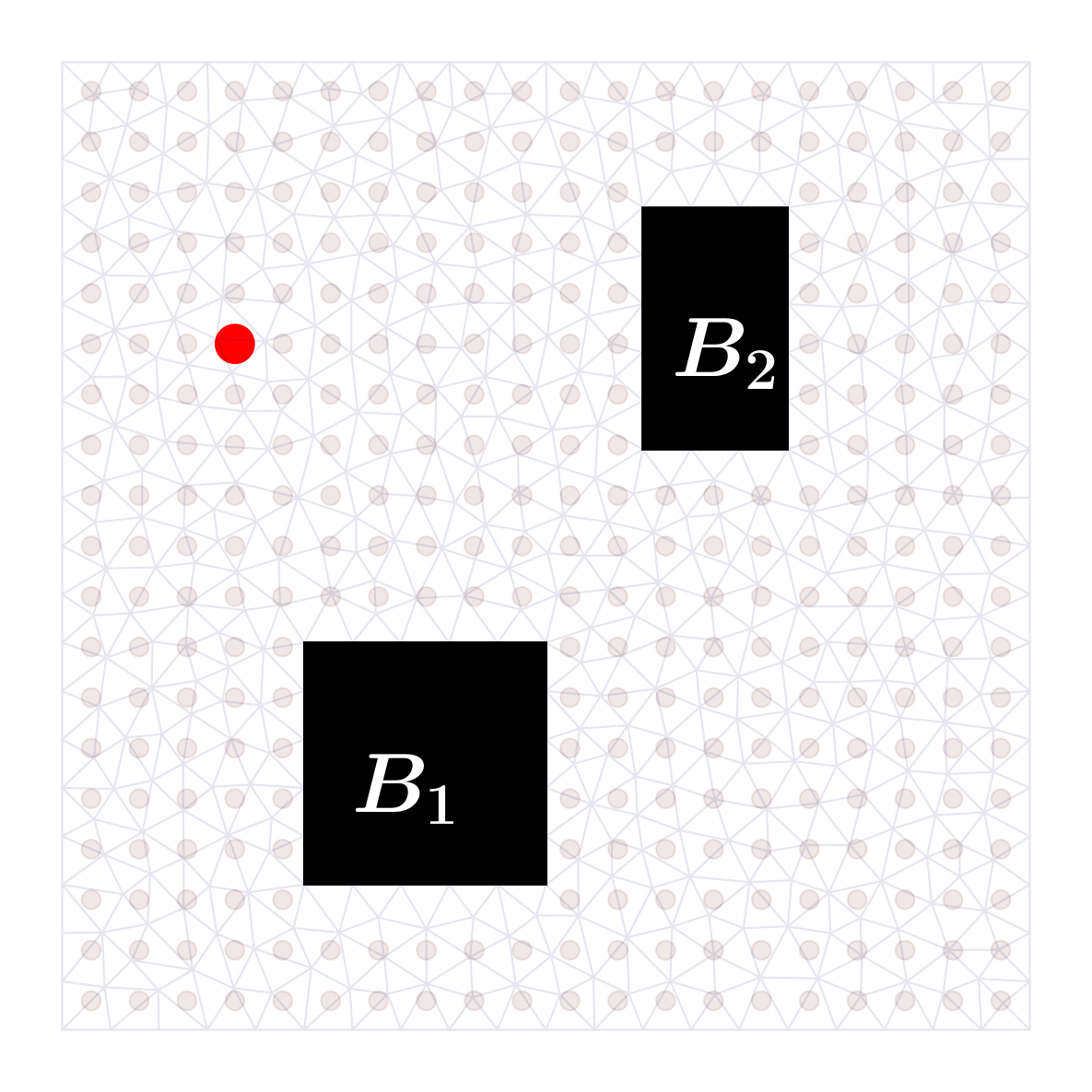}
        \includegraphics[width=0.235\linewidth]{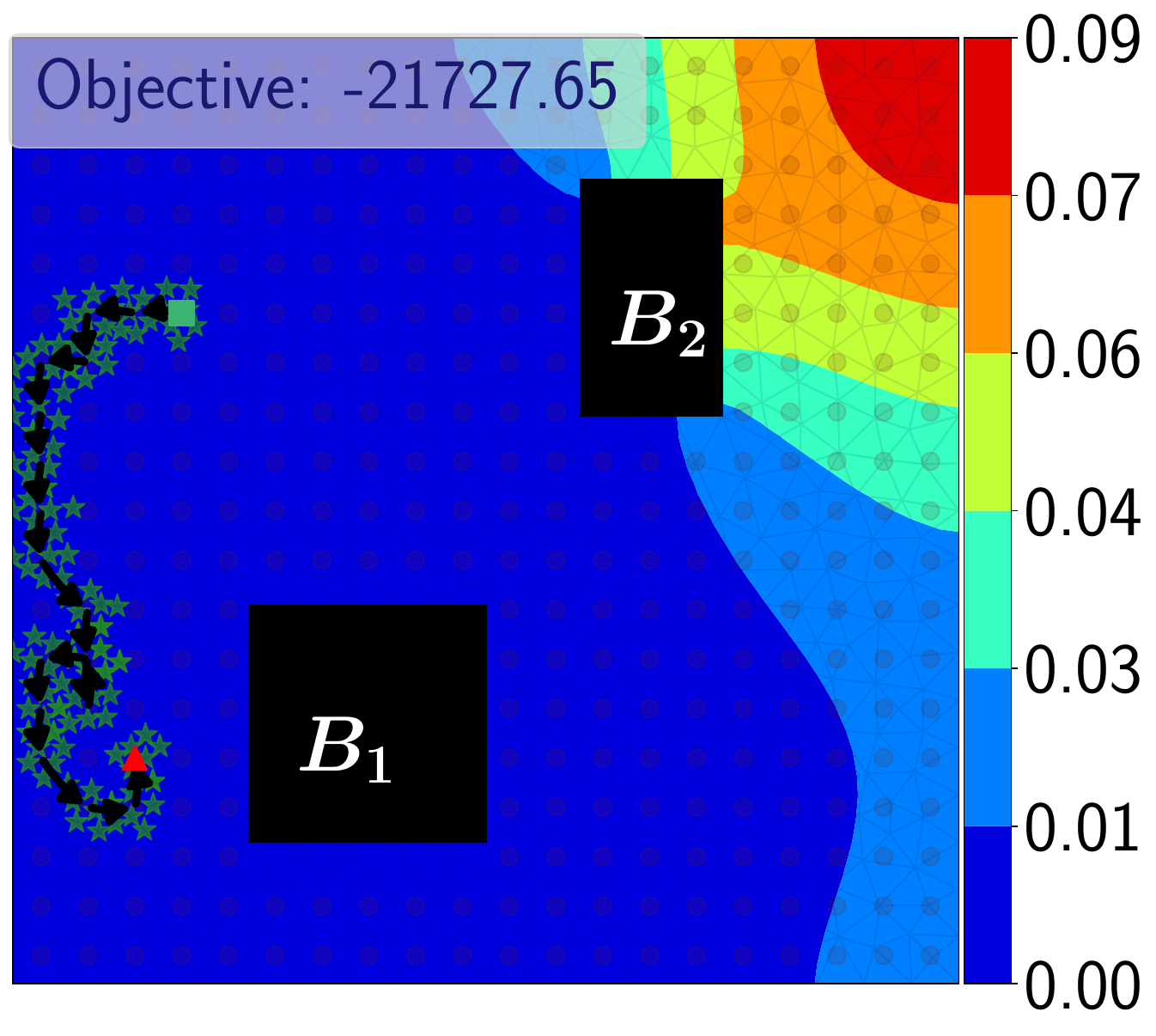}
        \caption{
          Results of \Cref{alg:probabilistic_path_optimization} with the first-order 
          policy model (\Cref{defn:first_order_path_model}) applied to the fine
          navigation mesh (\Cref{fig:navigation_meshes}, right) with 
          the starting point of the trajectory fixed to $(x, y)=(0.2, 0.7)$.
          Trajectory length is $n=19$, and the sensor size is $s=7$.
        }\label{sup:fig:fine_first_order_fixed_start_point_7_sensors}
      \end{figure}

      \Cref{sup:fig:fine_higher_order_fixed_start_point_7_sensors} 
      shows the performance of the optimization procedure with the higher-order 
      policy given by \Cref{defn:higher_order_path_model} with order $k=3$ and $k=5$, respectively.
      The resulting optimal initial parameter (first row) and optimal trajectories (second row)
      are shown in 
      \Cref{sup:fig:fine_higher_order_fixed_start_point_7_sensors_optimal_trajectories}.

      \begin{figure}[H]
        \centering
        \includegraphics[width=0.495\linewidth]{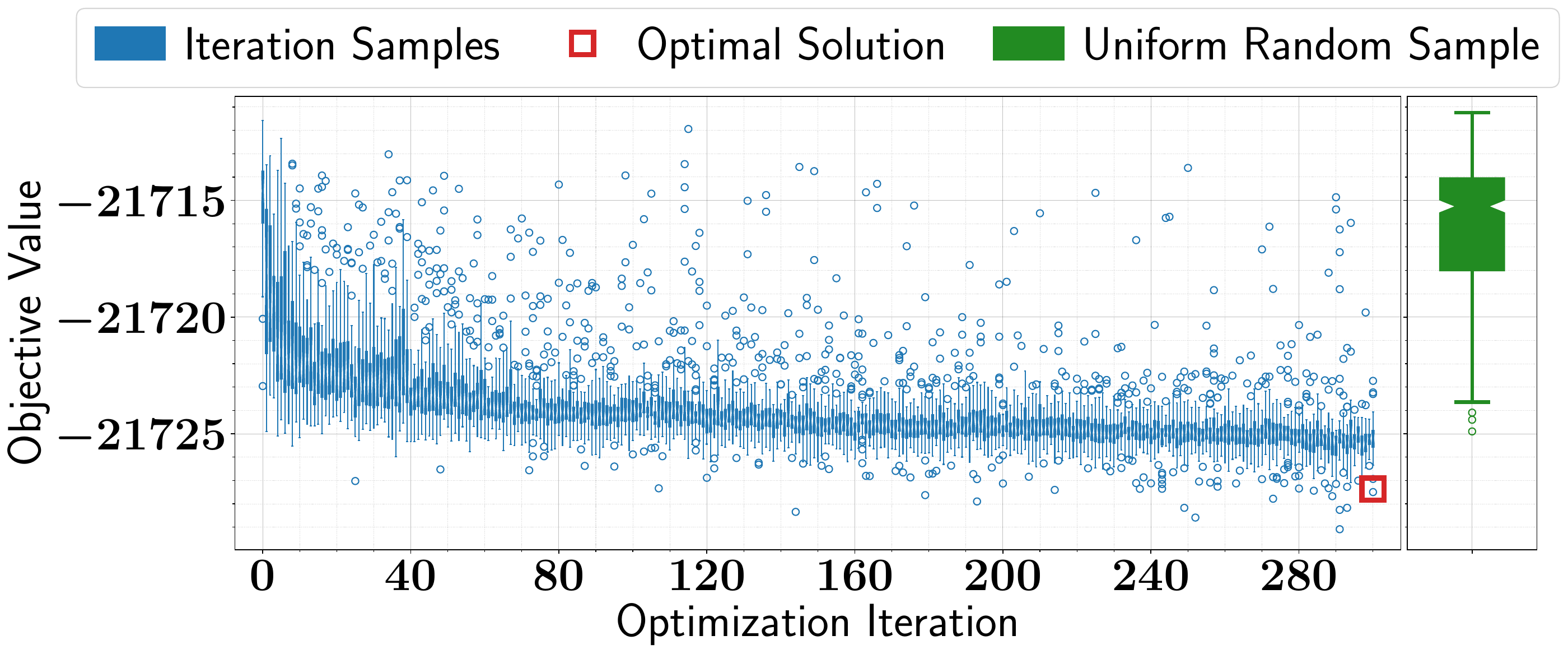}
        \includegraphics[width=0.495\linewidth]{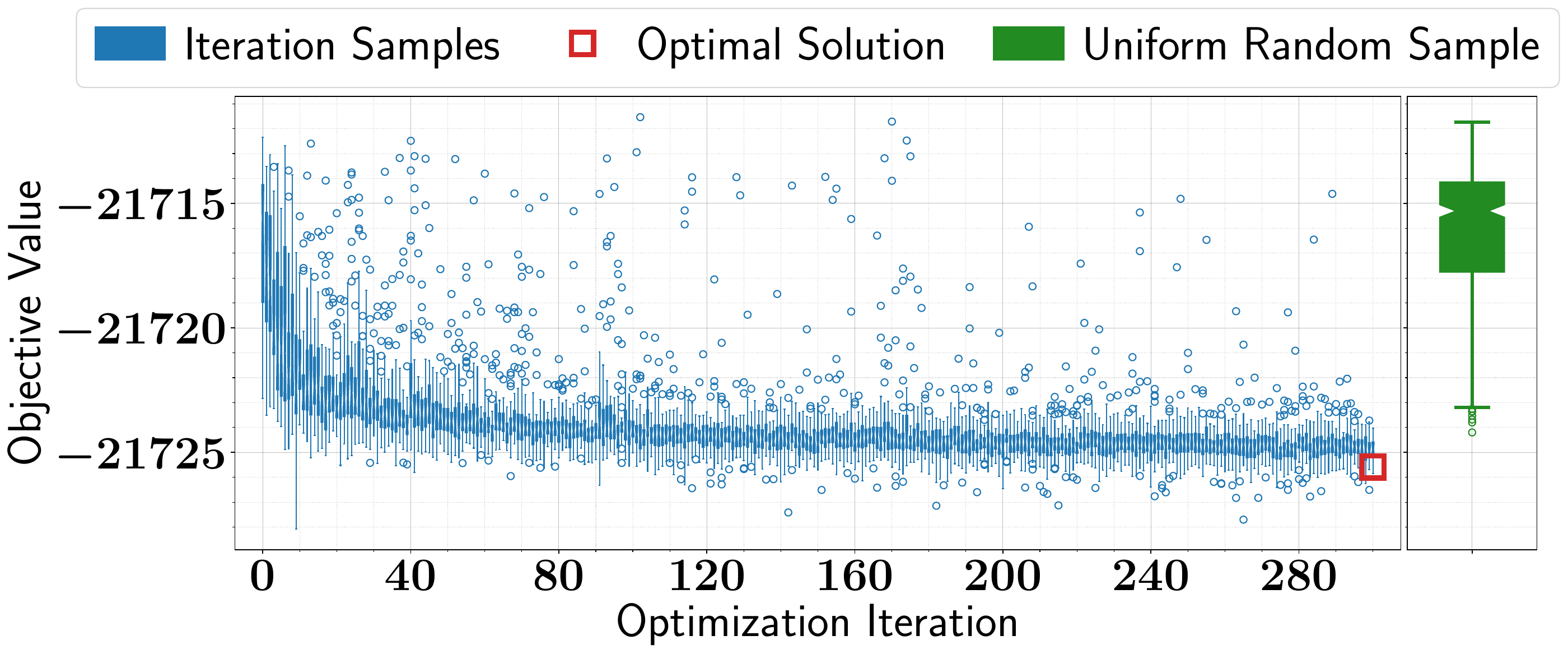}
        \includegraphics[width=0.495\linewidth]{plot_124_fine_higher_order_fixed_start_point_7_sensors_Reinforce_Optimization_Objective_Values_Box_with_Uniform_Random_Sample}
        \includegraphics[width=0.495\linewidth]{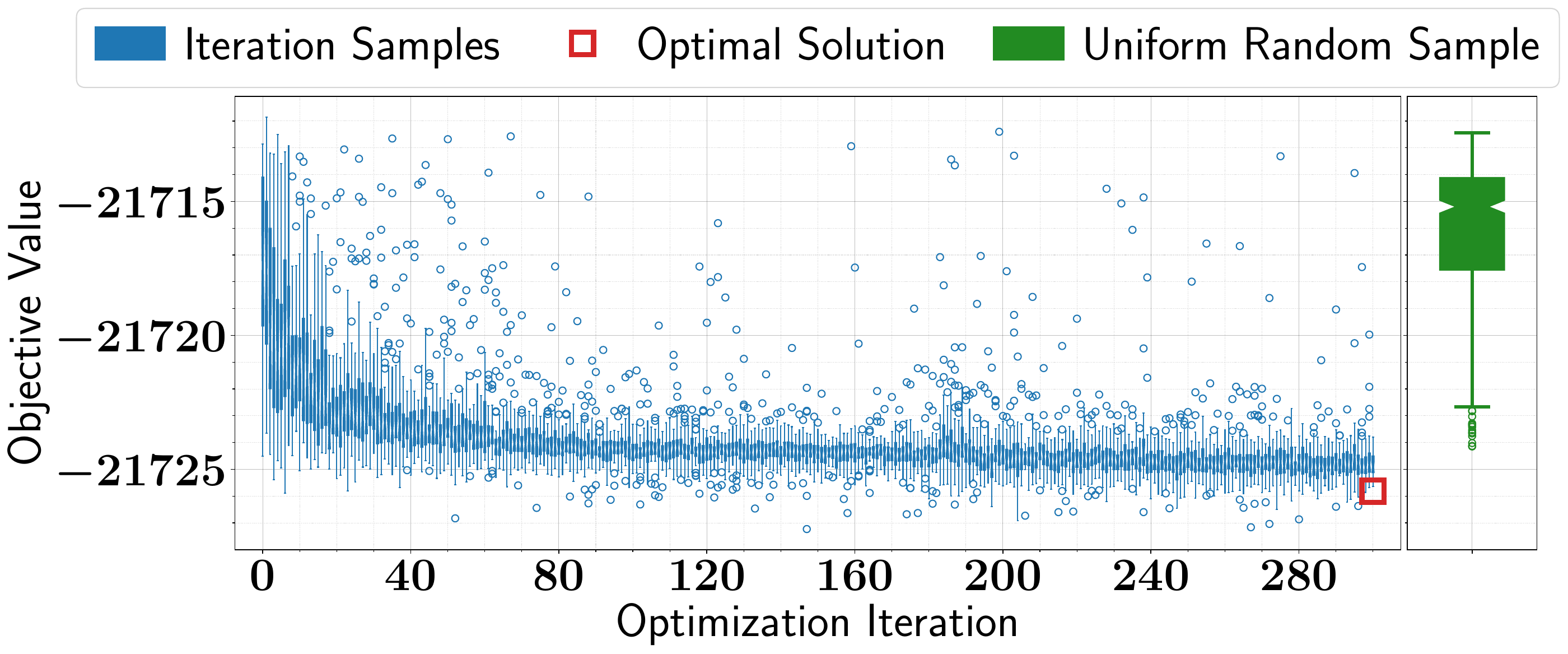}
        \caption{
          Results of \Cref{alg:probabilistic_path_optimization}  applied to the fine
          navigation mesh (\Cref{fig:navigation_meshes}, right) with 
          the starting point of the trajectory fixed to $(x, y)=(0.2, 0.7)$.
          Trajectory length is $n=19$, and the sensor size is $s=7$.
          Here the higher-order policy 
          (\Cref{defn:higher_order_path_model}) is used 
          with orders $k=3$ (first row) and $k=5$ (second row).
          The first column shows results with lag weights being optimized, and 
          the second column shows results with lag weights modeled 
          by \eqref{eqn:decreasing_lag_weights}.
        }\label{sup:fig:fine_higher_order_fixed_start_point_7_sensors}
      \end{figure}
      \begin{figure}[H]
        \centering
        \includegraphics[width=0.235\linewidth]{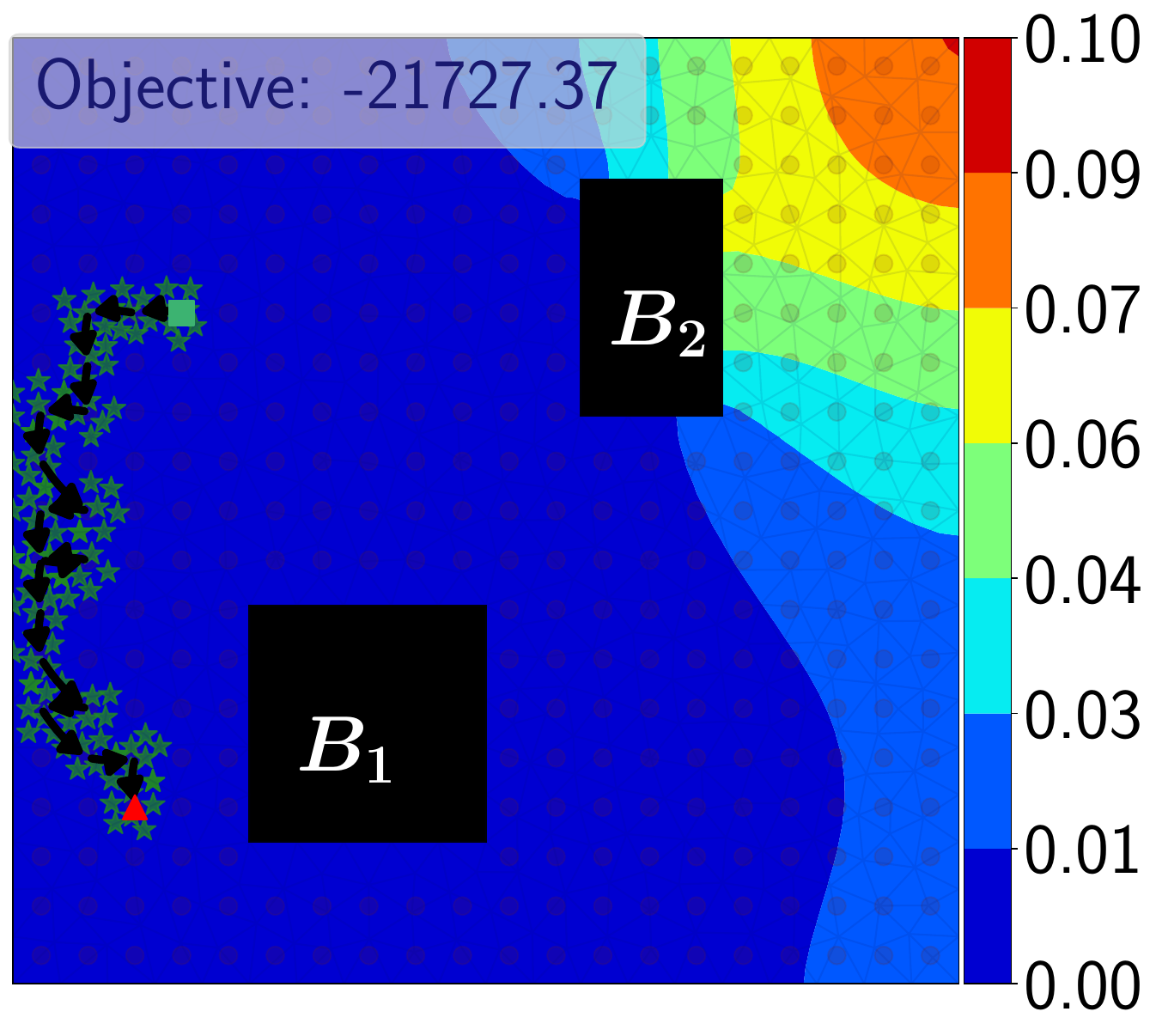}
        \includegraphics[width=0.235\linewidth]{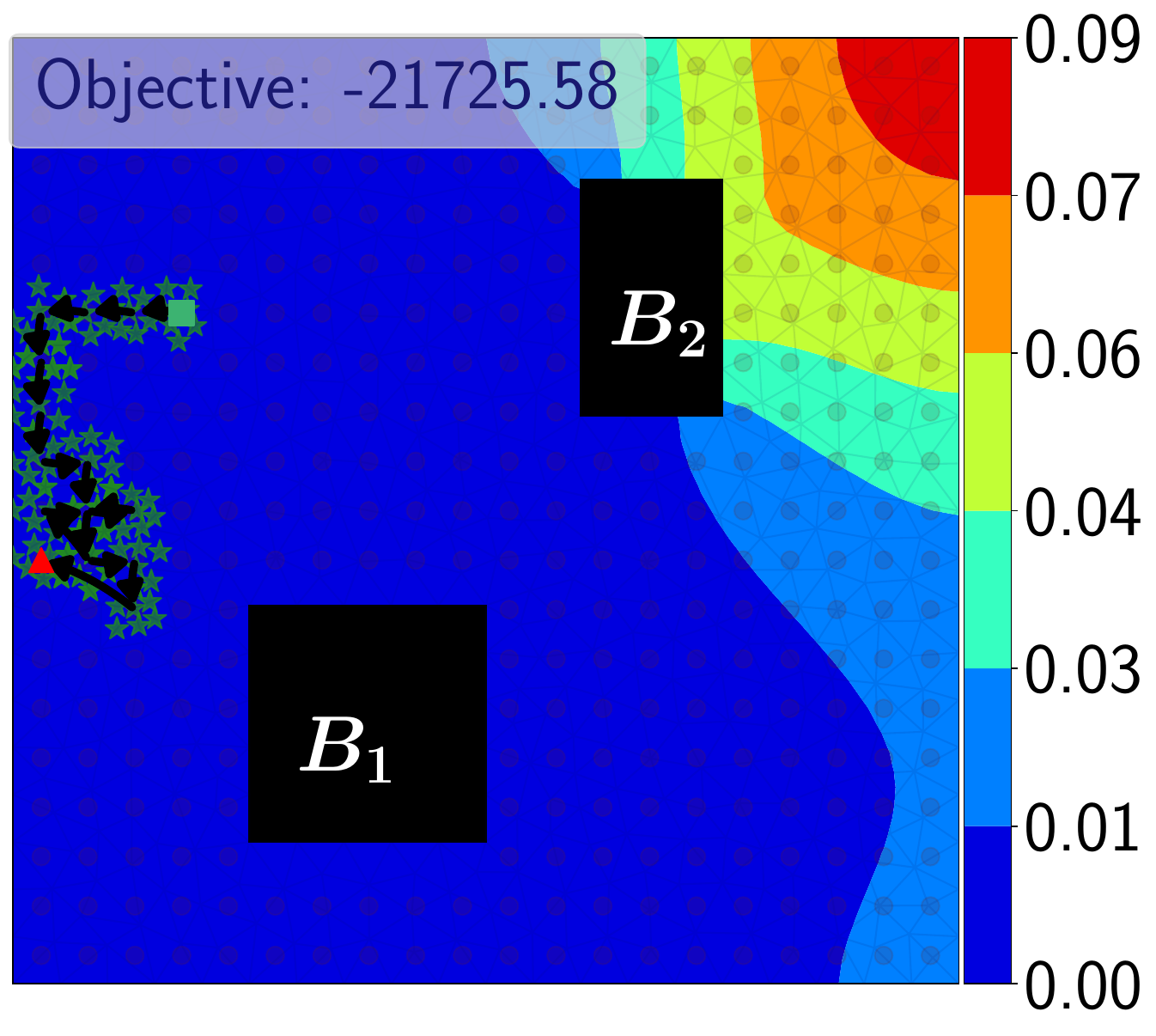}
        \\
        \includegraphics[width=0.235\linewidth]{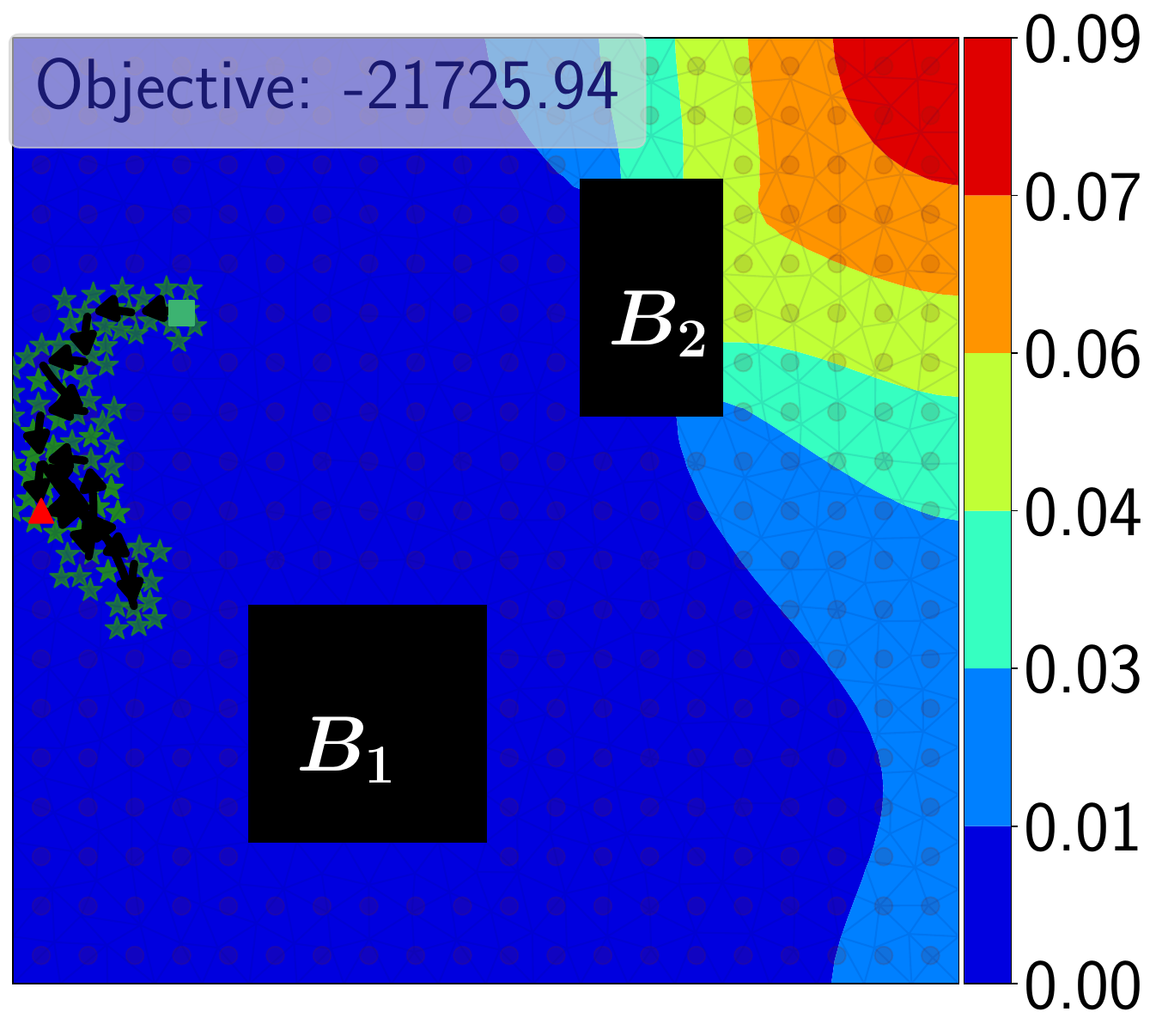}
        \includegraphics[width=0.235\linewidth]{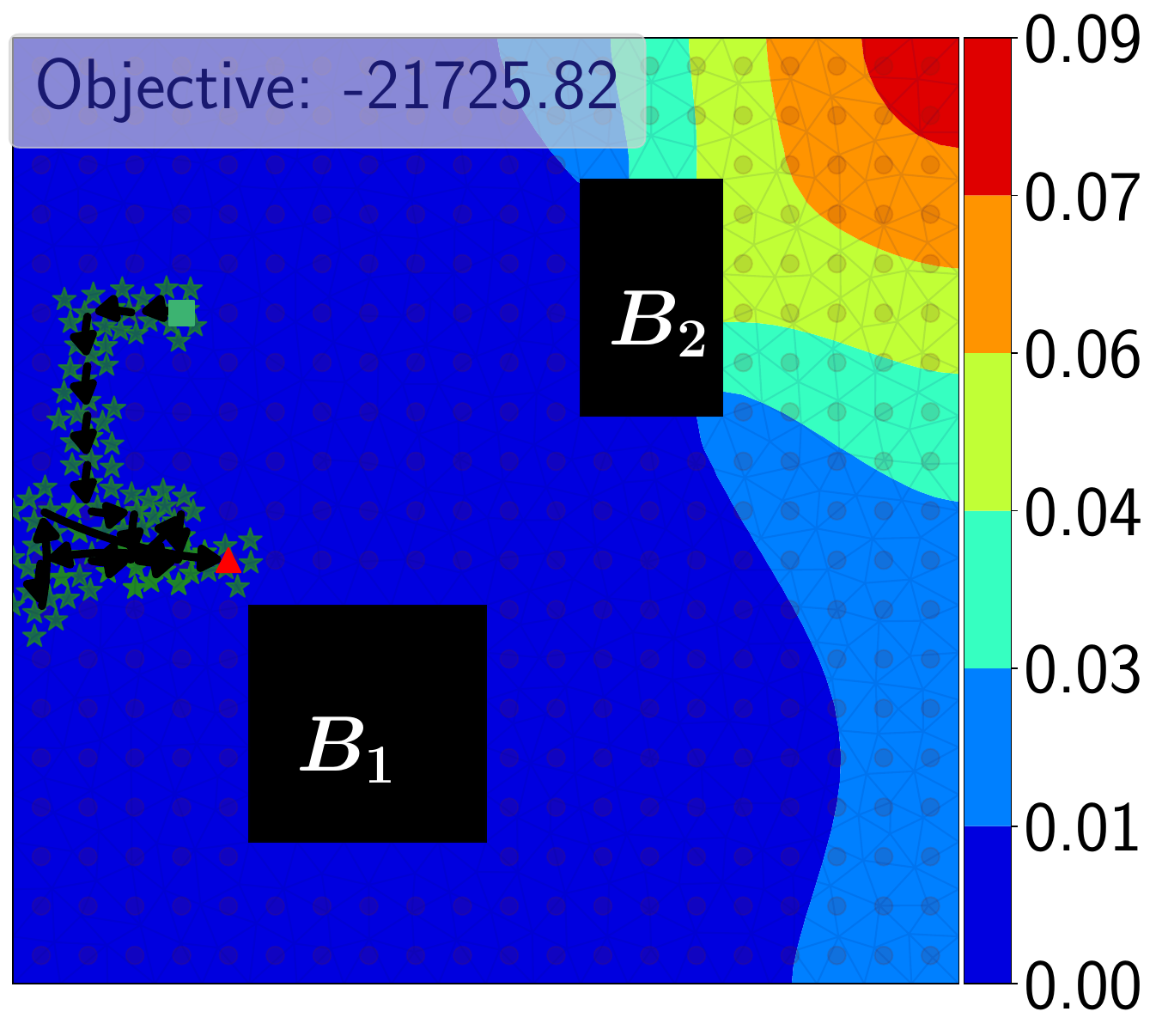}
        \caption{
          Optimal trajectories corresponding to results shown in 
          \Cref{sup:fig:fine_higher_order_fixed_start_point_7_sensors}.
        }\label{sup:fig:fine_higher_order_fixed_start_point_7_sensors_optimal_trajectories}
      \end{figure}
      Similarly, results obtained by the policy given by  \Cref{defn:generalized_higher_order_path_model} 
      are shown in 
      \Cref{sup:fig:fine_generalized_higher_order_fixed_start_point_7_sensors} 
      and 
      \Cref{sup:fig:fine_generalized_higher_order_fixed_start_point_7_sensors_optimal_trajectories}.

      \begin{figure}[H]
        \centering
        \includegraphics[width=0.495\linewidth]{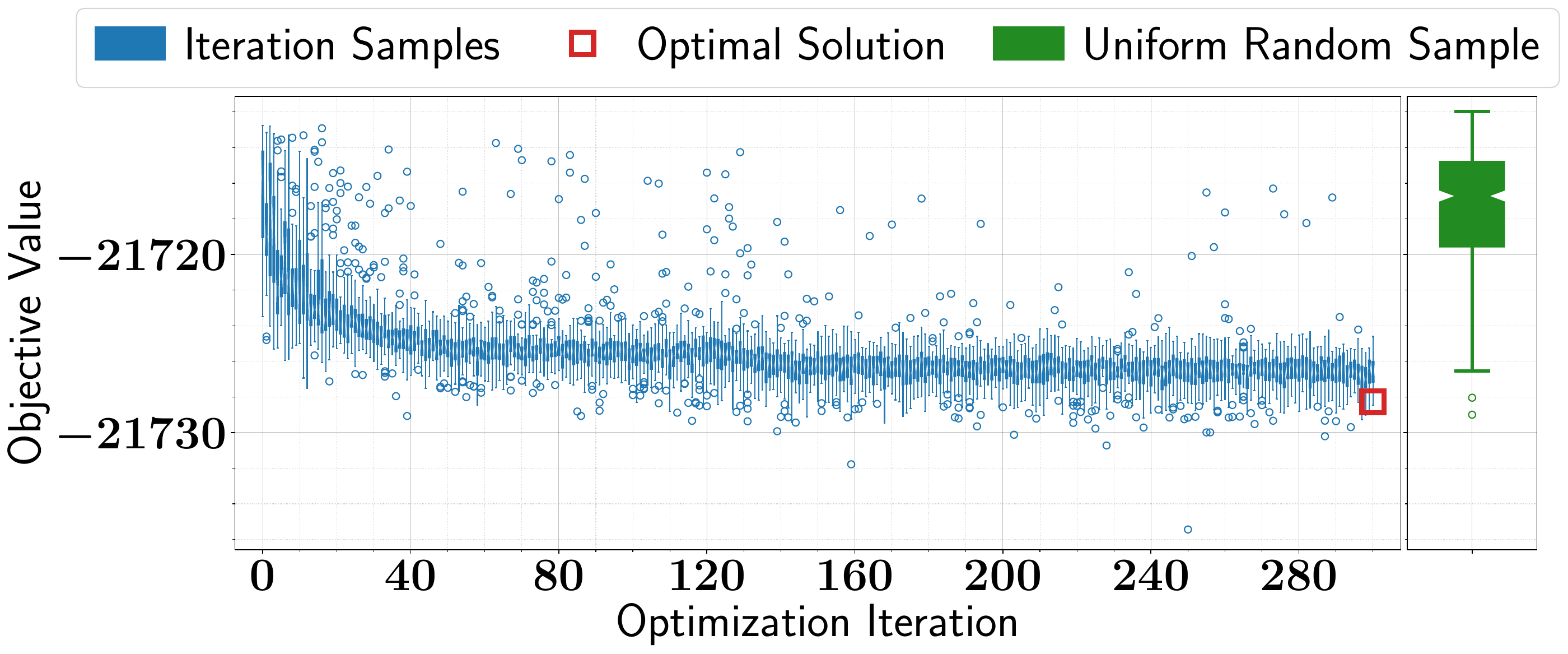}
        \includegraphics[width=0.495\linewidth]{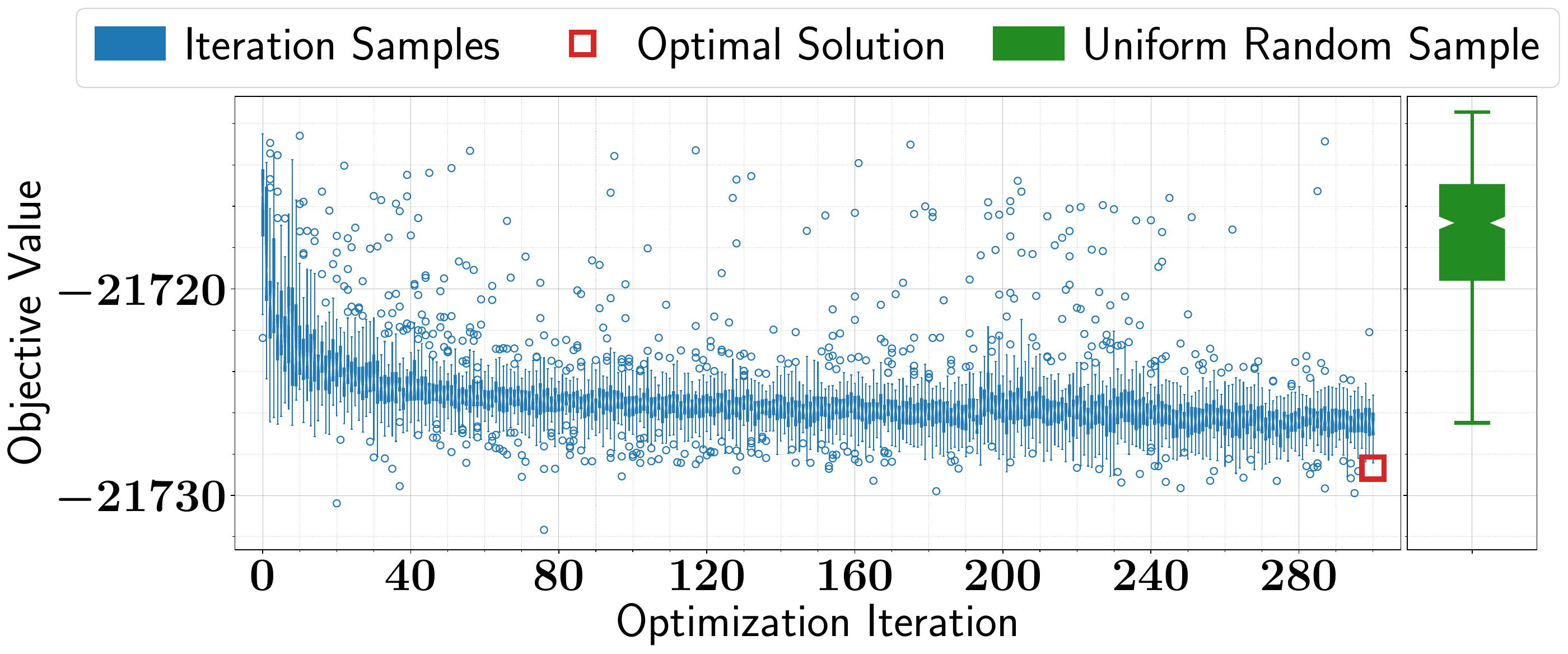}
        \includegraphics[width=0.495\linewidth]{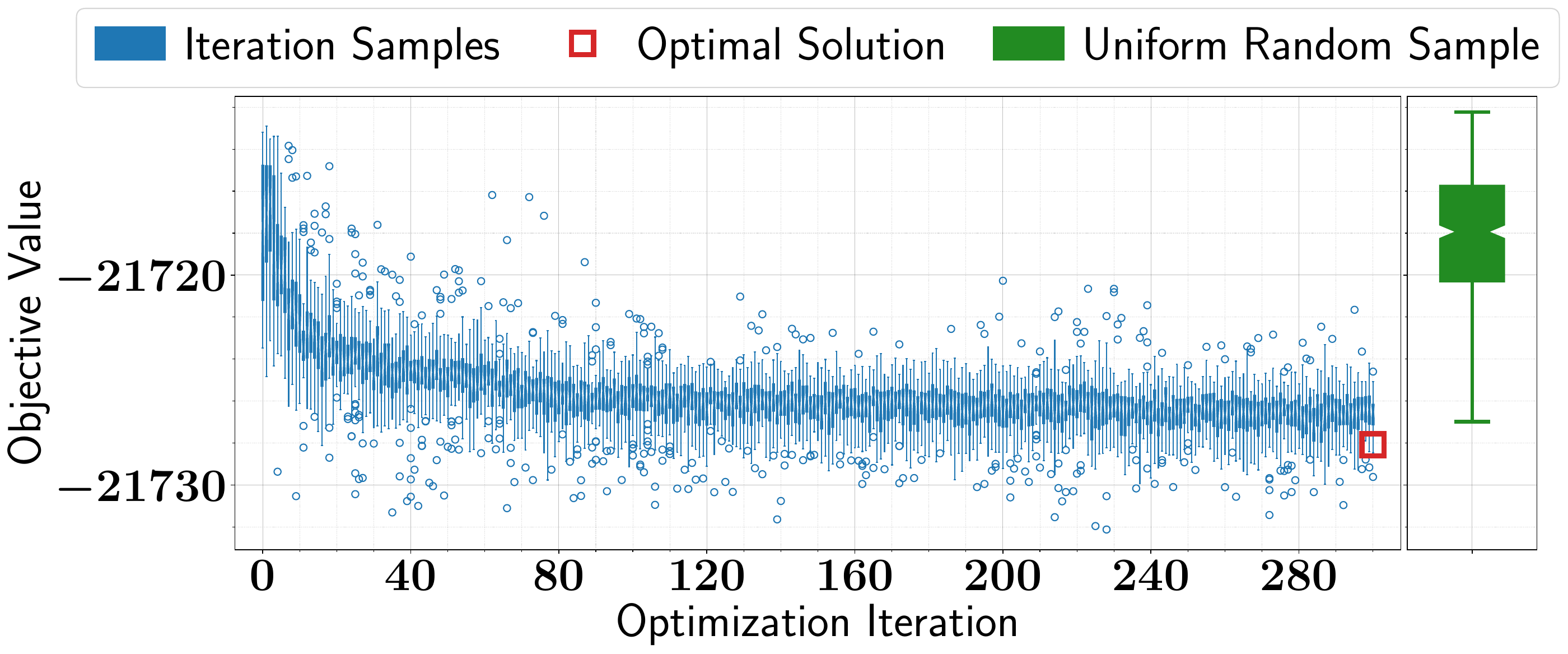}
        \includegraphics[width=0.495\linewidth]{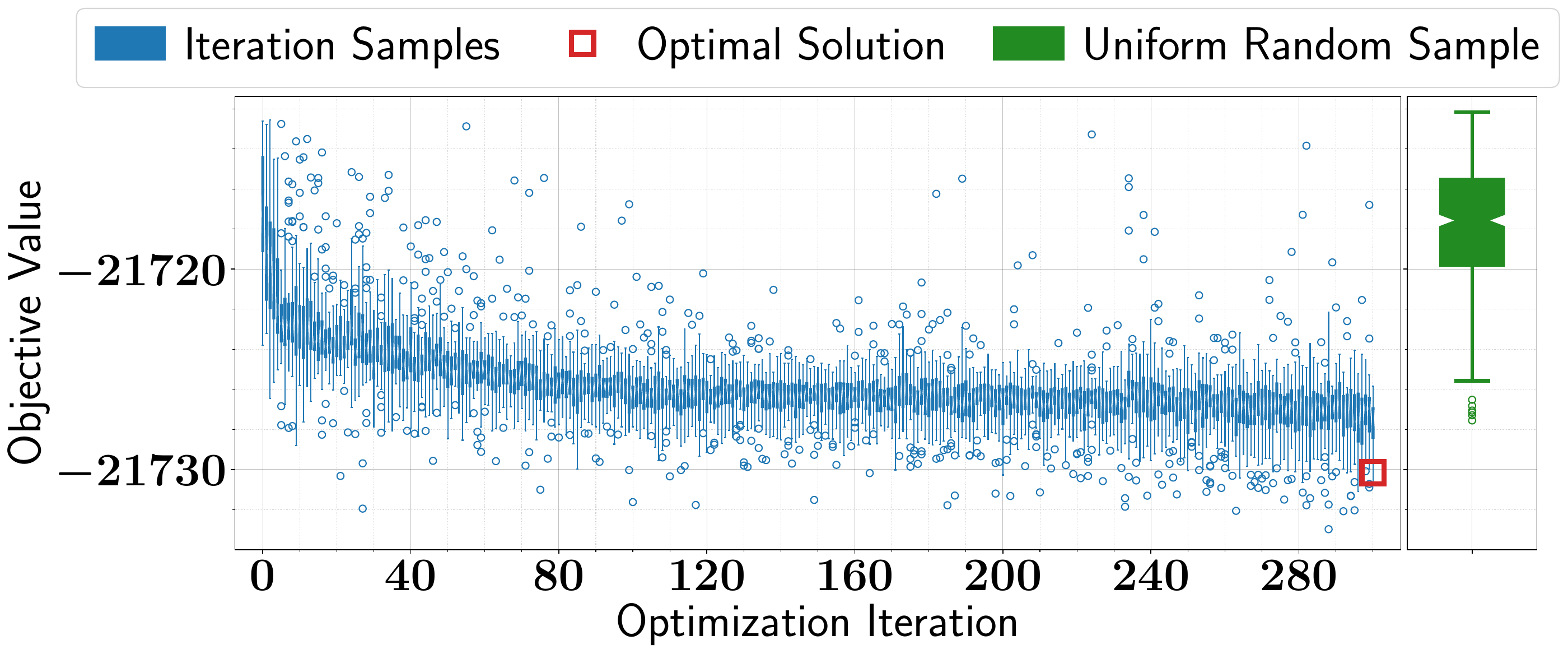}
        \caption{
          Results of \Cref{alg:probabilistic_path_optimization} with the generalized higher-order 
          policy model (\Cref{defn:generalized_higher_order_path_model}) applied to the fine
          navigation mesh (\Cref{fig:navigation_meshes}, right) with 
          the starting point of the trajectory fixed to $(x, y)=(0.2, 0.7)$.
          Trajectory length is $n=19$, and the sensor size is $s=7$.
          Here the generalized higher-order policy 
          (\Cref{defn:generalized_higher_order_path_model}) is used 
          with orders $k=3$ (first row) and $k=5$ (second row).
          The first column shows results with lag weights being optimized, and 
          the second column shows results with lag weights modeled 
          by \eqref{eqn:decreasing_lag_weights}.
        }\label{sup:fig:fine_generalized_higher_order_fixed_start_point_7_sensors}
      \end{figure}
      \begin{figure}[H]
        \centering
        \includegraphics[width=0.235\linewidth]{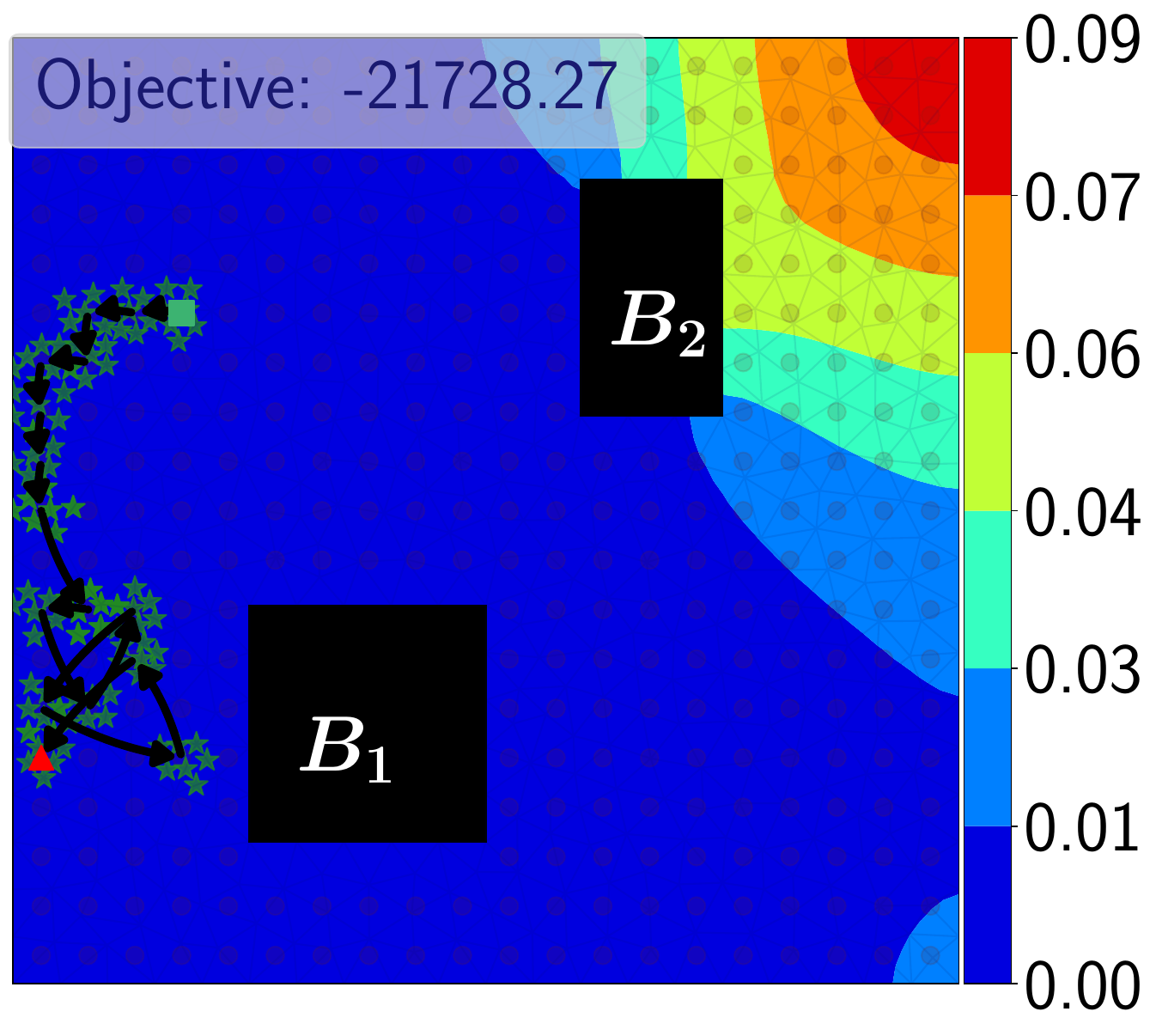}
        \includegraphics[width=0.235\linewidth]{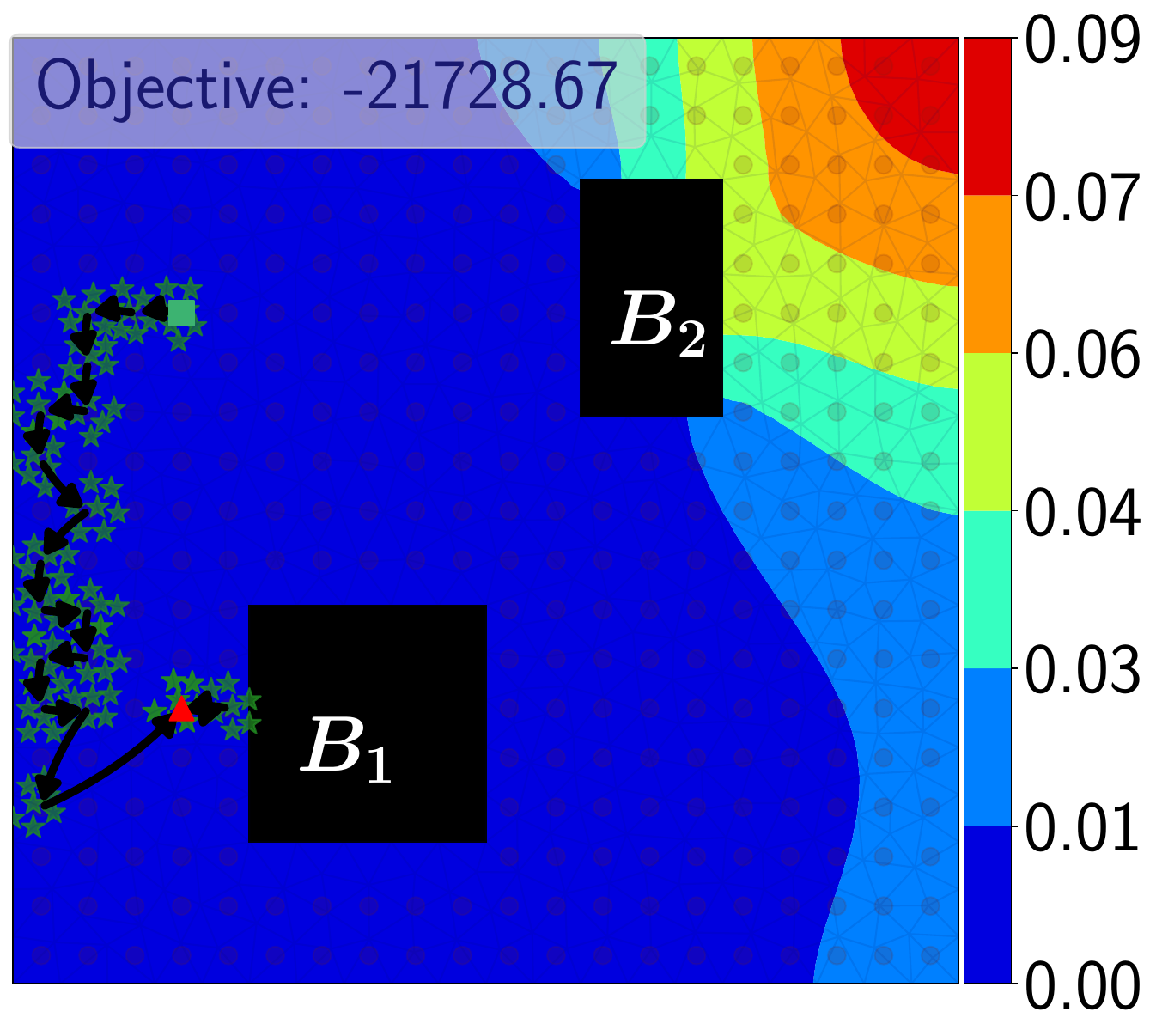}
        \\
        \includegraphics[width=0.235\linewidth]{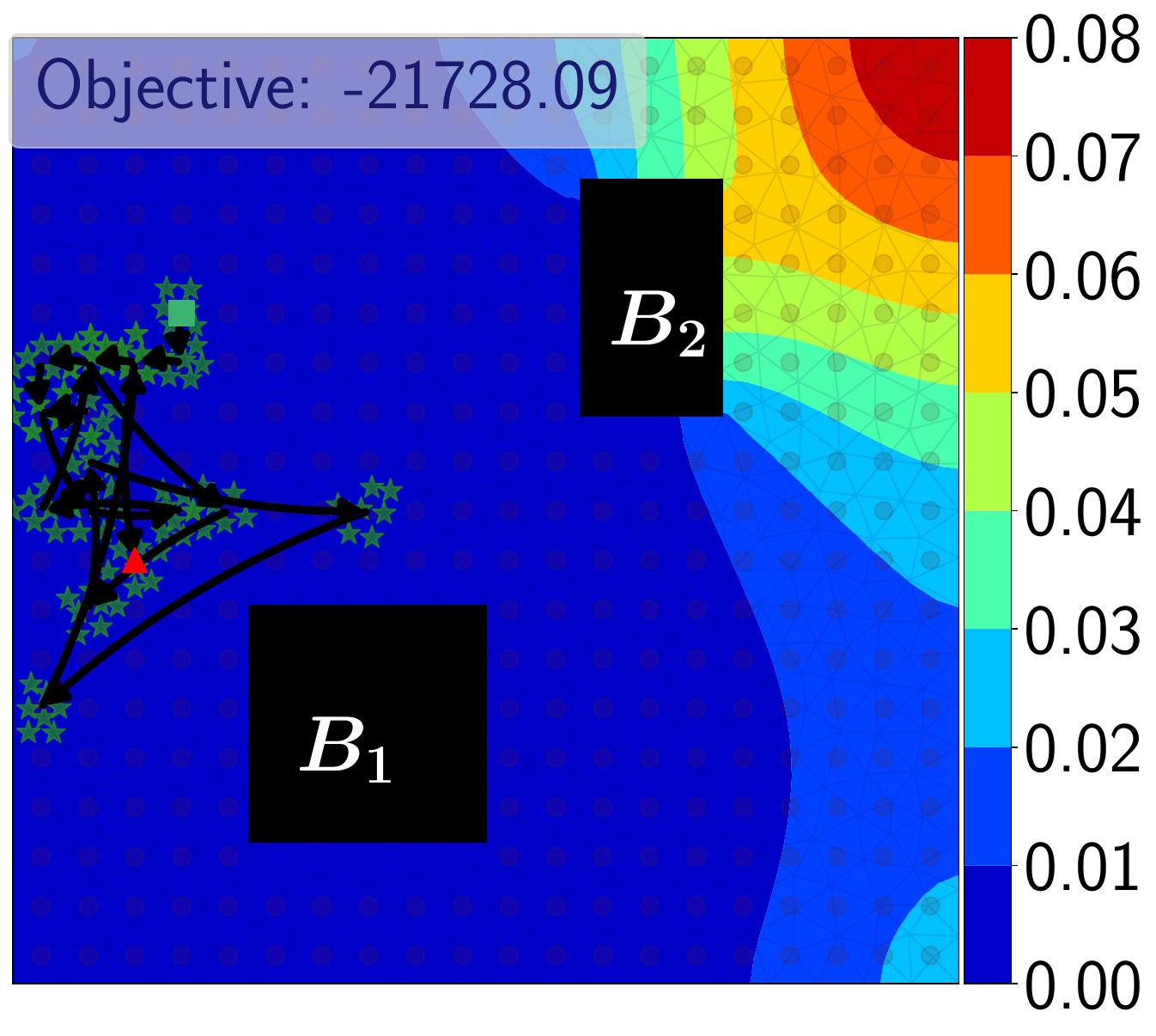}
        \includegraphics[width=0.235\linewidth]{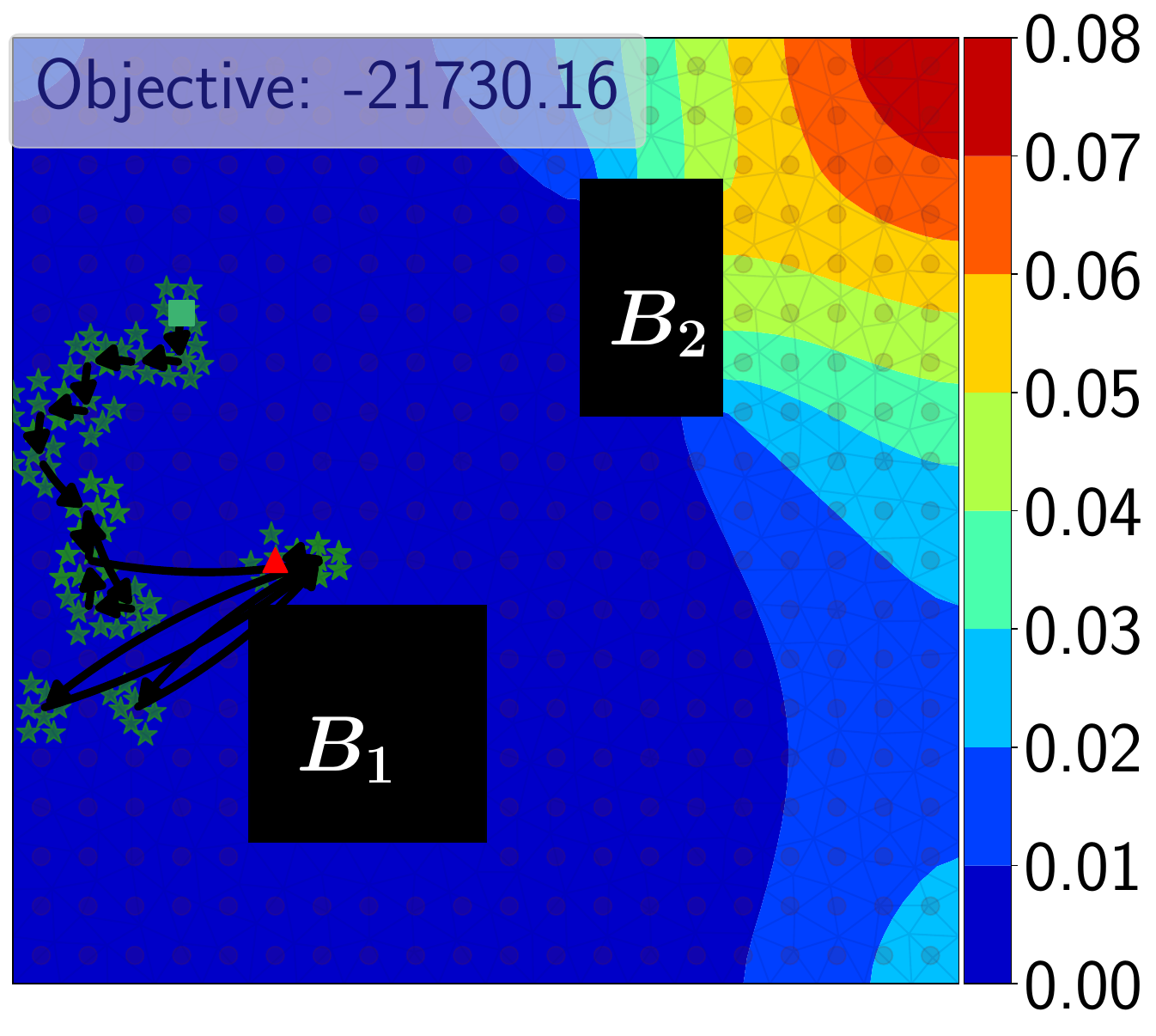}
        \caption{
          Optimal trajectories corresponding to results shown in 
          \Cref{sup:fig:fine_generalized_higher_order_fixed_start_point_7_sensors}.
        }\label{sup:fig:fine_generalized_higher_order_fixed_start_point_7_sensors_optimal_trajectories}
      \end{figure}

    The results displayed  in 
    \Cref{sup:subsubsec:fine_mesh_results_fixed_start_point_7_sensors} for $s=7$ moving sensors 
    are consistent with the results obtained with $s=1$ moving sensor in   
    \Cref{sup:subsubsec:fine_mesh_results_fixed_start_point_1_sensor}, 
    and with the coarse mesh experiment discussed in \Cref{subsec:numerical_results_coarse}.

  \subsection{Results with A- and E-Optimality Criteria}
    \label{sup:subsec:Other_Utility_Functions}
    In this section we show a partial set of results obtained by 
    using optimality criteria other than the D-optimality used primarily in 
    the numerical experiments.
    Specifically, we present in \Cref{sup:subsubsec:A-optimality_results} results obtained
    by using the A-optimality (trace/average of the posterior variances) criterion. 
    \Cref{sup:subsubsec:E-optimality_results} presents results obtained by employing 
    the E-optimality (eigenvalue) criterion.

    Similar to \Cref{subsubsec:fine_mesh_results_unspecified_start_point}, 
    this experiment employs the fine navigation mesh \Cref{fig:navigation_meshes} (right).
    The trajectory length is set to $n=19$ nodes (18 edges), and 
    the observation frequency is set to $f=1$. 
    The number of sensors is set to $s=1$ and $s=7$, respectively, and is stated clearly in each figure.
    First we show results  with the starting point of the path restricted  
    to the coordinates $(x, y)=(0.2, 0.7)$, and then 
    we allow the trajectory to start anywhere in the domain,  allowing the 
    starting point to be optimized as well.

    For clarity, we show results obtained only by using the first-order policy 
    (\Cref{defn:first_order_path_model}). We note that the same behavior as 
    in the case of D-optimality is observed here for higher-order policies in 
    comparison with the first-order policy.

    \subsubsection{Results with A-Optimality Criterion}
    \label{sup:subsubsec:A-optimality_results}
      For the A-optimality, the objective is to minimize the sum of posterior variances.
      The probabilistic OED optimization
      (\Cref{prbl:probabilistic_path_oed_problem}) in this case takes the form
      \begin{equation}\label{eqn:A-opt_probabilistic_optimization}
        \hyperparamvec\opt \in \argmin_{\designvec\sim\CondProb{\designvec}{\hyperparamvec}}
          \Expect{\designvec\sim\CondProb{\designvec}{\hyperparamvec}}{\utilityfunc(\designvec)}
          \,; \qquad \utilityfunc(\designvec) := \Trace{\Cparampost(\designvec)}
            \,,
      \end{equation}
      where $\Cparampost(\designvec)$ is the posterior covariance matrix \eqref{eqn:Posterior_Params}.

      \paragraph{Results with fixed starting point}
        \Cref{sup:fig:a_opt_fixed_start_1_sensor} 
        shows the results of \Cref{alg:probabilistic_path_optimization} with the first-order 
        policy defined by \Cref{defn:first_order_path_model} and with the number of 
        sensors set to $s=1$. Results obtained with $s=7$ moving sensors are shown in 
        \Cref{sup:fig:a_opt_fixed_start_7_sensors}. 
        \begin{figure}[H]
          \centering
          \includegraphics[width=0.53\linewidth]{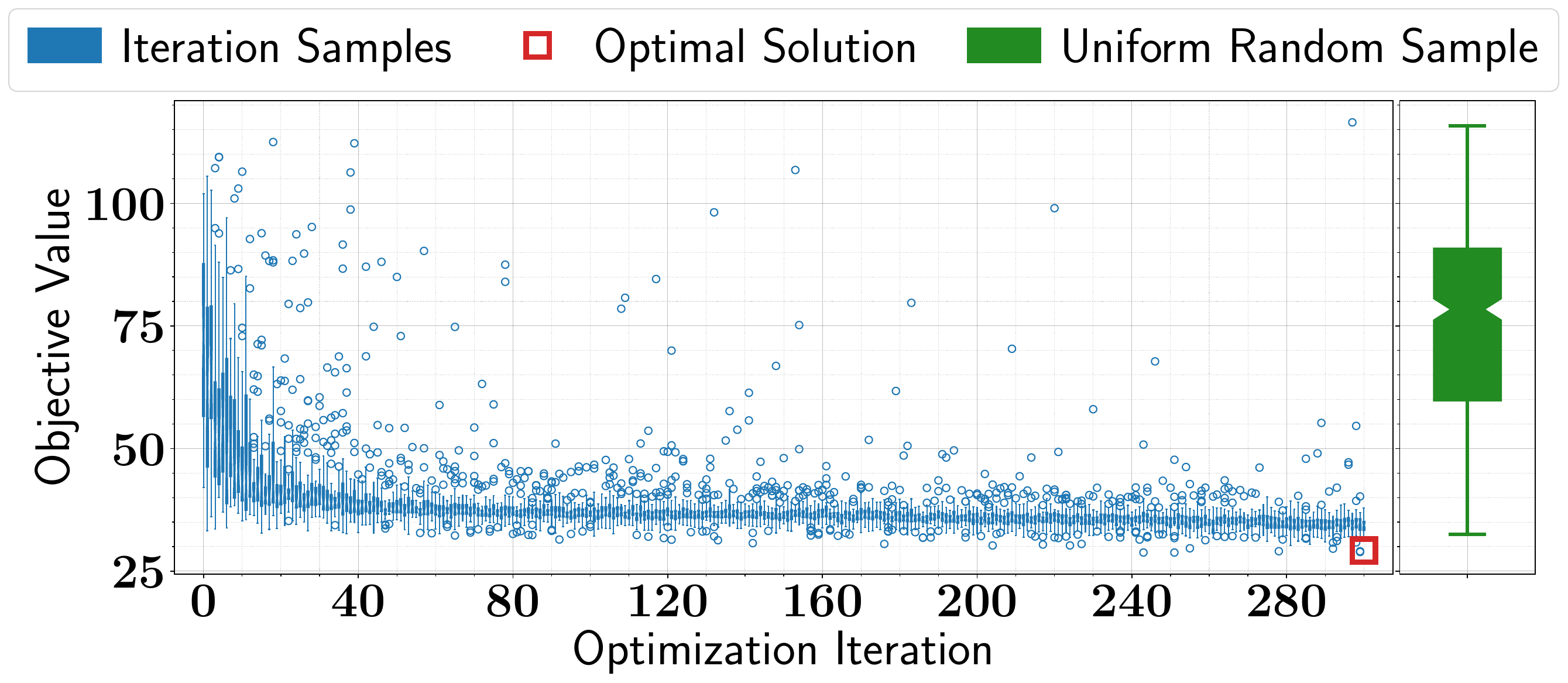}
          \includegraphics[width=0.215\linewidth]{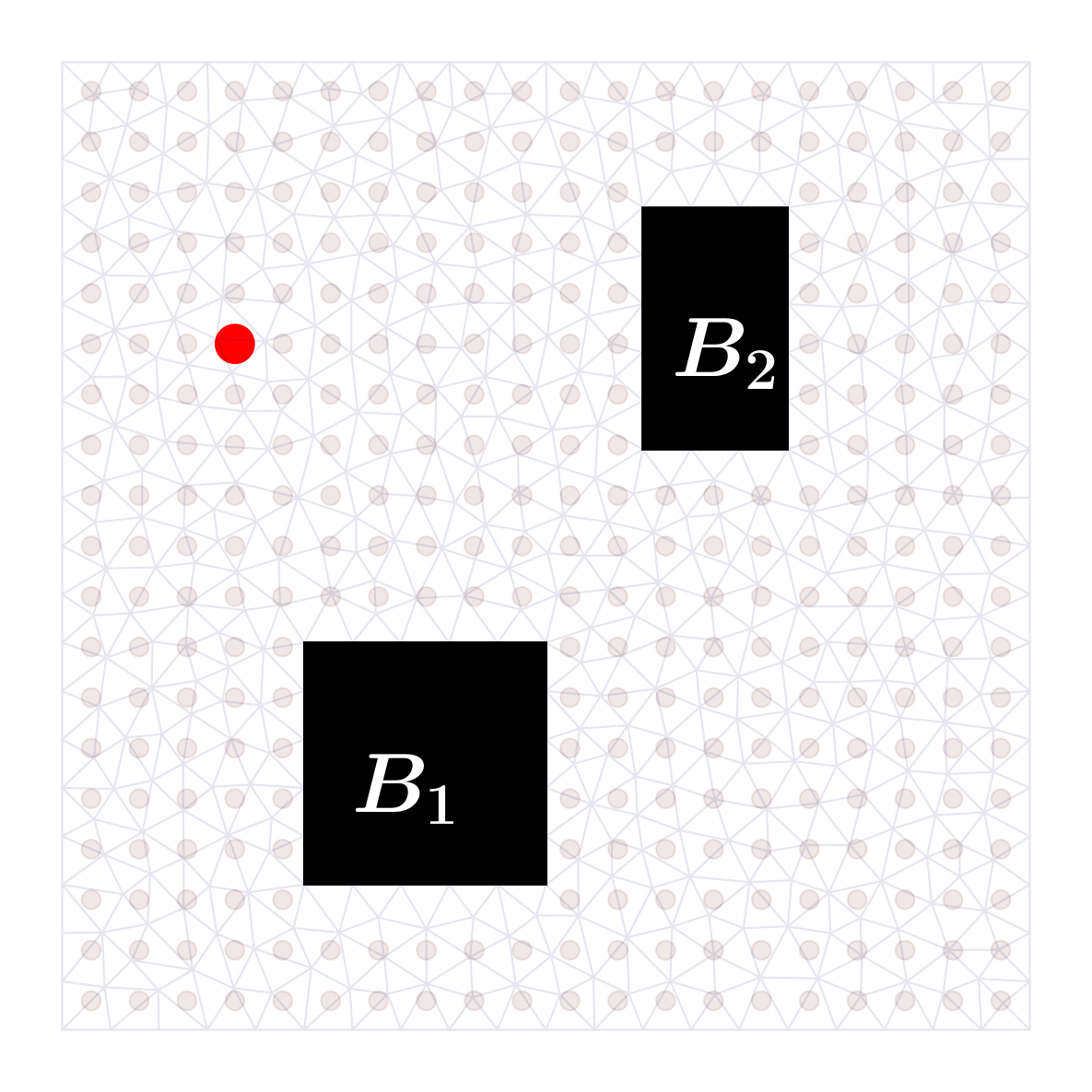}
          \includegraphics[width=0.235\linewidth]{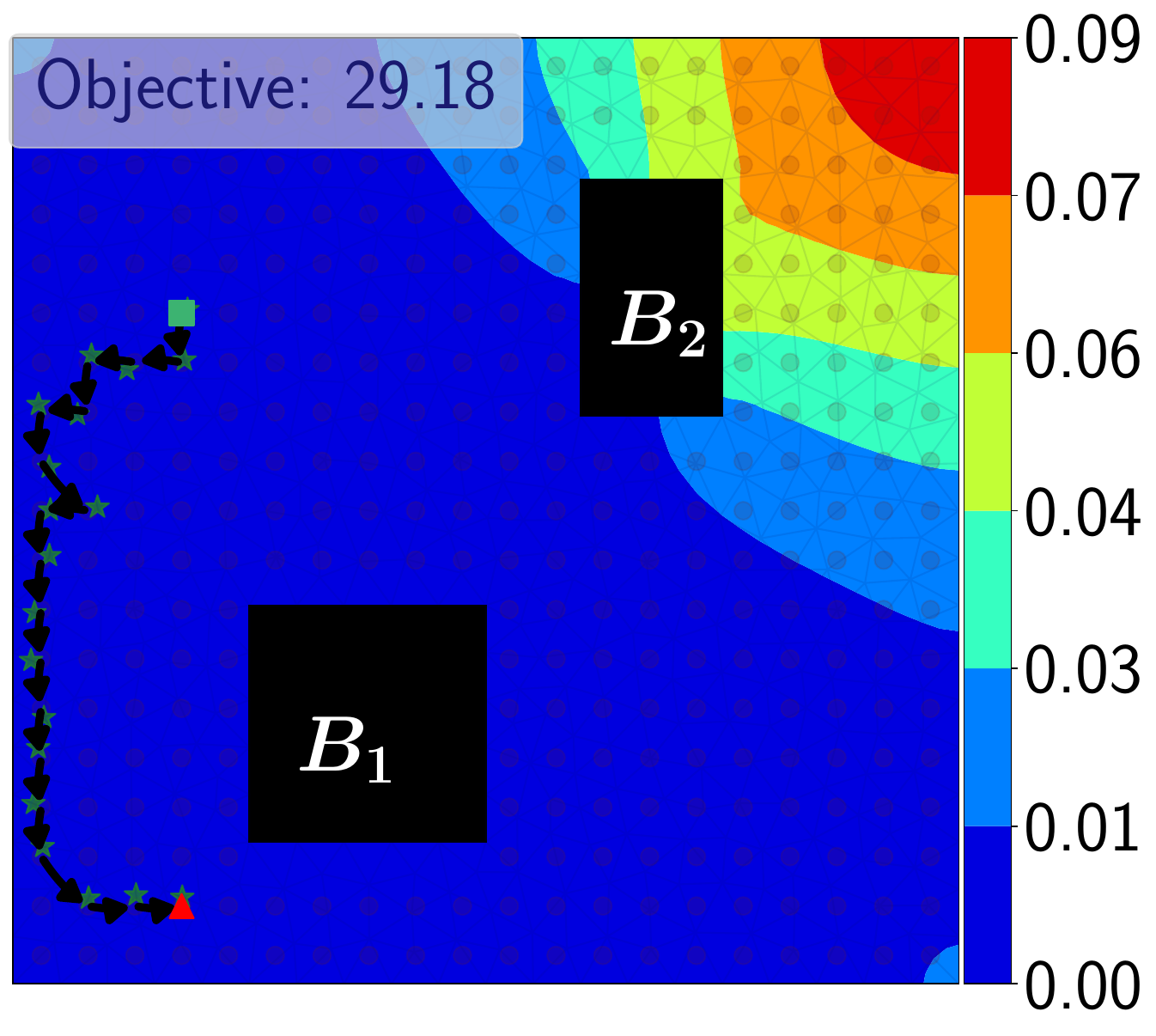}
          \caption{
            Results of \Cref{alg:probabilistic_path_optimization} with the first-order 
            policy model (\Cref{defn:first_order_path_model}) applied to the fine
            navigation mesh (\Cref{fig:navigation_meshes}, right) with 
            the starting point of the trajectory fixed to $(x, y)=(0.2, 0.7)$.
            The optimality criterion used is the A-optimality, 
            that is, the trace of the posterior covariance matrix.
            The number of sensors is set to $s=1$, the observation frequency is $f=1$, 
            and the trajectory length is $n=19$.
          }\label{sup:fig:a_opt_fixed_start_1_sensor}
        \end{figure}
        \begin{figure}[H]
          \centering
          \includegraphics[width=0.53\linewidth]{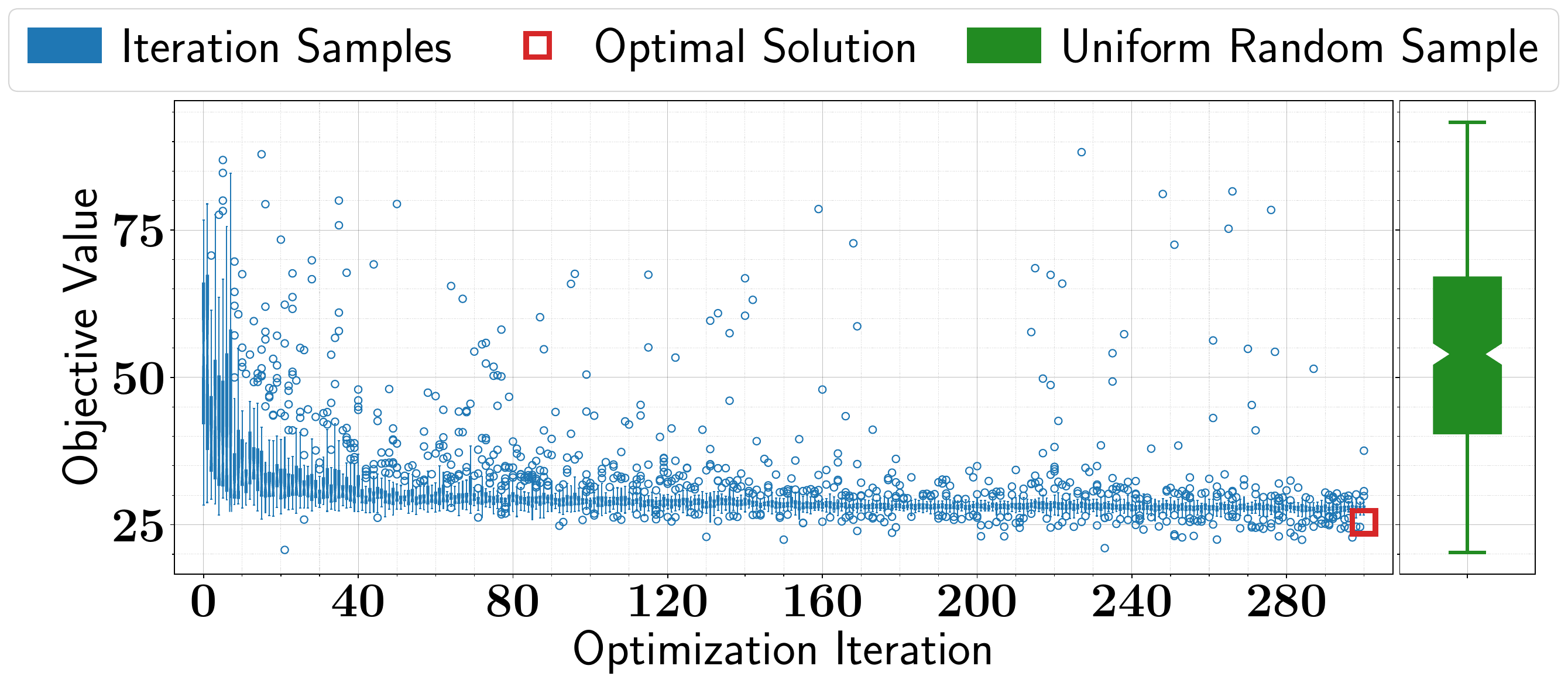}
          \includegraphics[width=0.215\linewidth]{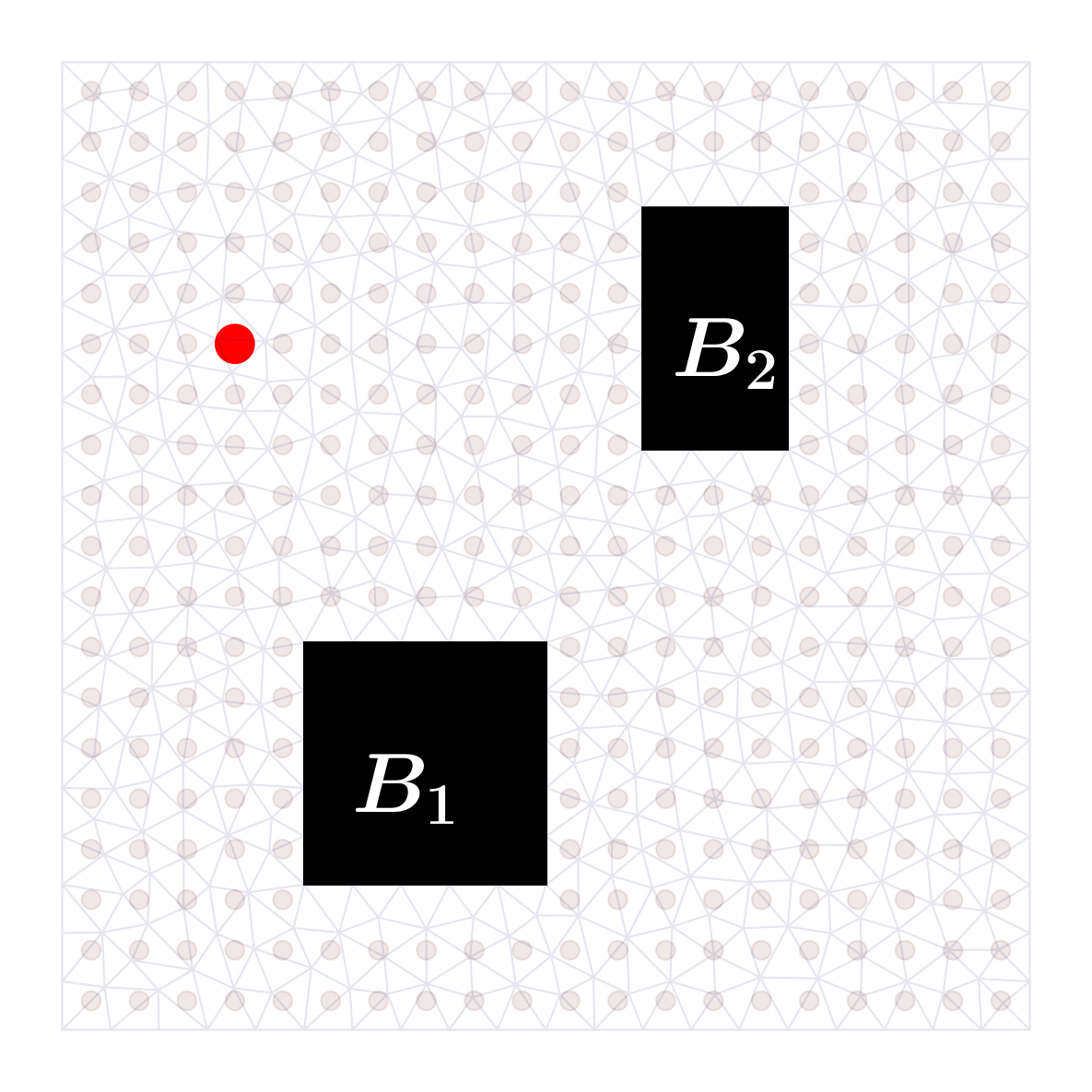}
          \includegraphics[width=0.235\linewidth]{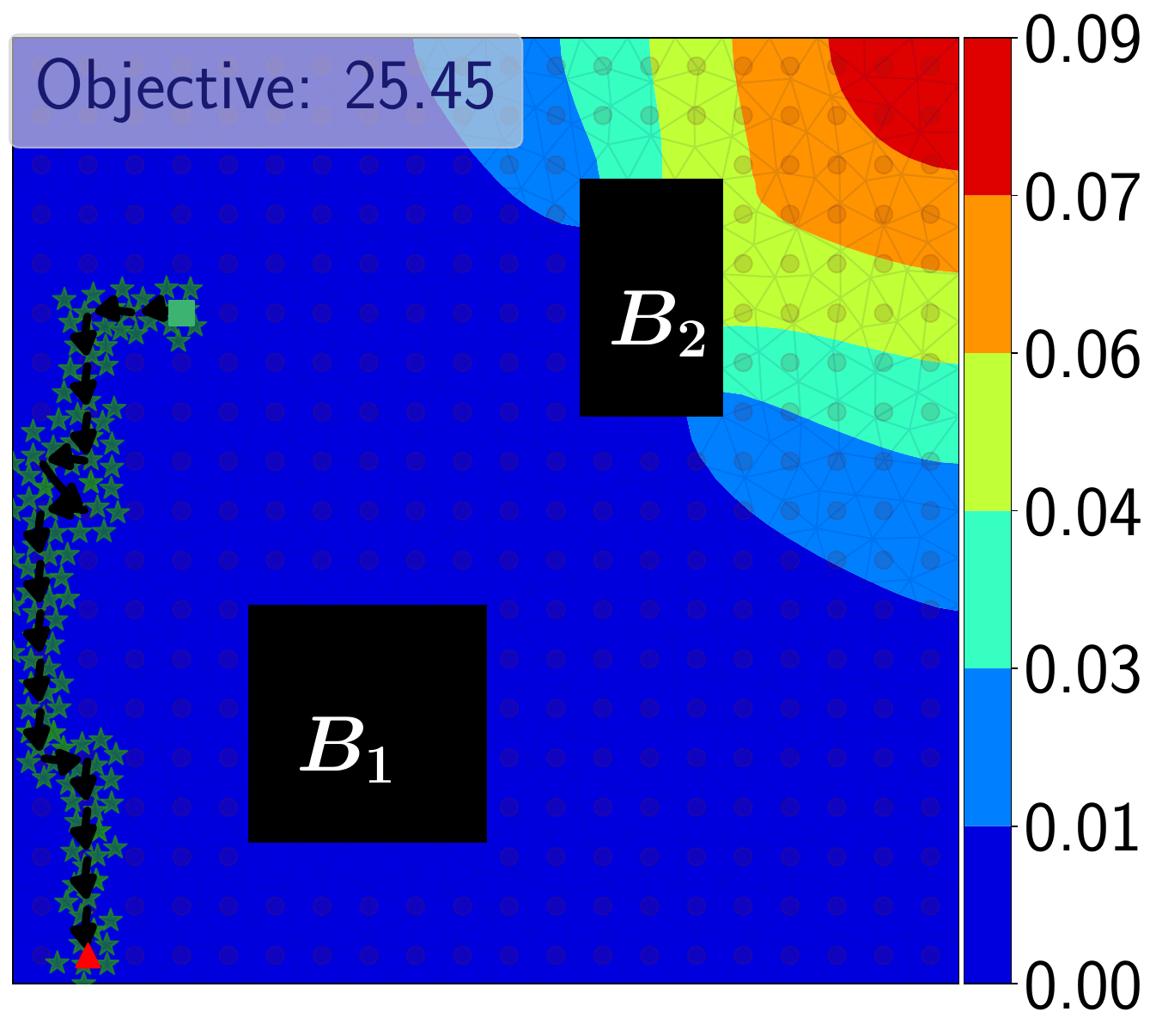}
          \caption{
            Similar to \Cref{sup:fig:a_opt_fixed_start_1_sensor}. 
            Here the number of moving sensors is set to $s=7$.
          }\label{sup:fig:a_opt_fixed_start_7_sensors}
        \end{figure}

      \paragraph{Results with unspecified starting point}
        \Cref{sup:fig:a_opt_unspecified_start_1_sensor} 
        shows the results of \Cref{alg:probabilistic_path_optimization} with the first-order 
        policy defined by \Cref{defn:first_order_path_model} and with number of 
        sensors set to $s=1$.
        \Cref{sup:fig:a_opt_unspecified_start_7_sensors} 
        shows the results obtained by setting the number of moving sensors to $s=7$.
        \begin{figure}[H]
          \centering
          \includegraphics[width=0.53\linewidth]{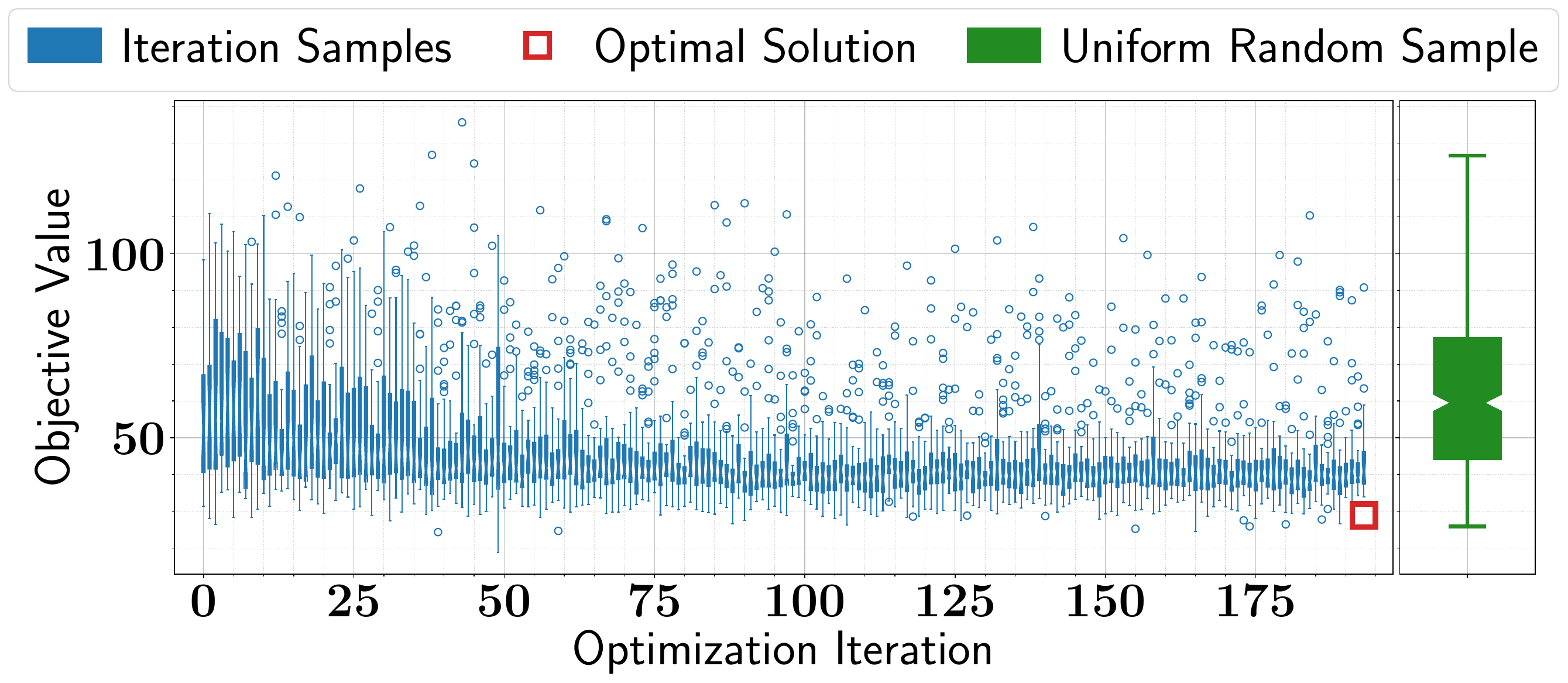}
          \includegraphics[width=0.215\linewidth]{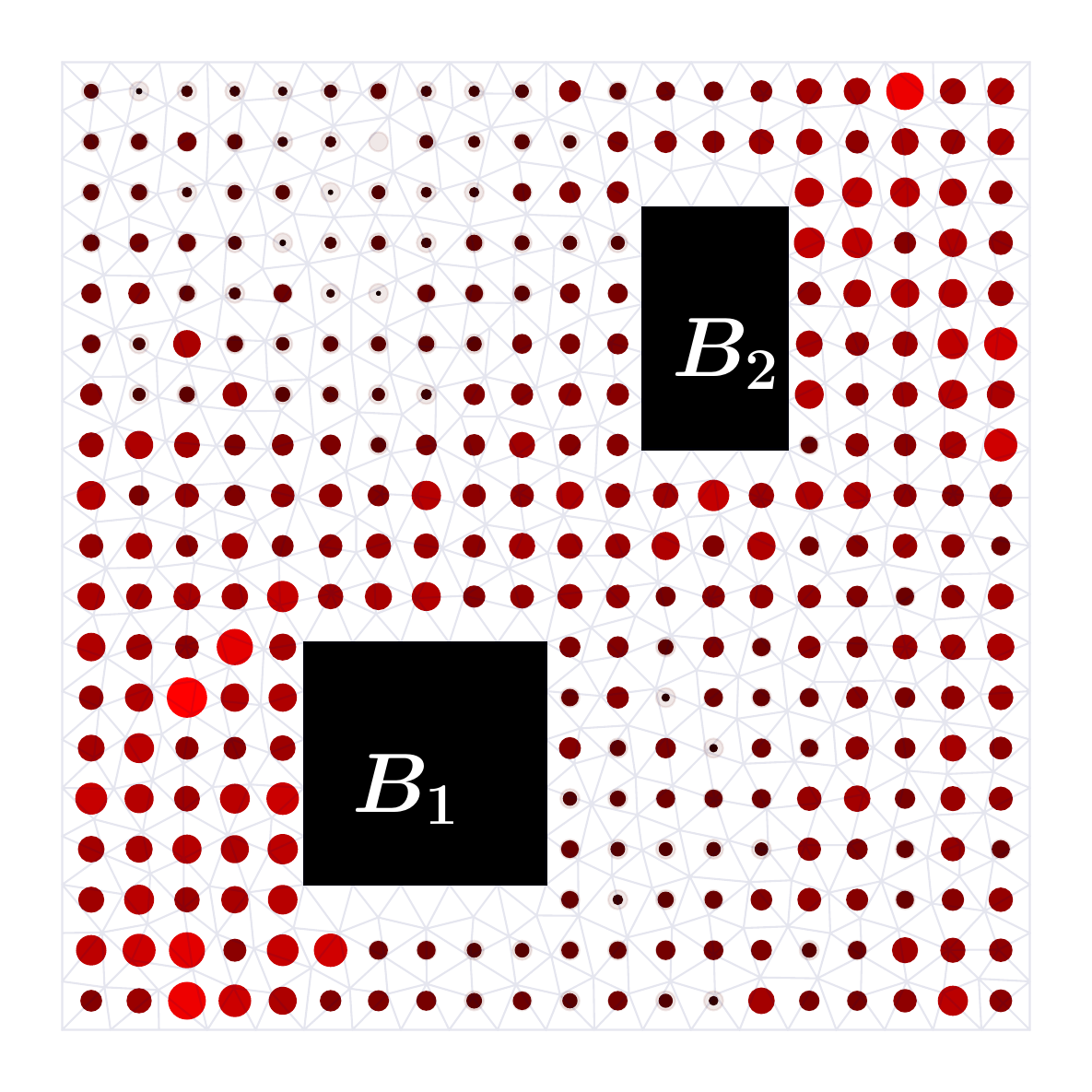}
          \includegraphics[width=0.235\linewidth]{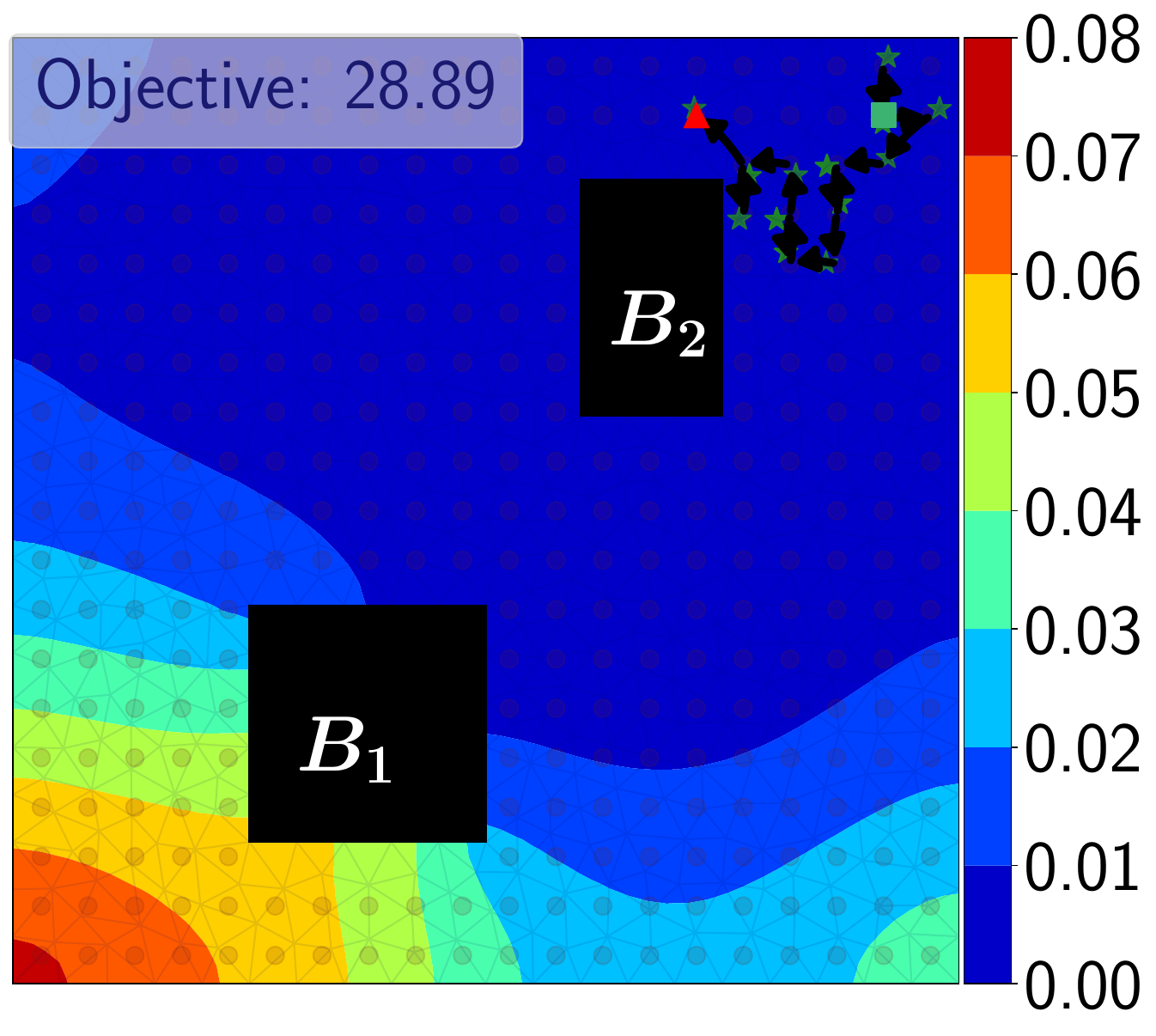}
          \caption{
            Results of \Cref{alg:probabilistic_path_optimization} with the first-order 
            policy model (\Cref{defn:first_order_path_model}) applied to the fine
            navigation mesh (\Cref{fig:navigation_meshes}, right) with unspecified
            starting point of the trajectory.
            The optimality criterion used is the A-optimality, 
            that is, the trace of the posterior covariance matrix.
            The number of sensors is set to $s=1$, the observation frequency is $f=1$, 
            and the trajectory length is $n=19$.
          }\label{sup:fig:a_opt_unspecified_start_1_sensor}
        \end{figure}
        \begin{figure}[H]
          \centering
          \includegraphics[width=0.53\linewidth]{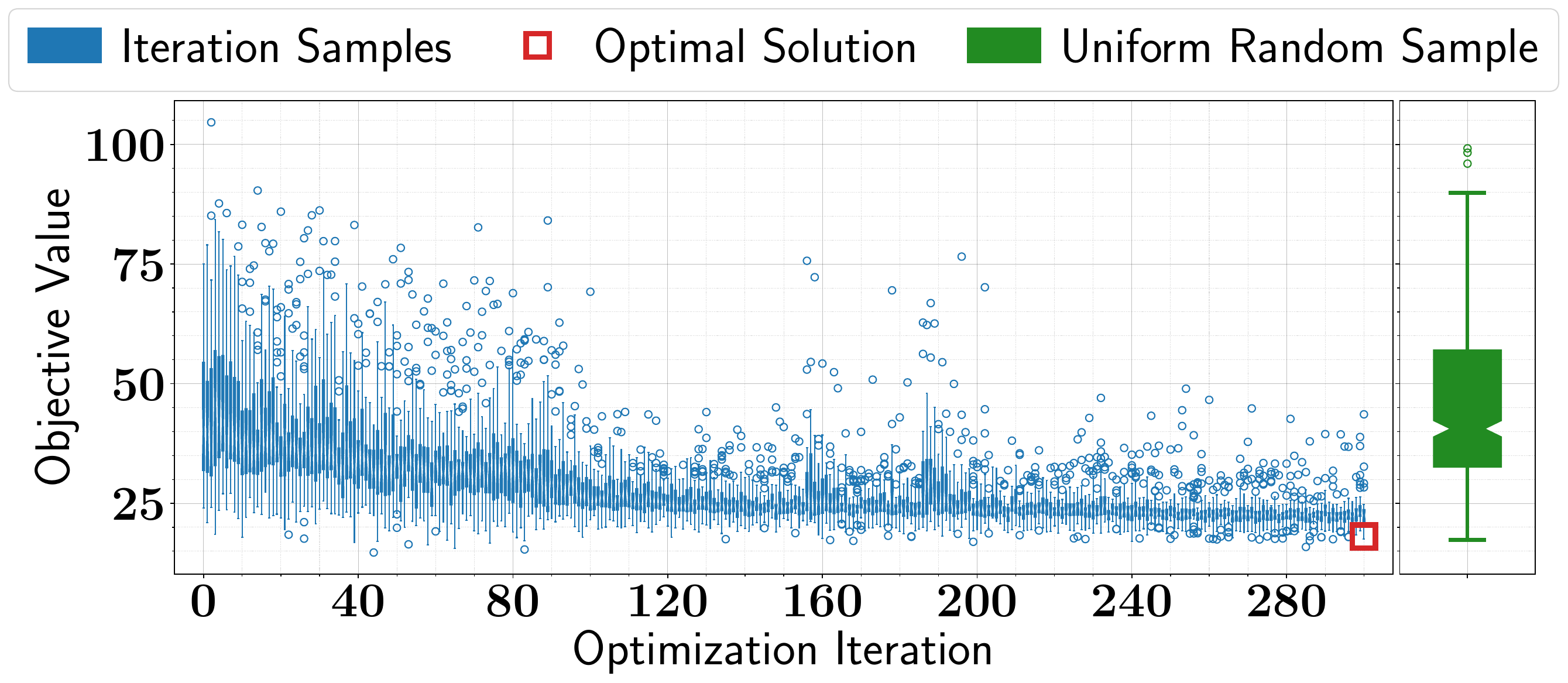}
          \includegraphics[width=0.215\linewidth]{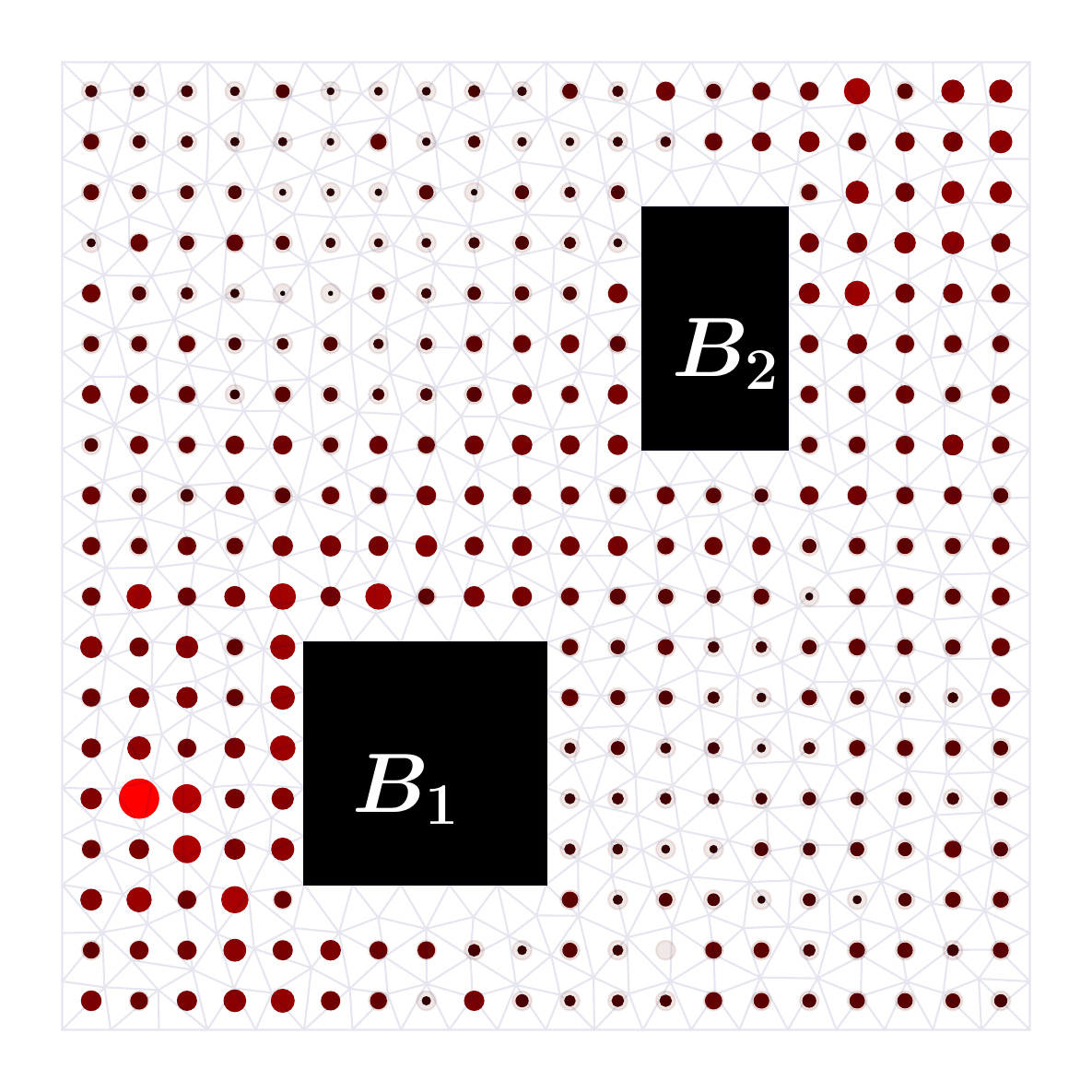}
          \includegraphics[width=0.235\linewidth]{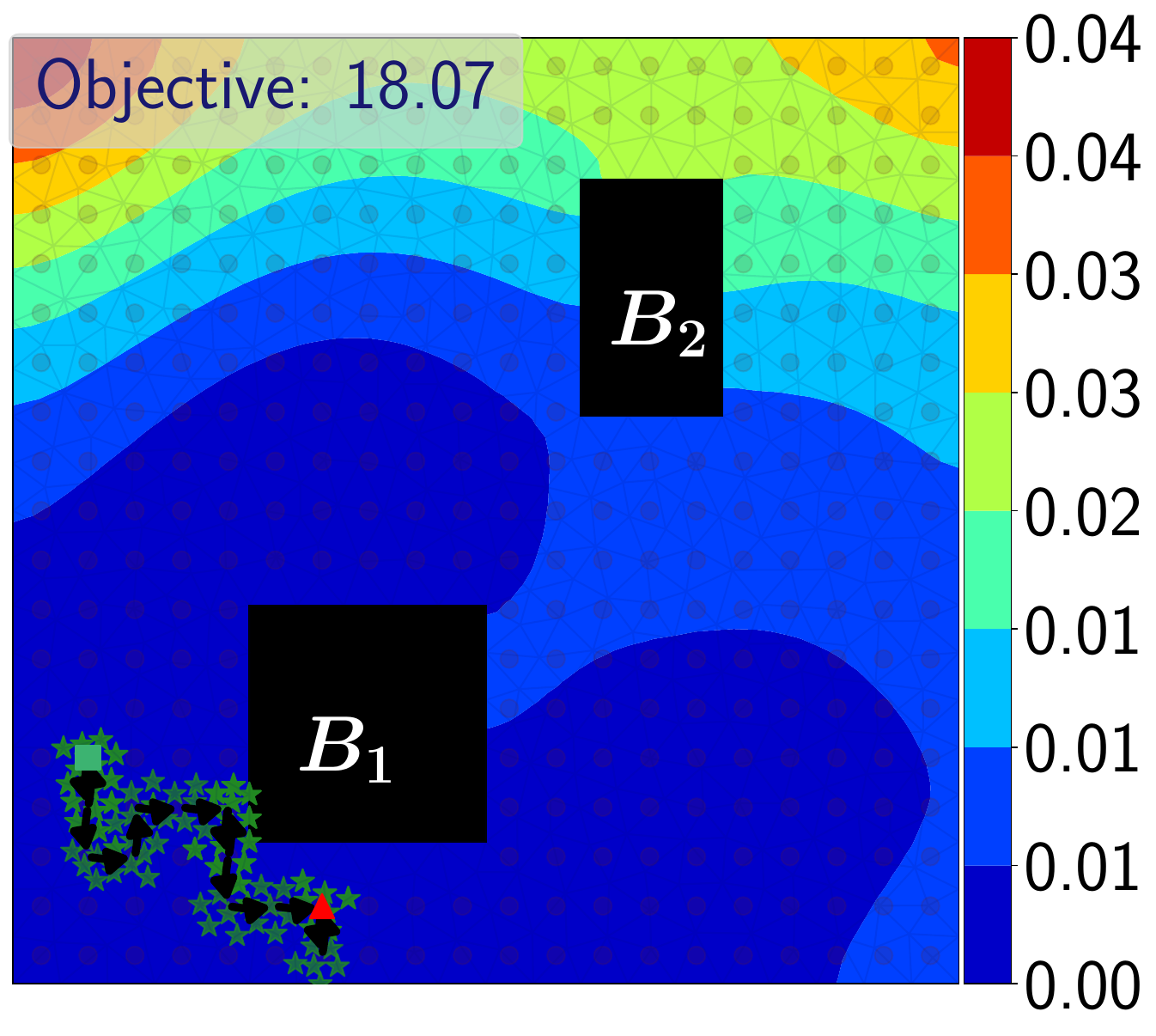}
          \caption{
            Similar to \Cref{sup:fig:a_opt_unspecified_start_1_sensor}. 
            Here the number of moving sensors is set to $s=7$.
          }\label{sup:fig:a_opt_unspecified_start_7_sensors}
        \end{figure}

    \subsubsection{Results with E-Optimality Criterion}
    \label{sup:subsubsec:E-optimality_results}
      For the E-optimality, the objective is to minimize the maximum eigenvalue of 
      posterior covariance matrix, that is, $\lambda_{max}\left(\Cparampost(\designvec)\right)$.
      This function (maximum eigenvalue) is not typically used in classical model-based OED and 
      in general in gradient-based optimization because of challenges in developing 
      its derivative \cite{andrew1993derivatives,rogers1970derivatives}.
      In these approaches the derivative of the utility function $\utilityfunc$ 
      with respect to the design $\designvec$ is needed, and thus the discrete design 
      is relaxed to allow values in the real-valued space.
      The optimal relaxed design is then rounded to retrieve an estimate of the solution 
      of the original discrete optimization problem; see, for example, \cite{attia2022optimal}.

      Since the proposed algorithms treat the objective as a black box, 
      this function can be used in the proposed framework out of the box without 
      the need for such complications.
      The probabilistic OED optimization
      (\Cref{prbl:probabilistic_path_oed_problem}) in this case takes the form
      \begin{equation}\label{eqn:E-opt_probabilistic_optimization}
        \hyperparamvec\opt \in \argmin_{\designvec\sim\CondProb{\designvec}{\hyperparamvec}}
          \Expect{\designvec\sim\CondProb{\designvec}{\hyperparamvec}}{\utilityfunc(\designvec)}
          \,; \qquad \utilityfunc(\designvec) := \lambda_{max}\left({\Cparampost(\designvec)}\right)
            \,,
      \end{equation}
      where $\Cparampost(\designvec)$ is the posterior covariance matrix \eqref{eqn:Posterior_Params}.

      \paragraph{Results with fixed starting point}
        \Cref{sup:fig:e_opt_fixed_start_1_sensor} 
        shows the results of \Cref{alg:probabilistic_path_optimization} with the first-order 
        policy defined by \Cref{defn:first_order_path_model} and with the number of 
        sensors set to $s=1$. 
        Results obtained with $s=7$ moving sensors are shown in 
        \Cref{sup:fig:e_opt_fixed_start_7_sensors} 
        \begin{figure}[H]
          \centering
          \includegraphics[width=0.53\linewidth]{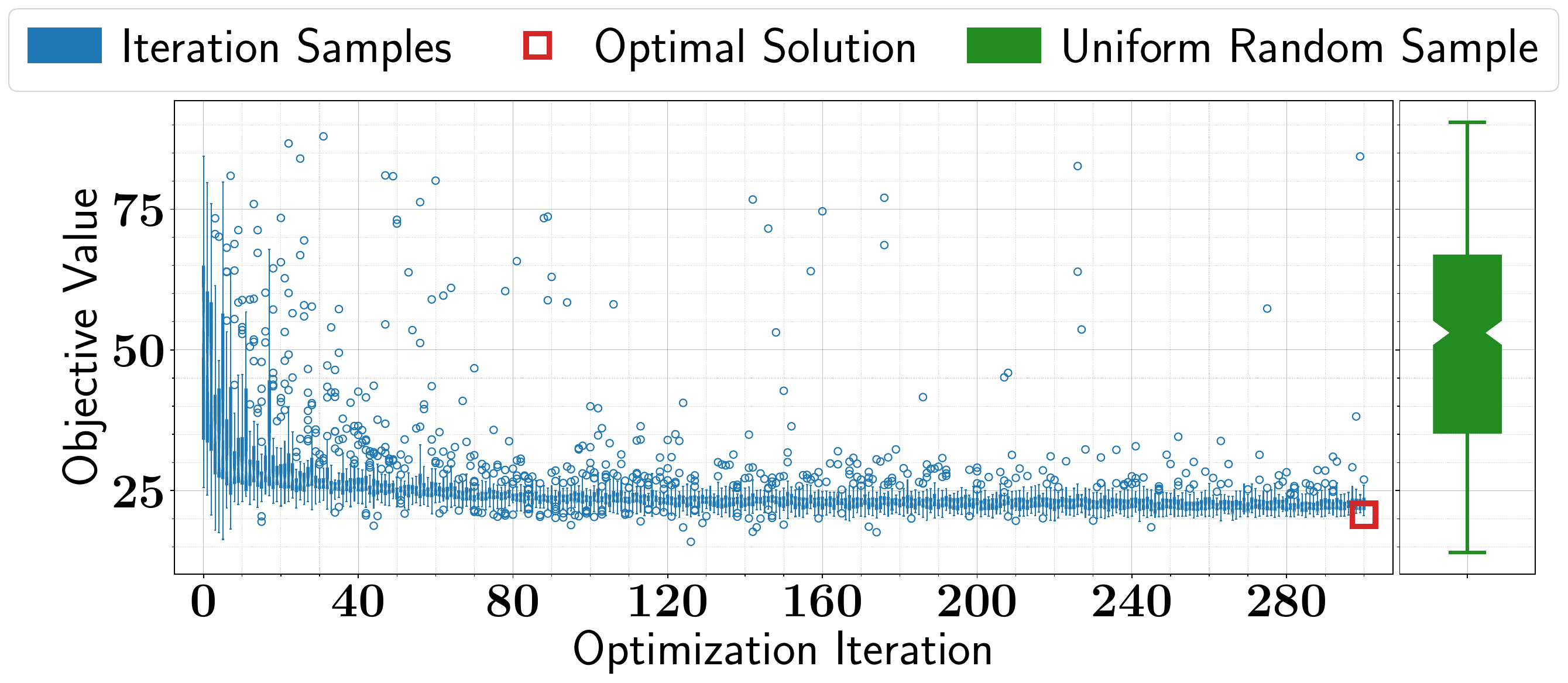}
          \includegraphics[width=0.215\linewidth]{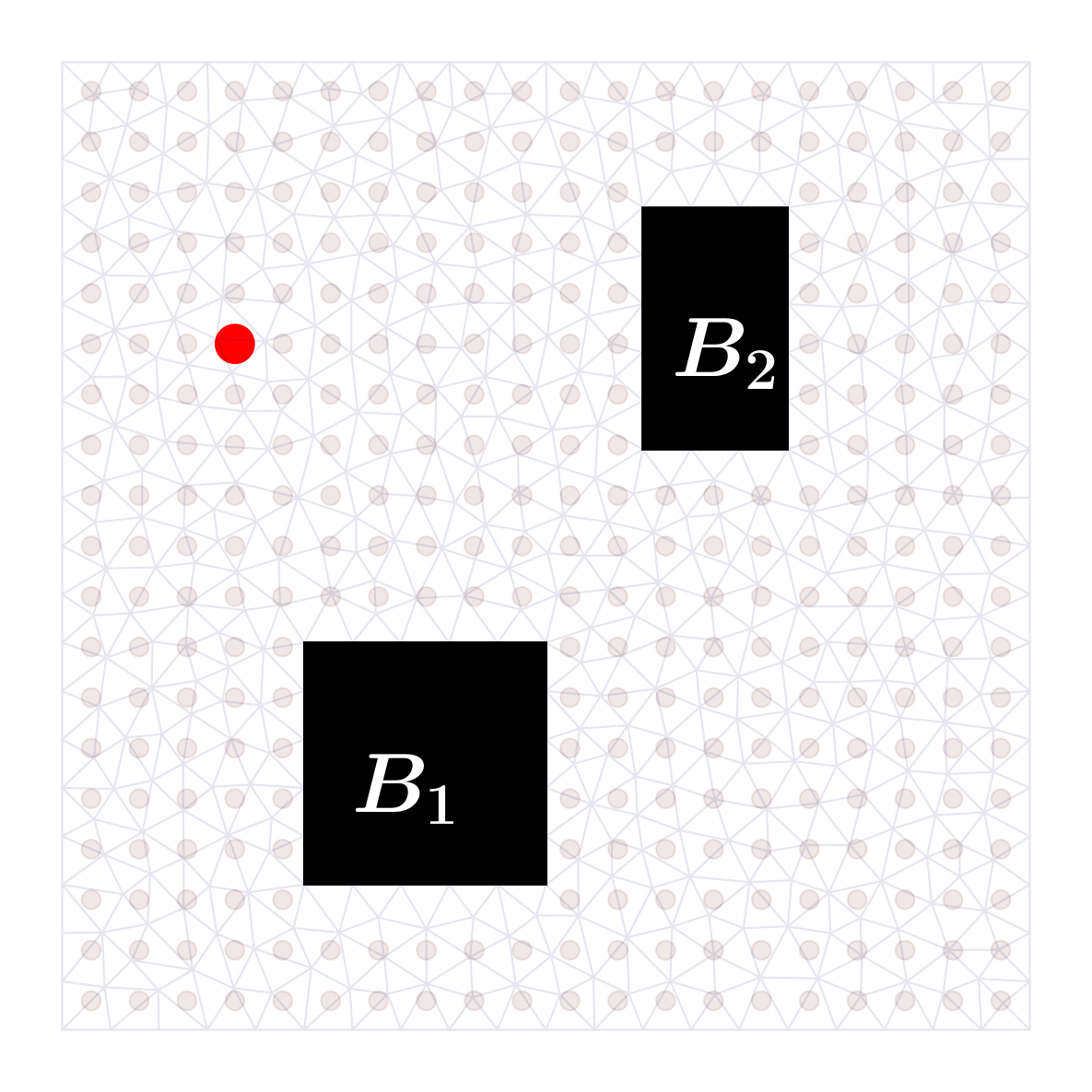}
          \includegraphics[width=0.235\linewidth]{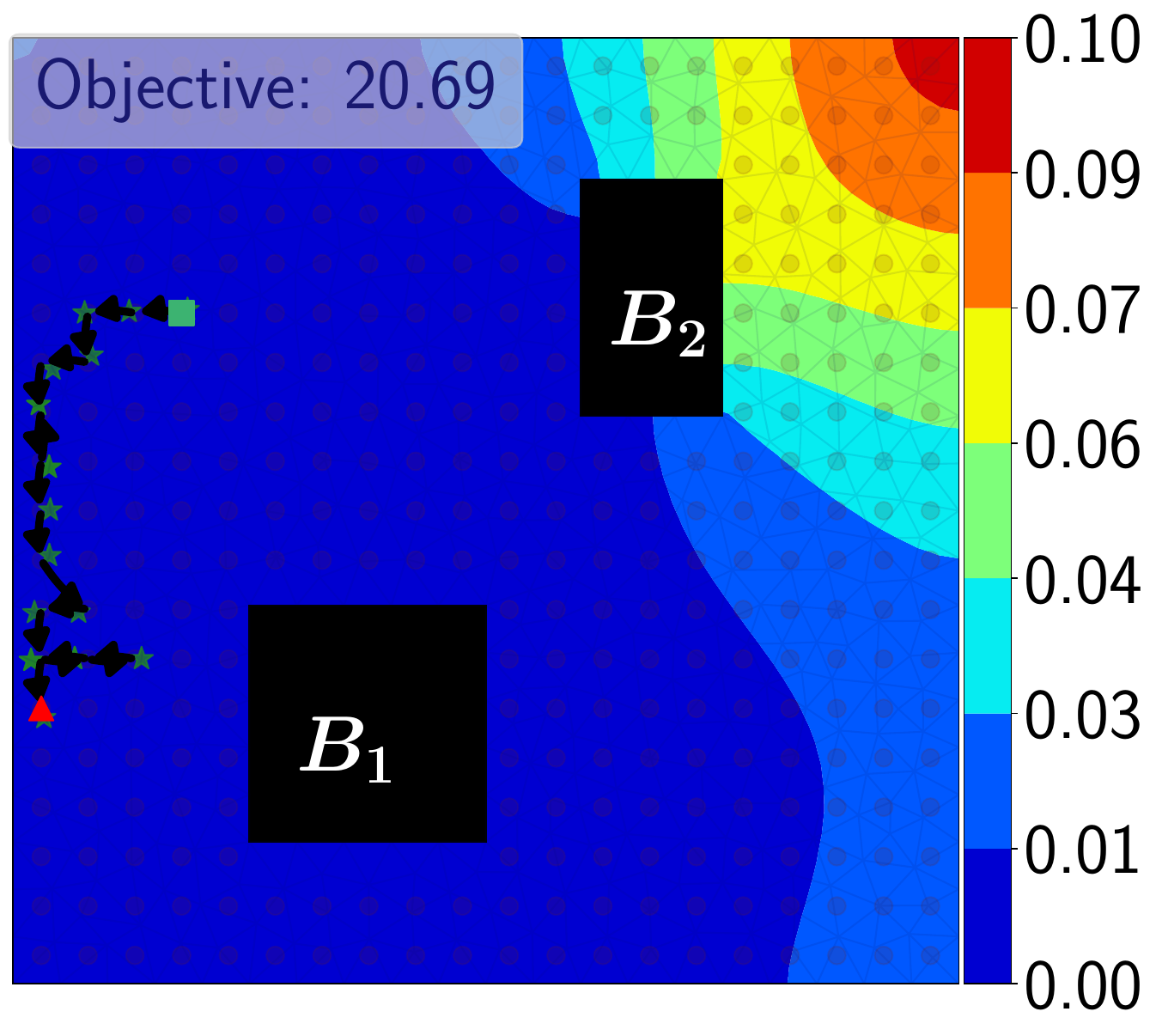}
          \caption{
            Results of \Cref{alg:probabilistic_path_optimization} with the first-order 
            policy model (\Cref{defn:first_order_path_model}) applied to the fine
            navigation mesh (\Cref{fig:navigation_meshes}, right) with 
            the starting point of the trajectory fixed to $(x, y)=(0.2, 0.7)$.
            The optimality criterion used is the E-optimality, 
            that is, the trace of the posterior covariance matrix.
            The number of sensors is set to $s=1$, the observation frequency is $f=1$, 
            and the trajectory length is $n=19$.
          }\label{sup:fig:e_opt_fixed_start_1_sensor}
        \end{figure}
        \begin{figure}[H]
          \centering
          \includegraphics[width=0.53\linewidth]{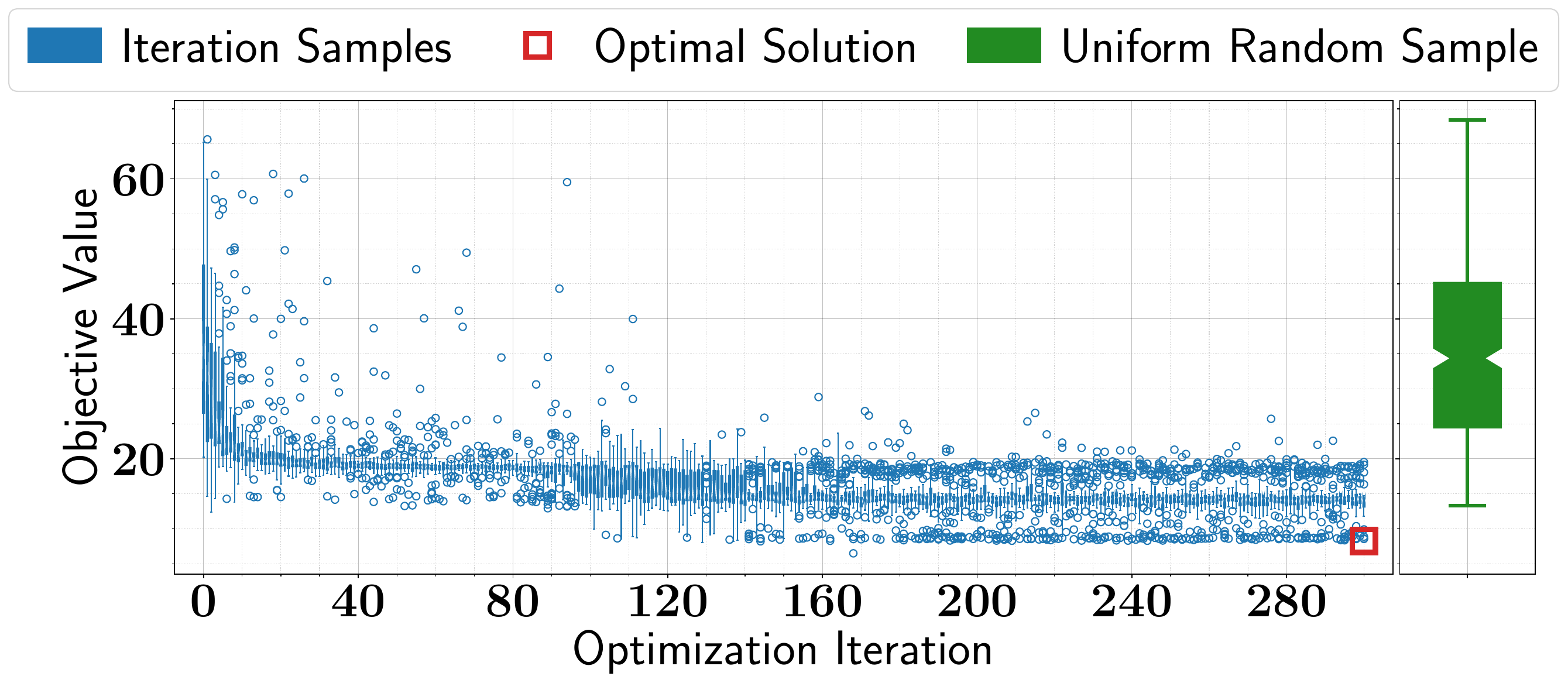}
          \includegraphics[width=0.215\linewidth]{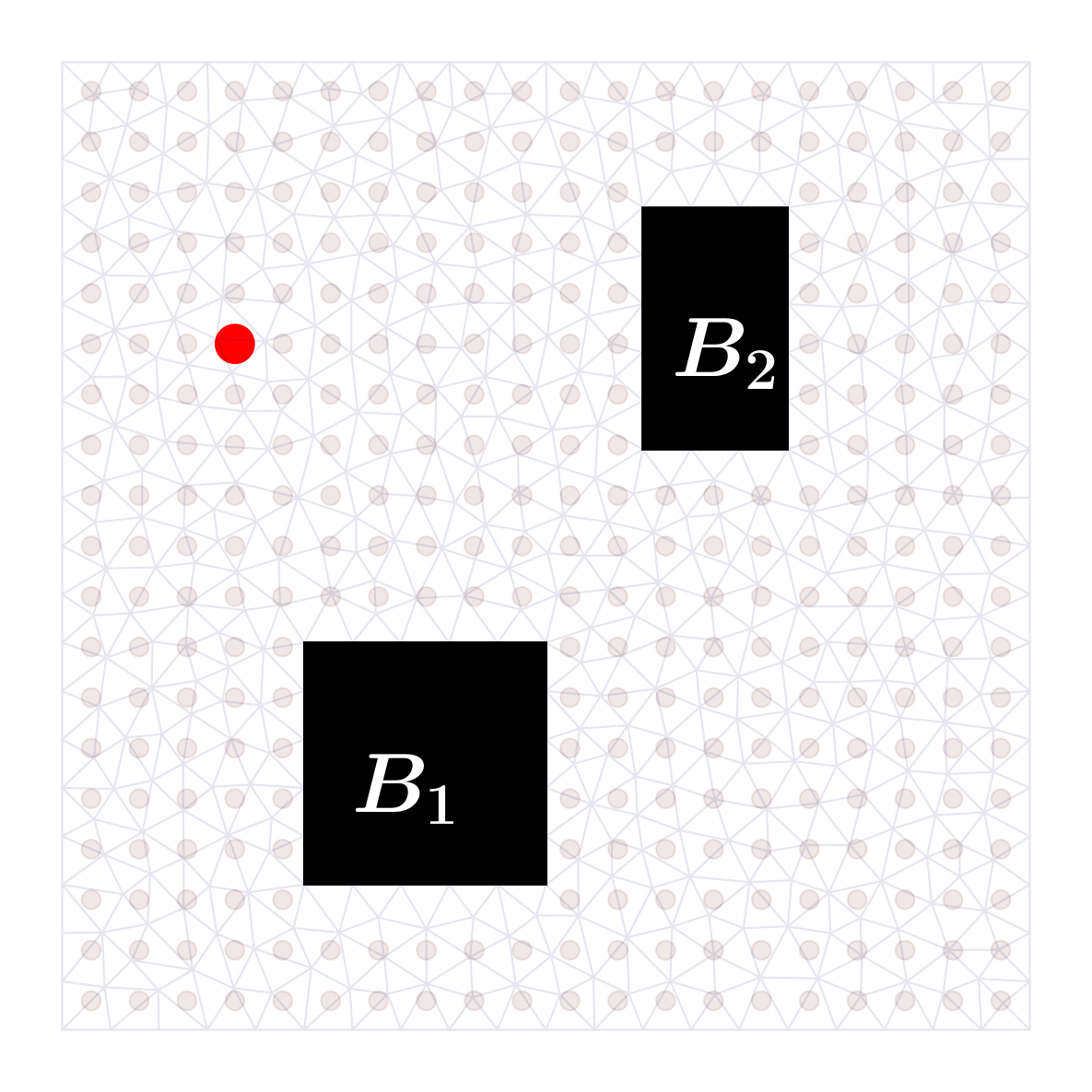}
          \includegraphics[width=0.235\linewidth]{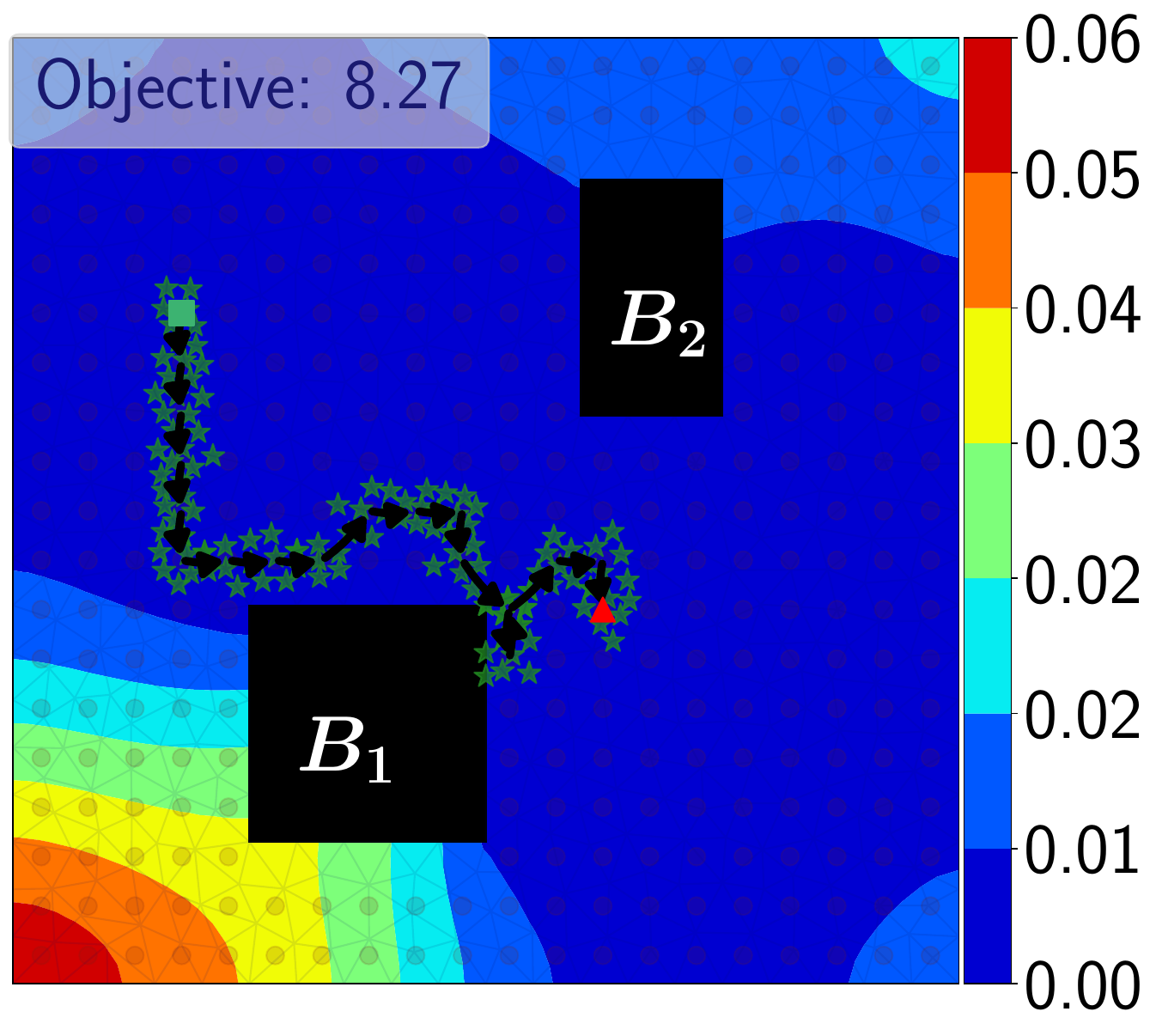}
          \caption{
            Similar to \Cref{sup:fig:e_opt_fixed_start_1_sensor}. 
            Here the number of moving sensors is set to $s=7$.
          }\label{sup:fig:e_opt_fixed_start_7_sensors}
        \end{figure}

      \paragraph{Results with unspecified starting point}
        \Cref{sup:fig:e_opt_unspecified_start_1_sensor} 
        shows the results of \Cref{alg:probabilistic_path_optimization} with the first-order 
        policy defined by \Cref{defn:first_order_path_model} and with the number of 
        sensors set to $s=1$.
        \Cref{sup:fig:e_opt_unspecified_start_7_sensors} 
        shows the results obtained by setting the number of moving sensors to $s=7$.
        \begin{figure}[H]
          \centering
          \includegraphics[width=0.53\linewidth]{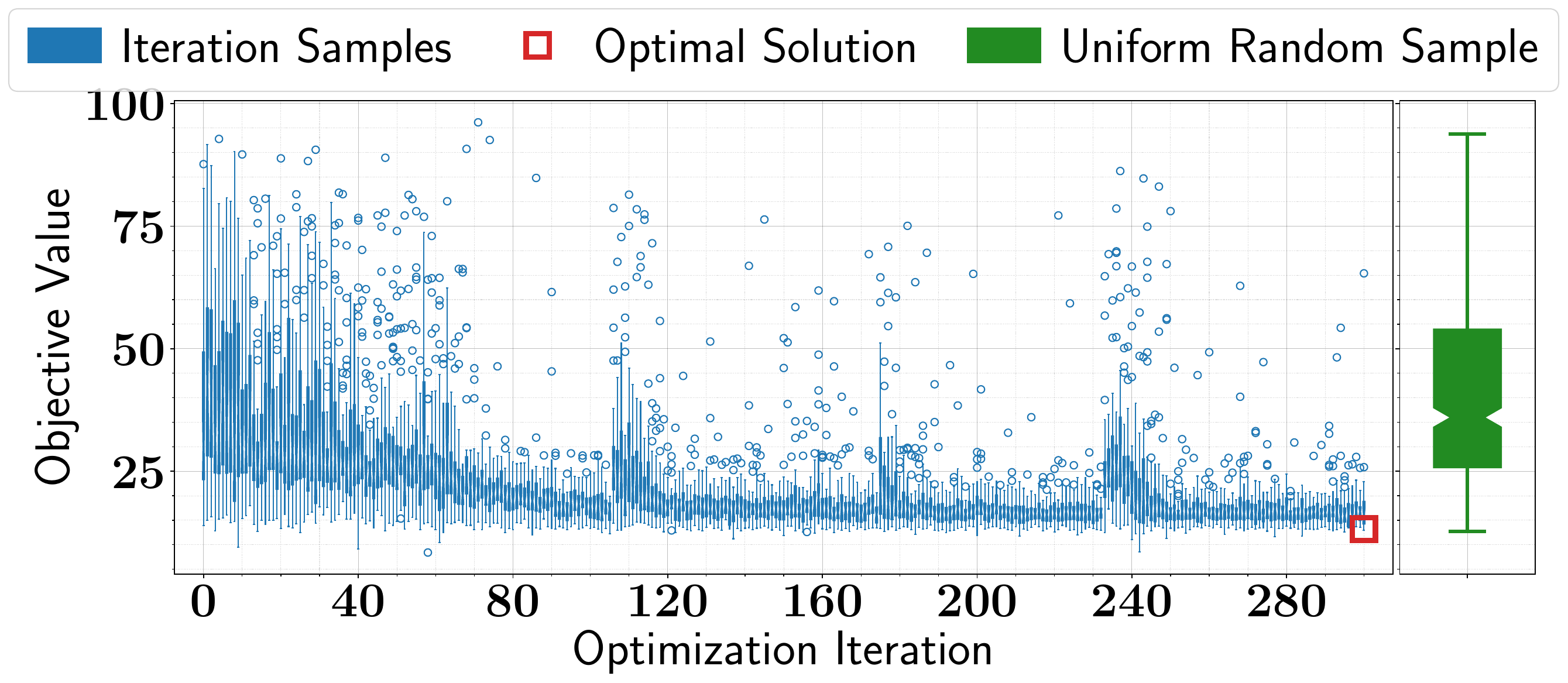}
          \includegraphics[width=0.215\linewidth]{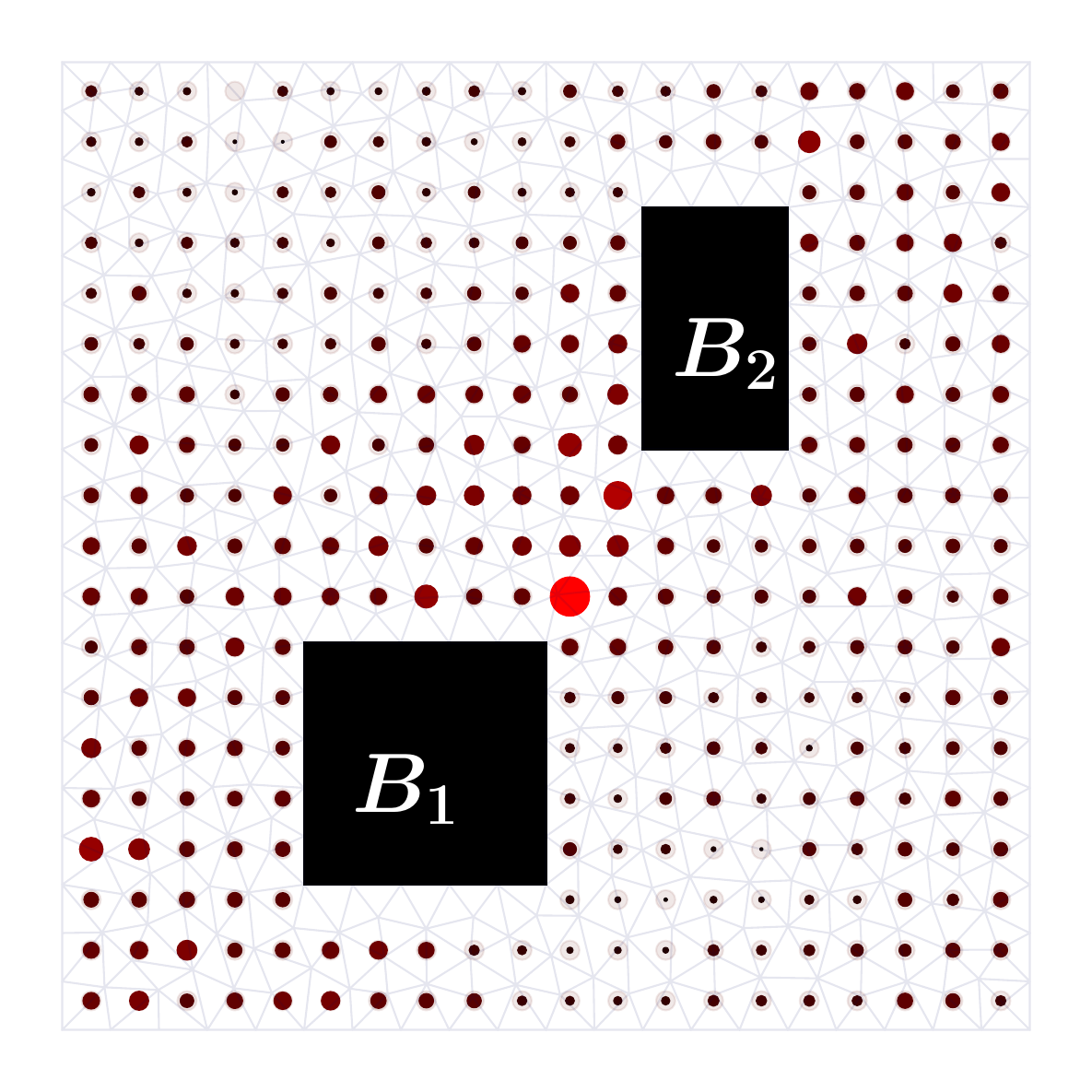}
          \includegraphics[width=0.235\linewidth]{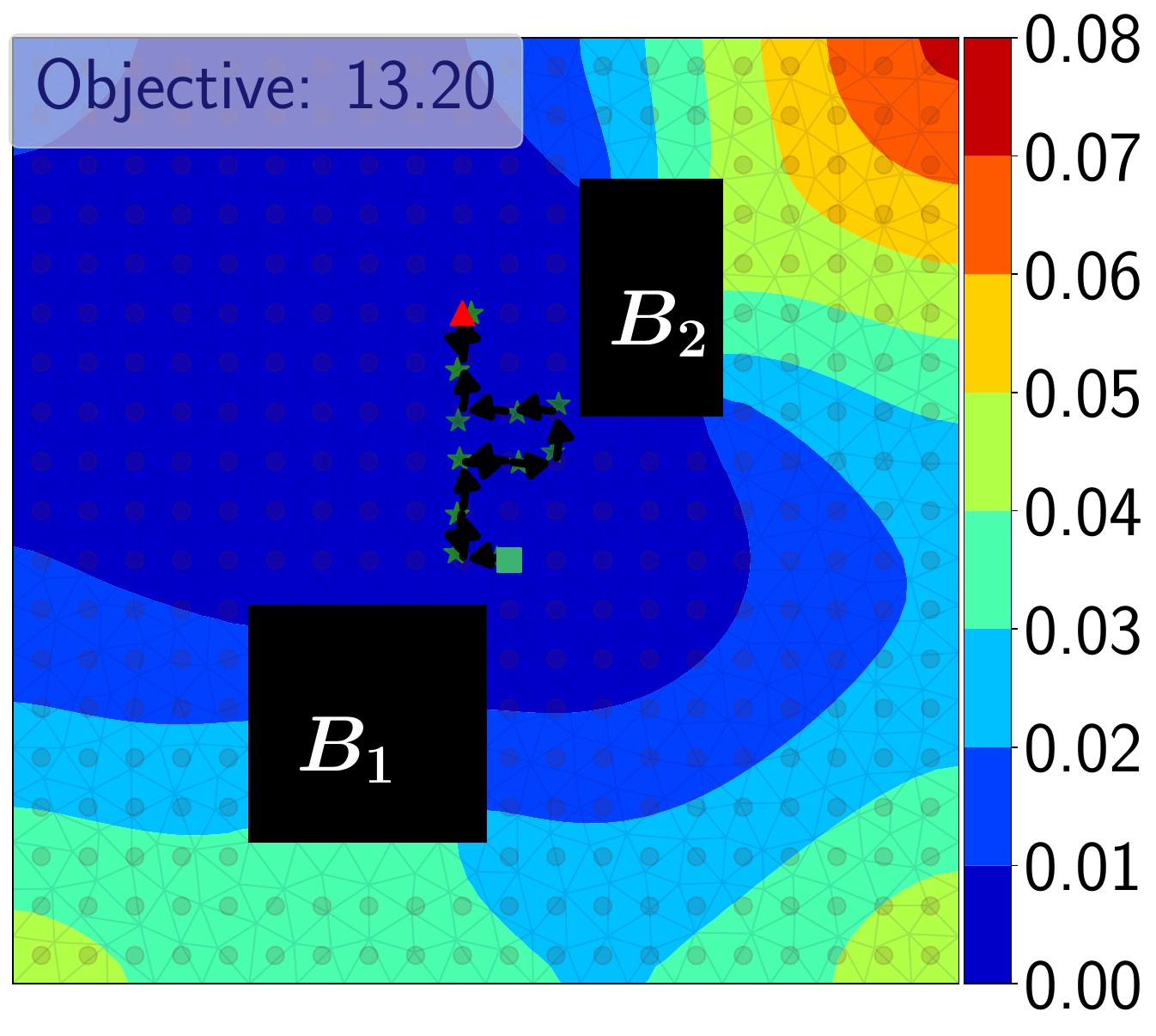}
          \caption{
            Results of \Cref{alg:probabilistic_path_optimization} with the first-order 
            policy model (\Cref{defn:first_order_path_model}) applied to the fine
            navigation mesh (\Cref{fig:navigation_meshes}, right) with unspecified
            starting point of the trajectory.
            The optimality criterion used is the E-optimality, 
            that is, the maximum eigenvalue of the posterior covariance matrix.
            The number of sensors is set to $s=1$, the observation frequency is $f=1$, 
            and the trajectory length is $n=19$.
          }\label{sup:fig:e_opt_unspecified_start_1_sensor}
        \end{figure}
        \begin{figure}[H]
          \centering
          \includegraphics[width=0.53\linewidth]{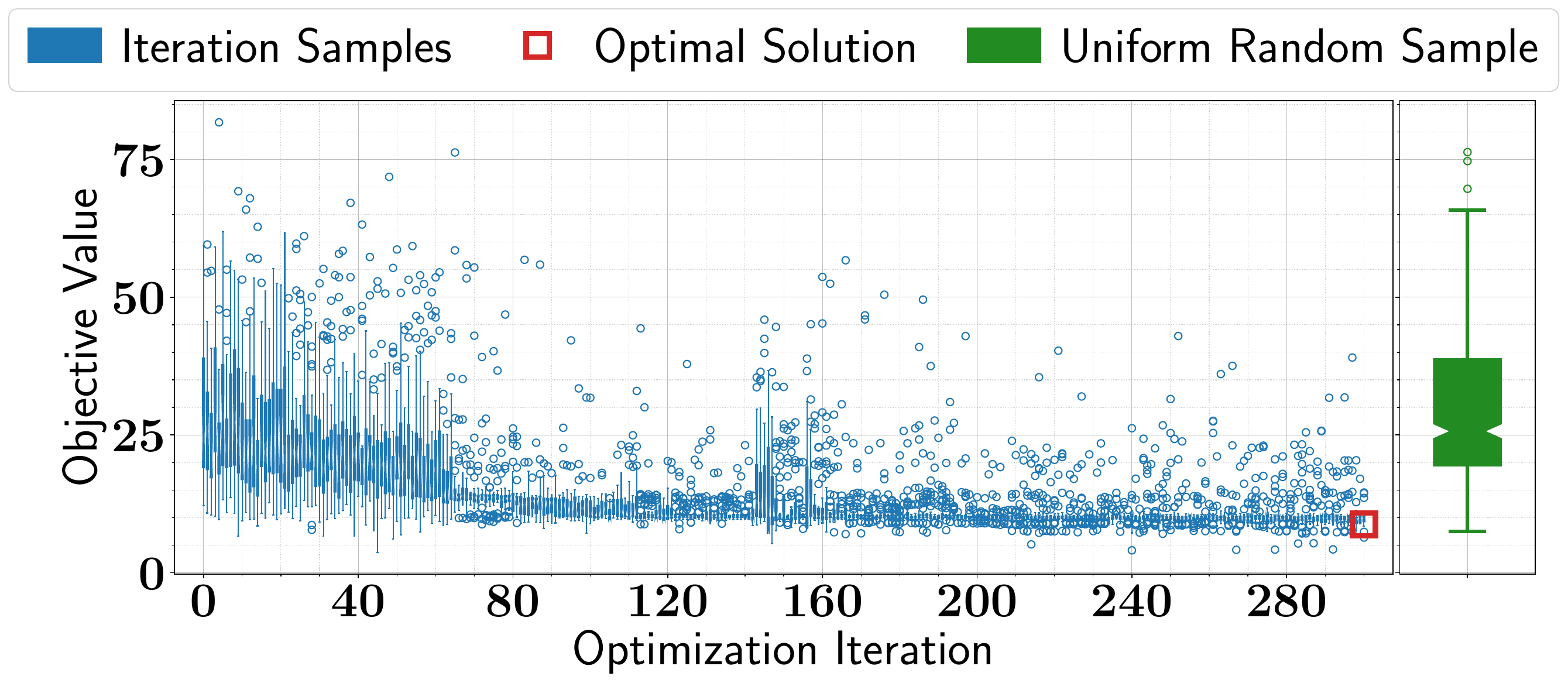}
          \includegraphics[width=0.215\linewidth]{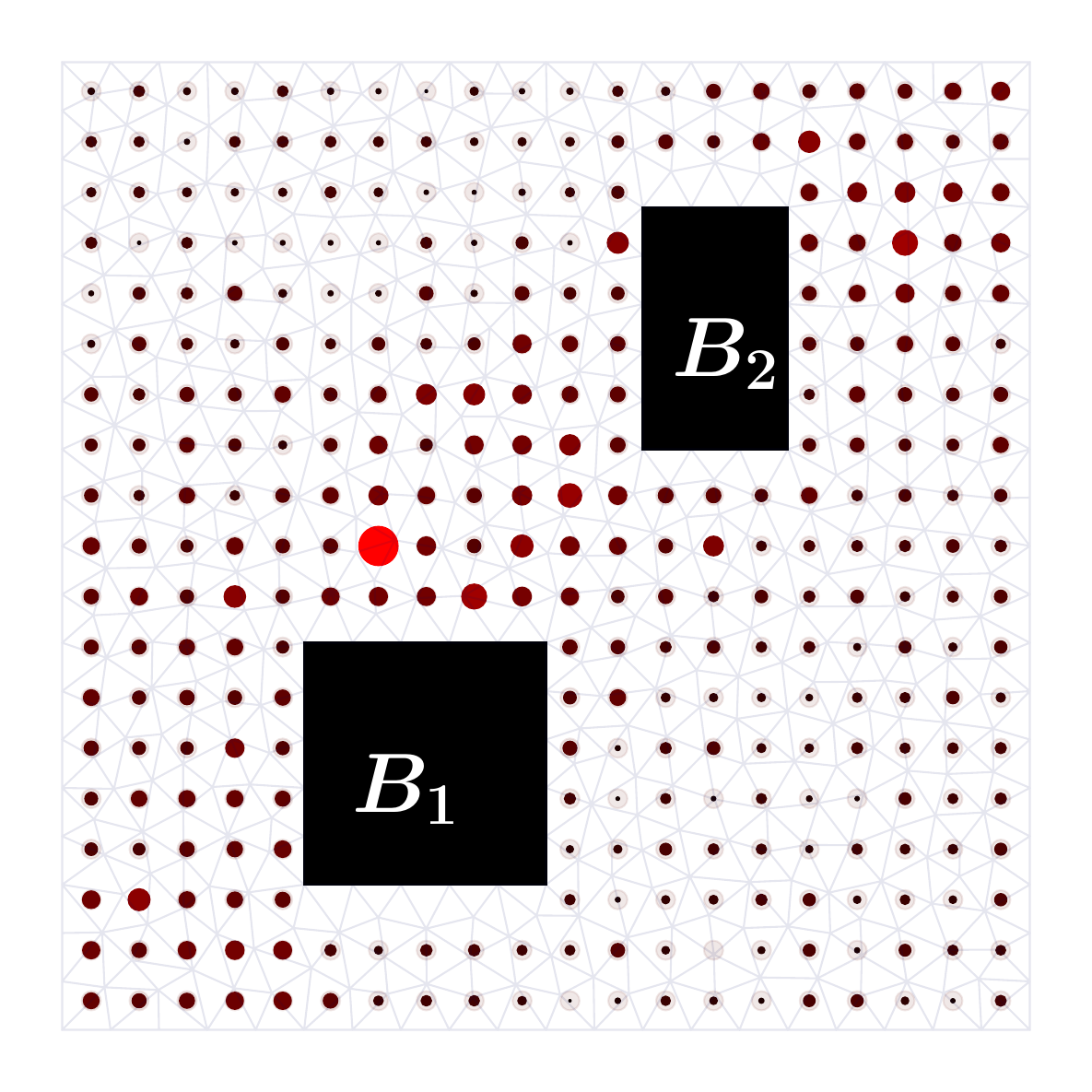}
          \includegraphics[width=0.235\linewidth]{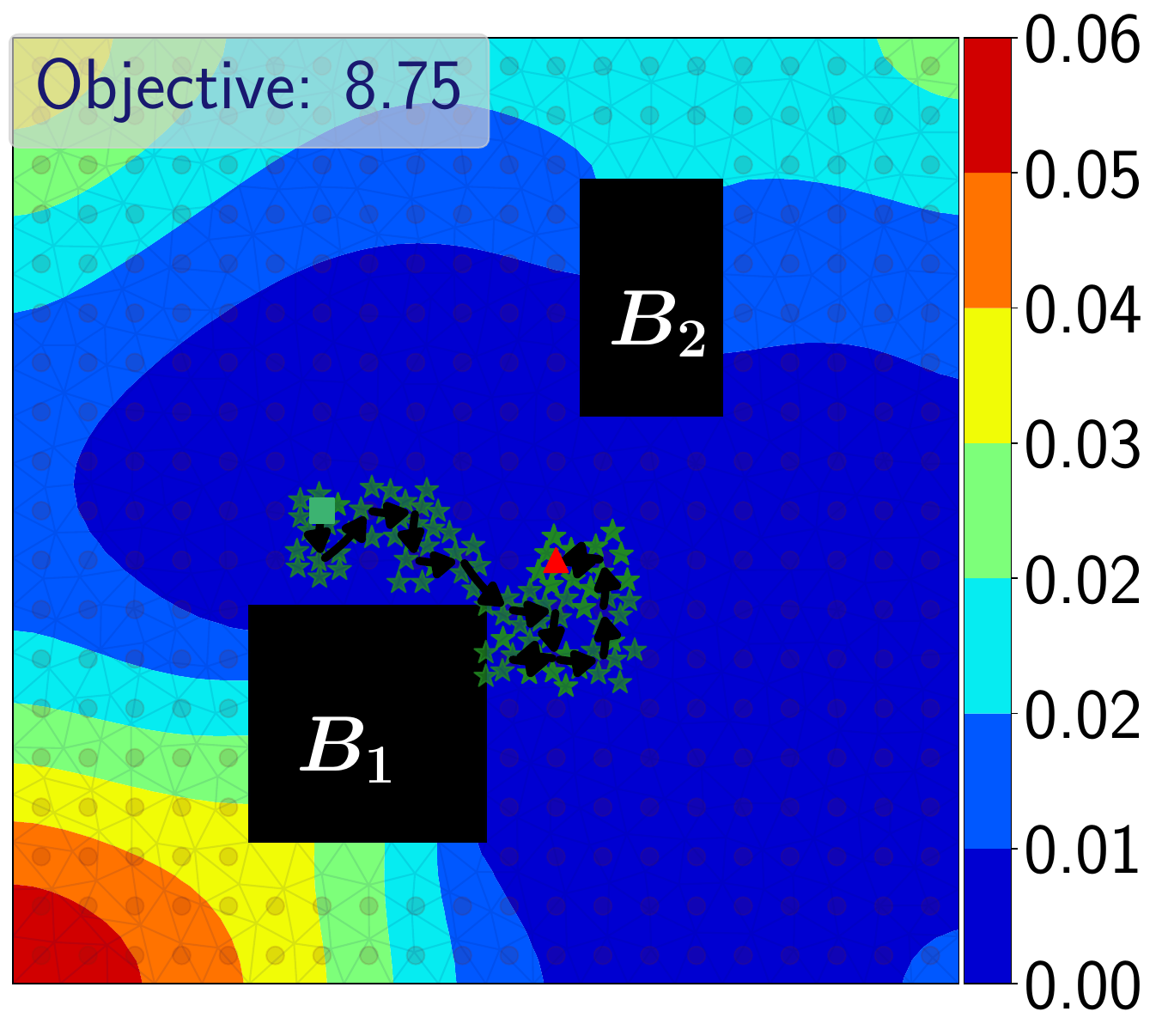}
          \caption{
            Similar to \Cref{sup:fig:e_opt_unspecified_start_1_sensor}. 
            Here the number of moving sensors is set to $s=7$.
          }\label{sup:fig:e_opt_unspecified_start_7_sensors}
        \end{figure}

  The results presented here show that the optimization procedure behaves similarly for various choices
  of the optimality criterion and are thus consistent with the experiment carried out with the D-optimality 
  objective.
  The resulting optimal trajectories, however, are potentially different in shape, exploring relatively 
  different parts of the space.
  Nevertheless, for all choices of the criteria here, the optimal trajectory falls in regions with minimum 
  activity because these locations are associated with minimum uncertainty due to the way the 
  synthetic data and noise covariance matrix is constructed.
  The choice of the utility function, however, is independent from the proposed approach and
  should be the user's choice. 

  For all the utility functions employed here, overall 
  the optimizer quickly reduces the average value of the objective/utility function, converging
  to the tail of the utility function distribution.
  This enables exploring the design space near global optima. Reaching the global  optimal value, however, 
  is not guaranteed.


  \section*{Acknowledgments}
    Thanks to Alen Alexanderian (NCSU) for discussion and feedback, 
    and to Gail Piper (ANL) for editing this manuscript.

  \bibliography{main_reference}{}
  \bibliographystyle{siamplain}

  \null \vfill
    \begin{flushright}
      \scriptsize 
      \framebox{\parbox{5.0in}{
        The submitted manuscript has been created by UChicago Argonne, LLC,
        Operator of Argonne National Laboratory (``Argonne"). Argonne, a
        U.S. Department of Energy Office of Science laboratory, is operated
        under Contract No. DE-AC02-06CH11357. The U.S. Government retains for
        itself, and others acting on its behalf, a paid-up nonexclusive,
        irrevocable worldwide license in said article to reproduce, prepare
        derivative works, distribute copies to the public, and perform
        publicly and display publicly, by or on behalf of the Government.
        The Department of Energy will provide public access to these results 
        of federally sponsored research in accordance with the DOE Public Access Plan. 
        http://energy.gov/downloads/doe-public-access-plan. 
      }}\normalsize
    \end{flushright}

\end{document}